\documentclass[11pt,a4paper]{article}
\usepackage[utf8]{inputenc}
\usepackage{authblk}
\usepackage{xeCJK} 

\usepackage[margin=1in]{geometry}
\usepackage{amsmath, amssymb, amsthm, bm}
\usepackage{booktabs}
\usepackage{enumitem}
\usepackage{hyperref}
\usepackage{xcolor}
\usepackage{comment}

\hypersetup{
	colorlinks=true,
	linkcolor=blue,
	citecolor=teal,
	urlcolor=cyan
}

\newcommand{\coloneqq}{\mathrel{\mathop{:}}=}

\newtheorem{definition}{Definition}
\newtheorem{theorem}{Theorem}
\newtheorem{corollary}{Corollary}
\newtheorem{lemma}{Lemma}

\newtheorem{proposition}{Proposition}

\newtheorem{assumption}{Assumption}

\newtheorem{axiom}{Axiom}

\newtheorem{philosophical}[theorem]{Philosophical View}
\newtheorem{methodology}{Methodology}

\usepackage{etoolbox}

\theoremstyle{remark}
\newtheorem{remark}{Remark}

\pretocmd{\endremark}{\hfill\ensuremath{\blacktriangle}}{}{}

\usepackage{tikz}

\usetikzlibrary{shapes.geometric, arrows.meta, positioning, calc, decorations.pathmorphing}
\usetikzlibrary{3d, calc, arrows.meta}

\title{\textbf{The Second Edge Theorem: The Asymptotic Collapse of Sample-Dependent Information Geometry to the Canonical Flat Canvas of Conventional Statistics in Large Sample Limits}}

\author[1,2,8]{Bing Cheng}
\author[3]{Yi-Shuai Niu} 
\author[5,6,7]{Howell Tong}
\author[3,4]{Shing-Tung Yau}

\affil[1]{Academy of Mathematics and Systems Science, Chinese Academy of Sciences, Beijing, China;
	bc2@amss.ac.cn}
\affil[2]{AMSS Center for Forecasting Science, Chinese Academy of Sciences, Beijing, China}

\affil[3]{Beijing Institute of Mathematical Sciences and Applications (BIMSA), Beijing, China}
\affil[4]{Yau Mathematical Sciences Center, Tsinghua University, Beijing, China}

\affil[5]{Department of Statistics and Data Science, Tsinghua University, Beijing 100084, China}
\affil[6]{Paula and Gregory Chow Institute for the Studies in Economics, Xiamen University, Xiamen 361005, China}
\affil[7]{Department of Statistics, London School of Economics and Political Science, London WC2A 2AE, UK}

\affil[8] {State Key Laboratory of Mathematical Science, Academy of Mathematics and Systems Science, Chinese Academy of Sciences}

\date{\today}

\begin{document}
	
	\maketitle
	
\begin{abstract}
	This paper establishes the complete global proof of the \textbf{Second Edge Theorem}, proving that as the sample size grows to infinity, the sequence of curved, sample-dependent joint information-geometric manifolds---formulated by the parameter space, sample-scaled Fisher metric, and dual alpha-connections---undergoes an asymptotic phase transition, metric-topological shrinkage, and global geometric collapse onto the flat canonical tangent canvas of Conventional Statistics defined by the tangent space at the true parameter, the base Fisher metric, and the zero-curvature Levi-Civita connection.
	
	First, we discover and prove the universal tensor valence scaling law, which dictates that tensor fields across valences one through four degenerate according to the sample size raised to the power of one minus half the valence. Under this unified scaling law, score one-form fluctuations stabilize to a normalized Gaussian field, the sample-scaled Fisher metric freezes uniformly to the constant single-observation Fisher metric at the true parameter, Amari affine connections dissolve at an inverse square-root sample rate to zero Christoffel symbols, and intrinsic Riemann curvature undergoes accelerated annihilation down to identical zero at an inverse sample rate.
	
	Second, we resolve the long-standing geometric-operational duality by bridging differential-geometric manifold collapse with statistical decision theory. Through the formulation of the Taylor-Geometry Algebraic Functional Identity and the Double Completeness Architecture, we prove that Cheeger-Gromov spatial manifold flattening and Le Cam decision risk condensation are mathematically non-separable dual projections of the exact same underlying asymptotic phase transition.
	
	Third, by constructing a Fisher-Compatible Ehresmann Connection over the statistical fiber bundle, we extend the theory to over-parameterized and singular models without requiring ambient metric non-degeneracy. On the non-singular horizontal distribution, we demonstrate Gromov-Hausdorff leaf space collapse, accelerated annihilation of the restricted holonomy group diameter, and global uniformization of local asymptotic normality under the Le Cam deficiency distance.
	
	Fourth, by unifying the First Edge Theorem and the Second Edge Theorem, we establish the nested dual-edge hierarchy of statistical science, demonstrating that Conventional Statistics forms the boundary of Information Geometry, which itself forms the boundary of Statistical Mechanics and Geometry. This proves that Conventional Statistics is not a heuristic approximation, but the unique, zero-curvature thermodynamic attractor state toward which all regular parametric information manifolds dynamically collapse in the infinite-sample limit.
	
	Finally, by extending our horizon to Statistically  Meaningful Geometry (SMG), modern statistics transcends its classical role as a static theory of flat asymptotic inference. When formulated within the full SMG framework, modern statistics evolves into a dynamic, non-equilibrium field theory capable of explicitly modeling finite-sample fluctuations, phase transitions, topological structures, over-parametrization models, and gauge-invariant interactions—redefining statistical inference as the active, thermodynamic geometry of complex systems rather than a passive infinite-sample limit.

\end{abstract}	
\newpage
	\tableofcontents
	
	\newpage
	\section{Epistemological Foundations, Regularity Axioms, and Target Universes}
	\label{sec:module1}
	
	\subsection{Motivations, Backgrounds, and Architectural Targets}
	\label{subsec:mod1_motivation}
	
	Classical large-sample statistical theory implicitly operates under the foundational premise that as sample size $N \to \infty$, estimation, inference, and decision testing occur within a linear, flat Euclidean workspace. However, finite-sample statistical parametric families are intrinsically curved Riemannian-affine manifolds \cite{amari2000, rao1945}. Historical literature frequently treats the transition from curved finite-sample models to flat asymptotic limit experiments as an ad-hoc analytic convenience or a passive collection of scalar point limits \cite{fisher1922, vandervaart1998}. 
	
	To eliminate circular reasoning, topological ambiguities, and informal coordinate heuristics, this section constructs a rigorous axiomatic baseline. We define the ambient statistical universe, establish the natural linear vector space derivation isomorphism between differential tangent spaces and arithmetic coordinates, formalize the Triadic Variable Spectrum, and establish the Epistemological Equivalence Theorem proving that General Conventional Statistics ($\text{CS}$) is mathematically equivalent to the canonical flat tangent space $\mathcal{M}_\infty$.
	
	\subsection{Taxonomy of Concepts, Objects, and Mathematical Relationships}
	\label{subsec:mod1_taxonomy}
	
	To ensure complete logical transparency, Table~\ref{tab:mod1_summary} details the taxonomy of mathematical concepts, geometric objects, operational relationships, and foundational theorems established in this section.
	
	\begin{table}[h!]
		\centering
		\small
		\begin{tabular}{|p{3.2cm}|p{11.8cm}|}
			\hline
			\textbf{Framework Element} & \textbf{Detailed Specification in Section 1} \\ \hline
			\textbf{Core Concepts} & Ambient statistical universe, High-Order Local Regularity, Natural Tangent Derivation Isomorphism, Triadic Variable Spectrum, General Conventional Statistics ($\text{CS}$). \\ \hline
			\textbf{Mathematical Objects} & Universe $\mathcal{M}_{\text{univ}}$, manifolds $IG_1 = (\Theta, g^{(1)}, \nabla^{(\alpha)})$, $IG_N = (\Theta, G^{(N)}, \nabla^{(\alpha,N)})$, tangent derivation space $T_{\theta_0}IG_1$, frozen Fisher metric $g_0 \equiv g^{(1)}(\theta_0)$, flat tangent space $\mathcal{M}_\infty \equiv (T_{\theta_0}IG_1, g_0, \nabla^{(0)})$, log-likelihood field $\Lambda_N(h)$. \\ \hline
			\textbf{Core Relationships} & Tangent derivation map $T_{\theta_0}IG_1 \cong \mathbb{R}^d$ via derivation chart map $\Psi_\phi$; metric capacity scaling $G^{(N)}(\theta) = N g^{(1)}(\theta)$; Triadic Variable Spectrum $\mathcal{T}_N = (\mathcal{E}_N, \mathcal{S}_N, \mathcal{F}_N)$ coupling measure concentration $\beta_N = 1/\sqrt{N}$ with metric dilation $\alpha_N = \sqrt{N}$. \\ \hline
			\textbf{Prior Literature Results} & Fisher \cite{fisher1922, fisher1925} score derivations; Rao \cite{rao1945} Riemannian Fisher metric; Le Cam \cite{lecam1960, lecam1986} LAN shift experiments; Amari \cite{amari2000} dual $\alpha$-connections; van der Vaart \cite{vandervaart1998} asymptotic regularity. \\ \hline
			\textbf{New Theoretical Results} & Definition~\ref{def:ambient_spaces} (Ambient Statistical Universe); Assumption~\ref{asm:local_regularity} (High-Order Regularity); Proposition~\ref{prop:isomorphism} (Linear Vector Space Derivation Isomorphism); Definition~\ref{def:triadic_spectrum} (Triadic Variable Spectrum); Theorem~\ref{thm:theorem2_full} (Epistemological Equivalence Theorem $\text{CS} \iff \mathcal{M}_\infty$); Theorem~\ref{thm:lan_linear_quadratic_proof} (Rigorous Linear-Quadratic LAN Decomposition). 
			\\ \hline
		\end{tabular}
		\caption{Taxonomy of Concepts, Objects, Relationships, and Mathematical Results in Section~\ref{sec:module1}.}
		\label{tab:mod1_summary}
	\end{table}
	
	\subsection{Foundational Regularity Axioms and Ambient Statistical Spaces}
	\label{subsec:mod1_axioms}
	
	We begin by formalizing the overarching ambient mathematical universe containing all single-sample and product statistical spaces.
	
	\begin{definition}[Ambient Statistical Universe]
		\label{def:ambient_spaces}
		The ambient geometric and statistical universe is formalized as the structured tuple:
		\begin{equation}
			\mathcal{M}_{\text{univ}} \equiv \left( \mathcal{X}, \mu, \Theta, \mathcal{P}, IG_1, IG_N, (\mathcal{X}^N, \mathcal{A}^{\otimes N}, P_{\theta_0}^{\otimes N}) \right)
			\label{eq:universe_tuple}
		\end{equation}
		where $\mathcal{X}$ is a sample space endowed with $\sigma$-finite measure $\mu$, $\Theta \subset \mathbb{R}^d$ is an open parameter domain of dimension $d$, and $\mathcal{P} \equiv \left\{ P_\theta : dP_\theta = p(x;\theta)d\mu(x), \; \theta \in \Theta \right\}$ is a dominated parametric family. The single-sample information manifold $IG_1$ and $N$-sample joint product space sequence $IG_N$ are defined respectively by:
		\begin{equation}
			IG_1 \equiv (\Theta, g^{(1)}, \nabla^{(\alpha)}), \quad IG_N \equiv (\Theta, G^{(N)}, \nabla^{(\alpha, N)})
			\label{eq:ig1_ign_def}
		\end{equation}
		where $g^{(1)}$ is the single-sample Fisher Information Metric (FIM) tensor, $G^{(N)}(\theta) = N \cdot g^{(1)}(\theta)$ is the $N$-sample joint Fisher metric tensor, $\nabla^{(\alpha)}$ is Amari's dual connection on $IG_1$, and $\nabla^{(\alpha,N)}$ is the joint affine connection on $IG_N$.
	\end{definition}

\begin{lemma}[Additivity of the Fisher Information Metric under i.i.d. Product Measures]
	\label{lem:additivity_fisher_metric}
~

	Let $\mathcal{M}_{\mathrm{univ}} \equiv \left( \mathcal{X}, \mu, \Theta, \mathcal{P}, IG_1, IG_N, (\mathcal{X}^N, \mathcal{A}^{\otimes N}, P_{\theta_0}^{\otimes N}) \right)$ be the ambient statistical universe as given in Definition \ref{def:ambient_spaces}, where the $N$-sample product probability measure on $\mathcal{X}^N$ is given by $P_\theta^{\otimes N}$ with product density
	\begin{equation}
	p_N(X^N; \theta) = \prod_{k=1}^N p(X_k; \theta), \quad X^N = (X_1, X_2, \dots, X_N) \in \mathcal{X}^N.
	\end{equation}
	Then, the joint $N$-sample Fisher Information Metric tensor $G^{(N)}(\theta)$ on $IG_N$ and the single-sample Fisher Information Metric tensor $g^{(1)}(\theta)$ on $IG_1$ satisfy:
	\begin{equation}
	G^{(N)}(\theta) = N \cdot g^{(1)}(\theta), \quad \forall \theta \in \Theta.
	\end{equation}
\end{lemma}

\begin{proof}
	Let $\theta = (\theta_1, \theta_2, \dots, \theta_d)^\top \in \Theta \subset \mathbb{R}^d$. The single-sample Fisher Information Metric components $g_{ij}^{(1)}(\theta)$ on $IG_1$ are defined by:
	\[
	g_{ij}^{(1)}(\theta) \equiv \mathbb{E}_{X \sim P_\theta} \left[ \frac{\partial \log p(X; \theta)}{\partial \theta_i} \frac{\partial \log p(X; \theta)}{\partial \theta_j} \right], \quad 1 \le i, j \le d.
	\]
	For the $N$-sample product space, the joint log-likelihood function $\ell_N(X^N; \theta) \equiv \log p_N(X^N; \theta)$ decomposes linearly into a sum of single-sample log-likelihoods $\ell(X_k; \theta) \equiv \log p(X_k; \theta)$:
	\[
	\ell_N(X^N; \theta) = \log \left( \prod_{k=1}^N p(X_k; \theta) \right) = \sum_{k=1}^N \log p(X_k; \theta) = \sum_{k=1}^N \ell(X_k; \theta).
	\]
	By linearity of partial differentiation, the components of the joint score vector are:
	\[
	\frac{\partial \ell_N(X^N; \theta)}{\partial \theta_i} = \sum_{k=1}^N \frac{\partial \ell(X_k; \theta)}{\partial \theta_i}.
	\]
	The $N$-sample joint Fisher Information Metric tensor $G_{ij}^{(N)}(\theta)$ on $IG_N$ is defined as the covariance matrix of the joint score vector:
	\[
	G_{ij}^{(N)}(\theta) \equiv \mathbb{E}_{X^N \sim P_\theta^{\otimes N}} \left[ \frac{\partial \ell_N(X^N; \theta)}{\partial \theta_i} \frac{\partial \ell_N(X^N; \theta)}{\partial \theta_j} \right].
	\]
	Substituting the score expansion into the expectation:
	\[
	G_{ij}^{(N)}(\theta) = \mathbb{E}_{X^N \sim P_\theta^{\otimes N}} \left[ \left( \sum_{k=1}^N \frac{\partial \ell(X_k; \theta)}{\partial \theta_i} \right) \left( \sum_{m=1}^N \frac{\partial \ell(X_m; \theta)}{\partial \theta_j} \right) \right] = \sum_{k=1}^N \sum_{m=1}^N \mathbb{E}_{X^N \sim P_\theta^{\otimes N}} \left[ \frac{\partial \ell(X_k; \theta)}{\partial \theta_i} \frac{\partial \ell(X_m; \theta)}{\partial \theta_j} \right].
	\]
	We evaluate the expectation for two cases:
	\begin{enumerate}
		\item \textbf{Diagonal terms ($k = m$):} Since $X_k \sim P_\theta$, each term equals the single-sample Fisher metric:
		\[
		\mathbb{E}_{X^N \sim P_\theta^{\otimes N}} \left[ \frac{\partial \ell(X_k; \theta)}{\partial \theta_i} \frac{\partial \ell(X_k; \theta)}{\partial \theta_j} \right] = \mathbb{E}_{X_k \sim P_\theta} \left[ \frac{\partial \ell(X_k; \theta)}{\partial \theta_i} \frac{\partial \ell(X_k; \theta)}{\partial \theta_j} \right] = g_{ij}^{(1)}(\theta).
		\]
		\item \textbf{Off-diagonal terms ($k \neq m$):} Since $X_k$ and $X_m$ are independent under the product measure $P_\theta^{\otimes N}$, expectation factors into the product of individual expectations:
		\[
		\mathbb{E}_{X^N \sim P_\theta^{\otimes N}} \left[ \frac{\partial \ell(X_k; \theta)}{\partial \theta_i} \frac{\partial \ell(X_m; \theta)}{\partial \theta_j} \right] = \mathbb{E}_{X_k \sim P_\theta} \left[ \frac{\partial \ell(X_k; \theta)}{\partial \theta_i} \right] \cdot \mathbb{E}_{X_m \sim P_\theta} \left[ \frac{\partial \ell(X_m; \theta)}{\partial \theta_j} \right].
		\]
		Under standard regularity conditions, the expectation of the score function vanishes for any single observation:
		\[
		\mathbb{E}_{X_k \sim P_\theta} \left[ \frac{\partial \ell(X_k; \theta)}{\partial \theta_i} \right] = \int_{\mathcal{X}} \frac{\partial \log p(x; \theta)}{\partial \theta_i} p(x; \theta) \, d\mu(x) = \frac{\partial}{\partial \theta_i} \int_{\mathcal{X}} p(x; \theta) \, d\mu(x) = \frac{\partial}{\partial \theta_i}(1) = 0.
		\]
		Therefore, for all $k \neq m$, the cross terms evaluate to $0 \cdot 0 = 0$.
		\end{enumerate}
			
Combining the diagonal and off-diagonal terms:
			\[
			G_{ij}^{(N)}(\theta) = \sum_{k=1}^N g_{ij}^{(1)}(\theta) + \sum_{k \neq m} 0 = N \cdot g_{ij}^{(1)}(\theta).
			\]
			In coordinate-free matrix notation across all indices $1 \le i, j \le d$, this establishes:
			\[
			G^{(N)}(\theta) = N \cdot g^{(1)}(\theta).
			\]
\end{proof}

	To guarantee that local geometric operations and stochastic expansions remain well-behaved, we enforce High-Order Local Regularity around the background parameter.
	
	\begin{assumption}[High-Order Local Regularity]
		\label{asm:local_regularity}
		The parametric family $\mathcal{P}$ satisfies the following conditions across an open neighborhood $\mathcal{U}(\theta_0) \subset \text{Int}(\Theta)$ centered at the true background environmental parameter $\theta_0$:
		\begin{enumerate}[label=(\roman*)]
			\item \textbf{Smooth $L_2$-Differentiability:} The mapping $\theta \mapsto \sqrt{p(x;\theta)}$ is three times continuously differentiable in $L_2(\mu)$.
			\item \textbf{Strict Metric Positivity:} The single-sample Fisher Information Metric tensor field, defined component-wise by:
			\begin{equation}
				g_{ij}^{(1)}(\theta) \equiv \mathbb{E}_\theta \left[ \frac{\partial \log p(X;\theta)}{\partial \theta^i} \frac{\partial \log p(X;\theta)}{\partial \theta^j} \right]
				\label{eq:fim_component_def}
			\end{equation}
			is strictly positive-definite ($\lambda_{\min}(g^{(1)}(\theta_0)) > 0$), smooth, and uniformly bounded for all $\theta \in \mathcal{U}(\theta_0)$.
			\item \textbf{Uniform Third-Order Integrability:} Third-order log-density derivatives satisfy the uniform envelope domination bound:
			\begin{equation}
				\sup_{\theta \in \mathcal{U}(\theta_0)} \mathbb{E}_\theta \left| \frac{\partial^3 \log p(X;\theta)}{\partial \theta^i \partial \theta^j \partial \theta^k} \right|^2 < \infty, \quad \forall i, j, k \in \{1, \dots, d\}
				\label{eq:third_order_bound}
			\end{equation}
		\end{enumerate}
	\end{assumption}
	
	\subsection{Natural Linear Vector Space Isomorphism $T_{\theta_0}IG_1 \cong \mathbb{R}^d$}
	\label{subsec:mod1_isomorphism}
	
	Let $C^\infty(\Theta)$ denote the algebra of smooth real-valued functions defined on an open neighborhood of $\theta_0$. Tangent vectors $v \in T_{\theta_0}IG_1$ are defined coordinate-freely as linear derivations $v: C^\infty(\Theta) \to \mathbb{R}$ satisfying $v(f f') = v(f)f'(\theta_0) + f(\theta_0)v(f')$.
	
	\begin{proposition}[Natural Linear Vector Space Derivation Isomorphism]
		\label{prop:isomorphism}
		Let $\phi = (\theta^1, \dots, \theta^d)$ be a local coordinate chart mapping $\mathcal{U}(\theta_0) \subset \Theta$ into $\mathbb{R}^d$. There exists a unique linear derivation isomorphism $\Psi_\phi: T_{\theta_0}IG_1 \to \mathbb{R}^d$ mapping derivations $v = \sum_{i=1}^d v^i \left.\frac{\partial}{\partial \theta^i}\right|_{\theta_0}$ to Cartesian $d$-tuples:
		\begin{equation}
			\Psi_\phi(v) = (v^1, v^2, \dots, v^d)^T \in \mathbb{R}^d, \quad \Psi_\phi^{-1}(u) = \sum_{i=1}^d u^i \left.\frac{\partial}{\partial \theta^i}\right|_{\theta_0}
			\label{eq:isomorphism_maps}
		\end{equation}
	\end{proposition}
The map $\Psi_\phi$ and its inverse $\Psi_\phi^{-1}$ establish a formal \textbf{linear vector space isomorphism} between the abstract tangent space $T_{\theta_0} IG_1$ (whose elements are derivations) and the concrete vector space $\mathbb{R}^d$. They allow one to rigorously transfer computations back and forth between abstract geometric derivations on the manifold $IG_1$ and standard linear algebra in $\mathbb{R}^d$.
	
\begin{proof}
	Let $\{e_i\}_{i=1}^d$ represent the standard canonical basis of $\mathbb{R}^d$. For the coordinate chart $\phi$, the partial algebraic derivations evaluated at $\theta_0$ form a natural basis for $T_{\theta_0}IG_1$:
	\begin{equation}
		\label{eq:partial_derivations_basis}
		\partial_i \big|_{\theta_0} \equiv \frac{\partial}{\partial \theta^i} \bigg|_{\theta_0}, \quad i \in \{1, \dots, d\}
	\end{equation}
	Any tangent vector $v \in T_{\theta_0}IG_1$ expands uniquely as $v = v^i \partial_i \big|_{\theta_0}$ (summing over repeated indices). We construct the linear bijection $\Psi_\phi$ and its inverse $\Psi_\phi^{-1}$ by setting:
	\begin{align}
		\label{eq:isomorphism_forward}
		\Psi_\phi(v) &= (v^1, v^2, \dots, v^d)^T \in \mathbb{R}^d \\
		\label{eq:isomorphism_inverse}
		\Psi_\phi^{-1}(u) &= u^i \partial_i \big|_{\theta_0}, \quad \forall u = (u^1, \dots, u^d)^T \in \mathbb{R}^d
	\end{align}
	Linearity follows directly from derivation operations. Bijectivity is guaranteed because the basis derivations $\{\partial_i \big|_{\theta_0}\}$ defined in Equation~\eqref{eq:partial_derivations_basis} form a linearly independent spanning set of dimension $d = \dim(\Theta) = \dim(\mathbb{R}^d)$.
\end{proof}

\begin{definition}[Metric-Induced Inner Product Space and Isometric Isomorphism]
~

	Let $(T_{\theta_0} IG_1, g_0)$ be the tangent derivation space at $\theta_0 \in \Theta$ equipped with the frozen metric tensor $g_0 \coloneqq g^{(1)}(\theta_0)$, whose component matrix relative to the coordinate basis $\left\{ \left.\frac{\partial}{\partial \theta^i}\right|_{\theta_0} \right\}_{i=1}^d$ is given by:
	\[
	g_{ij}^{(1)}(\theta_0) \coloneqq g_0 \left( \left.\frac{\partial}{\partial \theta^i}\right|_{\theta_0}, \left.\frac{\partial}{\partial \theta^j}\right|_{\theta_0} \right), \quad \forall i, j \in \{1, \dots, d\}.
	\]
	Let $\Psi_\phi : T_{\theta_0} IG_1 \to \mathbb{R}^d$ denote the natural coordinate isomorphism mapping a derivation $v = \sum_{i=1}^d v^i \left.\frac{\partial}{\partial \theta^i}\right|_{\theta_0}$ to its component column vector $(v^1, \dots, v^d)^T$.
	
\begin{enumerate}

\item \textbf{Induced Inner Product on $\mathbb{R}^d$:} \\

We equip the derivation coordinate space $\mathbb{R}^d$ with the inner product $\langle \cdot, \cdot \rangle_{g_0}$, defined as the pullback of $g_0$ under the inverse isomorphism $\Psi_\phi^{-1}$:
	\[
	\langle x, y \rangle_{g_0} \coloneqq \left( (\Psi_\phi^{-1})^* g_0 \right)(x, y) = g_0\left( \Psi_\phi^{-1}(x), \Psi_\phi^{-1}(y) \right) = \sum_{i=1}^d \sum_{j=1}^d x^i y^j g_{ij}^{(1)}(\theta_0),
	\]
	for all coordinate vectors $x = (x^1, \dots, x^d)^T, \, y = (y^1, \dots, y^d)^T \in \mathbb{R}^d$. Specifically:
\begin{itemize}
	\item $g_0$ is the tensor being pulled back.\footnote{{\bf The Pullback of a Tensor Field / Metric}  Let $M$ and $N$ be smooth manifolds, $F: M \to N$ a smooth map, and $g$ a covariant tensor field (such as a Riemannian metric) on $N$. The \textbf{pullback} of $g$ under $F$, denoted $F^*g$, is a covariant tensor field on $M$ defined for any point $p \in M$ and tangent vectors $X, Y \in T_p M$ by:
		\[
		(F^* g)_p(X, Y) = g_{F(p)}\!\left( dF_p(X),\, dF_p(Y) \right),
		\]
		where $dF_p: T_p M \to T_{F(p)} N$ is the differential (pushforward) of $F$ at $p$ \cite{lee2013}. 
		
		In the case where $F = \Psi_\phi^{-1}$, the map pulling back the metric $g_0$ is the inverse isomorphism $\Psi_\phi^{-1}$, and the pulled-back tensor on the domain space is $(\Psi_\phi^{-1})^* g_0$.}
	\item $\Psi_\phi^{-1}$ is the mapping under which the pullback is performed.
\end{itemize}
In operator notation, the resulting inner product on $\mathbb{R}^d$ is given by the pullback tensor $(\Psi_\phi^{-1})^* g_0$.

\item \textbf{Isometric Isomorphism Property:} \\
	The inner product space $(\mathbb{R}^d, \langle \cdot, \cdot \rangle_{g_0})$ renders $\Psi_\phi : (T_{\theta_0} IG_1, g_0) \to (\mathbb{R}^d, \langle \cdot, \cdot \rangle_{g_0})$ an \textbf{isometric isomorphism}, preserving inner products across spaces:
	\begin{equation}
		g_0(v, w) = \langle \Psi_\phi(v), \Psi_\phi(w) \rangle_{g_0}, \quad \forall v, w \in T_{\theta_0} IG_1.
	\end{equation}
\end{enumerate}
\end{definition}

\subsection{The Triadic Variable Spectrum}

To decouple sample-size-dependent empirical measure concentration from structural parameter capacity expansion, we formalize the Triadic Variable Spectrum governing the dynamic metamorphosis of $IG_N$. This framework explicitly balances the micro-scale statistical fluctuations governed by the Central Limit Theorem against the macro-scale geometric dilation of the underlying statistical manifold.

\begin{definition}[Triadic Variable Spectrum $\mathcal{T}_N$]
	\label{def:triadic_spectrum}
	
	The dynamic tracking mechanism governing $IG_N$ is formalized as the interacting triple:
	\begin{equation}
		\mathcal{T}_N \equiv \left( \mathcal{E}_N, \mathcal{S}_N, \mathcal{F}_N \right)
	\end{equation}
	where each component variable represents an operational layer defined as follows:
	\begin{enumerate}
		\item[\textnormal{(i)}] \textbf{Environment Variable ($\mathcal{E}_N$):} 
		\[
		\mathcal{E}_N \equiv \left( P_N(X^N), \theta_0, \beta_N \right)
		\]
		tracks micro-empirical measure concentration of the sample sequence $X^N \sim P_N(X^N)$ around the true parameter $\theta_0$ at the fluctuation rate $\beta_N \equiv 1/\sqrt{N}$.
		
		\item[\textnormal{(ii)}] \textbf{System Variable ($\mathcal{S}_N$):} 
		\[
		\mathcal{S}_N \equiv \left( G_{ij}^{(N)}(\theta), \alpha_N \right)
		\]
		tracks macro-informational capacity expansion $G_{ij}^{(N)}(\theta) = N \cdot  g_{ij}^{(1)}(\theta)$ governed by the additive metric dilation speed $\alpha_N \equiv \sqrt{N}$.
		
		\item[\textnormal{(iii)}] \textbf{Mechanism Variable ($\mathcal{F}_N$):} 
		\[
		\mathcal{F}_N \equiv \left( h, \psi_N(h), \widetilde{\Gamma}_{ijk}^{(\alpha, N)}(h) \right)
		\]
		tracks the bridging geometry via localized coordinate dilation $\psi_N(h) = \theta_0 + \frac{h}{\sqrt{N}}$ and the pulled-back connection symbols $\widetilde{\Gamma}_{ijk}^{(\alpha, N)}(h) \equiv \psi_N^* \Gamma_{ijk}^{(N)}(h)$. We will discuss them in more details of   following subsections because they are critical in this paper to bridge statistics and differential geometry frameworks.
	\end{enumerate}
\end{definition}

By establishing the operational interplay across $\mathcal{T}_N$, the micro-scale convergence rate $\beta_N = 1/\sqrt{N}$ perfectly offsets the macro-scale metric expansion $\alpha_N = \sqrt{N}$ through the pulling-back mechanism $\psi_N$, yielding a stable, non-degenerate asymptotic geometric limit for $IG_N$ as $N \to \infty$.


\subsection{Epistemological Equivalence Theorem: $\text{CS} \iff \mathcal{M}_\infty$}
	\label{subsec:epistemological_equivalence}
	
	In classical large-sample statistical theory, the transition from finite-sample curved statistical parametric families to flat asymptotic limit experiments is often treated as an informal analytic convenience or a passive collection of scalar point limits \cite{fisher1922, vandervaart1998}. To eliminate topological ambiguities and establish a rigorous axiomatic baseline, this section formalizes the \textbf{Epistemological Equivalence Theorem}, proving that the operational paradigm of Conventional Statistics ($\text{CS}$) is mathematically identical to the canonical flat tangent space $\mathcal{M}_\infty$.
	
	To ensure complete logical transparency and avoid circular reasoning, our exposition follows a four-step canonical architecture:
	\begin{enumerate}
		\item \textbf{Step 1:} Operational characterization of Conventional Statistics ($\text{CS}_{\text{stat}}$) from the perspective of classical asymptotic decision theory (Fisher, Le Cam, Wald, Wilks, Rao).
		\item \textbf{Step 2:} Differential-geometric axiomatization of the target space $\mathcal{M}_\infty$ and the Constitutional Axioms of Geometric Conventional Statistics ($\text{CS}_{\text{geom}}$).
		\item \textbf{Step 3:} Rigorous phase transition mechanics establishing metric freezing, connection dissolution, and LAN linear-quadratic expansion.
\item \textbf{Step 4:} Deduction of the Operational-Geometric Bridge Lemma and the Complemented Epistemological Equivalence Theorem  $\text{CS}_{\text{stat}} \iff \text{CS}_{\text{geom}} \iff \mathcal{M}_\infty$.
	\end{enumerate}
	
	\subsubsection{Step 1: The Operational Definition of Conventional Statistics ($\text{CS}_{\text{stat}}$)}
	\label{subsubsec:cs_stat_definition}
	
	Before invoking differential geometry, we formalize the operational paradigm of Conventional Statistics ($\text{CS}_{\text{stat}}$) as established by Fisher \cite{fisher1922, fisher1925}, Le Cam \cite{lecam1960, lecam1986}, and van der Vaart \cite{vandervaart1998}.
	
	\begin{definition}[Operational Paradigm of Conventional Statistics ($\text{CS}_{\text{stat}}$)]
		\label{def:cs_stat_operational}
		A localized statistical workspace centered at a base parameter $\theta_0 \in \mathrm{Int}(\Theta)$ belongs to the \textbf{Conventional Statistics Paradigm} ($\text{CS}_{\text{stat}}$) if local inference over compact displacement cages $h \in \mathbb{K} \subset \mathbb{R}^d$ satisfies the following three operational principles:
		\begin{enumerate}[label=\textnormal{(\roman*)}]
			\item \textbf{Fisher's Score Linearity and Quadratic Likelihood:} The localized log-likelihood ratio process
			\begin{equation}
				\Lambda_N(h) \equiv \sum_{n=1}^N \log \frac{p(X_n; \theta_0 + h/\sqrt{N})}{p(X_n; \theta_0)}
				\label{eq:Lambda_N_definition}
			\end{equation}
			admits a strictly linear-quadratic representation:
			\begin{equation}
				\Lambda_N(h) = h^\top \Delta_N(\theta_0) - \frac{1}{2} h^\top g^{(1)}(\theta_0) h + o_p(1),
				\label{eq:fisher_quadratic_likelihood}
			\end{equation}
			where $\Delta_N(\theta_0) \equiv \frac{1}{\sqrt{N}} \sum_{n=1}^N \nabla_\theta \log p(X_n; \theta_0)$ is the normalized score vector.
			
			\item \textbf{Le Cam's Local Asymptotic Normality (LAN) Limit Experiment:} As $N \to \infty$, the sequence of statistical experiments locally converges in the sense of Le Cam to a Gaussian shift experiment:
			\begin{equation}
				\mathcal{E}_\infty \equiv \left( \mathbb{R}^d, \, \mathcal{B}^d, \, \left\{ \mathcal{N}\left(h, \, [g^{(1)}(\theta_0)]^{-1}\right) : h \in \mathbb{R}^d \right\} \right).
				\label{eq:lecam_lan_experiment}
			\end{equation}
			
			\item \textbf{Wald-Wilks-Rao Test Trinity Invariance:} The three classical hypothesis test statistics---Rao's Score Test ($W_S$), Wilks' Likelihood Ratio Test ($W_{LR}$), and Wald's Test ($W_W$)---are identically equal and quadratic in the local score vector $V \sim \mathcal{N}(0, g_0)$:
			\begin{equation}
				W_S = W_{LR} = W_W = V^\top [g^{(1)}(\theta_0)]^{-1} V + o_p(1).
				\label{eq:trinity_invariance}
			\end{equation}
		\end{enumerate}
		where the variable $h \in \mathbb{K} \subset \mathbb{R}^d$ represents the \textbf{localizing parameter displacement vector}, defined formally via the asymptotic scaling relation $\theta_N = \theta_0 + \frac{h}{\sqrt{N}} \iff h = \sqrt{N}(\theta_N - \theta_0)$.
	\end{definition}
	
	The statistical interpretations of the variable $h$ are:
	\begin{itemize}
		\item \textbf{Canonical Asymptotic Rate Scaling ($\sqrt{N}$-Fluctuation Scale):} The parameter $h$ normalizes local parameter deviations from $\theta_0$ by the canonical $\sqrt{N}$ sampling rate governed by the Central Limit Theorem, zooming in on contiguous alternatives $P_{\theta_0 + h/\sqrt{N}}$.
		\item \textbf{Mean Shift of Asymptotic Gaussian Experiment:} In Le Cam's LAN framework, $h$ serves as the location parameter (mean shift) of the limiting Gaussian experiment $Y \sim \mathcal{N}\left(h, [g^{(1)}(\theta_0)]^{-1}\right)$.
		\item \textbf{Tangent Vector Coordinate:} Differential-geometrically, $h$ acts as the Cartesian coordinate of local displacements on the flat tangent space $T_{\theta_0}\Theta$. Restricting $h$ to a compact cage $\mathbb{K} \subset \mathbb{R}^d$ isolates the local geometry from non-linear global manifold pathologies.
	\end{itemize}
	
	\subsubsection{Step 2: Target Space $\mathcal{M}_\infty$ and Geometric Axiomatization ($\text{CS}_{\text{geom}}$)}
	\label{subsubsec:cs_geom_definition}
	
	We now introduce the exact differential-geometric representation ($\text{CS}_{\text{geom}}$) corresponding to the operational paradigm defined in Step 1.
	
	\begin{definition}[Canonical Asymptotic Flat Tangent Space $\mathcal{M}_\infty$]
		\label{def:m_infty_canonical}
		The \textbf{Canonical Asymptotic Flat Tangent Space} is defined as the flat Riemannian-affine triple:
		\begin{equation}
			\mathcal{M}_\infty \equiv \left( T_{\theta_0}IG_1, \, g_0, \, \nabla^{(0)} \right)
			\label{eq:m_infty_canonical_tuple}
		\end{equation}
		anchored at $\theta_0 \in \mathrm{Int}(\Theta)$, equipped with the position-invariant Fisher metric $g_0 \equiv g^{(1)}(\theta_0)$ and the flat Levi-Civita connection $\nabla^{(0)}$ characterized by identically vanishing Christoffel symbols ($\Gamma_{ijk}^{(0)} \equiv 0$) and zero Riemann curvature tensor ($\mathcal{R}_{ijmn}^{(0)} \equiv 0$).
	\end{definition}
	
	\begin{definition}[Constitutional Axioms of Geometric Conventional Statistics ($\text{CS}_{\text{geom}}$)]
		\label{def:cs_geom_axioms}
		A statistical workspace $\mathcal{M}_{\mathrm{stat}} \equiv (\mathcal{N}, g, \nabla)$ anchored at $\theta_0 \in \mathrm{Int}(\Theta)$ is a \textbf{Geometric Conventional Statistics Universe} ($\text{CS}$) corresponding to the Fisher-Le Cam operational paradigm (Definition~\ref{def:cs_stat_operational}) if and only if it fulfills three constitutional pillars:
		\begin{enumerate}[leftmargin=*, label={\textbf{Pillar \Roman*:}}]
			\item \textbf{Flat Tangent Action:} The underlying domain $\mathcal{N}$ acts strictly on local displacement vectors $h = \sqrt{N}(\theta - \theta_0) \in \mathbb{R}^d$, forming a vector space isomorphic to the flat tangent space $T_{\theta_0}IG_1$ anchored at $\theta_0$.
			\item \textbf{Frozen Metric Field:} The metric tensor field $g(h) \equiv g_0 = g^{(1)}(\theta_0)$ is spatially invariant across compact coordinate cages $\mathbb{K} \subset \mathbb{R}^d$.
			\item \textbf{Flat Affine Parallel Transport:} The connection $\nabla$ has identically zero Christoffel symbols ($\Gamma_{ijk} \equiv 0$) and zero Riemann curvature ($\mathcal{R}_{ijmn} \equiv 0$).
		\end{enumerate}
	\end{definition}
	
	\subsubsection{Step 3: Phase Transition Mechanics and Preliminary Proofs}
	\label{subsubsec:phase_transition_proofs}
	
	To establish the bridge between operational statistics and differential geometry, we prove two fundamental phase-transition lemmas under High-Order Local Regularity (Assumption 1).
	
	\begin{lemma}[Asymptotic Metric Freezing under LAN Scaling]
		\label{lem:metric_freezing_proof}
		Let $g^{(1)}(\theta)$ be a $C^2$-smooth single-sample Fisher Information Metric tensor field on $\mathcal{U}(\theta_0) \subset \Theta$. Under the local coordinate dilation $\psi_N(h) = \theta_0 + \frac{h}{\sqrt{N}}$, define the pulled-back local metric field $\widetilde{g}_N(h) \equiv \psi_N^* G^{(N)}(h)$ on a compact coordinate cage $\mathbb{K} \subset \mathbb{R}^d$. Then:
		\begin{equation}
			\sup_{h \in \mathbb{K}} \left\| \widetilde{g}_N(h) - g_0 \right\|_\infty = O\left(\frac{1}{\sqrt{N}}\right) \longrightarrow 0 \quad \text{as } N \to \infty,
			\label{eq:metric_freezing_bound}
		\end{equation}
		where $g_0 \equiv g^{(1)}(\theta_0)$.
	\end{lemma}
	
	\begin{proof}
		By tensor pullback mechanics under the dilation map $\psi_N(h) = \theta_0 + \frac{h}{\sqrt{N}}$, the multivariable Jacobian matrix is $J(\psi_N) = \frac{1}{\sqrt{N}} \mathbf{I}_{d \times d}$. Combining this with the additivity relation $G^{(N)}(\theta) = N g^{(1)}(\theta)$ yields:
		\begin{equation}
			\widetilde{g}_{ij, N}(h) = \sum_{a, b=1}^d \left(\frac{1}{\sqrt{N}} \delta_i^a\right) \left(\frac{1}{\sqrt{N}} \delta_j^b\right) N g_{ab}^{(1)}\left(\theta_0 + \frac{h}{\sqrt{N}}\right) = g_{ij}^{(1)}\left(\theta_0 + \frac{h}{\sqrt{N}}\right).
			\label{eq:metric_cancellation}
		\end{equation}
		Applying a multivariate Taylor series expansion to $g_{ij}^{(1)}\left(\theta_0 + \frac{h}{\sqrt{N}}\right)$ around $h = 0$:
		\begin{equation}
			g_{ij}^{(1)}\left(\theta_0 + \frac{h}{\sqrt{N}}\right) = g_{ij}^{(1)}(\theta_0) + \frac{1}{\sqrt{N}} \sum_{k=1}^d h_k \int_0^1 \frac{\partial g_{ij}^{(1)}}{\partial \theta_k}\left(\theta_0 + \frac{t h}{\sqrt{N}}\right) dt.
			\label{eq:taylor_expansion_metric}
		\end{equation}
		Subtracting $g_{ij, 0} \equiv g_{ij}^{(1)}(\theta_0)$ and taking the supremum norm over the compact cage $\mathbb{K}$:
		\begin{equation}
			\sup_{h \in \mathbb{K}} \left| \widetilde{g}_{ij, N}(h) - g_{ij}^{(1)}(\theta_0) \right| \le \frac{1}{\sqrt{N}} \left( \sup_{h \in \mathbb{K}} \|h\|_2 \right) \cdot \max_{i,j,k} \sup_{\theta \in \mathcal{U}(\theta_0)} \left| \frac{\partial g_{ij}^{(1)}(\theta)}{\partial \theta_k} \right|.
			\label{eq:sup_norm_metric_bound}
		\end{equation}
		Because $\mathbb{K}$ is compact, $C_\mathbb{K} \equiv \sup_{h \in \mathbb{K}} \|h\|_2 < \infty$. Smoothness of $g^{(1)}$ on the closed neighborhood $\overline{\mathcal{U}}(\theta_0)$ guarantees $M_g \equiv \max_{i,j,k} \sup_{\theta \in \mathcal{U}(\theta_0)} \left| \frac{\partial g_{ij}^{(1)}(\theta)}{\partial \theta_k} \right| < \infty$. Therefore:
		\begin{equation}
			\sup_{h \in \mathbb{K}} \left\| \widetilde{g}_N(h) - g_0 \right\|_\infty \le \frac{C_\mathbb{K} M_g}{\sqrt{N}} = O\left(\frac{1}{\sqrt{N}}\right).
			\label{eq:metric_freezing_final}
		\end{equation}
		Taking $N \to \infty$ proves that $\widetilde{g}_N(h)$ converges uniformly to the constant metric $g_0$ across $\mathbb{K}$, satisfying Pillar II.
	\end{proof}
	
	\begin{lemma}[Curvature and Connection Collapse Lemma]
		\label{lem:curvature_collapse_proof}
		Let $IG_1 = (\Theta, g^{(1)}, \nabla^{(\alpha)})$ be an information manifold with $C^3$-smooth metric $g^{(1)}$ and $\alpha$-connection $\nabla^{(\alpha)}$ whose Christoffel symbols are $\Gamma_{ijk}^{(\alpha)}(\theta)$. Let $\widetilde{\Gamma}_{ijk}^{(\alpha, N)}(h) \equiv [\psi_N^* \nabla^{(\alpha,N)}]_{ijk}(h)$ and $\widetilde{\mathcal{R}}_{ijmn}^{(\alpha, N)}(h) \equiv [\psi_N^* \mathcal{R}^{(\alpha,N)}]_{ijmn}(h)$ denote the pulled-back connection symbols and Riemann curvature components on $\mathbb{K}$. Then:
		\begin{enumerate}[label=\textnormal{(\roman*)}]
			\item $\sup_{h \in \mathbb{K}} \left| \widetilde{\Gamma}_{ijk}^{(\alpha, N)}(h) \right| = O\left(\frac{1}{\sqrt{N}}\right) \longrightarrow 0 \quad \text{as } N \to \infty$.
			\item $\sup_{h \in \mathbb{K}} \left| \widetilde{\mathcal{R}}_{ijmn}^{(\alpha, N)}(h) \right| = O\left(\frac{1}{N}\right) \longrightarrow 0 \quad \text{as } N \to \infty$.
		\end{enumerate}
	\end{lemma}
	
	\begin{proof}
		\textbf{Proof of (i):} The Christoffel symbols of the first kind for the pulled-back metric $\widetilde{g}_N(h)$ are defined by:
		\begin{equation}
			\widetilde{\Gamma}_{ij, k}^{(N)}(h) = \frac{1}{2} \left( \frac{\partial \widetilde{g}_{ik, N}(h)}{\partial h^j} + \frac{\partial \widetilde{g}_{jk, N}(h)}{\partial h^i} - \frac{\partial \widetilde{g}_{ij, N}(h)}{\partial h^k} \right).
			\label{eq:christoffel_def}
		\end{equation}
		Applying the chain rule to $\widetilde{g}_{ik, N}(h) = g_{ik}^{(1)}\left(\theta_0 + \frac{h}{\sqrt{N}}\right)$:
		\begin{equation}
			\frac{\partial \widetilde{g}_{ik, N}(h)}{\partial h^j} = \frac{1}{\sqrt{N}} \left. \frac{\partial g_{ik}^{(1)}(\theta)}{\partial \theta^j} \right|_{\theta = \theta_0 + \frac{h}{\sqrt{N}}}.
			\label{eq:chain_rule_christoffel}
		\end{equation}
		Substituting Equation~\eqref{eq:chain_rule_christoffel} into Equation~\eqref{eq:christoffel_def} yields:
		\begin{equation}
			\widetilde{\Gamma}_{ij, k}^{(N)}(h) = \frac{1}{\sqrt{N}} \Gamma_{ij, k}^{(1)}\left(\theta_0 + \frac{h}{\sqrt{N}}\right).
			\label{eq:connection_dissolution_scaling}
		\end{equation}
		Since $\Gamma_{ij, k}^{(1)}(\theta)$ is continuous and bounded on compact sets, taking the supremum over $\mathbb{K}$ gives:
		\begin{equation}
			\sup_{h \in \mathbb{K}} \left| \widetilde{\Gamma}_{ij, k}^{(N)}(h) \right| \le \frac{1}{\sqrt{N}} \sup_{\theta \in \mathcal{U}(\theta_0)} \left| \Gamma_{ij, k}^{(1)}(\theta) \right| = O\left(\frac{1}{\sqrt{N}}\right) \longrightarrow 0.
		\end{equation}
		
		\textbf{Proof of (ii):} The Riemann curvature tensor components are given by:
		\begin{equation}
			\widetilde{\mathcal{R}}_{ijmn}^{(N)}(h) = \frac{\partial \widetilde{\Gamma}_{in, j}^{(N)}}{\partial h^m} - \frac{\partial \widetilde{\Gamma}_{im, j}^{(N)}}{\partial h^n} + \sum_{s,t} \widetilde{g}_N^{st} \left( \widetilde{\Gamma}_{mj, s}^{(N)} \widetilde{\Gamma}_{ni, t}^{(N)} - \widetilde{\Gamma}_{nj, s}^{(N)} \widetilde{\Gamma}_{mi, t}^{(N)} \right).
			\label{eq:riemann_curvature_components}
		\end{equation}
		Differentiating $\widetilde{\Gamma}_{in, j}^{(N)}(h)$ with respect to $h^m$ introduces a second spatial derivative chain rule factor of $1/\sqrt{N}$:
		\begin{equation}
			\frac{\partial \widetilde{\Gamma}_{in, j}^{(N)}(h)}{\partial h^m} = \frac{\partial}{\partial h^m} \left[ \frac{1}{\sqrt{N}} \Gamma_{in, j}^{(1)}\left(\theta_0 + \frac{h}{\sqrt{N}}\right) \right] = \frac{1}{N} \left. \frac{\partial \Gamma_{in, j}^{(1)}(\theta)}{\partial \theta^m} \right|_{\theta = \theta_0 + \frac{h}{\sqrt{N}}}.
			\label{eq:second_derivative_chain_rule}
		\end{equation}
		Furthermore, the quadratic connection products scale as $O(1/\sqrt{N}) \times O(1/\sqrt{N}) = O(1/N)$. Thus:
		\begin{equation}
			\sup_{h \in \mathbb{K}} \left| \widetilde{\mathcal{R}}_{ijmn}^{(N)}(h) \right| = O\left(\frac{1}{N}\right) \longrightarrow 0,
		\end{equation}
		proving Pillar III ($\Gamma \equiv 0$ and $\mathcal{R} \equiv 0$) identically in the limit workspace.
	\end{proof}
	
	\begin{theorem}[Rigorous Linear-Quadratic LAN Decomposition]
		\label{thm:lan_linear_quadratic_proof}
		Under High-Order Local Regularity, the localized log-likelihood ratio process $$\Lambda_N(h) \equiv \sum_{n=1}^N \log \frac{p(X_n; \theta_0 + h/\sqrt{N})}{p(X_n; \theta_0)}$$ admits the exact expansion:
		\begin{equation}
			\Lambda_N(h) = h^\top \Delta_N(\theta_0) - \frac{1}{2} h^\top g_0 h + \mathcal{R}_N(h),
			\label{eq:lan_expansion_exact}
		\end{equation}
		where $\Delta_N(\theta_0) \equiv \frac{1}{\sqrt{N}} \sum_{n=1}^N \nabla_\theta \log p(X_n; \theta_0) \xrightarrow{d} \mathcal{N}(0, g_0)$, and the non-linear stochastic remainder satisfies:
		\begin{equation}
			\sup_{h \in \mathbb{K}} |\mathcal{R}_N(h)| = \mathcal{O}_p\left(\frac{1}{\sqrt{N}}\right) \xrightarrow{P_{\theta_0}} 0.
			\label{eq:remainder_stochastic_bound}
		\end{equation}
	\end{theorem}
	
	\begin{proof}
		Performing a second-order multivariate Taylor expansion of $\log p\left(X_n; \theta_0 + \frac{h}{\sqrt{N}}\right)$ around $h = 0$:
		\begin{equation}
			\log p\left(X_n; \theta_0 + \frac{h}{\sqrt{N}}\right) - \log p(X_n; \theta_0) = \frac{h^i}{\sqrt{N}} \partial_i \log p(X_n;\theta_0) + \frac{h^i h^j}{2N} \partial_i \partial_j \log p(X_n;\theta_0) + R_n^{(3)}(h).
			\label{eq:taylor_log_density}
		\end{equation}
		Summing over $n = 1, \dots, N$:
		\begin{equation}
			\Lambda_N(h) = h^\top \left( \frac{1}{\sqrt{N}} \sum_{n=1}^N \nabla_\theta \log p(X_n;\theta_0) \right) + \frac{1}{2} h^\top \left( \frac{1}{N} \sum_{n=1}^N \nabla_\theta^2 \log p(X_n;\theta_0) \right) h + \sum_{n=1}^N R_n^{(3)}(h).
			\label{eq:summed_log_density}
		\end{equation}
		By the Multivariate Central Limit Theorem, $\Delta_N(\theta_0) \xrightarrow{d} \mathcal{N}(0, g_0)$. By the Weak Law of Large Numbers, $\frac{1}{N} \sum_{n=1}^N \nabla_\theta^2 \log p(X_n;\theta_0) \xrightarrow{P_{\theta_0}} \mathbb{E}_{\theta_0}[\nabla_\theta^2 \log p(X;\theta_0)] = -g_0$. 
		
		By uniform envelope domination of third-order derivatives (Assumption 1), the remainder satisfies:
		\begin{equation}
			\sup_{h \in \mathbb{K}} \left| \sum_{n=1}^N R_n^{(3)}(h) \right| \le \frac{\|h\|_2^3}{6\sqrt{N}} \left( \frac{1}{N} \sum_{n=1}^N M(X_n) \right) = \mathcal{O}_p\left(\frac{1}{\sqrt{N}}\right),
			\label{eq:third_order_remainder_bound}
		\end{equation}
		completing the proof of Equation~\eqref{eq:lan_expansion_exact}.
	\end{proof}
	
\subsubsection{Step 4: The Operational-Geometric Bridge and Full Equivalence Theorem}
\label{subsubsec:bridge_and_full_theorem}
Continuing from the foundational definitions and phase-transition mechanics established above, we now present the pivotal \textbf{Operational-Geometric Bridge Lemma} (Lemma~\ref{lem:operational_geometric_bridge}), the complete \textbf{Complemented Epistemological Equivalence Theorem} (Theorem~\ref{thm:theorem2_full}), the structured \textbf{Defenses to Modern Statistical Objections} (Section~\ref{subsec:statistical_implications}).

To close the logical loop between the decision-theoretic operational paradigm ($\text{CS}_{\text{stat}}$) and the differential-geometric target space ($\mathcal{M}_\infty$), we establish the equivalence between operational principles and geometric constitutional pillars.

\begin{lemma}[Operational-Geometric Bridge Lemma: $\text{CS}_{\text{stat}} \iff \text{CS}_{\text{geom}}$]
	\label{lem:operational_geometric_bridge}
	Under High-Order Local Regularity (Assumption 1), a localized statistical workspace satisfies the operational paradigm of Conventional Statistics $\text{CS}_{\text{stat}}$ (Definition~\ref{def:cs_stat_operational}) if and only if its underlying geometric triple $\mathcal{M}_{\text{stat}} \equiv (\mathcal{N}, g, \nabla)$ fulfills the three constitutional pillars of $\text{CS}_{\text{geom}}$ (Definition~\ref{def:cs_geom_axioms}).
\end{lemma}
\begin{remark}

This Operational-Geometric Bridge Lemma establishes a formal, bi-directional logical equivalence ($\text{CS}_{\text{stat}} \iff \text{CS}_{\text{geom}}$) between the classical decision-theoretic operational paradigm of statistics and the differential geometry of information manifolds. This equivalence carries several profound statistical implications:

\begin{enumerate}
	\item \textbf{Ontological Justification for Local Asymptotic Normality (LAN):} 
	The lemma proves that Le Cam's Local Asymptotic Normality (LAN) and Fisher score linearity are not merely passive analytical heuristics or convenient linear approximations. Instead, they are the direct operational expressions of a flat, metric-frozen geometric space $(\mathcal{N}, g_0, \nabla^{(0)})$. This provides an exact geometric mechanism explaining why curved parametric families locally behave like linear Gaussian shift experiments under $1/\sqrt{N}$ rate scaling.
	
	\item \textbf{Geometric Mechanics of the Classical Test Trinity Equivalence:} 
	In finite-sample regimes, the three classical hypothesis test statistics—Rao's Score Test ($W_S$), Wilks' Likelihood Ratio Test ($W_{LR}$), and Wald's Test ($W_W$)—diverge due to non-zero Christoffel symbols $\Gamma_{ijk}^{(\alpha)}$ and Amari-Riemann curvature $\mathcal{R}_{ijmn}^{(\alpha)}$. Lemma \ref{lem:operational_geometric_bridge} shows that the asymptotic identity $W_S = W_{LR} = W_W + o_p(1)$ holds if and only if parallel transport is flat ($\Gamma_{ijk} \equiv 0$). Thus, flat affine transport is both necessary and sufficient for the unconditional invariance of the classical test trinity.
	
	\item \textbf{Elimination of Second-Order Information Loss:} 
	By enforcing a position-invariant metric field $g(h) \equiv g_0$ across local parameter perturbations (Pillar II), spatial metric derivatives vanish ($\nabla g = 0$). In asymptotic estimation theory, this sets Efron's exponential and mixture statistical curvatures ($\gamma_e, \gamma_m$) identically to zero, implying that first-order efficient estimators suffer zero second-order information loss within the localized contiguous neighborhood.
	
	\item \textbf{Universal Asymptotic Risk Bounds:} 
	Because the geometric workspace $\mathcal{M}_{\text{stat}}$ is isomorphic to a flat inner-product space equipped with frozen Fisher metric $g_0$, local decision rules and risk bounds derived under quadratic loss functions become globally uniform and shift-invariant across compact parameter displacement cages $\mathbb{K} \subset \mathbb{R}^d$.
\end{enumerate}

\end{remark}

\begin{proof}
	We establish bi-directional necessity and sufficiency step-by-step.
	
	\paragraph{\textbf{Necessity ($\text{CS}_{\text{stat}} \implies \text{CS}_{\text{geom}}$):}}
	Assume the localized workspace satisfies operational principles (i)--(iii) of Definition~\ref{def:cs_stat_operational}:
	\begin{enumerate}[label=(\alph*)]
\item \textbf{Verification of Pillar I (Flat Tangent Action):} 
In $\text{CS}_{\text{stat}}$, local statistical inference is parametrized over compact displacement cages $h = \sqrt{N}(\theta - \theta_0) \in \mathbb{K} \subset \mathbb{R}^d$. The domain of parameter perturbations is the linear vector space $\mathbb{R}^d$. By Proposition~\ref{prop:isomorphism},
		$\mathbb{R}^d$ is naturally isometric and isomorphic to the tangent derivation space $T_{\theta_0} IG_1$. Thus, the underlying domain $\mathcal{N}$ acts strictly on tangent vectors $h \in T_{\theta_0} IG_1 \cong \mathbb{R}^d$, fulfilling Pillar I.
		
		\item \textbf{Verification of Pillar II (Frozen Metric Field):} 
		Principle (i) enforces that the localized log-likelihood ratio process $\Lambda_N(h)$ is strictly quadratic with constant matrix $g^{(1)}(\theta_0) \equiv g_0$. By Le Cam's Local Asymptotic Normality (LAN) principle (ii), the limiting experiment has a spatially invariant covariance structure $g_0^{-1}$ across all local displacements $h \in \mathbb{K}$. Differential-geometrically, the metric tensor field $g(h)$ representing local Fisher information capacity must be spatially uniform across $\mathbb{K}$:
		\begin{equation}
			g(h) \equiv g_0 = g^{(1)}(\theta_0), \quad \forall h \in \mathbb{K},
		\end{equation}
		satisfying Pillar II.
		
\item \textbf{Verification of Pillar III (Flat Affine Parallel Transport):} 
	Principle (iii) enforces the exact identity of the test trinity: $W_S = W_{LR} = W_W + o_p(1)$ across all contiguous alternatives $h \in \mathbb{K}$. In information geometry \cite{amari1985, amari2000, kass1997}, the Taylor expansions of these three test statistics up to order $\mathcal{O}_p(N^{-1/2})$ reveal discrepancies governed by the dual Christoffel connections $\Gamma_{ijk}^{(\alpha)}$:
		\begin{align}
			W_{LR}(h) - W_S(h) &= \frac{1}{3\sqrt{N}} \sum_{i,j,k=1}^d \Gamma_{ijk}^{(1)}(\theta_0) h^i h^j h^k + o_p\left(N^{-1/2}\right), \label{eq:test_diff_1} \\
			W_W(h) - W_{LR}(h) &= \frac{1}{3\sqrt{N}} \sum_{i,j,k=1}^d \Gamma_{ijk}^{(-1)}(\theta_0) h^i h^j h^k + o_p\left(N^{-1/2}\right). \label{eq:test_diff_2}
		\end{align}
	Unconditional asymptotic identity $W_S = W_{LR} = W_W + o_p(1)$ for all displacement vectors $h \in \mathbb{K}$ requires the cubic forms in Equations~\eqref{eq:test_diff_1} and \eqref{eq:test_diff_2} to vanish identically. This implies that the Christoffel symbols vanish:
		\begin{equation}
			\Gamma_{ijk}^{(\alpha)}(\theta_0) \equiv 0, \quad \forall \alpha \in [-1, 1],
		\end{equation}
		which in turn forces the Amari-Riemann curvature tensor to vanish identically:
		\begin{equation}
			\mathcal{R}_{ijmn}^{(\alpha)}(\theta_0) \equiv 0.
		\end{equation}
		Hence, parallel transport is path-independent and flat ($\nabla = \nabla^{(0)}$), fulfilling Pillar III.
	\end{enumerate}
	
	\paragraph{\textbf{Sufficiency ($\text{CS}_{\text{geom}} \implies \text{CS}_{\text{stat}}$):}}
	Assume $\mathcal{M}_{\text{stat}} \equiv (\mathcal{N}, g, \nabla)$ fulfills Pillars I, II, and III of Definition~\ref{def:cs_geom_axioms}:
	\begin{enumerate}[label=(\alph*)]
		\item By Pillar I ($\mathcal{N} \cong \mathbb{R}^d$), local displacements are vectors $h \in \mathbb{R}^d$. By Theorem~\ref{thm:lan_linear_quadratic_proof}, any smooth probability density process on a flat tangent workspace admits the expansion $\Lambda_N(h) = h^\top \Delta_N(\theta_0) - \frac{1}{2} h^\top g_0 h + o_p(1)$, satisfying principle (i).
		\item By Pillar II ($g(h) \equiv g_0$), the metric capacity is position-invariant, guaranteeing that the limit experiment is the Gaussian shift experiment $\mathcal{E}_\infty = (\mathbb{R}^d, \mathcal{B}^d, \{\mathcal{N}(h, g_0^{-1})\})$, satisfying principle (ii).
		\item By Pillar III ($\Gamma_{ijk} \equiv 0, \mathcal{R}_{ijmn} \equiv 0$), parallel transport is flat. The geodesic distances along all $\alpha$-connections coincide identically with Euclidean distance under $g_0$, forcing $W_S = W_{LR} = W_W + o_p(1)$, satisfying principle (iii).
	\end{enumerate}
	Therefore, $\text{CS}_{\text{stat}} \iff \text{CS}_{\text{geom}}$.
\end{proof}

\begin{lemma}[Geometric Isomorphism Lemma: $\text{CS}_{\text{geom}} \iff \mathcal{M}_\infty$]
	\label{lem:lemma5_isomorphism}
~

	A localized statistical workspace triple $\mathcal{M}_{\mathrm{stat}} \equiv (\mathcal{N}, g, \nabla)$ centered at $\theta_0 \in \mathrm{Int}(\Theta)$ fulfills the three Constitutional Pillars of Geometric Conventional Statistics $\text{CS}_{\text{geom}}$ (Definition 6) if and only if $\mathcal{M}_{\mathrm{stat}}$ is isometric and affine-isomorphic to the canonical asymptotic flat tangent space $\mathcal{M}_\infty \equiv \left(T_{\theta_0}IG_1, g_0, \nabla^{(0)}\right)$ (Definition 5).
\end{lemma}

\begin{proof}
	We construct a step-by-step bi-directional proof covering domain diffeomorphism, metric tensor pullback isometry, and affine connection structure preservation.
	
	\subsubsection*{Part 1: Sufficiency ($\text{CS}_{\text{geom}} \implies \mathcal{M}_\infty$)}
	
	Assume $\mathcal{M}_{\mathrm{stat}} \equiv (\mathcal{N}, g, \nabla)$ satisfies Pillars I, II, and III of Definition 6.
	
	\paragraph{\textbf{Step 1: Construction of the Global Diffeomorphism $\Phi$.}}
	By Pillar I (Flat Tangent Action), the underlying domain $\mathcal{N}$ acts on local displacement vectors $h = (h^1, \dots, h^d)^\top \in \mathbb{K} \subset \mathbb{R}^d$. By Proposition 1 (Linear Vector Space Derivation Isomorphism), there exists a unique linear derivation isomorphism $\Psi_\phi^{-1}: \mathbb{R}^d \to T_{\theta_0}IG_1$ given by:
	\begin{equation}
		\Psi_\phi^{-1}(h) \coloneqq \sum_{i=1}^d h^i \left.\frac{\partial}{\partial \theta^i}\right|_{\theta_0} \in T_{\theta_0}IG_1.
		\label{eq:derivation_map_proof}
	\end{equation}
	Define the candidate isomorphism map $\Phi: \mathcal{N} \to T_{\theta_0}IG_1$ pointwise by $\Phi(h) \coloneqq \Psi_\phi^{-1}(h)$. Because $\Phi$ is a linear bijection between finite-dimensional vector spaces of identical dimension $d = \dim(\Theta)$, its Jacobian matrix in standard coordinate bases is the $d \times d$ identity matrix $\mathbf{I}_{d \times d}$. Consequently, $\Phi$ is smooth ($C^\infty$) with a smooth inverse $\Phi^{-1} = \Psi_\phi$, rendering $\Phi$ a global smooth diffeomorphism between manifolds $\mathcal{N}$ and $T_{\theta_0}IG_1$.
	
	\paragraph{\textbf{Step 2: Proof of Riemannian Isometry ($\Phi^* g_0 = g$).}}
	To prove that $\Phi$ is a Riemannian isometry, we must show that the pullback of the frozen Fisher metric tensor $g_0$ under $\Phi$ coincides identically with the metric tensor field $g$ on $\mathcal{N}$:
	\begin{equation}
		(\Phi^* g_0)_h(u, v) = g(h)(u, v), \quad \forall h \in \mathbb{K}, \; \forall u, v \in T_h \mathcal{N} \cong \mathbb{R}^d.
		\label{eq:isometry_goal}
	\end{equation}
	By definition of the dual covariant tensor pullback, for coordinate basis vector fields $\left\{ \frac{\partial}{\partial h^i} \right\}_{i=1}^d$ on $\mathcal{N}$:
	\begin{equation}
		(\Phi^* g_0)_h \left( \frac{\partial}{\partial h^i}, \frac{\partial}{\partial h^j} \right) \coloneqq (g_0)_{\Phi(h)} \left( d\Phi_h \left(\frac{\partial}{\partial h^i}\right), \, d\Phi_h \left(\frac{\partial}{\partial h^j}\right) \right).
		\label{eq:pullback_def_step}
	\end{equation}
	Applying the differential pushforward $d\Phi_h$ to the derivation basis yields:
	\begin{equation}
		d\Phi_h \left(\frac{\partial}{\partial h^i}\right) = \left.\frac{\partial}{\partial \theta^i}\right|_{\theta_0} \in T_{\theta_0}IG_1.
		\label{eq:pushforward_basis}
	\end{equation}
	Substituting Equation~\eqref{eq:pushforward_basis} into Equation~\eqref{eq:pullback_def_step} and applying Definition 2 (Metric-Induced Inner Product Space):
	\begin{equation}
		(\Phi^* g_0)_h \left( \frac{\partial}{\partial h^i}, \frac{\partial}{\partial h^j} \right) = g_0 \left( \left.\frac{\partial}{\partial \theta^i}\right|_{\theta_0}, \left.\frac{\partial}{\partial \theta^j}\right|_{\theta_0} \right) \coloneqq g_{ij}^{(1)}(\theta_0).
		\label{eq:pullback_component_eval}
	\end{equation}
	By Pillar II (Frozen Metric Field), the metric field $g(h)$ on $\mathcal{N}$ is spatially uniform across $\mathbb{K}$ and satisfies $g_{ij}(h) \equiv g_{ij}^{(1)}(\theta_0)$ for all $h \in \mathbb{K}$. Therefore:
	\begin{equation}
		(\Phi^* g_0)_h \left( \frac{\partial}{\partial h^i}, \frac{\partial}{\partial h^j} \right) = g_{ij}(h), \quad \forall i, j \in \{1, \dots, d\},
		\label{eq:metric_identity_proven}
	\end{equation}
	which proves that $\Phi^* g_0 = g$ everywhere on $\mathcal{N}$. Hence, $\Phi: (\mathcal{N}, g) \to \left(T_{\theta_0}IG_1, g_0\right)$ is an exact Riemannian isometry.
	
	\paragraph{\textbf{Step 3: Proof of Affine Isomorphism ($\Phi_* (\nabla_X Y) = \nabla^{(0)}_{\Phi_* X} (\Phi_* Y)$).}}
	Let $X = X^i \frac{\partial}{\partial h^i}$ and $Y = Y^j \frac{\partial}{\partial h^j}$ be smooth vector fields on $\mathcal{N}$. The general coordinate expression for the affine connection $\nabla$ on $\mathcal{N}$ is given by:
	\begin{equation}
		\nabla_X Y = \left( X^i \frac{\partial Y^k}{\partial h^i} + X^i Y^j \Gamma_{ij}^k(h) \right) \frac{\partial}{\partial h^k},
		\label{eq:general_affine_connection}
	\end{equation}
	where $\Gamma_{ij}^k(h)$ are the Christoffel symbols of $\nabla$. By Pillar III (Flat Affine Parallel Transport), the Christoffel symbols vanish identically across $\mathcal{N}$:
	\begin{equation}
		\Gamma_{ij}^k(h) \equiv 0, \quad \forall h \in \mathbb{K}, \; \forall i, j, k \in \{1, \dots, d\}.
		\label{eq:pillar_iii_zero_christoffel}
	\end{equation}
	Thus, Equation~\eqref{eq:general_affine_connection} simplifies strictly to the directional derivative:
	\begin{equation}
		\nabla_X Y = \left( X^i \frac{\partial Y^k}{\partial h^i} \right) \frac{\partial}{\partial h^k}.
		\label{eq:simplified_connection}
	\end{equation}
	Pushing forward $\nabla_X Y$ under $\Phi$ via Equation~\eqref{eq:pushforward_basis}:
	\begin{equation}
		\Phi_* (\nabla_X Y) = \left( X^i \frac{\partial Y^k}{\partial h^i} \right) \left.\frac{\partial}{\partial \theta^k}\right|_{\theta_0}.
		\label{eq:pushed_forward_connection}
	\end{equation}
	Now consider the flat Levi-Civita connection $\nabla^{(0)}$ on $T_{\theta_0}IG_1$. By Definition \ref{def:m_infty_canonical}, its Christoffel symbols in Cartesian derivation coordinates are identically zero ($\Gamma_{ij}^{(0)k} \equiv 0$). Evaluating $\nabla^{(0)}$ on the pushed-forward vector fields $\Phi_* X = X^i \left.\frac{\partial}{\partial \theta^i}\right|_{\theta_0}$ and $\Phi_* Y = Y^j \left.\frac{\partial}{\partial \theta^j}\right|_{\theta_0}$:
	\begin{equation}
		\nabla^{(0)}_{\Phi_* X} (\Phi_* Y) = \left( X^i \frac{\partial Y^k}{\partial h^i} \right) \left.\frac{\partial}{\partial \theta^k}\right|_{\theta_0}.
		\label{eq:target_flat_connection_eval}
	\end{equation}
	Comparing Equation~\eqref{eq:pushed_forward_connection} and Equation~\eqref{eq:target_flat_connection_eval} establishes:
	\begin{equation}
		\Phi_* (\nabla_X Y) = \nabla^{(0)}_{\Phi_* X} (\Phi_* Y), \quad \forall X, Y \in \mathfrak{X}(\mathcal{N}).
		\label{eq:affine_isomorphism_proven}
	\end{equation}
	Furthermore, the Riemann curvature tensor associated with $\nabla$ evaluates to:
	\begin{equation}
		\mathcal{R}(X, Y)Z \coloneqq \nabla_X \nabla_Y Z - \nabla_Y \nabla_X Z - \nabla_{[X, Y]} Z \equiv 0,
		\label{eq:zero_curvature_proven}
	\end{equation}
	matching the zero curvature tensor $\mathcal{R}^{(0)} \equiv 0$ of $\mathcal{M}_\infty$. Thus, $\Phi$ is an affine isomorphism.
	
	Combining Steps 1, 2, and 3 proves that $\mathcal{M}_{\mathrm{stat}} \equiv (\mathcal{N}, g, \nabla)$ is isometric and affine-isomorphic to $\mathcal{M}_\infty \equiv \left(T_{\theta_0}IG_1, g_0, \nabla^{(0)}\right)$.

	\subsubsection*{Part 2: Necessity ($\mathcal{M}_\infty \implies \text{CS}_{\text{geom}}$)}
	
	Conversely, assume $\mathcal{M}_{\mathrm{stat}}$ is isometric and affine-isomorphic to $\mathcal{M}_\infty \equiv \left(T_{\theta_0}IG_1, g_0, \nabla^{(0)}\right)$ via an isomorphism $\Phi$. We verify Pillars I, II, and III directly on $\mathcal{M}_\infty$:
	
	\begin{enumerate}[label=\arabic*.]
		\item \textbf{Verification of Pillar I (Flat Tangent Action):} 
		The base manifold of $\mathcal{M}_\infty$ is $T_{\theta_0}IG_1$. By Proposition \ref{prop:isomorphism}, $T_{\theta_0}IG_1 \cong \mathbb{R}^d$, acting on displacement vectors $h = \sqrt{N}(\theta - \theta_0) \in \mathbb{R}^d$. Hence, Pillar I holds.
		
		\item \textbf{Verification of Pillar II (Frozen Metric Field):} 
		The metric tensor of $\mathcal{M}_\infty$ is defined as $g_0 \coloneqq g^{(1)}(\theta_0)$, which is position-invariant across all displacement coordinates $h \in \mathbb{K} \subset \mathbb{R}^d$. Under the pullback $\Phi^*$, $g(h) \equiv g_0 = g^{(1)}(\theta_0)$ holds everywhere. Hence, Pillar II holds.
		
		\item \textbf{Verification of Pillar III (Flat Affine Parallel Transport):} 
		The connection $\nabla^{(0)}$ is the canonical flat Levi-Civita connection on $T_{\theta_0}IG_1$. In standard derivation coordinates, its Christoffel symbols vanish identically ($\Gamma_{ijk}^{(0)} \equiv 0$), and its Riemann curvature tensor vanishes identically ($\mathcal{R}_{ijmn}^{(0)} \equiv 0$). Hence, Pillar III holds.
	\end{enumerate}
	
	Combining Part 1 and Part 2 completes the full, step-by-step mathematical proof of this Lemma.
\end{proof}

\begin{theorem}[Epistemological Equivalence Theorem]
	\label{thm:theorem2_full}
	Under High-Order Local Regularity, the decision-theoretic operational paradigm $\text{CS}_{\text{stat}}$, the geometric constitutional universe $\text{CS}_{\text{geom}}$, and the canonical flat tangent space $\mathcal{M}_\infty$ are mutually equivalent:
	\begin{equation}
		\text{CS}_{\text{stat}} \iff \text{CS}_{\text{geom}} \iff \mathcal{M}_\infty.
	\end{equation}
\end{theorem}

\begin{proof}
	The proof follows directly from transitivity using Lemma \ref{lem:operational_geometric_bridge} and Lemma \ref{lem:lemma5_isomorphism}:
	\[
		\text{CS}_{\text{stat}} \iff [\text{Lemma } \ref{lem:operational_geometric_bridge}]{\quad \text{Operational Bridge} \quad} \text{CS}_{\text{geom}}(\mathcal{M}_{\text{stat}}) 
		\]
\begin{equation}		
		\iff [\text{Lemma } \ref{lem:lemma5_isomorphism}]{\quad \text{Geometric Isomorphism} \quad} \mathcal{M}_\infty.
	\end{equation}
	Combining both lemmas closes the epistemological loop, establishing that Conventional Statistics is mathematically identical to the flat tangent space $\mathcal{M}_\infty$.
\end{proof}


	\subsubsection{Statistical Implications of Epistemological Equivalence}
	\label{subsec:statistical_implications}
	
	The full epistemological equivalence established in Theorem~\ref{thm:theorem2_full} 
	($\text{CS}_{\text{stat}} \iff \text{CS}_{\text{geom}} \iff \mathcal{M}_\infty$) 
	transcends a purely formal unification. By establishing that classical decision-theoretic 
	statistics is mathematically identical to the flat, metric-frozen tangent space $\mathcal{M}_\infty$, 
	we uncover three fundamental statistical implications governing asymptotic inference:
	
	\begin{description}[leftmargin=*]
		\item[\textbf{Implication I: Geometric Mechanism of Le Cam's LAN Limits}] 
		In classical asymptotic decision theory, Le Cam's Local Asymptotic Normality (LAN) paradigm 
		asserts that regular statistical experiments converge locally to Gaussian shift experiments 
		$\mathcal{N}(h, g_0^{-1})$. In Amari's information geometry, exponential families are $e$-flat 
		($\alpha=1$) and mixture families are $m$-flat ($\alpha=-1$). Theorem~\ref{thm:theorem2_full} 
		demonstrates that a Gaussian shift experiment with fixed covariance $g_0^{-1}$ is 
		\textbf{simultaneously $e$-flat and $m$-flat} ($\nabla^{(\alpha)}$-flat for all $\alpha \in [-1, 1]$) 
		with identically vanishing Amari-Riemann curvature ($\mathcal{R}^{(\alpha)} \equiv 0$). 
		Thus, Pillar III ($\Gamma \equiv 0, \mathcal{R} \equiv 0$) provides the exact information-geometric 
		characterization of Le Cam's limiting experiment.
		
		\item[\textbf{Implication II: Complete Elimination of Second-Order Information Loss}] 
		In asymptotic estimation theory, first-order efficiency (attaining the Cram\'er-Rao bound) 
		depends strictly on the point evaluation of the metric tensor $g_0 \equiv g^{(1)}(\theta_0)$. 
		Spatial metric variations $\nabla g \neq 0$ do not alter first-order asymptotic variance; 
		rather, as shown by Efron and Amari, they generate exponential and mixture statistical 
		curvatures $(\gamma_e, \gamma_m)$ that dictate second-order information loss:
		\begin{equation}
			\Delta I_N \equiv I_N(\theta_0) - I_{\mathrm{Fisher}}(\hat{\theta}_N) 
			= \gamma_e^2 + \frac{1}{2}\gamma_m^2 + \mathcal{O}(N^{-1}).
			\label{eq:second_order_info_loss_impl}
		\end{equation}
		Because Theorem~\ref{thm:theorem2_full} forces the metric field to be spatially frozen 
		$g(h) \equiv g_0$ across contiguous neighborhoods (Pillar II), spatial derivatives vanish 
		$(\nabla g = 0)$, setting $\gamma_e = \gamma_m = 0$. Consequently, first-order efficient 
		estimators in $\text{CS}_{\text{stat}}$ suffer \textbf{zero second-order information loss} 
		$(\Delta I_N = 0)$ precisely because the target workspace is $\mathcal{M}_\infty$.
		
		\item[\textbf{Implication III: Unconditional Identity of the Classical Test Trinity}] 
		The three pillars of classical hypothesis testing---Rao's Score Test ($W_S$), Wilks' 
		Likelihood Ratio Test ($W_{LR}$), and Wald's Test ($W_W$)---are asymptotically equivalent. 
		Information-geometrically, these three test statistics correspond to measuring divergence 
		distances along different $\alpha$-geodesics ($\alpha = 1$ for Score, $\alpha = 0$ for LRT, 
		and $\alpha = -1$ for Wald). On a curved finite-sample statistical manifold, these tests 
		diverge due to non-zero Christoffel symbols $\Gamma_{ijk}^{(\alpha)}$. Theorem~\ref{thm:theorem2_full} 
		proves that the classical test trinity becomes \textbf{unconditionally identical} 
		($W_S = W_{LR} = W_W + o_p(1)$) if and only if the underlying workspace exhibits flat affine 
		parallel transport ($\Gamma_{ijk} \equiv 0$).
	\end{description}
	

	\section{The Localized Microscope Map $\psi_N$, Jacobian Scaffolding, and Cotangent Pullbacks}
	\label{sec:module2}
In this section, we will develop the localized dilation microscope map $\psi_N$, the multivariable Jacobian scaffold $J(\psi_N)$, and the dual cotangent pullback operator $\psi_N^*$. We demonstrate how the microscope map resolves ${R}^d, g_0)$ from the dynamically contracting and diverging information manifold sequence $IG_N = (\Theta, G^{(N)}, \nabla^{(\alpha,N)})$. We prove the \textbf{Universal Tensor Valence Scaling Law} $N^{1-r/2}$ for all $r$-covariant $i.i.d.$ additive tensor fields. Applying this fundamental law across vthe \emph{Moving Target Crisis} by decoupling the stationary observer canvas $\mathcal{M}_\infty \equiv ( r \in \{1, 2, 3, 4\}$, we establish the exact, step-by-step mathematical degeneration rules governing the four core information-geometric fields: score 1-form fluctuation ($r=1$, scaling $N^{1/2}$), Fisher metric stabilization ($r=2$, scaling $N^0=1$), Amari connection dissolution ($r=3$, scaling $N^{-1/2}$), and accelerated quadratic Riemann curvature annihilation ($r=4$, scaling $N^{-1}$). All results build continuously upon the regularity axioms and epistemological framework established in Section~\ref{sec:module1}.
	
	In Section~\ref{sec:module1}, we established the ambient statistical universe $\mathcal{M}_{\text{univ}}$ (Definition~\ref{def:ambient_spaces}), High-Order Local Regularity (Assumption~\ref{asm:local_regularity}), the natural derivation tangent space isomorphism $T_{\theta_0} IG_1 \cong \mathbb{R}^d$ (Proposition~\ref{prop:isomorphism}), and the Epistemological Equivalence Theorem asserting that Conventional Statistics ($\text{CS}$) is mathematically identical to the flat canonical tangent canvas $\mathcal{M}_\infty \equiv \left( T_{\theta_0} IG_1, g_0, \nabla^{(0)} \right)$ (Theorem~\ref{thm:theorem2_full}). 
	
	To construct the global phase transition $IG_N \to \mathcal{M}_\infty$, we must analyze the asymptotic behavior of the $N$-sample joint information manifold sequence $IG_N = (\Theta, G^{(N)}, \nabla^{(\alpha,N)})$ as $N \to \infty$. However, direct analysis on the fixed parameter domain $\Theta$ encounters severe analytical barriers. In this section, we construct the localized dilation microscope map $\psi_N$, derive its multivariable Jacobian scaffold $J(\psi_N)$, establish the dual cotangent pullback operator $\psi_N^*$, and prove the exact degeneration laws for all four fundamental information-geometric fields.

	\subsection{Resolution of the Moving Target Crisis}
	\label{subsec:mod2_moving_target}
	
	Under $i.i.d.$ product sampling from a dominated family $\mathcal{P}$, direct asymptotic evaluation of the sequence $IG_N = (\Theta, G^{(N)}, \nabla^{(\alpha,N)})$ encounters an intrinsic, dual-faceted analytical impasse defined as the \textbf{Moving Target Crisis}:
	
	\begin{enumerate}[label=\arabic*.]
		\item \textbf{Domain Collapse (Microscopic Signal Shrinkage):} By the Law of Large Numbers and the Central Limit Theorem, active sample probability mass concentrates around the background true parameter state $\theta_0 \in \text{Int}(\Theta)$ into a rapidly contracting topological coordinate neighborhood $\mathcal{U}_N(\theta_0) \subset \Theta$ whose geometric diameter shrinks at rate $\mathcal{O}_p(N^{-1/2})$. As $N \to \infty$, the active parameter domain contracts topologically to a single singular point $\{\theta_0\}$.
		
		\item \textbf{Metric Divergence (Macro-Capacity Explosion):} By the additivity lemma for $i.i.d.$ product measures, the joint $N$-sample Fisher Information Metric tensor field satisfies $G_{ab}^{(N)}(\theta) = N \cdot  g_{ab}^{(1)}(\theta)$ for all $\theta \in \Theta$. Consequently, every component of the Fisher metric tensor diverges to infinity across the domain at linear rate $\mathcal{O}(N)$:
		\begin{equation}
			\lim_{N \to \infty} G_{ab}^{(N)}(\theta) = \infty, \quad \forall \theta \in \mathcal{U}_N(\theta_0).
			\label{eq:metric_divergence_impasse}
		\end{equation}
	\end{enumerate}
	
	Taking pointwise limits $\lim_{N \to \infty} G^{(N)}(\theta)$ directly on the unscaled parameter domain $\Theta$ fails because the domain contracts to a zero-dimensional point while the metric tensor components blow up. 
	
	We resolve this crisis by decoupling the stationary observer tracking canvas $\mathcal{M}_\infty \equiv (\mathbb{R}^d, g_0)$ from the dynamic space $IG_N$ via a localized dilation microscope map $\psi_N$, transporting covariant geometric fields in reverse from $IG_N$ back to $\mathcal{M}_\infty$ via the dual cotangent pullback operator $\psi_N^*$.
	
	\subsection{Microscope Map Mechanics and Jacobian Scaffold}
	\label{subsec:mod2_mechanics}
	
	Let $\mathcal{M}_\infty \equiv (T_{\theta_0} IG_1, g_0, \nabla^{(0)})$ be the canonical flat tangent space anchored at $\theta_0 \in \text{Int}(\Theta)$, equipped with the frozen metric $g_0 \equiv g^{(1)}(\theta_0)$ and standard Cartesian coordinates $h = (h^1, \dots, h^d)^T \in \mathbb{R}^d$ via the natural derivation isomorphism $\Psi_\phi$ (Proposition~\ref{prop:isomorphism}).
	
	\begin{definition}[Localized Dilation Microscope Map]
		\label{def:microscope_map}
		Let $\mathbb{K} \subset \mathcal{M}_\infty \cong (\mathbb{R}^d, g_0)$ be a fixed, $N$-invariant compact coordinate cage enclosing the origin $h = \mathbf{0}$. For each sample size $N \ge 1$, the forward \textbf{localized dilation microscope map} $\psi_N : \mathbb{K} \to \Theta$ is defined pointwise by:
		\begin{equation}
			\psi_N(h) \equiv \theta_0 + \frac{h}{\sqrt{N}} \in \Theta \subset \mathbb{R}^d.
			\label{eq:microscope_map_def}
		\end{equation}
		The image $\psi_N(\mathbb{K}) \equiv \mathcal{U}_N(\theta_0) \subset \Theta$ defines a contracting coordinate neighborhood centered at $\theta_0$ with diameter $\text{diam}(\mathcal{U}_N(\theta_0)) = \mathcal{O}(N^{-1/2})$.
	\end{definition}
	
	\begin{lemma}[Multivariable Spatial Jacobian Matrix Scaffold]
		\label{lem:jacobian_scaffold}
		The differential push-forward map $d\psi_N : T_h \mathcal{M}_\infty \to T_{\psi_N(h)} IG_N$ acts as an isotropic linear scaling transformation. In local Cartesian coordinates, its multivariable spatial Jacobian matrix $J(\psi_N) \in \mathbb{R}^{d \times d}$ is spatially uniform across $\mathbb{K}$ and given by:
		\begin{equation}
			J_i^a(h) \equiv \frac{\partial \psi_N^a(h)}{\partial h^i} = \frac{1}{\sqrt{N}} \delta_i^a, \quad \forall a, i \in \{1, \dots, d\}.
			\label{eq:jacobian_scaffold_formula}
		\end{equation}
	\end{lemma}
	
	\begin{proof}
		Represent component $a$ of the vector-valued map $\psi_N(h)$ explicitly in Cartesian coordinates:
		\begin{equation}
			\psi_N^a(h) = \theta_0^a + \frac{1}{\sqrt{N}} h^a.
			\label{eq:psi_component_exp}
		\end{equation}
		Differentiating $\psi_N^a(h)$ with respect to the local spatial displacement coordinate $h^i$:
		\begin{equation}
			J_i^a(h) = \frac{\partial \psi_N^a(h)}{\partial h^i} = \frac{\partial}{\partial h^i} \left( \theta_0^a + \frac{1}{\sqrt{N}} h^a \right) = \frac{\partial \theta_0^a}{\partial h^i} + \frac{1}{\sqrt{N}} \frac{\partial h^a}{\partial h^i}.
			\label{eq:jacobian_derivation_step1}
		\end{equation}
		Because the background parameter $\theta_0$ is a fixed constant vector, $\frac{\partial \theta_0^a}{\partial h^i} = 0$. By definition of Cartesian coordinates, $\frac{\partial h^a}{\partial h^i} = \delta_i^a$, where $\delta_i^a$ is the Kronecker delta. Substituting these back into Equation~\eqref{eq:jacobian_derivation_step1}:
		\begin{equation}
			J_i^a(h) = 0 + \frac{1}{\sqrt{N}} \delta_i^a = \frac{1}{\sqrt{N}} \delta_i^a.
			\label{eq:jacobian_derivation_step2}
		\end{equation}
		In matrix notation, $J(\psi_N) = \frac{1}{\sqrt{N}} \mathbf{I}_{d \times d}$, where $\mathbf{I}_{d \times d}$ is the $d \times d$ identity matrix. Because this matrix is completely independent of $h$, the Jacobian scaffold is spatially uniform across the entire canvas $\mathbb{K}$.
	\end{proof}

\begin{proposition}[Resolution Mechanics of the Microscope Map]
	\label{prop:resolution_mechanics}
	The microscope map $\psi_N : \mathcal{M}_\infty \to IG_N$ resolves the Moving Target Crisis through a dual mathematical mechanism:
	\begin{enumerate}[label=\roman*.]
		\item \textbf{Domain Fixation:} The inverse map $\psi_N^{-1}(\theta) = \sqrt{N}(\theta - \theta_0)$ maps the shrinking parameter neighborhood $\mathcal{U}_N(\theta_0) \subset \Theta$ back onto the stationary, $N$-invariant coordinate cage $\mathbb{K} \subset \mathcal{M}_\infty$, eliminating domain collapse.
		
		\item \textbf{Exact Dilation Matching and SNR Constancy:} The spatial expansion factor $\sqrt{N}$ of the coordinate transform $\psi_N^{-1}$ and the dual cotangent pullback operator $\psi_N^*$ pull back the diverging joint Fisher metric $G^{(N)}$ to a stabilized metric tensor $\widetilde{G}^{(N)} \equiv \psi_N^* G^{(N)}$ on $\mathcal{M}_\infty$. This counteracts the probability mass concentration speed $\mathcal{O}(N^{-1/2})$, maintaining a constant Signal-to-Noise Ratio ($\mathrm{SNR}$) across $N$:
		\begin{equation}
			\small 
			\mathrm{SNR}_N(h) \sim \sqrt{\text{Capacity} \times \text{Displacement}^2} = \sqrt{\mathcal{O}(N) \times \left(\frac{h}{\sqrt{N}}\right)^2} = \sqrt{\mathcal{O}(N) \times \mathcal{O}\left(\frac{1}{N}\right)} = \mathcal{O}(1).
			\label{eq:snr_constancy_revised}
		\end{equation}
	\end{enumerate}
\end{proposition}

\begin{proof}
	Under $G^{(N)}(\theta) = N \cdot g^{(1)}(\theta)$, the intrinsic Riemannian geodesic distance between $\theta_0$ and a local parameter state $\theta = \theta_0 + \delta\theta$ on $IG_N$ expands as:
	\begin{equation}
		d_{G^{(N)}}(\theta_0, \theta_0 + \delta\theta) = \sqrt{\delta\theta^T G^{(N)}(\theta_0) \delta\theta} + o(\|\delta\theta\|) = \sqrt{N}\sqrt{\delta\theta^T g^{(1)}(\theta_0)\delta\theta} + o(\|\delta\theta\|).
		\label{eq:geodesic_expansion_ign_revised}
	\end{equation}
	Evaluating this distance for local contiguous alternatives $\delta\theta = \psi_N(h) - \theta_0 = \frac{h}{\sqrt{N}}$, the intrinsic Riemannian distance under the pulled-back Fisher metric field $\widetilde{G}^{(N)} \equiv \psi_N^* G^{(N)}$ on $\mathcal{M}_\infty$ satisfies:
	\begin{equation}
		\small 
		d_{\psi_N^* G^{(N)}}(\mathbf{0}, h) \equiv d_{G^{(N)}}\left(\theta_0, \theta_0 + \frac{h}{\sqrt{N}}\right) = \sqrt{N}\sqrt{\left(\frac{h}{\sqrt{N}}\right)^T g^{(1)}(\theta_0)\left(\frac{h}{\sqrt{N}}\right)} + o\left(\frac{1}{\sqrt{N}}\right) = \sqrt{h^T g_0 h} + o(1).
		\label{eq:pulled_back_distance_proof_revised}
	\end{equation}
	Thus, the pulled-back distance $d_{\psi_N^* G^{(N)}}(\mathbf{0}, h)$ on $\mathcal{M}_\infty$ is $\mathcal{O}(1)$ and non-degenerate for every fixed $h \in \mathbb{K}$, proving exact dilation matching.
\end{proof}

	\subsection{Cotangent Pullback Mechanics and Universal Scaling Laws}
	\label{subsec:mod2_pullbacks}
	
Having constructed the forward microscope diffeomorphism $\psi_N : \mathcal{M}_\infty \to IG_N$ (which maps the compact coordinate cage $\mathbb{K} \subset \mathcal{M}_\infty$ onto the contracting neighborhood $\mathcal{U}_N(\theta_0) \subset IG_N$), we employ its dual pullback operator $\psi_N^*$ to transport covariant geometric fields ({\bf tensors, differential forms, connections, and curvatures}) in reverse from $IG_N$ back to $\mathcal{M}_\infty$.
	
	\begin{definition}[Dual Cotangent Pullback Operator]
		\label{def:cotangent_pullback}
		Let $T^{(N)}$ be a purely $r$-covariant tensor field defined on $IG_N$. The \textbf{dual cotangent pullback operator} \footnote{For a rigorous geometric treatment of the general pullback operator on covariant tensor fields and differential forms under smooth mappings, see Lee~\cite{lee2013}, Chapters 12 and 13.}
		\begin{equation}
			\psi_N^* : \Gamma\left( \bigotimes_{m=1}^r T^* IG_N \right) \longrightarrow \Gamma\left( \bigotimes_{m=1}^r T^* \mathcal{M}_\infty \right)
			\label{eq:pullback_operator_mapping}
		\end{equation}
		constructs a pulled-back $r$-covariant tensor field $\widetilde{T}^{(N)} \equiv \psi_N^* T^{(N)}$ on $\mathcal{M}_\infty$ whose component representation at $h \in \mathbb{K}$ is defined by contracting $T^{(N)}$ against $r$ copies of the multivariable Jacobian scaffold matrix $J_i^a(h)$:
		\begin{equation}
			\left[\psi_N^* T^{(N)}\right]_{i_1 i_2 \dots i_r}(h) \equiv \sum_{a_1, a_2, \dots, a_r = 1}^d \left( \prod_{m=1}^r J_{i_m}^{a_m}(h) \right) T_{a_1 a_2 \dots a_r}^{(N)}\left( \psi_N(h) \right).
			\label{eq:pullback_tensor_component_def}
		\end{equation}
	\end{definition}
	
	\begin{theorem}[Universal Tensor Valence Scaling Law $N^{1-r/2}$]
		\label{thm:universal_scaling_law}
		Let $T^{(N)}$ be an $r$-covariant joint tensor field on $IG_N$ satisfying $i.i.d.$ linear additivity $T^{(N)}(\theta) = N \cdot T^{(1)}(\theta)$, where $T^{(1)}$ is the single-sample tensor field on $IG_1$. The pulled-back tensor field $\widetilde{T}^{(N)} \equiv \psi_N^* T^{(N)}$ on $\mathcal{M}_\infty$ obeys the universal asymptotic scaling law:
		\begin{equation}
			\widetilde{T}_{i_1 i_2 \dots i_r}^{(N)}(h) = N^{1 - \frac{r}{2}} \cdot T_{i_1 i_2 \dots i_r}^{(1)}\left( \theta_0 + \frac{h}{\sqrt{N}} \right).
			\label{eq:universal_scaling_formula}
		\end{equation}
	\end{theorem}
	
	\begin{proof}
		Substitute the Jacobian scaffold matrix $J_{i_m}^{a_m}(h) = \frac{1}{\sqrt{N}} \delta_{i_m}^{a_m}$ (Lemma~\ref{lem:jacobian_scaffold}) and the additivity relation $T_{a_1 \dots a_r}^{(N)}(\theta) = N T_{a_1 \dots a_r}^{(1)}(\theta)$ into Definition~\ref{def:cotangent_pullback}:
		\begin{align}
			\left[\psi_N^* T^{(N)}\right]_{i_1 i_2 \dots i_r}(h) &= \sum_{a_1, a_2, \dots, a_r = 1}^d \left( \prod_{m=1}^r \frac{1}{\sqrt{N}} \delta_{i_m}^{a_m} \right) \left[ N \cdot  T_{a_1 a_2 \dots a_r}^{(1)}\left( \theta_0 + \frac{h}{\sqrt{N}} \right) \right] \label{eq:scaling_proof_step1} \\
			&= \left( \prod_{m=1}^r \frac{1}{\sqrt{N}} \right) \sum_{a_1, \dots, a_r = 1}^d \left( \prod_{m=1}^r \delta_{i_m}^{a_m} \right) \left[ N \cdot T_{a_1 \dots a_r}^{(1)}\left( \theta_0 + \frac{h}{\sqrt{N}} \right) \right] \label{eq:scaling_proof_step2} \\
			&= \left( \frac{1}{(\sqrt{N})^r} \right) \cdot N \cdot T_{i_1 i_2 \dots i_r}^{(1)}\left( \theta_0 + \frac{h}{\sqrt{N}} \right) \label{eq:scaling_proof_step3} \\
			&= N^{1} \cdot N^{-\frac{r}{2}} \cdot T_{i_1 i_2 \dots i_r}^{(1)}\left( \theta_0 + \frac{h}{\sqrt{N}} \right) \label{eq:scaling_proof_step4} \\
			&= N^{1 - \frac{r}{2}} \cdot T_{i_1 i_2 \dots i_r}^{(1)}\left( \theta_0 + \frac{h}{\sqrt{N}} \right). \label{eq:scaling_proof_step5}
		\end{align}
		Equation~\eqref{eq:scaling_proof_step5} establishes the theorem.
	\end{proof}
	
	\begin{remark}
		The exponent $1 - r/2$ reflects a precise physical competition between two opposing geometric forces:
		\begin{itemize}
			\item The factor $N^1$ represents \textbf{sample information capacity growth}, scaling linearly with sample size $N$ under $i.i.d.$ observation accumulation.
			\item The factor $N^{-r/2} = (N^{-1/2})^r$ represents \textbf{spatial cotangent contraction}, arising from contracting $r$ dual cotangent vector slots against the Jacobian scaffold $J(\psi_N) = N^{-1/2} \mathbf{I}$.
		\end{itemize}
	\end{remark}
	
	\subsection{Systematic Degeneration of Four Geometric Fields}
	\label{subsec:mod2_degeneration}
	
	Applying Theorem~\ref{thm:universal_scaling_law} across valences $r \in \{1, 2, 3, 4\}$ yields the exact, deterministic degeneration rules governing all four fundamental information-geometric fields.
	
	\subsubsection{Object 1: Score 1-Form Transformation ($r=1$, Scaling $N^{1/2}$)}

\begin{lemma}[Distributional Limit of the Score $1$-Form at the Origin]
	Let $dL^{(N)} \equiv \sum_{n=1}^N d\log p(X_n; \theta)$ be the joint log-likelihood score $1$-form under $i.i.d.$ sampling. At $h = \mathbf{0}$, the pulled-back score $1$-form field $\psi_N^*(dL^{(N)})(0)$ on $T_{\mathbf{0}}^* \mathcal{M}_\infty$ converges in distribution to the Gaussian $1$-form field:
	\begin{equation}
		\psi_N^*\left( dL^{(N)} \right)(0) \xrightarrow{d} \Delta_\infty \equiv \sum_{a=1}^d Z_a dh^a \sim \mathcal{N}\left(\mathbf{0}, g_0\right).
		\label{eq:score_limit_at_zero}
	\end{equation}
	\label{lm:normal_limit_lemma}
\end{lemma}

\begin{proof}
	The proof follows from four direct mathematical steps:
	
	\begin{enumerate}
		\item \textbf{Coefficient Vector Extraction:}
		Evaluating the pulled-back score $1$-form $\psi_N^*(dL^{(N)})(h) = \sum_{a=1}^d \left( \frac{1}{\sqrt{N}} \sum_{n=1}^N \frac{\partial \log p(X_n; \theta_0 + h/\sqrt{N})}{\partial \theta^a} \right) dh^a$ at $h = \mathbf{0}$ yields:
		\begin{equation}
			\psi_N^*\left( dL^{(N)} \right)(0) = \sum_{a=1}^d V_a^{(N)} dh^a, \quad \text{where } V^{(N)} \equiv \frac{1}{\sqrt{N}} \sum_{n=1}^N S_n \in \mathbb{R}^d,
		\end{equation}
		and $S_n \equiv \nabla_\theta \log p(X_n; \theta_0)$ denotes the single-sample score vector for observation $X_n$.
		
		\item \textbf{Moment Conditions of Single-Observation Scores:}
		Under High-Order Local Regularity, the $i.i.d.$ score vectors $\{S_n\}_{n=1}^N$ satisfy zero expectation (unbiasedness) and covariance equal to the single-sample Fisher Information Metric $g_0 \equiv g^{(1)}(\theta_0)$:
		\begin{equation}
			\mathbb{E}_{\theta_0}[S_n] = \mathbf{0}, \quad \text{and} \quad \mathrm{Var}_{\theta_0}(S_n) = \mathbb{E}_{\theta_0}\left[ S_n S_n^T \right] = g_0.
		\end{equation}
		
		\item \textbf{Multivariate Central Limit Theorem (CLT):}
		By the Classical Multivariate CLT applied to $V^{(N)} = \frac{1}{\sqrt{N}} \sum_{n=1}^N S_n$, the coefficient vector converges weakly in $\mathbb{R}^d$:
		\begin{equation}
			V^{(N)} \xrightarrow{d} Z \sim \mathcal{N}_d\left(\mathbf{0}, g_0\right).
		\end{equation}
		
		\item \textbf{Cotangent Space Isomorphism and Weak Limit:}
		Applying the Continuous Mapping Theorem under the linear isomorphism mapping coordinate vectors $Z \in \mathbb{R}^d$ to cotangent elements $\sum_{a=1}^d Z_a dh^a \in T_{\mathbf{0}}^* \mathcal{M}_\infty$, we establish the normal distribution relationship 
		\begin{equation}
			\psi_N^*\left( dL^{(N)} \right)(0) = \sum_{a=1}^d V_a^{(N)} dh^a \xrightarrow{d} \sum_{a=1}^d Z_a dh^a \equiv \Delta_\infty \sim \mathcal{N}\left(\mathbf{0}, g_0\right).
		\end{equation}
	\end{enumerate}
	This completes the proof.
\end{proof}
	
	\begin{theorem}[Object 1: Score 1-Form Fluctuation Field]
		\label{thm:score_transformation}
		Under High-Order Local Regularity (Assumption~\ref{asm:local_regularity}), the parameter differential 1-form $d\theta^a$ pulls back to $\mathcal{M}_\infty$ as $\psi_N^*(d\theta^a) = \frac{1}{\sqrt{N}} dh^a$. The joint log-likelihood score differential 1-form $dL^{(N)} \equiv \sum_{n=1}^N d \log p(X_n; \theta)$ pulls back as the normalized fluctuation field:
		\begin{equation}
			\psi_N^*\left( dL^{(N)} \right)(h) = \sum_{a=1}^d \left( \frac{1}{\sqrt{N}} \sum_{n=1}^N \frac{\partial \log p(X_n; \theta_0 + h/\sqrt{N})}{\partial \theta^a} \right) dh^a.
			\label{eq:score_pullback_explicit}
		\end{equation}
		At $h = \mathbf{0}$, $\psi_N^*(dL^{(N)})(0) \xrightarrow{d} \Delta_\infty \sim \mathcal{N}(\mathbf{0}, g_0)$, exhibiting normalized $\mathcal{O}_p(1)$ stochastic fluctuations matching Theorem~\ref{thm:lan_linear_quadratic_proof}.
	\end{theorem}
	
	\begin{proof}
		Setting $r=1$ in Theorem~\ref{thm:universal_scaling_law}, $N^{1 - 1/2} = N^{1/2} = \sqrt{N}$. For a single observation $X_n$, the score 1-form is $d\ell_n(\theta) = \sum_{a=1}^d \frac{\partial \log p(X_n;\theta)}{\partial \theta^a} d\theta^a$. Under $i.i.d.$ sampling, $dL^{(N)}(\theta) = \sum_{n=1}^N d\ell_n(\theta)$. Pulling back under $\psi_N(h) = \theta_0 + h/\sqrt{N}$:
		\begin{align}
			\psi_N^*\left( dL^{(N)} \right)(h) &= \sum_{n=1}^N \sum_{a=1}^d \frac{\partial \log p(X_n; \psi_N(h))}{\partial \theta^a} \psi_N^*(d\theta^a) \nonumber \\
			&= \sum_{n=1}^N \sum_{a=1}^d \frac{\partial \log p(X_n; \theta_0 + h/\sqrt{N})}{\partial \theta^a} \left( \frac{1}{\sqrt{N}} dh^a \right) \nonumber \\
			&= \sum_{a=1}^d \left( \frac{1}{\sqrt{N}} \sum_{n=1}^N \frac{\partial \log p(X_n; \theta_0 + h/\sqrt{N})}{\partial \theta^a} \right) dh^a.
		\end{align}
		
Evaluating the above equation at $h = \mathbf{0}$, the pulled-back score 1-form reduces to:
\begin{equation}
	\psi_N^*\left( dL^{(N)} \right)(0) = \sum_{a=1}^d \left( \frac{1}{\sqrt{N}} \sum_{n=1}^N \frac{\partial \log p(X_n; \theta_0)}{\partial \theta^a} \right) dh^a = \sum_{a=1}^d V_a^{(N)} dh^a,
	\label{eq:score_pullback_at_zero}
\end{equation}
where $V^{(N)} \equiv \frac{1}{\sqrt{N}} \sum_{n=1}^N S_n \in \mathbb{R}^d$ denotes the vector of component coefficients, and $S_n \equiv \nabla_\theta \log p(X_n; \theta_0)$ is the single-observation score vector.

Under the High-Order Local Regularity in Assumption  \ref{asm:local_regularity}, the independent and identically distributed ($i.i.d.$) random vectors $\{S_n\}_{n=1}^N$ satisfy:
\begin{enumerate}
	\item \textbf{Zero Expectation (Unbiasedness):} 
	\begin{equation}
		\mathbb{E}_{\theta_0}[S_n] = \int \nabla_\theta \log p(x; \theta_0) \, p(x; \theta_0) \, dx = \nabla_\theta \int p(x; \theta_0) \, dx = \mathbf{0}.
	\end{equation}
	\item \textbf{Covariance Structure (Fisher Information):} 
	\begin{equation}
		\operatorname{Var}_{\theta_0}(S_n) = \mathbb{E}_{\theta_0}\left[ S_n S_n^T \right] = g^{(1)}(\theta_0) \equiv g_0.
	\end{equation}
\end{enumerate}

Applying the classical Multivariate Central Limit Theorem to $V^{(N)} = \frac{1}{\sqrt{N}} \sum_{n=1}^N S_n$, we obtain weak convergence of the coefficient vector in $\mathbb{R}^d$:
\begin{equation}
	V^{(N)} \xrightarrow{d} Z \sim \mathcal{N}\left(\mathbf{0}, g_0\right).
\end{equation}

Finally, by the Continuous Mapping Theorem under the linear isomorphism mapping vector components in $\mathbb{R}^d$ to cotangent basis elements in $T_{\mathbf{0}}^* \mathcal{M}_\infty$, the random 1-form converges in distribution to the Gaussian 1-form field $\Delta_\infty \equiv \sum_{a=1}^d Z_a dh^a$, by Lemma \ref{lm:normal_limit_lemma}:
\begin{equation}
	\psi_N^*\left( dL^{(N)} \right)(0) = \sum_{a=1}^d V_a^{(N)} dh^a \xrightarrow{d} \sum_{a=1}^d Z_a dh^a \equiv \Delta_\infty \sim \mathcal{N}\left(\mathbf{0}, g_0\right),
\end{equation}
which completes the proof.

\end{proof}
	
	\subsubsection{Object 2: Fisher Metric Tensor Stabilization ($r=2$, Scaling $N^0 = 1$)}
	
	\begin{theorem}[Object 2: Uniform Fisher Metric Stabilization]
		\label{thm:metric_stabilization}
		Under High-Order Local Regularity (Assumption~\ref{asm:local_regularity}), the pulled-back joint Fisher Information Metric tensor field $\widetilde{G}^{(N)}(h) \equiv \psi_N^* G^{(N)}(h)$,  locally in $h$-coordinates on $\mathcal{M}_\infty$,  exhibits exact algebraic capacity-cotangent cancellation ($N^{1 - 2/2} = N^0 = 1$):
		\begin{equation}
			\widetilde{g}_{ij}^{(N)}(h) \equiv \left[\psi_N^* G^{(N)}\right]_{ij}(h) = g_{ij}^{(1)}\left( \theta_0 + \frac{h}{\sqrt{N}} \right).
			\label{eq:metric_cancellation_formula}
		\end{equation}
		Furthermore, $\widetilde{G}^{(N)}(h)$ converges uniformly over every compact coordinate cage $\mathbb{K} \subset \mathcal{M}_\infty$ to the constant, non-singular limit metric $g_0 \equiv g^{(1)}(\theta_0)$:
		\begin{equation}
			\lim_{N \to \infty} \sup_{h \in \mathbb{K}} \left\| \widetilde{G}^{(N)}(h) - g_0 \right\|_\infty = \mathcal{O}\left(N^{-1/2}\right) \longrightarrow 0.
			\label{eq:metric_uniform_stabilization_mod2}
		\end{equation}
	\end{theorem}
	
	\begin{proof}
		Setting $r=2$ in Theorem~\ref{thm:universal_scaling_law}, $N^{1 - 2/2} = N^0 = 1$. Contracting $G_{ab}^{(N)}(\theta) = N g_{ab}^{(1)}(\theta)$ against two copies of $J_i^a(h) = N^{-1/2} \delta_i^a$:
		\begin{equation}
			\widetilde{g}_{ij}^{(N)}(h) = \sum_{a, b = 1}^d \left( \frac{1}{\sqrt{N}} \delta_i^a \right) \left( \frac{1}{\sqrt{N}} \delta_j^b \right) \left[ N g_{ab}^{(1)}\left( \theta_0 + \frac{h}{\sqrt{N}} \right) \right] = g_{ij}^{(1)}\left( \theta_0 + \frac{h}{\sqrt{N}} \right).
			\label{eq:metric_cancellation_proof_step}
		\end{equation}
		Applying Lemma~\ref{lem:metric_freezing_proof} from Section~\ref{sec:module1}, the Taylor expansion around $h = \mathbf{0}$ yields:
		\begin{equation}
			\sup_{h \in \mathbb{K}} \left| \widetilde{g}_{ij}^{(N)}(h) - g_{ij}^{(1)}(\theta_0) \right| \le \frac{1}{\sqrt{N}} \left( \sup_{h \in \mathbb{K}} \|h\|_2 \right) \cdot \max_{i,j,k} \sup_{\theta \in \mathcal{U}(\theta_0)} \left| \frac{\partial g_{ij}^{(1)}(\theta)}{\partial \theta^k} \right| \le \frac{C_\mathbb{K} M_g}{\sqrt{N}}.
		\end{equation}
		Taking $N \to \infty$ proves uniform stabilization at rate $\mathcal{O}(N^{-1/2})$.
	\end{proof}
	
	\subsubsection{Object 3: Amari Affine Connection Dissolution ($r=3$, Scaling $N^{-1/2}$)}
	
	\begin{theorem}[Object 3: Amari Connection Dissolution]
		\label{thm:connection_dissolution}
		Under High-Order Local Regularity (Assumption~\ref{asm:local_regularity}), the pulled-back Amari $\alpha$-connection Christoffel symbols $\widetilde{\Gamma}_{ijk}^{(\alpha, N)}(h) \equiv \left[\psi_N^* \nabla^{(\alpha,N)}\right]_{ijk}(h)$ scale as $N^{1 - 3/2} = N^{-1/2}$:
		\begin{equation}
			\widetilde{\Gamma}_{ijk}^{(\alpha, N)}(h) = \frac{1}{\sqrt{N}} \Gamma_{ijk}^{(\alpha, 1)}\left( \theta_0 + \frac{h}{\sqrt{N}} \right) = \mathcal{O}\left(\frac{1}{\sqrt{N}}\right).
			\label{eq:connection_scaling_formula}
		\end{equation}
		As $N \to \infty$, connection friction dissolves uniformly to zero across $\mathbb{K}$, collapsing non-Euclidean parallel transport into flat Euclidean transport:
		\begin{equation}
			\lim_{N \to \infty} \sup_{h \in \mathbb{K}} \left| \widetilde{\Gamma}_{ijk}^{(\alpha, N)}(h) \right| = 0 \equiv \Gamma_{ijk}^{(0)}.
			\label{eq:connection_dissolution_limit}
		\end{equation}
	\end{theorem}
	
	\begin{proof}
		Setting $r=3$ in Theorem~\ref{thm:universal_scaling_law}, $N^{1 - 3/2} = N^{-1/2}$. Under $i.i.d.$ sampling, Amari's $\alpha$-connection Christoffel symbols of the first kind scale linearly: $\Gamma_{abc}^{(\alpha, N)}(\theta) = N \cdot  \Gamma_{abc}^{(\alpha, 1)}(\theta)$ \cite{amari1985}. Contracting against three copies of $J_i^a(h) = N^{-1/2} \delta_i^a$:
		\begin{align}
			\widetilde{\Gamma}_{ijk}^{(\alpha, N)}(h) &= \sum_{a, b, c = 1}^d \left( \frac{1}{\sqrt{N}} \delta_i^a \right) \left( \frac{1}{\sqrt{N}} \delta_j^b \right) \left( \frac{1}{\sqrt{N}} \delta_k^c \right) \left[ N \cdot \Gamma_{abc}^{(\alpha, 1)}\left( \theta_0 + \frac{h}{\sqrt{N}} \right) \right] \nonumber \\
			&= \frac{N}{N^{3/2}} \Gamma_{ijk}^{(\alpha, 1)}\left( \theta_0 + \frac{h}{\sqrt{N}} \right) = \frac{1}{\sqrt{N}} \Gamma_{ijk}^{(\alpha, 1)}\left( \theta_0 + \frac{h}{\sqrt{N}} \right).
		\end{align}
		By Assumption~\ref{asm:local_regularity}, $\Gamma_{ijk}^{(\alpha, 1)}(\theta)$ is continuous and uniformly bounded on $\overline{\mathcal{U}}(\theta_0)$ by $M_\Gamma < \infty$. Taking the uniform supremum over $\mathbb{K}$:
		\begin{equation}
			\sup_{h \in \mathbb{K}} \left| \widetilde{\Gamma}_{ijk}^{(\alpha, N)}(h) \right| \le \frac{M_\Gamma}{\sqrt{N}} = \mathcal{O}\left(N^{-1/2}\right) \xrightarrow{N \to \infty} 0.
		\end{equation}
		This establishes uniform dissolution of connection friction.
	\end{proof}
	
	\subsubsection{Object 4: Accelerated Quadratic Curvature Annihilation ($r=4$, Scaling $N^{-1}$)}
	
	\begin{theorem}[Object 4: Accelerated Quadratic Curvature Annihilation]
		\label{thm:curvature_annihilation}
		Under High-Order Local Regularity (Assumption~\ref{asm:local_regularity}), the pulled-back Riemann curvature tensor field $\widetilde{\mathcal{R}}_{ijmn}^{(\alpha, N)}(h) \equiv \left[\psi_N^* \mathcal{R}^{(\alpha,N)}\right]_{ijmn}(h)$ scales as $N^{1 - 4/2} = N^{-1}$:
		\begin{equation}
			\widetilde{\mathcal{R}}_{ijmn}^{(\alpha, N)}(h) = \frac{1}{N} \mathcal{R}_{ijmn}^{(\alpha, 1)}\left( \theta_0 + \frac{h}{\sqrt{N}} \right) = \mathcal{O}_p\left(\frac{1}{N}\right).
			\label{eq:curvature_scaling_formula}
		\end{equation}
		Riemann curvature undergoes accelerated quadratic decay ($\mathcal{O}_p(N^{-1})$)---decaying twice as fast as connection dissolution ($\mathcal{O}(N^{-1/2})$)---annihilating all intrinsic manifold curvature:
		\begin{equation}
			\lim_{N \to \infty} \sup_{h \in \mathbb{K}} \left\| \widetilde{\mathcal{R}}_{ijmn}^{(\alpha, N)}(h) \right\| = 0 \equiv \mathcal{R}_{ijmn}^{(0)}.
			\label{eq:curvature_annihilation_limit}
		\end{equation}
	\end{theorem}
	
	\begin{proof}
		Setting $r=4$ in Theorem~\ref{thm:universal_scaling_law}, $N^{1 - 4/2} = N^{-1}$. We provide a detailed step-by-step verification using the coordinate formula for the Riemann curvature tensor.
		
		In local coordinates on $IG_N$, the $(0,4)$-covariant Riemann curvature tensor $\mathcal{R}_{ijmn}^{(\alpha, N)}$ expands as \cite{amari1985, frankel2011}:
		\begin{equation}
			\mathcal{R}_{ijmn}^{(\alpha, N)}(\theta) = \frac{\partial \Gamma_{nj, i}^{(\alpha, N)}(\theta)}{\partial \theta^m} - \frac{\partial \Gamma_{mj, i}^{(\alpha, N)}(\theta)}{\partial \theta^n} + \sum_{s,t} G_{(N)}^{st}(\theta) \left( \Gamma_{mj, s}^{(\alpha, N)} \Gamma_{ni, t}^{(-\alpha, N)} - \Gamma_{nj, s}^{(\alpha, N)} \Gamma_{mi, t}^{(-\alpha, N)} \right).
			\label{eq:riemann_full_formula_ign}
		\end{equation}
		Under $i.i.d.$ sampling, $G_{(N)}^{st}(\theta) = \frac{1}{N} g_{(1)}^{st}(\theta)$ and $\Gamma_{abc}^{(\alpha, N)}(\theta) = N \cdot  \Gamma_{abc}^{(\alpha, 1)}(\theta)$, where $G_{(N)}^{st}(\theta)$ is the $(s,t)$-component of the inverse (contravariant) $N$-sample joint Fisher Information Metric tensor on $IG_N$. Substituting these into Equation~\eqref{eq:riemann_full_formula_ign}:
		\begin{align}
			\mathcal{R}_{ijmn}^{(\alpha, N)}(\theta) &= N \frac{\partial \Gamma_{nj, i}^{(\alpha, 1)}(\theta)}{\partial \theta^m} - N \frac{\partial \Gamma_{mj, i}^{(\alpha, 1)}(\theta)}{\partial \theta^n} \nonumber \\
			&\quad + \sum_{s,t} \left( \frac{1}{N} g_{(1)}^{st}(\theta) \right) \left[ N \Gamma_{mj, s}^{(\alpha, 1)}(\theta) \right] \left[ N \Gamma_{ni, t}^{(-\alpha, 1)}(\theta) - N \Gamma_{nj, s}^{(\alpha, 1)}(\theta) \Gamma_{mi, t}^{(-\alpha, 1)}(\theta) \right] \nonumber \\
			&= N \left[ \frac{\partial \Gamma_{nj, i}^{(\alpha, 1)}(\theta)}{\partial \theta^m} - \frac{\partial \Gamma_{mj, i}^{(\alpha, 1)}(\theta)}{\partial \theta^n} + \sum_{s,t} g_{(1)}^{st}(\theta) \left( \Gamma_{mj, s}^{(\alpha, 1)} \Gamma_{ni, t}^{(-\alpha, 1)} - \Gamma_{nj, s}^{(\alpha, 1)} \Gamma_{mi, t}^{(-\alpha, 1)} \right) \right] \nonumber \\
			&= N \cdot \mathcal{R}_{ijmn}^{(\alpha, 1)}(\theta).
			\label{eq:riemann_linear_growth}
		\end{align}
Let $\widetilde{\mathcal{R}}_{ijmn}^{(\alpha, N)}(h)$ is the $(i,j,m,n)$-component of the pulled-back $(0,4)$-covariant Riemann curvature tensor field (associated with the Amari $\alpha$-connection) defined on the tangent canvas $\mathcal{M}_\infty$. Now pull back $\mathcal{R}^{(\alpha, N)}$ under $\psi_N(h) = \theta_0 + h/\sqrt{N}$ by contracting against four copies of $J_i^a(h) = N^{-1/2} \delta_i^a$:
		\begin{align}
			\widetilde{\mathcal{R}}_{ijmn}^{(\alpha, N)}(h) \equiv \left[\psi_N^* \mathcal{R}^{(\alpha,N)}\right]_{ijmn}(h) &= \sum_{a,b,c,d=1}^d \frac{\partial \psi_N^a(h)}{\partial h^i} \frac{\partial \psi_N^b(h)}{\partial h^j} \frac{\partial \psi_N^c(h)}{\partial h^m} \frac{\partial \psi_N^d(h)}{\partial h^n} \mathcal{R}_{abcd}^{(\alpha,N)}\big(\psi_N(h)\big) \nonumber \\ 
			&= {\small \sum_{a,b,c,d=1}^d \left( \frac{1}{\sqrt{N}} \delta_i^a \right) \left( \frac{1}{\sqrt{N}} \delta_j^b \right) \left( \frac{1}{\sqrt{N}} \delta_m^c \right) \left( \frac{1}{\sqrt{N}} \delta_n^d \right) \mathcal{R}_{abcd}^{(\alpha, N)}\left( \theta_0 + \frac{h}{\sqrt{N}} \right)} \nonumber \\
			&= \left( \frac{1}{\sqrt{N}} \right)^4 \left[ N \cdot \mathcal{R}_{ijmn}^{(\alpha, 1)}\left( \theta_0 + \frac{h}{\sqrt{N}} \right) \right] \nonumber \\
			&= \frac{1}{N^2} \cdot N \cdot \mathcal{R}_{ijmn}^{(\alpha, 1)}\left( \theta_0 + \frac{h}{\sqrt{N}} \right) \nonumber \\
			&= \frac{1}{N} \mathcal{R}_{ijmn}^{(\alpha, 1)}\left( \theta_0 + \frac{h}{\sqrt{N}} \right).
			\label{eq:riemann_pullback_final_step}
		\end{align}
		By Assumption~\ref{asm:local_regularity}, $\|\mathcal{R}^{(\alpha, 1)}(\theta)\|_\infty \le M_R < \infty$ on $\overline{\mathcal{U}}(\theta_0)$. Taking the uniform supremum over $\mathbb{K}$:
		\begin{equation}
			\sup_{h \in \mathbb{K}} \left\| \widetilde{\mathcal{R}}^{(\alpha, N)}(h) \right\|_\infty \le \frac{M_R}{N} = \mathcal{O}_p\left(N^{-1}\right) \xrightarrow{N \to \infty} 0.
		\end{equation}
		This proves accelerated quadratic curvature decay ($\mathcal{O}_p(N^{-1})$).
	\end{proof}
	
	\subsection{Summary of Section 2 Degeneration Laws}
	\label{subsec:mod2_summary_table}
	
	Table~\ref{tab:degeneration_summary} summarizes the universal scaling law and exact degeneration rules established across the four fundamental geometric fields in Section 2.
	
	\begin{table}[h!]
		\centering
		\small
		\begin{tabular}{|p{1.6cm}|p{2.8cm}|c|c|p{3.5cm}|}
			\hline
			\textbf{Geometric Object} & \textbf{Geometric Tensor Field} & \textbf{Valence ($r$)} & \textbf{Scaling Law ($N^{1-r/2}$)} & \textbf{Asymptotic Limit on $\mathcal{M}_\infty$} \\ \hline
			\textbf{Obj 1} & Score 1-Form $dL^{(N)}$ & $r=1$ & $N^{1 - 1/2} = N^{1/2}$ & Gaussian Fluctuation $\Delta_\infty \sim \mathcal{N}(\mathbf{0}, g_0)$ \\ \hline
			\textbf{Obj 2} & Fisher Metric $G^{(N)}$ & $r=2$ & $N^{1 - 2/2} = N^0 = 1$ & Frozen Metric $g_0 \equiv g^{(1)}(\theta_0)$ \\ \hline
			\textbf{Obj 3} & Amari Connection $\nabla^{(\alpha,N)}$ & $r=3$ & $N^{1 - 3/2} = N^{-1/2}$ & Dissolved Transport $\Gamma_{ijk}^{(0)} \equiv 0$ \\ \hline
			\textbf{Obj 4} & Riemann Curvature $\mathcal{R}^{(\alpha,N)}$ & $r=4$ & $N^{1 - 4/2} = N^{-1}$ & Flat Canvas $\mathcal{R}_{ijmn}^{(0)} \equiv 0$ \\ \hline
		\end{tabular}
		\caption{Summary of Universal Tensor Scaling Laws and Field Degenerations on $\mathcal{M}_\infty$.}
		\label{tab:degeneration_summary}
	\end{table}
	
	The systematic decay of connection friction at rate $\mathcal{O}(N^{-1/2})$ and quadratic curvature at rate $\mathcal{O}(N^{-1})$ provides the explicit differential-geometric foundation required for next Section, where we axiomatize manifold convergence paradigms (Geometric Priority vs. Statistical Operationalism) and prove the Double Completeness Architecture.

\section{Foundational Axiomatization of Manifold Convergence: Geometric Priority vs. Statistical Operationalism}
\label{sec:module3}


\subsection{The Primary Ontological Question: Geometric Priority vs. Statistical Operationalism}
\label{subsec:ontological_question}


Having established the localized tensor degeneration laws in Section~\ref{sec:module2}—namely, score $1$-form fluctuation normalization ($\mathcal{O}_p(1)$), Fisher metric stabilization ($\mathcal{O}(N^{-1/2})$), Amari connection dissolution ($\mathcal{O}(N^{-1/2})$), and accelerated quadratic curvature annihilation ($\mathcal{O}_p(N^{-1})$)—we arrive at the pivotal architectural threshold of Section \ref{sec:module3}. To synthesize these localized field limits into a unified global theory of manifold convergence $IG_N \to \mathcal{M}_\infty$, we must raise and answer a foundational ontological question:

\begin{quote}
	\textbf{The Primary Ontological Question:} \textit{When defining the asymptotic limit of a sequence of joint information manifolds $IG_N = (\Theta, G^{(N)}, \nabla^{(\alpha,N)})$, which paradigm holds foundational ontological priority? Is differential geometry primary—such that the physical spatial flattening of the Riemannian-affine manifold drives statistical decision efficiency downstream? Or is statistical decision theory primary—such that operational risk condensation under Le Cam deficiency distance is the sole physical reality, reducing differential geometry to an auxiliary mathematical visualization tool?}
\end{quote}

\subsubsection*{1. Formulation of the Ontological Dichotomy}

Establishing the primary position dictates both the causal arrow of explanation and the axiomatic direction of mathematical proof across asymptotic statistics:

\begin{enumerate}[label=\textnormal{(\roman*)}]
	\item \textbf{The Geometric Priority Hypothesis (Way 1):} 
	Grounded in differential geometry and geometric analysis (Cartan, Cheeger, Gromov), this view posits that the statistical manifold $IG_N$ is an objective, physical geometric space. Metric stabilization, parallel transport flattening, and curvature annihilation are intrinsic spatial properties. Under this paradigm, classical statistical limit theorems—such as Local Asymptotic Normality (LAN), the Central Limit Theorem, and Cramér-Rao efficiency—are \textit{downstream operational manifestations} of spatial curvature collapse ($IG_N \xrightarrow{\text{geom}} \mathcal{M}_\infty$).
	
	\item \textbf{The Statistical Operationalism Hypothesis (Way 2):} 
	Grounded in asymptotic decision theory (Le Cam, Wald, van der Vaart), this view asserts that data distributions, decision rules, and loss functions are the sole operational realities. Information manifolds are merely formal coordinate representations constructed via Taylor expansions of log-likelihood ratios. Curvature collapse is not a physical process, but a mathematical byproduct of sample averaging under Le Cam deficiency convergence ($\mathcal{E}_N \xrightarrow{\text{oper}} \mathcal{E}_\infty$).
\end{enumerate}

\begin{table}[h!]
	\centering
	\small
	\begin{tabular}{|p{2.4cm}| p{6.8cm}| p{6.0cm}|}
		\toprule
		\textbf{Dimension} & \textbf{Geometric Priority (Way 1)} & \textbf{Statistical Operationalism (Way 2)} \\ \hline 
		\midrule
		\textbf{Ontological Baseline} & Smooth metric-affine tensor fields on $T_{\theta_0}IG_1$ & Sequence of statistical decision experiments $\mathcal{E}_N$ \\ \hline 
		\addlinespace
		\textbf{Primary Convergence} & Cheeger-Gromov $C^\infty$-manifold collapse: & Le Cam experiment deficiency $\Delta$ distance collapse: \\
		& $\lim_{N \to \infty} \left( \psi_N^* G^{(N)},~ \psi_N^* \nabla^{(\alpha,N)}, ~\psi_N^* \mathcal{R}^{(\alpha,N)} \right) = (g_0, \nabla^{(0)}, 0)~~$ & $~~~~~\lim_{N \to \infty} \Delta\left(\mathcal{E}_N(\mathbb{K}),~ \mathcal{E}_\infty(\mathbb{K})\right) = 0$ \\ \hline 
		\addlinespace
		\textbf{Causal Arrow} & Spatial manifold flattening $\implies$ Risk condensation & Decision risk condensation $\implies$ Tensor field stabilization \\
		\bottomrule
	\end{tabular}
	\caption{Divergent axiomatic deductive paths under alternative primary paradigms.}
	\label{tab:ontological_deductive_paths}
\end{table}

\subsubsection*{2. The Rigorous Mathematical Answer: A Unified Duality Paradigm}

To resolve this foundational question, Section \ref{sec:module3} introduces a radical departure from traditional  precedent in both differential geometry and statistics worlds. Historically, both classical geometers and mathematical statisticians have operated under a biased, reductionist ontological dichotomy:
\begin{itemize}
	\item \textbf{Geometric Monism (The Geometer's Bias):} Classical differential geometers (e.g., Cartan, Cheeger, Gromov) implicitly assumed that Riemannian-affine manifolds represent the primary physical reality, viewing statistical decision rules, estimators, and empirical distributions as downstream computational secondary effects.
	\item \textbf{Statistical Reductionism (The Statistician's Bias):} Asymptotic decision theorists (e.g., Le Cam, Wald, van der Vaart) asserted that decision risk bounds and probability measures are the sole operational reality, dismissing geometric concepts—such as Fisher metrics, Amari connections, and Riemann curvature—as mere auxiliary artifacts of Taylor expansion approximations.
\end{itemize}		
In this paper, we assert that \textbf{this historical dichotomy is fundamentally flawed}. Neither paradigm holds unilateral ontological priority. In Section \ref{sec:module3}, we establish a {\bf Unified Geometric-Operational Duality Paradigm}: differential geometric structure and statistical decision operationalism are non-separable dual manifestations—literally \textit{two sides of the exact same underlying mathematical coin}—governing the asymptotic phase transition $IG_N \to \mathcal{M}_\infty$.
		
		This new ontological synthesis is formalized through the three foundational mathematical pillars of Section \ref{sec:module3}:
		
\begin{itemize}
\item \textbf{Pillar I: Category-Theoretic Duality (Trajectory Duality):} 
	The limit process $IG_N \to \mathcal{M}_\infty$ is an asymptotic phase transition driven by the fluctuation parameter $\beta_N = N^{-1/2}$. Under a category-theoretic lens, the geometric pathway of tensor field collapse in the category of smooth manifolds $\mathbf{Man}$ and the operational pathway of risk deficiency condensation in the Le Cam category $\mathbf{Exp}$ are functorially isomorphic. They trace the exact same evolutionary trajectory viewed through two dual categorical perspectives. \footnote{In category theory, $\mathbf{Man}$ denotes the canonical category whose objects are smooth manifolds and whose morphisms are smooth maps \cite{maclane1998, lee2018}. Analogously, $\mathbf{Exp}$ denotes the Le Cam category of statistical experiments, where objects are statistical models $\mathcal{E} = (\Omega, \mathcal{F}, \{P_\theta\}_{\theta \in \Theta})$ sharing a common parameter space $\Theta$, and morphisms are Markov transitions (randomization operators) that preserve operational decision risk bounds \cite{lecam1986, torgersen1991}.}
			
\item \textbf{Pillar II: Algebraic Equivalence via the Taylor-Geometry Identity:} 
			By the Taylor-Geometry Algebraic Functional Identity (Lemma~\ref{lem:pillar_2_taylor_geometry}), the localized Radon-Nikodym log-likelihood process $\Lambda_N(h)$ driving Le Cam's deficiency distance $\Delta$ is algebraically bound to the pulled-back geometric tensor fields ($r \in \{1,2,3,4\}$):
			\begin{equation}
				\Lambda_N(h) = \sum_{a=1}^d h^a \left[ \psi_N^*(dL^{(N)}) \right]_a(0) - \frac{1}{2} \sum_{i,j=1}^d h^i h^j \left[ \psi_N^* G^{(N)} \right]_{ij}(0) + \mathcal{R}_N(h),
			\end{equation}
			where the non-linear stochastic remainder $\mathcal{R}_N(h)$ is strictly controlled by the pulled-back connection symbols $\widetilde{\Gamma}_{ijk}^{(\alpha,N)}$ and Riemann curvature tensor $\widetilde{\mathcal{R}}_{ijmn}^{(\alpha,N)}$. Statistical likelihood operations and manifold geometry are thus algebraically co-dependent and structurally inseparable.
			
			\item \textbf{Pillar III: The Geometric-Operational Equivalence Theorem:} 
			Because decision deficiency $\Delta$ and Cheeger-Gromov metric-affine tensor collapse are functionally tied through $\Lambda_N(h)$, geometric field stabilization is both \textbf{necessary and sufficient} for operational decision convergence:
			\begin{equation}
				IG_N \xrightarrow{\text{geom}} \mathcal{M}_\infty \quad \Longleftrightarrow \quad IG_N \xrightarrow{\text{oper}} \mathcal{M}_\infty.
			\end{equation}
		\end{itemize}
		
		\textbf{Synthesis:} 
Differential geometry provides the \textit{structural-differential substrate} (the spatial topology, parallel transport, and curvature dynamics), while statistical decision theory provides its \textit{operational-observational measurement} (the decision risk, loss bounds, and estimation efficiency). They are co-primary, intrinsically synchronized, and mutually deterministic.


\subsection{Asymptotic Dual Phase Transition Paradigm: Physical Mechanics and Temperature Analogue}
\label{subsec:mod3_phase_transition_mechanics}

The transition of the sequence of $N$-sample joint information manifolds $IG_N = (\Theta, G^{(N)}, \nabla^{(\alpha,N)})$ onto the flat tangent canvas $\mathcal{M}_\infty$ as $N \to \infty$ is not merely a passive analytic limit of scalar parameters. Instead, it constitutes a genuine, structural {\it dual phase transition} spanning differential geometry and mathematical statistics \cite{amari1985, lecam1986}.

In statistical thermodynamics and quantum field theory \cite{landau1980, goldenfeld1992}, {\it a physical phase transition} is driven by a tuning parameter (e.g., temperature $T \to T_c$), during which a disordered, highly fluctuating macroscopic system undergoes structural condensation, symmetry restoration, and the suppression of non-linear fluctuations. In information geometry, the inverse sample size scaling parameter 
\begin{equation}
	\beta_N \equiv \frac{1}{\sqrt{N}}
	\label{eq:temperature_parameter_def}
\end{equation}
acts as the {\bf fluctuation temperature parameter}, while the large-sample limit $N \to \infty$ represents the {\it thermodynamic limit}.

\begin{definition}[Evolutionary Dynamics vs. Static Target State]
	\label{def:process_vs_state}
	The asymptotic metamorphosis $IG_N \to \mathcal{M}_\infty$ consists of two distinct mathematical entities:
	\begin{enumerate}[label=\textnormal{(\roman*)}]
		\item \textbf{The Limiting Target State (Destination):} The canonical flat tangent space $\mathcal{M}_\infty \equiv (T_{\theta_0} IG_1, g_0, \nabla^{(0)})$, representing a static, non-curved Euclidean-Gaussian workspace where classical decision theory operates.
		\item \textbf{The Asymptotic Evolutionary Dynamic (The Phase Transition Process):} The continuous, $N$-parameterized sequence of geometric-statistical state transformations $\{IG_N\}_{N=1}^\infty$. As the sample size $N$ increases, the scaling factor $\beta_N = 1/\sqrt{N}$ systematically cools the system, quenching non-Euclidean geometric distortions and non-Gaussian statistical fluctuations.
	\end{enumerate}
\end{definition}

As summarized in Table~\ref{tab:dual_phase_transition_mechanics}, this evolutionary process manifests through two dual, mutually synchronized operational perspectives.

\begin{table}[h!]
	\centering
	\small
	\begin{tabular}{|p{3.2cm}|p{5.8cm}|p{5.8cm}|}
		\hline
		\textbf{Physical Feature} & \textbf{Geometric Phase Transition (Way 1)} & \textbf{Statistical Phase Transition (Way 2)} \\ \hline
		\textbf{Microscopic State ($N=1$)} & Curved Riemannian-Affine Manifold $IG_1$ with non-zero curvature $\mathcal{R}^{(\alpha)}$ and path-dependent transport $\Gamma^{(\alpha)}$. & Asymmetric, non-Gaussian finite-sample experiment with non-linear likelihood ratios. \\ \hline
		\textbf{Thermodynamic Tuning} & Localized dilation map $\psi_N(h) = \theta_0 + \frac{h}{\sqrt{N}}$ zooming into tangent space $T_{\theta_0} IG_1$. & Rate scaling $\sqrt{N}(\theta - \theta_0) = h$ isolating contiguous alternative distributions $P_{\theta_0 + h/\sqrt{N}}$. \\ \hline
		\textbf{Cooling Mechanism} & Spatial cotangent contraction scaling $N^{-r/2}$ against $r$-covariant tensor fields. & Central Limit Theorem measure concentration suppressing higher-order cumulants. \\ \hline
		\textbf{Order Parameters} & Connection symbols $\widetilde{\Gamma}^{(\alpha,N)} \sim \mathcal{O}(N^{-1/2})$, Riemann curvature $\widetilde{\mathcal{R}}^{(\alpha,N)} \sim \mathcal{O}_p(N^{-1})$. & Non-linear likelihood remainder process $\mathcal{R}_N(h) \sim \mathcal{O}_p(N^{-1/2})$. \\ \hline
		\textbf{Condensed Target Phase ($N \to \infty$)} & Flat Euclidean space $\mathcal{M}_\infty$ with zero curvature ($\mathcal{R}^{(0)} \equiv 0$) and zero connection friction ($\Gamma^{(0)} \equiv 0$). & Canonical Gaussian Shift Experiment $\mathcal{E}_\infty = \{\mathcal{N}(h, g_0^{-1}) : h \in \mathbb{R}^d\}$ (LAN, Local Asymptotic Normality). \\ \hline
	\end{tabular}
	\caption{Duality Breakdown of the Asymptotic Information-Geometric Phase Transition.}
	\label{tab:dual_phase_transition_mechanics}
\end{table}

\subsection{Axiomatic Formalization of the Dual Transitions: Game Rules and Observables}
\label{subsec:mod3_axioms}

To prevent circular reasoning and establish a sharp mathematical foundation, we formalize the governing rules for both observable manifestations of the phase transition through two self-contained, independent axiomatic paradigms.

\subsubsection{Way 1: The Geometric Priority Paradigm (Cartan-Cheeger-Gromov)}
\label{subsubsec:way1_geometric_priority}

Under Geometric Priority, the information manifold sequence $IG_N$ is treated as an objective sequence of physical-geometric spaces \cite{cheeger1970, gromov2007}. Geometric field stabilization, connection flattening, and curvature annihilation occur intrinsically, independent of downstream decision tasks.

\begin{axiom}[Geometric Priority Axiom]
	\label{ax:geom_priority}
	The sequence of joint information-geometric manifolds $IG_N = (\Theta, G^{(N)}, \nabla^{(\alpha,N)})$ converges geometrically to the canonical flat tangent space $\mathcal{M}_\infty \equiv (T_{\theta_0} IG_1, g_0, \nabla^{(0)})$, denoted by $IG_N \xrightarrow{\text{geom}} \mathcal{M}_\infty$, if and only if the pulled-back tensor fields under the localized microscope map $\psi_N(h) = \theta_0 + \frac{h}{\sqrt{N}}$ satisfy uniform convergence across every compact coordinate cage $\mathbb{K} \subset \mathcal{M}_\infty \cong \mathbb{R}^d$:
	\begin{align}
		\lim_{N \to \infty} \sup_{h \in \mathbb{K}} \left\| \left[\psi_N^* G^{(N)}\right](h) - g_0 \right\|_\infty &= 0 \quad \text{\textnormal{(Uniform Metric Stabilization)}}, \label{eq:ax1_metric_stabilization} \\
		\lim_{N \to \infty} \sup_{h \in \mathbb{K}} \left\| \left[\psi_N^* \nabla^{(\alpha,N)}\right](h) - \nabla^{(0)} \right\|_\infty &= 0 \quad \text{\textnormal{(Uniform Connection Dissolution)}}, \label{eq:ax1_connection_dissolution} \\
		\lim_{N \to \infty} \sup_{h \in \mathbb{K}} \left\| \left[\psi_N^* \mathcal{R}^{(\alpha,N)}\right](h) \right\|_\infty &= 0 \quad \text{\textnormal{(Accelerated Curvature Annihilation)}}. \label{eq:ax1_curvature_annihilation}
	\end{align}
\end{axiom}

\subsubsection{Way 2: The Statistical Operationalism Paradigm (Le Cam Deficiency Distance)}
\label{subsubsec:way2_statistical_operationalism}

Under Statistical Operationalism, decision theory is primary \cite{lecam1960, lecam1986}. Information manifolds are operational representations whose convergence is defined strictly by the uniform convergence of decision-theoretic risk bounds under Le Cam's deficiency distance $\Delta$.

\begin{definition}[Localized Statistical Experiments]
	\label{def:localized_experiments}
	Let $\mathbb{K} \subset \mathcal{M}_\infty \cong \mathbb{R}^d$ be a compact coordinate cage enclosing the origin $h = \mathbf{0}$.
	\begin{enumerate}[label=\textnormal{(\roman*)}]
		\item The {\it empirical localized statistical experiment sequence} $\mathcal{E}_N(\mathbb{K})$ associated with $IG_N$ is defined by:
		\begin{equation}
			\mathcal{E}_N(\mathbb{K}) \equiv \left( \mathcal{X}^N, \, \mathcal{A}^{\otimes N}, \, \left\{ P_{\theta_0 + h/\sqrt{N}}^{(N)} : h \in \mathbb{K} \right\} \right).
			\label{eq:empirical_experiment_def}
		\end{equation}
		\item The canonical Gaussian shift experiment  $\mathcal{E}_\infty(\mathbb{K})$ associated with the target canvas $\mathcal{M}_\infty$ is defined by:
		\begin{equation}
			\mathcal{E}_\infty(\mathbb{K}) \equiv \left( \mathbb{R}^d, \, \mathcal{B}^d, \, \left\{ \mathcal{N}\left(h, \, g_0^{-1}\right) : h \in \mathbb{K} \right\} \right).
			\label{eq:gaussian_experiment_def}
		\end{equation}
	\end{enumerate}
\end{definition}

\begin{axiom}[Statistical Operationalism Axiom]
	\label{ax:stat_oper}
	The sequence of joint information manifolds $IG_N$ converges operationally to the Conventional Statistics workspace $\mathcal{M}_\infty \equiv \text{CS}$, denoted by $IG_N \xrightarrow{\text{oper}} \mathcal{M}_\infty$, if and only if for every compact coordinate cage $\mathbb{K} \subset \mathbb{R}^d$, Le Cam's experiment deficiency distance $\Delta$ between $\mathcal{E}_N(\mathbb{K})$ and $\mathcal{E}_\infty(\mathbb{K})$ vanishes asymptotically \cite{lecam1986, lecam2000}:
	\begin{equation}
		\lim_{N \to \infty} \Delta\left( \mathcal{E}_N(\mathbb{K}), \, \mathcal{E}_\infty(\mathbb{K}) \right) = 0.
		\label{eq:ax2_lecam_deficiency_limit}
	\end{equation}
\end{axiom}

\subsection{The Double Completeness Architecture and Taylor-Geometry Identity}
\label{subsec:mod3_double_completeness}

To unify Way 1 (Geometric Priority) and Way 2 (Statistical Operationalism), first we establish the {\it Double Completeness Architecture}. We prove that Élie Cartan's Structural Determinacy Theorem (Pillar I) and the Taylor-Geometry Algebraic Functional Identity (Pillar II) lock differential geometry and decision theory into a single functional framework.

\begin{proposition}[Pillar I: Differential-Geometric Completeness]
	\label{prop:pillar_1_completeness}
	Let $\mathcal{M}_\infty \equiv (T_{\theta_0} IG_1, g_0, \nabla^{(0)})$ be the target flat space. The fulfillment of Axiom~\ref{ax:geom_priority} uniquely and completely determines the smooth geometric limit of the sequence $\psi_N^*(IG_N)$ in the Cheeger-Gromov $C^\infty$-topology \cite{cheeger1970, gromov2007}:
	\begin{equation}
		\psi_N^*(IG_N) \xrightarrow{C^\infty} \mathcal{M}_\infty \equiv \left( T_{\theta_0} IG_1, \, g_0, \, \nabla^{(0)} \right).
		\label{eq:cheeger_gromov_convergence_limit}
	\end{equation}
	No additional geometric tensors or higher-order structural fields are required to specify the limiting manifold.
\end{proposition}

\begin{remark}
~
\begin{enumerate}
\item This Proposition (\emph{\bf Pillar I: Differential-Geometric Completeness}) serves as the completeness bridge for Way 1 (\emph{Geometric Priority}). It establishes that satisfying the conditions of the Geometric Priority Axiom is \textbf{necessary and sufficient} to uniquely and completely determine the smooth geometric limit of the sequence of pulled-back information-geometric manifolds $\psi_N^*(IG_N)$ in the Cheeger-Gromov $C^\infty$-topology onto the target canonical flat space $\mathcal{M}_\infty \equiv (T_{\theta_0}IG_1, g_0, \nabla^{(0)})$.

\item \textbf{Mechanism (Élie Cartan's Structural Determinacy Theorem):} A smooth Riemannian-affine manifold $(M, g, \nabla)$ is locally uniquely determined up to isometric affine diffeomorphisms by its metric tensor $g$, its connection symbols $\Gamma$, and its Riemann curvature tensor $\mathcal{R}$.
	
\item \textbf{Implications from this  Proposition:} By combining Axiom 1's \ref{ax:geom_priority} uniform field limits with Cheeger-Gromov compactness, this proposition proves \textbf{structural completeness}: {\it no additional higher-order geometric tensor fields or hidden topological degrees of freedom exist beyond these three fields to specify the limiting manifold $\mathcal{M}_\infty$.}
\end{enumerate}
\end{remark}

\begin{proof}
	By Élie Cartan's Structural Determinacy Theorem \cite{frankel2011, lee2013}, a smooth Riemannian-affine manifold $(M, g, \nabla)$ is locally uniquely determined up to isometric affine diffeomorphisms by its metric tensor field $g$, its Christoffel connection symbols $\Gamma_{ij}^k$, and its Riemann curvature tensor field $\mathcal{R}_{jmn}^i$.
	
	Under High-Order Local Regularity (Assumption~\ref{asm:local_regularity}):
	\begin{enumerate}[label=\arabic*.]
		\item By Theorem~\ref{thm:metric_stabilization}, the pulled-back Fisher metric field $\widetilde{G}^{(N)}(h) \equiv \psi_N^* G^{(N)}(h)$ satisfies $$\sup_{h \in \mathbb{K}} \|\widetilde{G}^{(N)}(h) - g_0\|_\infty = \mathcal{O}(N^{-1/2}) \to 0,$$ establishing uniform metric stabilization to the frozen tensor $g_0 \equiv g^{(1)}(\theta_0)$.
		\item By Theorem~\ref{thm:connection_dissolution}, the pulled-back Amari connection symbols satisfy $$\sup_{h \in \mathbb{K}} |\widetilde{\Gamma}_{ijk}^{(\alpha,N)}(h)| = \mathcal{O}(N^{-1/2}) \to 0,$$ forcing the affine connection to dissolve uniformly into the flat Levi-Civita connection $\nabla^{(0)}$ with zero Christoffel symbols $\Gamma_{ijk}^{(0)} \equiv 0$.
		\item By Theorem~\ref{thm:curvature_annihilation}, the pulled-back Riemann curvature tensor satisfies $$\sup_{h \in \mathbb{K}} \|\widetilde{\mathcal{R}}_{ijmn}^{(\alpha,N)}(h)\|_\infty = \mathcal{O}_p(N^{-1}) \to 0,$$ confirming accelerated quadratic annihilation of all intrinsic non-Euclidean curvature.
	\end{enumerate}
	According to Cheeger-Gromov compactness and convergence theory for smooth Riemannian manifolds \cite{cheeger1970, gromov2007}, uniform $C^\infty$-bounds on metric components $g_{ij}$, connection symbols $\Gamma_{ij}^k$, and curvature components $\mathcal{R}_{jmn}^i$ guarantee that the sequence of pulled-back manifolds $\psi_N^*(IG_N) = (\mathbb{K}, \widetilde{G}^{(N)}, \widetilde{\nabla}^{(\alpha,N)})$ converges smoothly in the Cheeger-Gromov topology to the canonical flat space $\mathcal{M}_\infty$. This completes the proof of Pillar I.
\end{proof}

Now, we prove that the decision-theoretic log-likelihood process $\Lambda_N(h)$ driving Le Cam's deficiency distance $\Delta$ is algebraically built from these exact same four geometric fields.

\begin{lemma}[Pillar II: Taylor-Geometry Algebraic Functional Identity]
	\label{lem:pillar_2_taylor_geometry}
	Under High-Order Local Regularity (Assumption~\ref{asm:local_regularity}), the localized Radon-Nikodym log-likelihood ratio process 
	\begin{equation}
		\Lambda_N(h) \equiv \log \frac{d P_{\theta_0 + h/\sqrt{N}}^{(N)}}{d P_{\theta_0}^{(N)}}(X^N)
		\label{eq:radon_nikodym_ratio_def}
	\end{equation}
	admits an exact algebraic functional representation built exclusively from the four pulled-back geometric tensors ($r \in \{1, 2, 3, 4\}$): \footnote{
		The variable $r \in \{1, 2, 3, 4\}$ appears in the text as a shorthand reference to the tensor rank (or valence) of each of the four pulled-back geometric tensors used to build the representation:
		\begin{itemize}
			\item $r = 1$: The pulled-back score 1-form $\left[\psi_N^*(dL^{(N)})\right]_a$ (1 index)
			
			\item $r = 2$: The pulled-back Fisher metric $\left[\psi_N^* G^{(N)}\right]_{ij}$ (2 indices)
			
			\item $r = 3$: The pulled-back connection $\left[\psi_N^* \nabla^{(\alpha,N)}\right]_{ijk}$ (3 indices)
			
			\item $r = 4$: The pulled-back curvature tensor $\left[\psi_N^* \mathcal{R}^{(\alpha,N)}\right]_{ijmn}$ (4 indices)
		\end{itemize}
	}
	\begin{equation}
		\Lambda_N(h) = \sum_{a=1}^d h^a \left[ \psi_N^*(dL^{(N)}) \right]_a(\mathbf{0}) - \frac{1}{2} \sum_{i,j=1}^d h^i h^j \left[ \psi_N^* G^{(N)} \right]_{ij}(\mathbf{0}) + \mathcal{R}_N(h),
		\label{eq:taylor_geometry_exact_identity}
	\end{equation}
	where $\psi_N^*(dL^{(N)})(\mathbf{0}) \xrightarrow{d} \Delta_\infty \sim \mathcal{N}(\mathbf{0}, g_0)$ is the pulled-back score $1$-form at $h = \mathbf{0}$ (Theorem~\ref{thm:score_transformation}), $\psi_N^* G^{(N)}(\mathbf{0}) = g_0$ is the frozen Fisher metric (Theorem~\ref{thm:metric_stabilization}), and the non-linear stochastic remainder process $\mathcal{R}_N(h)$ satisfies the uniform geometric envelope bound across $\mathbb{K}$:
	\begin{equation}
		\sup_{h \in \mathbb{K}} |\mathcal{R}_N(h)| \le \frac{\|h\|_2^3}{6} \sup_{h \in \mathbb{K}} \left| \left[ \psi_N^* \nabla^{(\alpha,N)} \right]_{ijk}(h) \right| + \mathcal{O}_p\left( \sup_{h \in \mathbb{K}} \left\| \left[ \psi_N^* \mathcal{R}^{(\alpha,N)} \right]_{ijmn}(h) \right\|_\infty \right).
		\label{eq:remainder_geometry_explicit_bound}
	\end{equation}
\end{lemma}

\begin{proof}
	Let $h \in \mathbb{K} \subset \mathbb{R}^d$ be an arbitrary displacement vector. For each observation $X_n$ ($n = 1, \dots, N$), define the single-observation log-likelihood function $l_n(\theta) \equiv \log p(X_n; \theta)$. The joint localized log-likelihood ratio process is:
	\begin{equation}
		\Lambda_N(h) = \sum_{n=1}^N \left( l_n\left(\theta_0 + \frac{h}{\sqrt{N}}\right) - l_n(\theta_0) \right).
		\label{eq:lambda_n_sum_proof}
	\end{equation}
	
	Performing a third-order multivariate Taylor expansion of $l_n\left(\theta_0 + \frac{h}{\sqrt{N}}\right)$ around $h = \mathbf{0}$:
	\begin{align}
		l_n\left(\theta_0 + \frac{h}{\sqrt{N}}\right) - l_n(\theta_0) &= \frac{1}{\sqrt{N}} \sum_{a=1}^d h^a \frac{\partial l_n(\theta_0)}{\partial \theta^a} + \frac{1}{2N} \sum_{a,b=1}^d h^a h^b \frac{\partial^2 l_n(\theta_0)}{\partial \theta^a \partial \theta^b} \nonumber \\
		&\quad + \frac{1}{6 N^{3/2}} \sum_{a,b,c=1}^d h^a h^b h^c \int_0^1 3(1-t)^2 \frac{\partial^3 l_n\left(\theta_0 + \frac{t h}{\sqrt{N}}\right)}{\partial \theta^a \partial \theta^b \partial \theta^c} dt.
		\label{eq:taylor_expansion_third_order}
	\end{align}
	
	Summing Equation~\eqref{eq:taylor_expansion_third_order} over $n = 1, \dots, N$ and substituting the pulled-back score $1$-form expression from Theorem~\ref{thm:score_transformation}:
	\begin{equation}
		\sum_{n=1}^N \frac{1}{\sqrt{N}} \sum_{a=1}^d h^a \frac{\partial l_n(\theta_0)}{\partial \theta^a} = \sum_{a=1}^d h^a \left[ \psi_N^*(dL^{(N)}) \right]_a(\mathbf{0}).
		\label{eq:score_term_match}
	\end{equation}
	
	For the second-order term, adding and subtracting the expected Hessian matrix $\mathbb{E}_{\theta_0}\left[\frac{\partial^2 l_n(\theta_0)}{\partial \theta^a \partial \theta^b}\right] = -g_{ab}^{(1)}(\theta_0)$:
	\begin{equation}
		\small 
		\frac{1}{2N} \sum_{n=1}^N \sum_{a,b=1}^d h^a h^b \frac{\partial^2 l_n(\theta_0)}{\partial \theta^a \partial \theta^b} = -\frac{1}{2} \sum_{a,b=1}^d h^a h^b g_{ab}^{(1)}(\theta_0) + \frac{1}{2} \sum_{a,b=1}^d h^a h^b \left( \frac{1}{N} \sum_{n=1}^N \frac{\partial^2 l_n(\theta_0)}{\partial \theta^a \partial \theta^b} + g_{ab}^{(1)}(\theta_0) \right).
		\label{eq:hessian_term_split}
	\end{equation}
	By Theorem~\ref{thm:metric_stabilization}, $g_{ab}^{(1)}(\theta_0) = [\psi_N^* G^{(N)}]_{ab}(\mathbf{0})$. The second term in Equation~\eqref{eq:hessian_term_split} converges to zero in probability at rate $\mathcal{O}_p(N^{-1/2})$ by the Law of Large Numbers.
	
	For the third-order term, Amari's $\alpha$-connection Christoffel symbols of the first kind on $IG_1$ are defined by \cite{amari1985, amari2000}:
	\begin{equation}
		\small 
		\Gamma_{abc}^{(\alpha, 1)}(\theta) = \mathbb{E}_\theta \left[ \frac{\partial^2 \log p(X;\theta)}{\partial \theta^a \partial \theta^b} \frac{\partial \log p(X;\theta)}{\partial \theta^c} \right] + \frac{1-\alpha}{2} \mathbb{E}_\theta \left[ \frac{\partial \log p(X;\theta)}{\partial \theta^a} \frac{\partial \log p(X;\theta)}{\partial \theta^b} \frac{\partial \log p(X;\theta)}{\partial \theta^c} \right].
		\label{eq:amari_christoffel_def}
	\end{equation}
	Under the pulled-back connection scaling law from Theorem~\ref{thm:connection_dissolution}, $$[\psi_N^* \nabla^{(\alpha,N)}]_{abc}(h) = \frac{1}{\sqrt{N}} \Gamma_{abc}^{(\alpha,1)}\left(\theta_0 + \frac{h}{\sqrt{N}}\right).$$ Thus, the expected third-order log-density derivative is directly proportional to the pulled-back connection symbols $\widetilde{\Gamma}_{abc}^{(\alpha,N)}(h)$.
	
	Finally, non-Euclidean path integrability distortions and parallel transport holonomy are governed by the Riemann curvature tensor field $\widetilde{\mathcal{R}}_{ijmn}^{(\alpha,N)}(h) \equiv [\psi_N^* \mathcal{R}^{(\alpha,N)}]_{ijmn}(h)$ via the Frobenius Integrability Theorem \cite{frankel2011}. By Theorem~\ref{thm:curvature_annihilation}, these curvature distortions decay at the accelerated quadratic rate $\mathcal{O}_p(N^{-1})$.
	
	Combining Equations~\eqref{eq:score_term_match}--\eqref{eq:amari_christoffel_def} and taking the supremum over the compact cage $\mathbb{K}$ yields the exact functional identity Equation~\eqref{eq:taylor_geometry_exact_identity} and remainder bound Equation~\eqref{eq:remainder_geometry_explicit_bound}, completing the proof of Pillar II.
\end{proof}

\begin{remark}[{\bf The Double Completeness Architecture}]
	\label{rem:double_completeness_architecture}
	The unification of Pillar I (Differential-Geometric Completeness) and Pillar II (Taylor-Geometry Algebraic Functional Identity) establishes the {\it Double Completeness Architecture}:
	\begin{itemize}
		\item {\bf Pillar I} guarantees that the four pulled-back geometric tensor fields ($r \in \{1, 2, 3, 4\}$) completely determine the asymptotic differential geometry of the manifold sequence $\psi_N^*(IG_N)$.
		\item {\bf Pillar II} guarantees that the same four pulled-back geometric tensor fields completely determine the asymptotic log-likelihood process $\Lambda_N(h)$ and decision-theoretic risk bounds of the statistical experiment sequence $\mathcal{E}_N(\mathbb{K})$.
	\end{itemize}
	{\it Consequently, no ``fifth'' degree of freedom, hidden topological anomaly, or unobserved geometric distortion exists outside these four fields.}
\end{remark}

\section{Proof of the Geometric-Operational Equivalence Theorem to prove Pillar III}
\label{sec:equivalence_proof}

In Section~\ref{sec:module3}, we established the dual axiomatic framework for manifold convergence: \textbf{Way 1 (Geometric Priority, Axiom~\ref{ax:geom_priority})} and \textbf{Way 2 (Statistical Operationalism, Axiom~\ref{ax:stat_oper})}. In this section, we present the rigorous, step-by-step mathematical proof of the {\it Geometric-Operational Equivalence Theorem}  (Theorem~\ref{thm:geom_oper_equivalence}), establishing that geometric curvature collapse and statistical decision risk condensation are two mathematically equivalent projections of a single underlying phase transition.

\begin{lemma}[Le Cam Experiment Deficiency Limit Lemma]
	\label{lem:lecam_deficiency_equivalence}

~	
	
Let $\mathcal{E}_N(\mathbb{K}) = (\mathcal{X}^N, \mathcal{A}^{\otimes N}, \{P_{\theta_0 + h/\sqrt{N}}^{(N)} : h \in \mathbb{K}\})$ be the localized statistical experiment sequence and $\mathcal{E}_\infty(\mathbb{K}) = (\mathbb{R}^d, \mathcal{B}^d, \{\mathcal{N}(h, g_0^{-1}) : h \in \mathbb{K}\})$ be the limit Gaussian shift experiment on a compact set $\mathbb{K} \subset \mathbb{R}^d$. 

If the localized log-likelihood ratio process $\Lambda_N(h) \equiv \log \frac{d P_{\theta_0 + h/\sqrt{N}}^{(N)}}{d P_{\theta_0}^{(N)}}$ satisfies uniform stochastic expansion over $\mathbb{K}$:
\begin{equation}
	\Lambda_N(h) = h^\top \Delta_N(\theta_0) - \frac{1}{2} h^\top g_0 h + o_{P_{\theta_0}^{(N)}}(1), \quad \text{with } \Delta_N(\theta_0) \xrightarrow{d} \mathcal{N}(0, g_0),
	\label{eq:lan_stochastic_expansion_lemma}
\end{equation}
then the probability measures $P_{\theta_0 + h/\sqrt{N}}^{(N)}$ and $P_{\theta_0}^{(N)}$ are mutually contiguous ($P_{\theta_0 + h/\sqrt{N}}^{(N)} \triangleleft \triangleright P_{\theta_0}^{(N)}$),\footnote{Two sequences of probability measures $\{P_N\}$ and $\{Q_N\}$ are \textit{mutually contiguous} (denoted $P_N \triangleleft \triangleright Q_N$) if for any sequence of measurable events $A_N$, $P_N(A_N) \to 0$ if and only if $Q_N(A_N) \to 0$ as $N \to \infty$. In asymptotic statistics, mutual contiguity serves as an asymptotic analogue to mutual absolute continuity (equivalence of measures), guaranteeing that asymptotically negligible events under $P_{\theta_0}^{(N)}$ remain asymptotically negligible under $P_{\theta_0 + h/\sqrt{N}}^{(N)}$, and vice versa.} and the Le Cam experiment deficiency distance converges to zero: \footnote{The Le Cam experiment deficiency distance $\Delta(\mathcal{E}, \mathcal{F}) \equiv \max\{\delta(\mathcal{E}, \mathcal{F}), \delta(\mathcal{F}, \mathcal{E})\}$ measures the operational distance between two statistical experiments $\mathcal{E} = (\mathcal{X}, \mathcal{A}, \{P_\theta : \theta \in \Theta\})$ and $\mathcal{F} = (\mathcal{Y}, \mathcal{B}, \{Q_\theta : \theta \in \Theta\})$. The deficiency $\delta(\mathcal{E}, \mathcal{F}) = \inf_K \sup_{\theta \in \Theta} \Vert{}K P_\theta - Q_\theta\Vert{}_{\text{TV}}$ quantifies the minimal worst-case total variation error in approximating the distributions of $\mathcal{F}$ by applying a Markov transition kernel (randomization operator) $K$ to $\mathcal{E}$. A vanishing deficiency distance ($\Delta(\mathcal{E}_N, \mathcal{E}_\infty) \to 0$) guarantees that any statistical decision rule in $\mathcal{E}_N$ can be asymptotically matched by a rule in $\mathcal{E}_\infty$ with identical risk bounds under arbitrary bounded loss functions.}
\begin{equation}
	\lim_{N \to \infty} \Delta\left( \mathcal{E}_N(\mathbb{K}), \, \mathcal{E}_\infty(\mathbb{K}) \right) = 0.
	\label{eq:lecam_deficiency_zero_lemma}
\end{equation}
\end{lemma}

\begin{proof}
The proof proceeds in three rigorous measure-theoretic steps:

\paragraph{Step 1: Contiguity Verification via Le Cam's First Lemma.}
From Equation \eqref{eq:lan_stochastic_expansion_lemma}, under $P_{\theta_0}^{(N)}$, the scalar random variable $\Lambda_N(h)$ converges weakly to a univariate Gaussian distribution:
\begin{equation}
	\Lambda_N(h) \xrightarrow{d} \Lambda_\infty(h) \sim \mathcal{N}\left( -\frac{1}{2} \sigma^2, \, \sigma^2 \right), \quad \text{where } \sigma^2 \equiv h^\top g_0 h.
\end{equation}
The expectation of the limiting Radon-Nikodym derivative $L(h) \equiv \exp(\Lambda_\infty(h))$ satisfies:
\begin{equation}
	\mathbb{E}\left[ e^{\Lambda_\infty(h)} \right] = \exp\left( -\frac{1}{2}\sigma^2 + \frac{1}{2}\sigma^2 \right) = e^0 = 1.
\end{equation}
By Le Cam's First Lemma, an expectation of 1 for the limit likelihood ratio guarantees that the sequence $P_{\theta_0 + h/\sqrt{N}}^{(N)}$ is mutually contiguous with $P_{\theta_0}^{(N)}$ ($\forall h \in \mathbb{K}$).

\paragraph{Step 2: Finite-Dimensional Convergence of Experiments.}
Let $\{h_1, h_2, \dots, h_k\} \subset \mathbb{K}$ be any finite collection of displacement vectors. The joint log-likelihood vector satisfies:
\begin{equation}
	\begin{pmatrix} \Lambda_N(h_1) \\ \vdots \\ \Lambda_N(h_k) \end{pmatrix} \xrightarrow{d} \begin{pmatrix} h_1^\top Z - \frac{1}{2} h_1^\top g_0 h_1 \\ \vdots \\ h_k^\top Z - \frac{1}{2} h_k^\top g_0 h_k \end{pmatrix} \quad \text{under } P_{\theta_0}^{(N)}, \quad Z \sim \mathcal{N}(0, g_0).
\end{equation}
By Le Cam's Representation Theorem, finite-dimensional weak convergence of log-likelihood ratio processes implies that the finite sub-experiments $\mathcal{E}_N(\{h_1, \dots, h_k\})$ converge in deficiency distance to $\mathcal{E}_\infty(\{h_1, \dots, h_k\})$.

\paragraph{Step 3: Uniform Compact Extension via Local Randomization.}
Because $\mathbb{K}$ is compact, for every $\epsilon > 0$, there exists a finite $\delta$-grid $\mathbb{K}_\delta = \{h_1, \dots, h_m\} \subset \mathbb{K}$ such that $\sup_{h \in \mathbb{K}} \min_{1 \le j \le m} \|h - h_j\| < \delta$.
The uniform continuity of the Gaussian covariance mapping $h \mapsto h^\top g_0 h$ together with $o_{P_{\theta_0}^{(N)}}(1)$ stochastic remainder bounds over $\mathbb{K}$ establishes that the randomization Markov kernels $K_N: \mathcal{X}^N \to \mathcal{P}(\mathbb{R}^d)$ constructed on the finite grid extend uniformly to the full compact set $\mathbb{K}$:
\begin{equation}
	\lim_{N \to \infty} \Delta\left( \mathcal{E}_N(\mathbb{K}), \, \mathcal{E}_\infty(\mathbb{K}) \right) \le \limsup_{N \to \infty} \Delta\left( \mathcal{E}_N(\mathbb{K}_\delta), \, \mathcal{E}_\infty(\mathbb{K}_\delta) \right) + C \delta = C \delta.
\end{equation}
Taking $\delta \to 0$ proves Equation \eqref{eq:lecam_deficiency_zero_lemma}.
\end{proof}

\begin{lemma}[Connection Dissolution Identity and Asymptotic Flatness]
	\label{lem:connection_dissolution}
	Let $(M, g, \nabla^{(\alpha)})$ be a smooth $d$-dimensional information manifold with $N$-sample log-likelihood $L^{(N)}(\theta) = \sum_{m=1}^N \ell(X_m; \theta)$ for i.i.d. observations $X_1, \dots, X_N \sim P_{\theta_0}$. Define the local scaling map $\psi_N: \mathbb{K} \to M$ on a compact set $\mathbb{K} \subset \mathbb{R}^d$ by $\psi_N(h) = \theta_0 + \frac{h}{\sqrt{N}}$, and let $\Lambda_N(h) \equiv L^{(N)}\left(\theta_0 + \frac{h}{\sqrt{N}}\right) - L^{(N)}(\theta_0)$ denote the localized log-likelihood ratio process.
	
	Under High-Order Local Regularity (Assumption \ref{asm:local_regularity}), the pulled-back $\alpha$-connection $3$-tensor field $\left[\psi_N^* \nabla^{(\alpha,N)}\right]_{ijk}(h)$ satisfies:
	\begin{enumerate}
		\item \textbf{Explicit Scaling Representation:} For all $h \in \mathbb{K}$,
		\begin{equation}
			\left[\psi_N^* \nabla^{(\alpha,N)}\right]_{ijk}(h) = \frac{1}{\sqrt{N}} \Gamma_{ijk}^{(\alpha,1)}\left(\theta_0 + \frac{h}{\sqrt{N}}\right),
		\end{equation}
		where $\Gamma_{ijk}^{(\alpha,1)}(\theta)$ denotes the single-sample $\alpha$-connection Christoffel symbol of the first kind.
		
		\item \textbf{Asymptotic Connection Dissolution:} As $N \to \infty$, the pulled-back connection vanishes uniformly on compact sets to the flat Euclidean connection $\nabla^{(0)}$:
		\begin{equation}
			\lim_{N \to \infty} \sup_{h \in \mathbb{K}} \left| \left[\psi_N^* \nabla^{(\alpha,N)}\right]_{ijk}(h) \right| = 0 \equiv \Gamma_{ijk}^{(0)}.
		\end{equation}
		
		\item \textbf{Third-Order Cumulant Identity:} The third derivative of the expected localized log-likelihood ratio evaluated at $h = 0$ satisfies:
		\begin{equation}
			\left. \frac{\partial^3 \mathbb{E}_{\theta_0}[\Lambda_N(h)]}{\partial h^i \partial h^j \partial h^k} \right|_{h=0} = -\sqrt{N} \cdot \left[\psi_N^* \nabla^{(\alpha,N)}\right]_{ijk}(0) = -\Gamma_{ijk}^{(\alpha,1)}(\theta_0).
		\end{equation}
	\end{enumerate}
\end{lemma}

\begin{proof}
	We proceed in three structural stages.
	
	\paragraph{Stage 1: Tensor Pull-Back and $N$-Sample Additivity.}
	The map $\psi_N: h \mapsto \theta(h) = \theta_0 + \frac{h}{\sqrt{N}}$ induces smooth coordinate transformations between local coordinates $h^i$ on $\mathbb{K}$ and manifold coordinates $\theta^a$ on $M$. The components of the Jacobi matrix of $\psi_N$ are:
	\begin{equation}
		J^a_i(h) \equiv \frac{\partial \theta^a}{\partial h^i} = \frac{1}{\sqrt{N}} \delta^a_i,
	\end{equation}
	where $\delta^a_i$ is the Kronecker delta.
	
	Applying the standard covariant pull-back transformation law to the rank-3 lower-index connection Christoffel symbols $\Gamma_{abc}^{(\alpha,N)}(\theta)$ yields:
	\begin{equation}
		\left[\psi_N^* \nabla^{(\alpha,N)}\right]_{ijk}(h) = \sum_{a,b,c=1}^d \frac{\partial \theta^a}{\partial h^i} \frac{\partial \theta^b}{\partial h^j} \frac{\partial \theta^c}{\partial h^k} \Gamma_{abc}^{(\alpha,N)}(\psi_N(h)) = \frac{1}{N^{3/2}} \Gamma_{ijk}^{(\alpha,N)}\left(\theta_0 + \frac{h}{\sqrt{N}}\right).
		\label{eq:proof_pullback_law}
	\end{equation}
	
	Because the $N$-sample log-likelihood is the sum of $N$ independent, identically distributed single-sample log-likelihoods $L^{(N)}(\theta) = \sum_{m=1}^N \ell(X_m;\theta)$, the metric tensor and $\alpha$-connection coefficients scale linearly with sample size $N$:
	\begin{equation}
		\Gamma_{ijk}^{(\alpha,N)}(\theta) = N \cdot \Gamma_{ijk}^{(\alpha,1)}(\theta).
		\label{eq:proof_additivity}
	\end{equation}
	
	Substituting \eqref{eq:proof_additivity} into \eqref{eq:proof_pullback_law} gives:
	\begin{equation}
		\left[\psi_N^* \nabla^{(\alpha,N)}\right]_{ijk}(h) = \frac{1}{N^{3/2}} \cdot N \cdot \Gamma_{ijk}^{(\alpha,1)}\left(\theta_0 + \frac{h}{\sqrt{N}}\right) = \frac{1}{\sqrt{N}} \Gamma_{ijk}^{(\alpha,1)}\left(\theta_0 + \frac{h}{\sqrt{N}}\right),
	\end{equation}
	completing the proof of Part 1.
	
	\paragraph{Stage 2: Uniform Dissolution to Flat Geometry.}
	Under High-Order Local Regularity (Assumption 1), the single-sample connection coefficients $\theta \mapsto \Gamma_{ijk}^{(\alpha,1)}(\theta)$ are $C^1$-continuous on the compact subset $V \subset M$ containing $\theta_0$. Thus, there exists a uniform bound $C \equiv \sup_{\theta \in V} \left| \Gamma_{ijk}^{(\alpha,1)}(\theta) \right| < \infty$.
	
	For sufficiently large $N$, $\psi_N(\mathbb{K}) \subset V$. Taking the supremum over $h \in \mathbb{K}$:
	\begin{equation}
		\sup_{h \in \mathbb{K}} \left| \left[\psi_N^* \nabla^{(\alpha,N)}\right]_{ijk}(h) \right| \le \frac{1}{\sqrt{N}} \sup_{\theta \in V} \left| \Gamma_{ijk}^{(\alpha,1)}(\theta) \right| = \frac{C}{\sqrt{N}}.
	\end{equation}
	Taking $N \to \infty$, $\frac{C}{\sqrt{N}} \to 0$, establishing uniform convergence to zero (the flat Euclidean connection $\Gamma_{ijk}^{(0)} = 0$). This completes Part 2.
	
	\paragraph{Stage 3: Third Variation of Expectation and Information Geometry Identity.}
	Consider the expected localized log-likelihood ratio under $P_{\theta_0}$:
	\begin{equation}
		\mathbb{E}_{\theta_0}[\Lambda_N(h)] = N \cdot \mathbb{E}_{\theta_0} \left[ \ell\left(X; \theta_0 + \frac{h}{\sqrt{N}}\right) - \ell(X; \theta_0) \right].
	\end{equation}
	By differentiation under the integral sign (justified by Assumption 1 local dominance bounds), applying the chain rule $\frac{\partial}{\partial h^i} = \frac{1}{\sqrt{N}} \frac{\partial}{\partial \theta^i}$ three times gives:
	\begin{equation}
		\frac{\partial^3 \mathbb{E}_{\theta_0}[\Lambda_N(h)]}{\partial h^i \partial h^j \partial h^k} = N \left(\frac{1}{\sqrt{N}}\right)^3 \mathbb{E}_{\theta_0} \left[ \left. \frac{\partial^3 \ell(X; \theta)}{\partial \theta^i \partial \theta^j \partial \theta^k} \right|_{\theta = \theta_0 + \frac{h}{\sqrt{N}}} \right].
	\end{equation}
	Evaluating at $h = 0$:
	\begin{equation}
		\left. \frac{\partial^3 \mathbb{E}_{\theta_0}[\Lambda_N(h)]}{\partial h^i \partial h^j \partial h^k} \right|_{h=0} = \frac{1}{\sqrt{N}} \mathbb{E}_{\theta_0} \left[ \partial_i \partial_j \partial_k \ell(X; \theta_0) \right].
		\label{eq:proof_third_deriv_eval}
	\end{equation}
	
	By Amari's fundamental information-geometric identity, the single-sample connection Christoffel symbols of the $\alpha$-connection are related to log-density derivatives by:
	\begin{equation}
		\Gamma_{ijk}^{(\alpha,1)}(\theta_0) = -\mathbb{E}_{\theta_0} \left[ \partial_i \partial_j \partial_k \ell(X; \theta_0) \right] + \frac{1-\alpha}{2} T_{ijk}(\theta_0),
	\end{equation}
	where $T_{ijk}(\theta_0) = \mathbb{E}_{\theta_0}[\partial_i \ell \cdot \partial_j \ell \cdot \partial_k \ell]$ is the Amari skewness tensor. Under the canonical exponential family or flat connection duality ($\alpha = 1$), or directly for the standard 1-connection formulation, $\mathbb{E}_{\theta_0}\left[\partial_i \partial_j \partial_k \ell(X;\theta_0)\right] = -\Gamma_{ijk}^{(\alpha,1)}(\theta_0)$.
	
	Substituting this identity into \eqref{eq:proof_third_deriv_eval}:
	\begin{equation}
		\left. \frac{\partial^3 \mathbb{E}_{\theta_0}[\Lambda_N(h)]}{\partial h^i \partial h^j \partial h^k} \right|_{h=0} = -\frac{1}{\sqrt{N}} \Gamma_{ijk}^{(\alpha,1)}(\theta_0).
	\end{equation}
	Recalling from Part 1 that $\left[\psi_N^* \nabla^{(\alpha,N)}\right]_{ijk}(0) = \frac{1}{\sqrt{N}} \Gamma_{ijk}^{(\alpha,1)}(\theta_0)$, we obtain:
	\begin{equation}
		\left. \frac{\partial^3 \mathbb{E}_{\theta_0}[\Lambda_N(h)]}{\partial h^i \partial h^j \partial h^k} \right|_{h=0} = -\sqrt{N} \left[ \psi_N^* \nabla^{(\alpha,N)} \right]_{ijk}(0) = -\Gamma_{ijk}^{(\alpha,1)}(\theta_0),
	\end{equation}
	which establishes Part 3 and completes the proof.
\end{proof}

\begin{lemma}[Accelerated Curvature Annihilation]
	\label{lem:accelerated_curvature_annihilation}
	Let $\mathcal{M}_N = (M, g^{(N)}, \nabla^{(\alpha,N)})$ be an information manifold sequence defined on a parameter domain $M \subset \mathbb{R}^d$ corresponding to $N$ independent and identically distributed (i.i.d.) observations $X_1, \dots, X_N \sim P_\theta$. Let $\mathcal{R}^{(\alpha,N)}$ denote the type-$(0,4)$ Riemann-Christoffel curvature tensor of Amari's $\alpha$-connection $\nabla^{(\alpha,N)}$. Let $\mathbb{K} \subset \mathbb{R}^d$ be a compact set, and let $\psi_N: \mathbb{K} \to M$ be the local coordinate scaling map defined by $\psi_N(h) = \theta_0 + \frac{h}{\sqrt{N}}$.
	
	Under High-Order Local Regularity (Assumption \ref{asm:local_regularity}), the pulled-back Riemann curvature tensor field $\left[\psi_N^* \mathcal{R}^{(\alpha,N)}\right]_{ijmn}(h)$ satisfies:
	\begin{enumerate}
		\item \textbf{Explicit $\mathcal{O}(N^{-1})$ Scaling Identity:} For all $h \in \mathbb{K}$,
		\begin{equation}
			\left[ \psi_N^* \mathcal{R}^{(\alpha,N)} \right]_{ijmn}(h) = \frac{1}{N} \mathcal{R}_{ijmn}^{(\alpha,1)}\left( \theta_0 + \frac{h}{\sqrt{N}} \right),
		\end{equation}
		where $\mathcal{R}_{ijmn}^{(\alpha,1)}(\theta)$ is the single-sample Riemann curvature tensor.
		
		\item \textbf{Uniform Asymptotic Annihilation:} As $N \to \infty$, the pulled-back curvature tensor vanishes uniformly on $\mathbb{K}$ at rate $\mathcal{O}(N^{-1})$ to the target flat Euclidean curvature $\mathcal{R}^{(0)} \equiv 0$:
		\begin{equation}
			\lim_{N \to \infty} \sup_{h \in \mathbb{K}} \left\| \left[ \psi_N^* \mathcal{R}^{(\alpha,N)} \right](h) \right\|_\infty = 0 \equiv \mathcal{R}_{ijmn}^{(0)}.
		\end{equation}
	\end{enumerate}
\end{lemma}

\begin{proof}
	We structure the proof into four rigorous mathematical stages.
	
	\paragraph{Stage 1: Tensor Pull-Back under Local Diffeomorphism.}
	The scaling map $\psi_N: h \mapsto \theta(h) = \theta_0 + \frac{h}{\sqrt{N}}$ is a smooth local diffeomorphism from $\mathbb{K}$ into $M$. Its Jacobi matrix has constant diagonal entries:
	\begin{equation}
		J^a_i(h) \equiv \frac{\partial \theta^a}{\partial h^i} = \frac{1}{\sqrt{N}} \delta^a_i.
	\end{equation}
	Applying the covariant tensor pull-back formula to the $4$-index covariant Riemann-Christoffel curvature tensor $\mathcal{R}^{(\alpha,N)}$ gives:
	\begin{equation}
		\left[ \psi_N^* \mathcal{R}^{(\alpha,N)} \right]_{ijmn}(h) = \sum_{a,b,c,d=1}^d \frac{\partial \theta^a}{\partial h^i} \frac{\partial \theta^b}{\partial h^j} \frac{\partial \theta^c}{\partial h^m} \frac{\partial \theta^d}{\partial h^n} \mathcal{R}_{abcd}^{(\alpha,N)}\left( \psi_N(h) \right).
		\label{eq:pullback_def}
	\end{equation}
	Substituting $J^a_i(h) = \frac{1}{\sqrt{N}} \delta^a_i$ into \eqref{eq:pullback_def} factors out the fourfold Jacobian product:
	\begin{equation}
		\left[ \psi_N^* \mathcal{R}^{(\alpha,N)} \right]_{ijmn}(h) = \left( \frac{1}{\sqrt{N}} \right)^4 \mathcal{R}_{ijmn}^{(\alpha,N)}\left( \theta_0 + \frac{h}{\sqrt{N}} \right) = \frac{1}{N^2} \mathcal{R}_{ijmn}^{(\alpha,N)}\left( \theta_0 + \frac{h}{\sqrt{N}} \right).
		\label{eq:jacobian_scaling}
	\end{equation}
	
	\paragraph{Stage 2: Sample-Size Linearity ($N$-Sample Additivity).}
	For $N$ i.i.d. observations, the $N$-sample log-likelihood function is $L^{(N)}(\theta) = \sum_{k=1}^N \ell(X_k; \theta)$. Consequently, the metric $G^{(N)}(\theta) = N \cdot G^{(1)}(\theta)$ and connection symbols $\Gamma_{ijk}^{(\alpha,N)}(\theta) = N \cdot \Gamma_{ijk}^{(\alpha,1)}(\theta)$ scale linearly with sample size $N$.
	
	The lower-index Riemann-Christoffel curvature tensor $\mathcal{R}_{ijmn}^{(\alpha,N)}$ is defined via Christoffel symbols by:
	\begin{equation}
		\mathcal{R}_{ijmn}^{(\alpha,N)} = \partial_m \Gamma_{inj}^{(\alpha,N)} - \partial_n \Gamma_{imj}^{(\alpha,N)} + \sum_{r,s} \left( g^{(N)} \right)^{rs} \left( \Gamma_{smj}^{(\alpha,N)} \Gamma_{irn}^{(\alpha,N)} - \Gamma_{snj}^{(\alpha,N)} \Gamma_{irm}^{(\alpha,N)} \right).
	\end{equation}
	Since $\left( g^{(N)} \right)^{rs} = \frac{1}{N} (g^{(1)})^{rs}$ while $\Gamma_{ijk}^{(\alpha,N)} = N \cdot \Gamma_{ijk}^{(\alpha,1)}$, the product term scales as $\frac{1}{N} \cdot N^2 = N$. Therefore, the entire expression obeys exact linear additivity:
	\begin{equation}
		\mathcal{R}_{ijmn}^{(\alpha,N)}(\theta) = N \cdot \mathcal{R}_{ijmn}^{(\alpha,1)}(\theta).
		\label{eq:curvature_additivity}
	\end{equation}
	
	\paragraph{Stage 3: Application of Universal Tensor Scaling Law ($r=4$).}
	By Theorem~\ref{thm:universal_scaling_law}, any rank-$r$ covariant tensor field $T^{(N)}$ satisfying $N$-linear additivity $T^{(N)} = N \cdot T^{(1)}$ obeys the pullback relation:
	\begin{equation}
		\left[ \psi_N^* T^{(N)} \right]_{i_1 \dots i_r}(h) = \frac{1}{N^{r/2 - 1}} T_{i_1 \dots i_r}^{(1)}\left( \theta_0 + \frac{h}{\sqrt{N}} \right).
	\end{equation}
	Substituting $r = 4$ for the 4-index Riemann curvature tensor and combining Equations~\eqref{eq:jacobian_scaling} and \eqref{eq:curvature_additivity}:
	\begin{equation}
		\left[ \psi_N^* \mathcal{R}^{(\alpha,N)} \right]_{ijmn}(h) = \frac{1}{N^2} \cdot N \cdot \mathcal{R}_{ijmn}^{(\alpha,1)}\left( \theta_0 + \frac{h}{\sqrt{N}} \right) = \frac{1}{N} \mathcal{R}_{ijmn}^{(\alpha,1)}\left( \theta_0 + \frac{h}{\sqrt{N}} \right),
	\end{equation}
	which proves Part 1.
	
	\paragraph{Stage 4: Compact Uniform Convergence.}
	Under Assumption 1 (High-Order Local Regularity), the single-sample curvature components $\theta \mapsto \mathcal{R}_{ijmn}^{(\alpha,1)}(\theta)$ are $C^0$-continuous on a compact neighborhood $V \subset M$ containing $\theta_0$. Thus, the maximum tensor norm is bounded by a finite constant:
	\begin{equation}
		M_{\mathcal{R}} \equiv \sup_{\theta \in V} \left\| \mathcal{R}^{(\alpha,1)}(\theta) \right\|_\infty < \infty.
	\end{equation}
	For all $N \ge N_0$ such that $\psi_N(\mathbb{K}) \subset V$, taking the supremum over $h \in \mathbb{K}$ yields:
	\begin{equation}
		\sup_{h \in \mathbb{K}} \left\| \left[ \psi_N^* \mathcal{R}^{(\alpha,N)} \right](h) \right\|_\infty \le \frac{1}{N} \sup_{\theta \in V} \left\| \mathcal{R}^{(\alpha,1)}(\theta) \right\|_\infty \le \frac{M_{\mathcal{R}}}{N}.
	\end{equation}
	Taking the limit as $N \to \infty$:
	\begin{equation}
		\lim_{N \to \infty} \sup_{h \in \mathbb{K}} \left\| \left[ \psi_N^* \mathcal{R}^{(\alpha,N)} \right](h) \right\|_\infty \le \lim_{N \to \infty} \frac{M_{\mathcal{R}}}{N} = 0 \equiv \mathcal{R}_{ijmn}^{(0)},
	\end{equation}
	which proves Part 2 and completes the proof of the lemma.
\end{proof}

\begin{theorem}[Geometric-Operational Equivalence Theorem]
	\label{thm:geom_oper_equivalence}
	Under High-Order Local Regularity (Assumption~\ref{asm:local_regularity}), the Geometric Priority Paradigm (\textnormal{Axiom~\ref{ax:geom_priority}}) and the Statistical Operationalism Paradigm (\textnormal{Axiom~\ref{ax:stat_oper}}) are logically and mathematically equivalent across every compact coordinate cage $\mathbb{K} \subset \mathbb{R}^d$:
	\begin{equation}
		IG_N \xrightarrow{\text{geom}} \mathcal{M}_\infty \iff IG_N \xrightarrow{\text{oper}} \mathcal{M}_\infty.
		\label{eq:theorem_equivalence_main}
	\end{equation}
\end{theorem}

\begin{proof}
	We prove necessity ($\implies$) and sufficiency ($\impliedby$) independently.
	
	\subsection*{Part 1: Necessity ($IG_N \xrightarrow{\text{geom}} \mathcal{M}_\infty \implies IG_N \xrightarrow{\text{oper}} \mathcal{M}_\infty$)}

	Assume Axiom \ref{ax:geom_priority} ($IG_N \xrightarrow{\text{geom}} \mathcal{M}_\infty$) holds. Then, under the localized microscope map $\psi_N(h) = \theta_0 + \frac{h}{\sqrt{N}}$, the pulled-back tensor fields satisfy uniform convergence on every compact cage $\mathbb{K} \subset \mathbb{R}^d$:
	\begin{align}
		\lim_{N \to \infty} \sup_{h \in \mathbb{K}} \left\| \left[\psi_N^* G^{(N)}\right](h) - g_0 \right\|_\infty &= 0, \label{eq:proof_metric_limit} \\
		\lim_{N \to \infty} \sup_{h \in \mathbb{K}} \left\| \left[\psi_N^* \nabla^{(\alpha,N)}\right](h) - \nabla^{(0)} \right\|_\infty &= 0, \label{eq:proof_connection_limit} \\
		\lim_{N \to \infty} \sup_{h \in \mathbb{K}} \left\| \left[\psi_N^* \mathcal{R}^{(\alpha,N)}\right](h) \right\|_\infty &= 0. \label{eq:proof_curvature_limit}
	\end{align}
	
	By Lemma \ref{lem:pillar_2_taylor_geometry} (Pillar II: Taylor-Geometry Algebraic Functional Identity), the localized Radon-Nikodym log-likelihood ratio process $\Lambda_N(h) \equiv \log \frac{d P_{\theta_0 + h/\sqrt{N}}^{(N)}}{d P_{\theta_0}^{(N)}}(X^N)$ admits the exact functional decomposition:
	\begin{equation}
		\Lambda_N(h) = \sum_{a=1}^d h^a \left[ \psi_N^*(dL^{(N)}) \right]_a(0) - \frac{1}{2} \sum_{i,j=1}^d h^i h^j \left[ \psi_N^* G^{(N)} \right]_{ij}(0) + \mathcal{R}_N(h),
		\label{eq:proof_taylor_decomp}
	\end{equation}
	where the non-linear stochastic remainder $\mathcal{R}_N(h)$ satisfies the uniform geometric envelope bound across $\mathbb{K}$:
	\begin{equation}
		\sup_{h \in \mathbb{K}} |\mathcal{R}_N(h)| \le \frac{\|h\|_2^3}{6} \sup_{h \in \mathbb{K}} \left| \left[ \psi_N^* \nabla^{(\alpha,N)} \right]_{ijk}(h) \right| + \mathcal{O}_p\left( \sup_{h \in \mathbb{K}} \left\| \left[ \psi_N^* \mathcal{R}^{(\alpha,N)} \right]_{ijmn}(h) \right\|_\infty \right).
		\label{eq:proof_remainder_envelope}
	\end{equation}
	
	Substituting the geometric convergence limits from Equations \eqref{eq:proof_connection_limit} and \eqref{eq:proof_curvature_limit} into Equation \eqref{eq:proof_remainder_envelope}:
	\begin{equation}
		\sup_{h \in \mathbb{K}} |\mathcal{R}_N(h)| = \mathcal{O}\left(N^{-1/2}\right) + \mathcal{O}_p\left(N^{-1}\right) = \mathcal{O}_p\left(N^{-1/2}\right) \xrightarrow{P_{\theta_0}^{(N)}} 0.
		\label{eq:proof_remainder_vanish}
	\end{equation}
	
	Furthermore, by Theorem \ref{thm:score_transformation}, the score 1-form at $h=0$ satisfies multivariate Central Limit Theorem weak convergence:
	\begin{equation}
		\psi_N^*(dL^{(N)})(0) \xrightarrow{d} \Delta_\infty \sim \mathcal{N}(0, g_0).
		\label{eq:proof_score_normal}
	\end{equation}
	
	Substituting Equations \eqref{eq:proof_metric_limit}, \eqref{eq:proof_remainder_vanish}, and \eqref{eq:proof_score_normal} into Equation \eqref{eq:proof_taylor_decomp}, the localized log-likelihood process satisfies the uniform linear-quadratic expansion:
	\begin{equation}
		\Lambda_N(h) = h^\top \Delta_N(\theta_0) - \frac{1}{2} h^\top g_0 h + o_{P_{\theta_0}^{(N)}}(1), \quad \text{where } \Delta_N(\theta_0) \xrightarrow{d} \mathcal{N}(0, g_0).
		\label{eq:proof_lan_form_established}
	\end{equation}
	
	Applying Lemma \ref{lem:lecam_deficiency_equivalence} (Le Cam Experiment Deficiency Limit Lemma) directly to Equation \eqref{eq:proof_lan_form_established} establishes contiguity and proves that Le Cam's experiment deficiency distance between $\mathcal{E}_N(\mathbb{K})$ and $\mathcal{E}_\infty(\mathbb{K})$ vanishes asymptotically:
	\begin{equation}
		\lim_{N \to \infty} \Delta\left( \mathcal{E}_N(\mathbb{K}), \, \mathcal{E}_\infty(\mathbb{K}) \right) = 0.
	\end{equation}
	This fulfills Axiom \ref{ax:stat_oper} ($IG_N \xrightarrow{\text{oper}} \mathcal{M}_\infty$), completing the proof of Part 1 (Necessity).

	


\subsection*{Part 2: Sufficiency ($IG_N \xrightarrow{\text{oper}} \mathcal{M}_\infty \implies IG_N \xrightarrow{\text{geom}} \mathcal{M}_\infty$)}

	Assume Axiom \ref{ax:stat_oper} holds. That is, for every compact coordinate cage $\mathbb{K} \subset \mathbb{R}^d$, the Le Cam deficiency distance between the empirical localized experiment $\mathcal{E}_N(\mathbb{K})$ and the canonical Gaussian shift experiment $\mathcal{E}_\infty(\mathbb{K})$ vanishes asymptotically:
	\begin{equation}
		\lim_{N \to \infty} \Delta\left( \mathcal{E}_N(\mathbb{K}), \, \mathcal{E}_\infty(\mathbb{K}) \right) = 0.
		\label{eq:proof_part2_assumption}
	\end{equation}
	
	By Le Cam's Representation Theorem, Equation \eqref{eq:proof_part2_assumption} implies that the localized Radon-Nikodym log-likelihood ratio process $\Lambda_N(h) \equiv \sum_{n=1}^N \log \frac{p(X_n; \theta_0 + h/\sqrt{N})}{p(X_n; \theta_0)}$ converges weakly under $P_{\theta_0}^{(N)}$ to the canonical quadratic Gaussian likelihood process:
	\begin{equation}
		\Lambda_N(h) \xrightarrow{d} h^\top Z - \frac{1}{2} h^\top A h, \quad \text{where } Z \sim \mathcal{N}(0, A),
		\label{eq:proof_part2_weak_limit}
	\end{equation}
	for a symmetric, positive-definite matrix $A \in \mathbb{R}^{d \times d}$. \footnote{In Le Cam's Local Asymptotic Normality (LAN) framework, the weak convergence of the localized log-likelihood ratio process $\Lambda_N(h) \xrightarrow{d} h^\top Z - \frac{1}{2} h^\top A h$ under the background probability measure $P_{\theta_0}^{(N)}$ dictates that the limiting random vector $Z$ is Gaussian with zero mean and covariance matrix $A = g^{(1)}(\theta_0)$:
		
		\begin{enumerate}
			\item \textbf{Zero Mean ($\mathbb{E}_{\theta_0}[Z] = \mathbf{0}$):} Under High-Order Local Regularity, the single-observation score vector satisfies $\mathbb{E}_{\theta_0}\left[\nabla_\theta \log p(X; \theta_0)\right] = \int \nabla_\theta p(x; \theta_0) d\mu(x) = \nabla_\theta (1) = \mathbf{0}$. By linearity of expectation, the normalized sample score vector $V^{(N)} \equiv \frac{1}{\sqrt{N}} \sum_{n=1}^N \nabla_\theta \log p(X_n; \theta_0)$ has zero mean for every $N$. Consequently, by the Multivariate Central Limit Theorem, the weak limit $Z$ must inherit a mean of zero.
			
			\item \textbf{Covariance Matrix ($\mathrm{Var}_{\theta_0}(Z) = A$):} The covariance matrix of the single-sample score vector is defined as the Fisher Information Metric $g^{(1)}(\theta_0) \equiv \mathbb{E}_{\theta_0}\left[ \nabla_\theta \log p(X; \theta_0) \nabla_\theta \log p(X; \theta_0)^\top \right]$. Under $i.i.d.$ sampling, the variance of the normalized sum $V^{(N)}$ is identically equal to $g^{(1)}(\theta_0)$ for all $N$. Identifying $A \equiv g^{(1)}(\theta_0)$ establishes that $A$ is the precise covariance matrix of $Z$.
			
			\item \textbf{Probability Preservation under Shift:} The zero-mean structure $Z \sim \mathcal{N}(\mathbf{0}, A)$ directly guarantees that the limiting Radon-Nikodym derivative $L(h) \equiv \exp\left(h^\top Z - \frac{1}{2} h^\top A h\right)$ satisfies $\mathbb{E}_{\theta_0}[L(h)] = \exp\left(\frac{1}{2} h^\top A h - \frac{1}{2} h^\top A h\right) = 1$, preserving total probability mass under Le Cam's contiguity principles.
	\end{enumerate}}
	
	We now extract and verify the convergence of the four pulled-back geometric tensor fields ($r \in \{1, 2, 3, 4\}$) on the tangent workspace $\mathcal{M}_\infty$:
	
	\paragraph{1. Metric Stabilization ($r = 2$):}
	Differentiating the expected log-likelihood ratio process $\mathbb{E}_{\theta_0}[\Lambda_N(h)]$ twice with respect to the local displacement coordinates $h^i$ and $h^j$ at $h = 0$ yields:
	\begin{equation}
		\left. \frac{\partial^2 \mathbb{E}_{\theta_0}[\Lambda_N(h)]}{\partial h^i \partial h^j} \right|_{h=0} = -\left[ \psi_N^* G^{(N)} \right]_{ij}(0).
		\label{eq:proof_part2_metric_extract}
	\end{equation}
	From the weak limit in Equation \eqref{eq:proof_part2_weak_limit}, the expected second variation equals $-A_{ij}$. Evaluating at the background state $\theta_0$ identifies $A \equiv g_0 = g^{(1)}(\theta_0)$.
	
	By tensor pullback mechanics under the localized microscope map $\psi_N(h) = \theta_0 + \frac{h}{\sqrt{N}}$ and $N$-sample metric additivity $G^{(N)}(\theta) = N g^{(1)}(\theta)$:
	\begin{equation}
		\left[\psi_N^* G^{(N)}\right]_{ij}(h) = g_{ij}^{(1)}\left(\theta_0 + \frac{h}{\sqrt{N}}\right).
	\end{equation}
	Under High-Order Local Regularity, $g_{ij}^{(1)}(\theta)$ is $C^2$-smooth on $\mathcal{U}(\theta_0)$. Applying a multivariate Taylor expansion around $h = 0$ over the compact set $\mathbb{K}$:
	\begin{equation}
		\lim_{N \to \infty} \sup_{h \in \mathbb{K}} \left\| \left[\psi_N^* G^{(N)}\right](h) - g_0 \right\|_\infty \le \lim_{N \to \infty} \frac{C_\mathbb{K} M_g}{\sqrt{N}} = 0,
		\label{eq:proof_part2_metric_proven}
	\end{equation}
	which fulfills condition \eqref{eq:ax1_metric_stabilization} of Axiom \ref{ax:geom_priority}.
	
	\paragraph{2. Score Fluctuation ($r = 1$):}
	Taking the first variation of $\Lambda_N(h)$ at $h = 0$ recovers the score 1-form $\psi_N^*(dL^{(N)})(0)$. By Equation \eqref{eq:proof_part2_weak_limit}, its limit covariance is $A = g_0$, matching the normalized Gaussian score field $\Delta_\infty \sim \mathcal{N}(0, g_0)$.
	
	\paragraph{3. Connection Dissolution ($r = 3$):}
	By Lemma \ref{lem:connection_dissolution}, the pulled-back Amari $\alpha$-connection Christoffel symbols $\left[\psi_N^* \nabla^{(\alpha,N)}\right]_{ijk}(h)$ obey the pullback scaling relation:
	\begin{equation}
		\left[\psi_N^* \nabla^{(\alpha,N)}\right]_{ijk}(h) = \frac{1}{\sqrt{N}} \Gamma_{ijk}^{(\alpha,1)}\left(\theta_0 + \frac{h}{\sqrt{N}}\right).
	\end{equation}
	Under Assumption \ref{asm:local_regularity}, the single-sample connection symbols $\Gamma_{ijk}^{(\alpha,1)}(\theta)$ are $C^1$-continuous and bounded by $M_\Gamma < \infty$ on a compact neighborhood containing $\theta_0$. Taking the uniform supremum over $\mathbb{K}$:
	\begin{equation}
		\lim_{N \to \infty} \sup_{h \in \mathbb{K}} \left| \left[ \psi_N^* \nabla^{(\alpha,N)} \right]_{ijk}(h) \right| = \lim_{N \to \infty} \frac{1}{\sqrt{N}} \sup_{h \in \mathbb{K}} \left| \Gamma_{ijk}^{(\alpha,1)}\left(\theta_0 + \frac{h}{\sqrt{N}}\right) \right| \le \lim_{N \to \infty} \frac{M_\Gamma}{\sqrt{N}} = 0 \equiv \Gamma_{ijk}^{(0)},
		\label{eq:proof_part2_connection_proven}
	\end{equation}
	which fulfills condition \eqref{eq:ax1_connection_dissolution} of Axiom \ref{ax:geom_priority}.
	
	\paragraph{4. Accelerated Curvature Annihilation ($r = 4$):}
	By Lemma \ref{lem:accelerated_curvature_annihilation}, the pulled-back type-$(0,4)$ Riemann curvature tensor field obeys the quadratic scaling law:
	\begin{equation}
		\left[ \psi_N^* \mathcal{R}^{(\alpha,N)} \right]_{ijmn}(h) = \frac{1}{N} \mathcal{R}_{ijmn}^{(\alpha,1)}\left( \theta_0 + \frac{h}{\sqrt{N}} \right).
	\end{equation}
	Under Assumption \ref{asm:local_regularity}, single-sample curvature components are continuous and bounded by $M_{\mathcal{R}} < \infty$ on compact neighborhoods. Taking the uniform supremum over $\mathbb{K}$:
	\begin{equation}
		\lim_{N \to \infty} \sup_{h \in \mathbb{K}} \left\| \left[ \psi_N^* \mathcal{R}^{(\alpha,N)} \right](h) \right\|_\infty \le \lim_{N \to \infty} \frac{M_{\mathcal{R}}}{N} = 0 \equiv \mathcal{R}_{ijmn}^{(0)},
		\label{eq:proof_part2_curvature_proven}
	\end{equation}
	which fulfills condition \eqref{eq:ax1_curvature_annihilation} of Axiom \ref{ax:geom_priority}.
	
	\paragraph{Conclusion:}
	Combining Equations \eqref{eq:proof_part2_metric_proven}, \eqref{eq:proof_part2_connection_proven}, and \eqref{eq:proof_part2_curvature_proven} establishes that all three constitutional pillars of Axiom \ref{ax:geom_priority} ($IG_N \xrightarrow{\text{geom}} \mathcal{M}_\infty$) are satisfied, completing the proof of Sufficiency and establishing this Theorem.

\end{proof}

\section{Fisher-Compatible Ehresmann Connections and Tangent Space Splitting}
\label{sec:ehresmann_connections}

In Section \ref{sec:equivalence_proof}, we established the Geometric-Operational Equivalence Theorem under the ambient regularity assumption that the single-sample Fisher Information Metric $g^{(1)}(\theta_0)$ is strictly positive-definite ($\lambda_{\min}(g^{(1)}(\theta_0)) > 0$) across the local neighborhood $\mathcal{U}(\theta_0) \subset \Theta$. However, in modern statistical physics, over-parameterized statistical models, latent variable models, singular exponential families, and deep learning architectures, parameter spaces $\Theta \subset \mathbb{R}^D$ ($D \gg d$) exhibit non-identifiable gauge symmetries and rank-deficient metric tensors ($\det g(\theta) = 0$).

To extend the Second Edge Theorem to over-parameterized statistical universes without enforcing artificial ambient non-degeneracy, this section constructs the \textbf{Fisher-Compatible Ehresmann Connection Framework}. We formalize the statistical fiber bundle $(\Theta, \mathcal{B}, \pi)$, isolate internal structural gauge directions within the vertical sub-bundle $V_\theta \equiv \ker(d\pi_\theta)$, and construct the horizontal distribution $H_\theta \equiv V_\theta^{\perp_g}$ as the Fisher-metric orthogonal complement. We prove that the restricted metric $g_\theta\big|_{H_\theta \times H_\theta}$ remains strictly positive-definite and non-singular, establishing the geometric foundation required for horizontal unification in Section \ref{sec:module5_unification}.

\subsection{Motivation, Extension Requirements, and Ambient Breakdown}
\label{subsec:ehresmann_motivation}

The primary objective of this section is to resolve a fundamental mathematical barrier that emerges when transitioning from strictly identifiable statistical models to high-dimensional, over-parameterized architectures. In an over-parameterized system where $\dim(\Theta) = D > d = \dim(\mathcal{B})$, distinct parameter vectors $\theta_1 \neq \theta_2 \in \Theta$ can map to identical probability distributions $P_{\theta_1} = P_{\theta_2} \in \mathcal{B}$. This structural redundancy induces non-trivial \emph{gauge orbits} (fibers) in $\Theta$, causing the ambient single-sample Fisher Information Metric tensor:
\begin{equation}
	g_{ij}^{(1)}(\theta) \equiv \mathbb{E}_\theta \left[ \frac{\partial \log p(X;\theta)}{\partial \theta^i} \frac{\partial \log p(X;\theta)}{\partial \theta^j} \right]
	\label{eq:ambient_fisher_def}
\end{equation}
to become strictly rank-deficient, with $\mathrm{rank}(g^{(1)}(\theta)) = d < D$ and $\det g^{(1)}(\theta) = 0$.

\subsubsection{Why Results in Section \ref{sec:module3}  and Section \ref{sec:equivalence_proof} Require Extension}

Directly applying the geometric and operational frameworks formulated in Section \ref{sec:module3} (Manifold Convergence Axiomatization) and Section \ref{sec:equivalence_proof} (Geometric-Operational Equivalence Theorem) to an ambient space with degenerate $g^{(1)}(\theta_0)$ leads to immediate mathematical collapse:

\begin{enumerate}[label=\textnormal{(\roman*)}]
	\item \textbf{Breakdown of Section \ref{sec:module3} (Manifold Convergence):} In Section \ref{sec:module3}, the local microscope map $\psi_N(h) = \theta_0 + \frac{h}{\sqrt{N}}$ was constructed under the implicit requirement that the pulled-back metric sequence $g_N \equiv \psi_N^* g^{(N)}$ induces a well-defined, non-degenerate Riemannian distance function on $\mathbb{R}^D$. When $g^{(1)}(\theta_0)$ has a non-trivial kernel $\ker(g^{(1)}(\theta_0)) \neq \{\mathbf{0}\}$, the metric distance along kernel directions vanishes identically ($d_g(\theta, \theta + v) = 0$ for $v \in \ker(g^{(1)}(\theta_0))$). Consequently, the metric topology fails to be Hausdorff on $\Theta$, and the double-completeness and uniform convergence bounds established on compact coordinate cages $\mathbb{K} \subset \mathbb{R}^D$ collapse along unidentifiable directions.
	
	\item \textbf{Breakdown of Section \ref{sec:equivalence_proof} (Geometric-Operational Equivalence):} In Section \ref{sec:equivalence_proof}, the equivalence between geometric curvature objects and operational estimator variances explicitly relied on the metric inverse tensor $(g^{(1)})^{ij}$ and a uniform lower bound $\lambda_{\min}(g^{(1)}(\theta_0)) \ge \kappa_0 > 0$. When $\det g^{(1)}(\theta_0) = 0$:
	\begin{itemize}
		\item The matrix inverse $(g^{(1)})^{-1}$ does not exist, causing Christoffel symbols $$\Gamma_{ij}^k = \frac{1}{2}(g^{(1)})^{kl}\left(\partial_i g_{jl}^{(1)} + \partial_j g_{il}^{(1)} - \partial_l g_{ij}^{(1)}\right)$$ and Riemann curvature contractions to become mathematically undefined or singular.
		\item The operational covariance matrix $\Sigma_N$ of any parameter estimator along gauge directions diverges to infinity ($\lambda_{\max}(\Sigma_N) \to \infty$), destroying the dual isomorphism between statistical estimation risk and geometric metric length.
	\end{itemize}
\end{enumerate}

\subsubsection{Mathematical Breakdown Lemma}

To rigorously demonstrate the failure of the ambient formulation under metric degeneracy, we state and prove the following breakdown lemma.

\begin{lemma}[Ambient Degeneracy and Operational Breakdown]
	\label{lem:ambient_breakdown}
	Let $\Theta \subset \mathbb{R}^D$ be an ambient parameter space with $D > d$, and suppose $g^{(1)}(\theta_0)$ is rank-deficient with kernel $V_{\theta_0} \equiv \ker(g^{(1)}(\theta_0)) \subset T_{\theta_0}\Theta$ of dimension $D - d \ge 1$. Then:
	\begin{enumerate}[label=\textnormal{(\roman*)}]
		\item Under the isotropic microscope chart $\psi_N(h) = \theta_0 + \frac{h}{\sqrt{N}}$ for $h \in \mathbb{R}^D$, the localized metric sequence $g_N(h) \equiv \psi_N^* g^{(N)}(h)$ satisfies:
		\begin{equation}
			\inf_{h \in \mathbb{R}^D, \, \|h\|=1} h^T g_N(h) h = 0, \quad \forall N \in \mathbb{N},
			\label{eq:zero_eigenvalue_microscope}
		\end{equation}
		preventing the uniform lower bound $\lambda_{\min}(g_N) \ge \kappa_0 > 0$ required for Section \ref{sec:module3} convergence.
		
		\item For any $v \in V_{\theta_0} \setminus \{\mathbf{0}\}$ and any locally unbiased parameter estimator $\hat{\theta}_N$ on $\Theta$, the operational estimation variance along $v$ satisfies the unbounded limit:
		\begin{equation}
			\lim_{N \to \infty} N \cdot \mathrm{Var}_{\theta_0}\left( v^T \hat{\theta}_N \right) = +\infty,
			\label{eq:divergent_operational_variance}
		\end{equation}
		whereas the geometric Fisher length $\|v\|_{g^{(1)}} \equiv \sqrt{g_{\theta_0}^{(1)}(v,v)} = 0$, violating the Geometric-Operational Equivalence ratio of Section \ref{sec:equivalence_proof}.
	\end{enumerate}
\end{lemma}

\begin{proof}
~
\begin{itemize}
\item \textbf{Proof of (i):} Let $v \in V_{\theta_0} \setminus \{\mathbf{0}\}$ be a non-zero kernel vector normalized such that $\|v\|_{\mathbb{R}^D} = 1$. By definition of $V_{\theta_0}$, $g^{(1)}(\theta_0)(v, v) = 0$  (since  $g^{(1)}(\theta_0) v = \mathbf{0}$) . Evaluating the pulled-back metric tensor $g_N(0) = \psi_N^* g^{(N)}(0)$ on $v$:
	\begin{equation}
		g_N(0)(v, v) = \left( \frac{1}{\sqrt{N}} J(\psi_N)^T \left[ N \cdot g^{(1)}(\theta_0) \right] J(\psi_N) \right)(v, v) = g^{(1)}(\theta_0)(v, v) = 0.
		\label{eqn:171}
	\end{equation}
Now the vector $v$ is explicitly normalized to be a unit vector ($\Vert{}v\Vert{}_{\mathbb{R}^D} = 1$), which means $v$ belongs to the set of unit vectors $\{h \in \mathbb{R}^D : \Vert{}h\Vert{} = 1\}$. 

By definition of the infimum, the greatest lower bound over a set $S$ is always less than or equal to the function evaluated at any specific element $x_0 \in S$:
$$\inf_{h \in S} f(h) \le f(x_0) \quad \text{for any } x_0 \in S.
$$
Choosing $x_0 = v$ gives:$$\inf_{\Vert{}h\Vert{}=1} h^T g_N(0) h \le v^T g_N(0) v.$$
Since equation \eqref{eqn:171} established that $v^T g_N(0) v = g_N(0)(v,v) = 0$, substituting this value yields $\inf_{\Vert{}h\Vert{}=1} h^T g_N(0) h \le 0$. Since $g_N(0)$ is positive semi-definite, the infimum is exactly $0$ for all $N \in \mathbb{N}$, proving \eqref{eq:zero_eigenvalue_microscope}.
	
\item \textbf{Proof of (ii):} Because $v \in V_{\theta_0} = \ker(g^{(1)}(\theta_0))$, the Lie derivative of the log-likelihood along $v$ vanishes $P_{\theta_0}$-almost everywhere:
	\begin{equation}
		v\left( \log p(X; \theta_0) \right) = \sum_{i=1}^D v^i \frac{\partial \log p(X; \theta_0)}{\partial \theta^i} = 0 \quad P_{\theta_0}\text{-a.e.}
	\end{equation}
	Consequently, sample data $X_1, \dots, X_N$ contains zero information regarding parameter displacements along $v$. By the generalized Cramér-Rao inequality, if $v \in \ker(g^{(1)}(\theta_0))$, the lower bound for the variance of any unbiased estimator $\hat{\theta}_N$ in direction $v$ is given by the pseudo-inverse limit:
	\begin{equation}
		N \cdot \mathrm{Var}_{\theta_0}\left( v^T \hat{\theta}_N \right) \ge v^T \left( g^{(1)}(\theta_0) \right)^+ v = +\infty.
	\end{equation}
	However, the geometric length of $v$ induced by $g^{(1)}(\theta_0)$ is $\|v\|_{g^{(1)}} = \sqrt{g^{(1)}(\theta_0)(v,v)} = 0$. The operational variance diverges to $+\infty$ while the geometric length is $0$, demonstrating that the finite operational-geometric equivalence ratio $\frac{\mathrm{Var}(v^T \hat{\theta}_N)}{\|v\|_g^2} \in (0, \infty)$ constructed in Section \ref{sec:equivalence_proof} breaks down completely on $V_{\theta_0}$.
	
\item {\bf Proof of (ii):} The proof proceeds in four steps to demonstrate the complete collapse of operational variance bounds along kernel directions.

\subparagraph{Step 1: Vanishing Score Function along Kernel Vectors.}
By assumption, $v \in V_{\theta_0} \setminus \{\mathbf{0}\} = \ker(g^{(1)}(\theta_0))$. Evaluating the single-sample Fisher Information quadratic form on $v$:
\begin{equation}
	g^{(1)}(\theta_0)(v, v) = \mathbb{E}_{\theta_0} \left[ \left( \sum_{i=1}^D v^i \frac{\partial \log p(X; \theta_0)}{\partial \theta^i} \right)^2 \right] = 0.
	\label{eq:fim_quadratic_zero}
\end{equation}
Because the integrand $\left( v(\log p(X; \theta_0)) \right)^2$ is a non-negative random variable under measure $P_{\theta_0}$, an expectation of zero strictly requires that the random variable vanishes almost surely. Thus, the directional Lie derivative of the log-likelihood function vanishes $P_{\theta_0}$-almost everywhere:
\begin{equation}
	v\left( \log p(X; \theta_0) \right) \equiv \sum_{i=1}^D v^i \frac{\partial \log p(X; \theta_0)}{\partial \theta^i} = 0 \quad P_{\theta_0}\text{-a.e.}
	\label{eq:score_vanishing_step1}
\end{equation}

\subparagraph{Step 2: Vanishing Total Sample Score.}
Extending from a single observation to an $i.i.d.$ sample $X_{1:N} = (X_1, \dots, X_N) \sim P_{\theta_0}^{\otimes N}$, the total joint sample log-likelihood is $L_N(X_{1:N}; \theta) = \sum_{n=1}^N \log p(X_n; \theta)$. Applying the directional derivative operator $v$ yields:
\begin{equation}
	v\left( L_N(X_{1:N}; \theta_0) \right) = \sum_{n=1}^N v\left( \log p(X_n; \theta_0) \right) = \sum_{n=1}^N 0 = 0 \quad P_{\theta_0}^{\otimes N}\text{-a.e.}
	\label{eq:total_score_vanishing}
\end{equation}
Equation \eqref{eq:total_score_vanishing} demonstrates that the sample score vector carries zero sensitivity along direction $v$, meaning the sample data $X_{1:N}$ provides identically zero Fisher information regarding parameter displacements along $v$.

\subparagraph{Step 3: Unbiasedness Contradiction and Divergent Variance Bound.}
Let $\hat{\theta}_N(X_{1:N})$ be any locally unbiased estimator of $\theta$ in a neighborhood of $\theta_0$, satisfying $\mathbb{E}_\theta\left[ \hat{\theta}_N \right] = \theta$. Differentiating the projected expectation $v^T \mathbb{E}_\theta[\hat{\theta}_N] = v^T \theta$ along direction $v$ at $\theta = \theta_0$:
\begin{equation}
	v\left( v^T \mathbb{E}_\theta[\hat{\theta}_N] \right) \Big|_{\theta_0} = v^T v = \|v\|_2^2 > 0.
	\label{eq:unbiased_lhs}
\end{equation}
On the other hand, interchanging differentiation and integration under standard measure-theoretic regularity conditions yields:
\begin{align}
	v\left( v^T \mathbb{E}_\theta[\hat{\theta}_N] \right) \Big|_{\theta_0} &= \int_{\mathcal{X}^N} \left( v^T \hat{\theta}_N(x_{1:N}) \right) v\left( \prod_{n=1}^N p(x_n; \theta_0) \right) d\mu^{\otimes N}(x_{1:N}) \nonumber \\
	&= \mathbb{E}_{\theta_0} \left[ \left( v^T \hat{\theta}_N \right) \cdot v\left( L_N(X_{1:N}; \theta_0) \right) \right].
	\label{eq:unbiased_rhs}
\end{align}
Substituting Equation \eqref{eq:total_score_vanishing} into Equation \eqref{eq:unbiased_rhs}:
\begin{equation}
	v\left( v^T \mathbb{E}_\theta[\hat{\theta}_N] \right) \Big|_{\theta_0} = \mathbb{E}_{\theta_0} \left[ \left( v^T \hat{\theta}_N \right) \cdot 0 \right] = 0.
	\label{eq:unbiased_rhs_zero}
\end{equation}
Equating Equation \eqref{eq:unbiased_lhs} and Equation \eqref{eq:unbiased_rhs_zero} produces the contradiction $\|v\|_2^2 = 0$, establishing that no finite-variance locally unbiased estimator can exist along kernel direction $v$ \cite{rao1973, stoica2001}. 

Formalizing this via the generalized Cramér-Rao inequality for singular information matrices \cite{stoica2001}, since $v \in \ker(g^{(1)}(\theta_0))$ implies $v \notin \mathrm{range}(g^{(1)}(\theta_0))$, the variance lower bound evaluated using the Moore-Penrose pseudo-inverse limit diverges:
\begin{equation}
	N \cdot \mathrm{Var}_{\theta_0} \left( v^T \hat{\theta}_N \right) \ge \lim_{\epsilon \to 0^+} v^T \left( g^{(1)}(\theta_0) + \epsilon \mathbf{I}_D \right)^{-1} v = v^T \left( g^{(1)}(\theta_0) \right)^+ v = +\infty.
	\label{eq:divergent_variance_formal}
\end{equation}

\subparagraph{Step 4: Operational-Geometric Ratio Breakdown.}
Conversely, evaluating the geometric length of $v$ under the ambient metric $g^{(1)}(\theta_0)$ yields:
\begin{equation}
	\|v\|_{g^{(1)}} \equiv \sqrt{g^{(1)}(\theta_0)(v, v)} = \sqrt{v^T g^{(1)}(\theta_0) v} = 0.
	\label{eq:zero_geometric_length}
\end{equation}
Comparing the operational estimation risk with the geometric metric length demonstrates the complete mathematical breakdown of the operational-geometric equivalence ratio constructed in Section \ref{sec:equivalence_proof}:
\begin{equation}
	\frac{\mathrm{Var}_{\theta_0}\left( v^T \hat{\theta}_N \right)}{\|v\|_{g^{(1)}}^2} = \frac{+\infty}{0} = +\infty \notin (0, \infty).
	\label{eq:ratio_breakdown}
\end{equation}
Thus, ambient geometry without Ehresmann horizontal splitting fails to bound operational estimation errors on singular parameter spaces. $\blacksquare$

\end{itemize}
\end{proof}

\subsubsection{The Ehresmann Connection Instrument as the Geometric Solution}

To resolve this failure without abandoning over-parameterized models, we introduce the \textbf{Ehresmann Connection Instrument}. Instead of treating $\Theta$ as an unstructured ambient manifold $\mathbb{R}^D$, we equip it with a fiber bundle structure $(\Theta, \mathcal{B}, \pi)$ where:
\begin{itemize}
	\item The base manifold $\mathcal{B} \equiv \Theta / \sim$ represents the $d$-dimensional quotient space of observationally identifiable probability distributions, where $g^{\mathcal{B}}$ is strictly non-singular ($\lambda_{\min}(g^{\mathcal{B}}) \ge \kappa_0 > 0$).
	\item The fiber $\mathcal{F}_p = \pi^{-1}(p)$ represents the $(D-d)$-dimensional unidentifiable gauge orbit. \footnote{
			Let $(\Theta, \mathcal{B}, \pi)$ be a statistical fiber bundle, where $\pi: \Theta \to \mathcal{B}$ is the canonical projection map from the ambient parameter space $\Theta \subset \mathbb{R}^D$ to the quotient base manifold $\mathcal{B} \equiv \Theta / \sim$ of identifiable probability distributions. 
			
			For any point $p \in \mathcal{B}$, the fiber (or preimage) $\pi^{-1}(p) \subset \Theta$ over $p$ is defined as:
			\begin{equation}
				\pi^{-1}(p) \equiv \left\{ \theta \in \Theta : \pi(\theta) = p \right\} = \left\{ \theta \in \Theta : P_\theta = p \quad \mu\text{-a.e.} \right\}
				\label{eq:fiber_point_def}
			\end{equation}
			
			More generally, as a set-valued map $\pi^{-1}: \mathcal{P}(\mathcal{B}) \to \mathcal{P}(\Theta)$ on the power set of $\mathcal{B}$, the inverse image map for any subset $\mathcal{U} \subseteq \mathcal{B}$ is defined as:
			\begin{equation}
				\pi^{-1}(\mathcal{U}) \equiv \left\{ \theta \in \Theta : \pi(\theta) \in \mathcal{U} \right\} = \bigcup_{p \in \mathcal{U}} \pi^{-1}(p)
				\label{eq:preimage_set_def}
			\end{equation}}
\end{itemize}

The Ehresmann connection provides a smooth, canonically invariant decomposition of the ambient tangent space into vertical and horizontal sub-bundles:
\begin{equation}
	T_\theta \Theta = H_\theta \oplus V_\theta, \quad \text{where } V_\theta \equiv \ker(d\pi_\theta) \text{ and } H_\theta \equiv V_\theta^{\perp_g}.
	\label{eq:ehresmann_splitting_intro}
\end{equation}
By projecting localized coordinate charts and operators onto the horizontal distribution $H_\theta$ via the projection operator $P_\theta^H \equiv \mathrm{id}_{T_\theta\Theta} - \omega_\theta$, we isolate the non-identifiable gauge modes in $V_\theta$ while preserving a strictly positive-definite restricted metric $g_\theta\big|_{H_\theta \times H_\theta}$. This horizontal splitting allows the convergence theorems of Section \ref{sec:module3} and the equivalence bounds of Section \ref{sec:equivalence_proof} to hold strictly on $H_\theta$, enabling the formulation of the Geometric-Operational Unification Theorem in Section \ref{sec:module5_unification}.

\subsection{Statistical Fiber Bundle Architecture and Gauge Orbits}
\label{subsec:fiber_bundle_architecture}

We formalize the ambient parameter space $\Theta$ as a smooth fiber bundle over the quotient manifold of observationally identifiable probability distributions.

\begin{definition}[Statistical Fiber Bundle]
	\label{def:statistical_fiber_bundle}
	Let $\Theta \subset \mathbb{R}^D$ be a smooth, open, $D$-dimensional ambient parameter space, and let $\mathcal{P} = \{P_\theta : dP_\theta = p(x;\theta)d\mu(x), \, \theta \in \Theta\}$ be a dominated parametric family on sample space $\mathcal{X}$. Define the observational equivalence relation $\sim$ on $\Theta$ by:
	\begin{equation}
		\theta_1 \sim \theta_2 \quad \Longleftrightarrow \quad P_{\theta_1} = P_{\theta_2} \quad \mu\text{-almost everywhere}.
		\label{eq:equivalence_relation_def}
	\end{equation}
	Let $\mathcal{B} \equiv \Theta / \sim$ denote the $d$-dimensional quotient base manifold ($d < D$) of identifiable probability distributions, equipped with canonical projection map $\pi: \Theta \to \mathcal{B}$ defined by $\pi(\theta) = [P_\theta]$. The structured tuple:
	\begin{equation}
		(\Theta, \, \mathcal{B}, \, \pi, \, \mathcal{F})
		\label{eq:fiber_bundle_tuple}
	\end{equation}
	forms a smooth statistical fiber bundle, where for each $p \in \mathcal{B}$, the fiber $\mathcal{F}_p \equiv \pi^{-1}(p) \subset \Theta$ represents the $(D-d)$-dimensional gauge orbit of observationally equivalent parameters.
\end{definition}

\begin{assumption}[Submersion and Smooth Regularity of Fiber Bundle]
	\label{asm:bundle_submersion}
	The projection map $\pi: \Theta \to \mathcal{B}$ is a smooth surjective submersion of constant rank $d = \dim(\mathcal{B}) < D = \dim(\Theta)$. Consequently, for each $p \in \mathcal{B}$, the fiber $\mathcal{F}_p = \pi^{-1}(p)$ is a smooth closed sub-manifold of $\Theta$ of dimension $\dim(\mathcal{F}_p) = D - d$.
\end{assumption}

\subsubsection{Vertical Sub-Bundle and Score Function Vanishing}
\label{subsec:vertical_subbundle}

\begin{definition}[Vertical Distribution / Kernel Sub-Bundle $V_\theta$]
	\label{def:vertical_subbundle}
	For each ambient parameter state $\theta \in \Theta$, the \textbf{vertical tangent space} $V_\theta \subset T_\theta\Theta$ is defined coordinate-freely as the kernel of the differential push-forward map $d\pi_\theta: T_\theta\Theta \to T_{\pi(\theta)}\mathcal{B}$:
	\begin{equation}
		V_\theta \equiv \ker(d\pi_\theta) = \left\{ v \in T_\theta\Theta : d\pi_\theta(v) = \mathbf{0} \in T_{\pi(\theta)}\mathcal{B} \right\}.
		\label{eq:vertical_subspace_def}
	\end{equation}
	The assignment $\theta \mapsto V_\theta$ defines an integrable smooth sub-bundle $V \subset T\Theta$ of rank $D - d$, called the \textbf{vertical sub-bundle}. Tangent vectors $v \in V_\theta$ represent unidentifiable internal gauge variations along the fiber $\mathcal{F}_{\pi(\theta)}$.
\end{definition}

\begin{lemma}[Score Function Vanishing along Vertical Directions]
	\label{lem:vertical_score_vanishing}
	Let $(\Theta, \mathcal{B}, \pi)$ be the statistical fiber bundle (Definition \ref{def:statistical_fiber_bundle}). For any parameter state $\theta \in \Theta$ and any vertical tangent vector $v \in V_\theta = \ker(d\pi_\theta)$, the directional Lie derivative of the single-sample log-likelihood density function vanishes $\mu$-almost everywhere:
	\begin{equation}
		v\left( \log p(X; \theta) \right) \equiv \sum_{i=1}^D v^i \frac{\partial \log p(X; \theta)}{\partial \theta^i} = 0 \quad P_\theta\text{-a.e.}
		\label{eq:vertical_score_vanishing_statement}
	\end{equation}
	Consequently, every vertical tangent vector $v \in V_\theta$ lies in the kernel of the ambient single-sample Fisher Information Metric tensor $g_\theta^{(1)}$:
	\begin{equation}
		V_\theta \subseteq \ker\left(g_\theta^{(1)}\right) \equiv \left\{ v \in T_\theta\Theta : g_\theta^{(1)}(v, w) = 0, \; \forall w \in T_\theta\Theta \right\}.
		\label{eq:vertical_metric_kernel_containment}
	\end{equation}
\end{lemma}

\begin{proof}
	Let $v \in V_\theta = \ker(d\pi_\theta)$ be a vertical tangent vector at $\theta \in \Theta$. By definition of $V_\theta$, $v$ is tangent to the fiber $\mathcal{F}_{\pi(\theta)} = \pi^{-1}(\pi(\theta))$. Let $\gamma: (-\epsilon, \epsilon) \to \mathcal{F}_{\pi(\theta)}$ be a smooth integral curve in $\Theta$ satisfying $\gamma(0) = \theta$ and $\dot{\gamma}(0) = v$.
	
	Because $\gamma(t) \in \mathcal{F}_{\pi(\theta)}$ for all $t \in (-\epsilon, \epsilon)$, all parameter states along $\gamma(t)$ correspond to the exact same probability distribution $P_{\pi(\theta)} \in \mathcal{B}$. Thus, the density function is identical across $t$:
	\begin{equation}
		p(x; \gamma(t)) = p(x; \theta) \quad \mu\text{-a.e.}, \; \forall t \in (-\epsilon, \epsilon).
		\label{eq:density_constant_along_fiber}
	\end{equation}
	Differentiating Equation \eqref{eq:density_constant_along_fiber} with respect to $t$ at $t = 0$ using the chain rule:
	\begin{equation}
		\left. \frac{d}{dt} \log p(x; \gamma(t)) \right|_{t=0} = \sum_{i=1}^D \dot{\gamma}^i(0) \frac{\partial \log p(x; \theta)}{\partial \theta^i} = \sum_{i=1}^D v^i \frac{\partial \log p(x; \theta)}{\partial \theta^i} = 0 \quad \mu\text{-a.e.}
		\label{eq:chain_rule_score_vanishing}
	\end{equation}
	This establishes Equation \eqref{eq:vertical_score_vanishing_statement}.
	
	Now, evaluate the ambient Fisher Information Metric component $g_\theta^{(1)}(v, w)$ for an arbitrary tangent vector $w \in T_\theta\Theta$:
	\begin{equation}
		g_\theta^{(1)}(v, w) \equiv \mathbb{E}_\theta \left[ \left( v(\log p(X; \theta)) \right) \cdot \left( w(\log p(X; \theta)) \right) \right].
		\label{eq:fim_vertical_evaluation}
	\end{equation}
	Substituting Equation \eqref{eq:chain_rule_score_vanishing} into Equation \eqref{eq:fim_vertical_evaluation} gives:
	\begin{equation}
		g_\theta^{(1)}(v, w) = \mathbb{E}_\theta \left[ 0 \cdot \left( w(\log p(X; \theta)) \right) \right] = 0, \quad \forall w \in T_\theta\Theta.
		\label{eq:fim_vertical_zero}
	\end{equation}
	Hence, $v \in \ker(g_\theta^{(1)})$, proving $V_\theta \subseteq \ker(g_\theta^{(1)})$.
\end{proof}

\subsection{Fisher-Compatible Ehresmann Connection and Horizontal Splitting}
\label{subsec:ehresmann_connection_splitting}

Having isolated the vertical kernel $V_\theta$, we construct a smooth complementary horizontal distribution $H_\theta \subset T_\theta\Theta$ that encodes all statistically informative parameter variations.

\begin{definition}[Fisher-Compatible Ehresmann Connection]
	\label{def:ehresmann_connection}
	An \textbf{Ehresmann connection} on the statistical fiber bundle $(\Theta, \mathcal{B}, \pi)$ is a smooth choice of a horizontal subspace $H_\theta \subset T_\theta\Theta$ at each state $\theta \in \Theta$ such that the tangent space decomposes into a smooth direct sum:
	\begin{equation}
		T_\theta\Theta = H_\theta \oplus V_\theta, \quad \forall \theta \in \Theta.
		\label{eq:tangent_direct_sum_splitting}
	\end{equation}
	An Ehresmann connection $H \subset T\Theta$ is called \textbf{Fisher-compatible} (or \textbf{metric-orthogonal}) if $H_\theta$ is defined as the Fisher-metric orthogonal complement of the vertical sub-bundle $V_\theta$:
	\begin{equation}
		H_\theta \equiv V_\theta^{\perp_g} = \left\{ x \in T_\theta\Theta : g_\theta^{(1)}(x, v) = 0, \; \forall v \in V_\theta \right\}.
		\label{eq:horizontal_orthogonal_def}
	\end{equation}
\end{definition}

\begin{definition}[Connection 1-Form $\omega_\theta$ and Horizontal Projection $P_\theta^H$]
	\label{def:connection_1form_projection}
	Let $T_\theta\Theta = H_\theta \oplus V_\theta$ be the Fisher-compatible Ehresmann connection (Definition \ref{def:ehresmann_connection}).
	\begin{enumerate}[label=\textnormal{(\roman*)}]
		\item The \textbf{Ehresmann connection 1-form} is the smooth $V_\theta$-valued 1-form $\omega_\theta: T_\theta\Theta \to V_\theta$ defined as the canonical projection onto the vertical space $V_\theta$ along $H_\theta$:
		\begin{equation}
			\omega_\theta(x + v) = v, \quad \forall x \in H_\theta, \; v \in V_\theta.
			\label{eq:connection_1form_def}
		\end{equation}
		\item The \textbf{horizontal projection operator} $P_\theta^H: T_\theta\Theta \to H_\theta$ is the smooth idempotent linear operator defined by:
		\begin{equation}
			P_\theta^H \equiv \mathrm{id}_{T_\theta\Theta} - \omega_\theta, \quad \text{satisfying } (P_\theta^H)^2 = P_\theta^H, \; \mathrm{im}(P_\theta^H) = H_\theta, \; \ker(P_\theta^H) = V_\theta.
			\label{eq:horizontal_projection_def}
		\end{equation}
		Where $\mathrm{id}_{T_\theta\Theta}$ is the identity operator (or identity mapping) on the tangent space $T_\theta\Theta$.
	\end{enumerate}
\end{definition}

\begin{lemma}[Horizontal Tangent Isomorphism]
	\label{lem:horizontal_isomorphism}
	Let $(\Theta, \mathcal{B}, \pi)$ be the statistical fiber bundle with Fisher-compatible Ehresmann connection $T_\theta\Theta = H_\theta \oplus V_\theta$. For every $\theta \in \Theta$, the restricted differential map:
	\begin{equation}
		d\pi_\theta\big|_{H_\theta} : H_\theta \longrightarrow T_{\pi(\theta)}\mathcal{B}
		\label{eq:restricted_differential_isomorphism}
	\end{equation}
	is a linear vector space isomorphism between the $d$-dimensional horizontal subspace $H_\theta$ and the $d$-dimensional tangent space $T_{\pi(\theta)}\mathcal{B}$ of identifiable distributions.
\end{lemma}

\begin{proof}
	By the First Isomorphism Theorem for vector spaces \cite{roman2008} applied to the linear differential map $d\pi_\theta: T_\theta\Theta \to T_{\pi(\theta)}\mathcal{B}$:
	\begin{equation}
		T_\theta\Theta / \ker(d\pi_\theta) \cong \mathrm{im}(d\pi_\theta), 
		\label{eq:vector_isomorphism_step1}
	\end{equation}
	where $T_\theta\Theta / \ker(d\pi_\theta)$ is the quotient vector space of the ambient tangent space $T_\theta\Theta$ modulo the vertical kernel subspace $V_\theta \equiv \ker(d\pi_\theta)$.
	
	By Assumption \ref{asm:bundle_submersion}, $\pi$ is a submersion, so $d\pi_\theta$ is surjective: $\mathrm{im}(d\pi_\theta) = T_{\pi(\theta)}\mathcal{B}$. By Definition \ref{def:vertical_subbundle}, $\ker(d\pi_\theta) = V_\theta$. Substituting these into Equation \eqref{eq:vector_isomorphism_step1}:
	\begin{equation}
		T_\theta\Theta / V_\theta \cong T_{\pi(\theta)}\mathcal{B}.
		\label{eq:vector_isomorphism_step2}
	\end{equation}
	By the direct sum decomposition $T_\theta\Theta = H_\theta \oplus V_\theta$ (Definition \ref{def:ehresmann_connection}), every vector $w \in T_\theta\Theta / V_\theta$ has a unique representative in $H_\theta$. Thus, $H_\theta \cong T_\theta\Theta / V_\theta$. Consequently, $d\pi_\theta\big|_{H_\theta}: H_\theta \to T_{\pi(\theta)}\mathcal{B}$ is a linear bijection between vector spaces of equal finite dimension $\dim(H_\theta) = \dim(T_{\pi(\theta)}\mathcal{B}) = d$, proving that $d\pi_\theta\big|_{H_\theta}$ is an exact linear isomorphism.
\end{proof}

\subsection{Strict Uniform Positivity and Non-Singularity on $H_\theta$}
\label{subsec:horizontal_strict_positivity}

We now prove that restricting the Fisher Information Metric to the horizontal distribution $H_\theta$ completely eliminates ambient non-degeneracy pathologies.

\begin{theorem}[Strict Positivity and Non-Singularity of Restricted Metric on $H_\theta$]
	\label{thm:horizontal_strict_positivity}
	Let $(\Theta, \mathcal{B}, \pi)$ be the statistical fiber bundle equipped with the Fisher-compatible Ehresmann connection $T_\theta\Theta = H_\theta \oplus V_\theta$. Assuming local regularity on the quotient base manifold $\mathcal{B}$, the restricted single-sample Fisher Information Metric tensor field:
	\begin{equation}
		g_\theta^{(1)}\big|_{H_\theta \times H_\theta} : H_\theta \times H_\theta \longrightarrow \mathbb{R}
		\label{eq:restricted_metric_def}
	\end{equation}
	is strictly positive-definite and non-singular across the localized neighborhood $\mathcal{U}(\theta_0) \subset \Theta$. Specifically, there exists a uniform constant $\kappa_0 > 0$ such that:
	\begin{equation}
		\lambda_{\min}\left( g_\theta^{(1)}\big|_{H_\theta \times H_\theta} \right) \ge \kappa_0 > 0, \quad \forall \theta \in \mathcal{U}(\theta_0).
		\label{eq:horizontal_min_eigenvalue_bound}
	\end{equation}
\end{theorem}

\begin{proof}
	The proof proceeds in four steps.
	
	\paragraph{Step 1: Non-Zero Score Vector for Non-Zero Horizontal Vectors.}
	Let $x \in H_\theta$ be an arbitrary non-zero horizontal vector ($x \neq \mathbf{0}$). By Lemma \ref{lem:horizontal_isomorphism}, $d\pi_\theta\big|_{H_\theta}$ is an isomorphism, so:
	\begin{equation}
		u \equiv d\pi_\theta(x) \neq \mathbf{0} \in T_{\pi(\theta)}\mathcal{B}.
		\label{eq:pushforward_nonzero}
	\end{equation}
	Because $u \neq \mathbf{0}$, $u$ represents a non-trivial variation of probability distributions on the quotient manifold $\mathcal{B}$. By regularity of the identifiable base family $\mathcal{B}$, the base Fisher metric $g^{\mathcal{B}}$ is strictly positive-definite:
	\begin{equation}
		g_{\pi(\theta)}^{\mathcal{B}}(u, u) > 0, \quad \forall u \neq \mathbf{0} \in T_{\pi(\theta)}\mathcal{B}.
		\label{eq:base_metric_positive}
	\end{equation}
	
	\paragraph{Step 2: Metric Preservation under Horizontal Projection.}
	By definition of the quotient metric on $\mathcal{B}$, the Fisher metric $g_{\pi(\theta)}^{\mathcal{B}}$ on base tangent vectors $u = d\pi_\theta(x)$ is defined as the pullback of ambient density variation under horizontal lifting $x \in H_\theta$ ({\it i.e.,} $x$ is the horizontal lift of $u$.):
	\begin{equation}
		g_\theta^{(1)}(x, x) = \mathbb{E}_\theta \left[ \left( x(\log p(X; \theta)) \right)^2 \right] = g_{\pi(\theta)}^{\mathcal{B}}(d\pi_\theta(x), \, d\pi_\theta(x)).
		\label{eq:metric_pullback_identity_horizontal}
	\end{equation}
	Substituting Equation \eqref{eq:base_metric_positive} into Equation \eqref{eq:metric_pullback_identity_horizontal}:
	\begin{equation}
		g_\theta^{(1)}(x, x) = g_{\pi(\theta)}^{\mathcal{B}}(u, u) > 0, \quad \forall x \in H_\theta \setminus \{\mathbf{0}\}.
		\label{eq:horizontal_quadratic_form_positive}
	\end{equation}
	This proves that $g_\theta^{(1)}\big|_{H_\theta \times H_\theta}$ is strictly positive-definite.
	
\paragraph{Step 3: Non-Singularity and Zero Intersection with Kernel.}
Suppose $x \in H_\theta$ belongs to the kernel of the restricted metric tensor $g_\theta^{(1)}\big|_{H_\theta \times H_\theta}$, which means $g_\theta^{(1)}(x, w_H) = 0$ for all horizontal testing vectors $w_H \in H_\theta$. Setting $w_H = x$, the strict positive-definiteness established in Equation~\eqref{eq:horizontal_quadratic_form_positive} forces $g_\theta^{(1)}(x, x) = 0 \implies x = \mathbf{0}$.
	
	Furthermore, by Definition \ref{def:ehresmann_connection}, $H_\theta \equiv V_\theta^{\perp_g}$. Thus, $g_\theta^{(1)}(x, v) = 0$ for all $v \in V_\theta$ by construction. Combining horizontal and vertical testing:
	
\begin{enumerate}
	\item \textit{Vector Decomposition:} Any ambient tangent vector $w \in T_\theta\Theta$ uniquely decomposes into horizontal and vertical components: $w = w_H + w_V$, where $w_H \in H_\theta$ and $w_V \in V_\theta$.
	\item \textit{Bilinear Evaluation:} $g_\theta^{(1)}(x, w) = g_\theta^{(1)}(x, w_H) + g_\theta^{(1)}(x, w_V) = 0 + 0 = 0$, where $g_\theta^{(1)}(x, w_H) = 0$ holds by hypothesis and $g_\theta^{(1)}(x, w_V) = 0$ holds by metric orthogonality ($H_\theta \equiv V_\theta^{\perp_g}$).
	\item \textit{Kernel Intersection Membership:} Since $g_\theta^{(1)}(x, w) = 0$ holds for all $w \in T_\theta\Theta$, $x \in \ker\left(g_\theta^{(1)}\right)$. Combined with $x \in H_\theta$, this completes the proof that $x \in \ker\left(g_\theta^{(1)}\right) \cap H_\theta$.
\end{enumerate}	
	
	Equation \eqref{eq:horizontal_quadratic_form_positive} guarantees that $\ker\left(g_\theta^{(1)}\right) \cap H_\theta = \{\mathbf{0}\}$. Thus, $g_\theta^{(1)}\big|_{H_\theta \times H_\theta}$ is non-singular with matrix rank equal to $d = \dim(H_\theta) = \dim(\mathcal{B})$. Hence we  establish that the zero vector is the only element in the kernel:
	$$\ker\left(g_\theta^{(1)}\big\vert{}_{H_\theta \times H_\theta}\right) = \{\mathbf{0}\}$$
	

\paragraph{Step 4: Uniform Lower Bound via Compact Sphere Isolation.}
To transition from point-wise strict positive-definiteness to a uniform lower bound on the minimum eigenvalue across the entire localized neighborhood $\mathcal{U}(\theta_0) \subset \Theta$, we formalize the proof through four transparent logical phases.

\subparagraph{Phase 1: Compact Domain Construction.}
Let $\overline{\mathcal{U}}(\theta_0) \subset \Theta$ denote the closed, bounded neighborhood (compact closure) of $\theta_0$ in the ambient space $\mathbb{R}^D$. Let $H_{\theta_0} \cong \mathbb{R}^d$ be the $d$-dimensional horizontal subspace at state $\theta_0$. We define the unit sphere in $H_{\theta_0}$ equipped with the standard Euclidean norm as:
\begin{equation}
	\mathbb{S}^{d-1} \equiv \left\{ x \in H_{\theta_0} : \|x\|_2 = 1 \right\} \subset \mathbb{R}^d.
	\label{eq:unit_sphere_def}
\end{equation}
By the Heine-Borel Theorem \cite{rudin1976}, $\overline{\mathcal{U}}(\theta_0)$ is closed and bounded in $\mathbb{R}^D$, hence compact. Similarly, $\mathbb{S}^{d-1}$ is closed and bounded in $\mathbb{R}^d$, hence compact \cite{rudin1976}. By Tychonoff's Theorem on product topologies \cite{munkres2000}, the Cartesian product space:
\begin{equation}
	K \equiv \overline{\mathcal{U}}(\theta_0) \times \mathbb{S}^{d-1} \subset \mathbb{R}^D \times \mathbb{R}^d
	\label{eq:compact_product_space}
\end{equation}
is a non-empty, compact topological space.

\subparagraph{Phase 2: Continuity and Strict Positivity of the Rayleigh Mapping. \cite{rayleigh1877}}
Define the generalized Rayleigh quotient mapping $f: K \longrightarrow \mathbb{R}$ on the compact product domain $K$ by:
\begin{equation}
	f(\theta, x) \equiv g_\theta^{(1)}(x, x), \quad \forall (\theta, x) \in \overline{\mathcal{U}}(\theta_0) \times \mathbb{S}^{d-1}.
	\label{eq:rayleigh_map_def}
\end{equation}
The mapping $f(\theta, x)$ satisfies two crucial properties:
\begin{enumerate}[label=\textnormal{(\roman*)}]
	\item \textbf{Continuity ($C^0$):} By local regularity, the Fisher metric tensor field $\theta \mapsto g_\theta^{(1)}$ is smooth (hence $C^0$-continuous) over $\overline{\mathcal{U}}(\theta_0)$ \cite{lee2013}. Furthermore, $f(\theta, x)$ is a quadratic form in $x$. Being a composition of continuous functions, $f(\cdot, \cdot)$ is jointly $C^0$-continuous on the compact metric space $K$ \cite{rudin1976}.
	\item \textbf{Strict Positivity:} For any $(\theta, x) \in K$, $x \in \mathbb{S}^{d-1}$ implies $x \neq \mathbf{0}$. Steps 1--3 established that the restricted metric $g_\theta^{(1)}\big|_{H_\theta \times H_\theta}$ is strictly positive-definite and non-singular at every $\theta \in \overline{\mathcal{U}}(\theta_0)$. Therefore:
	\begin{equation}
		f(\theta, x) = g_\theta^{(1)}(x, x) > 0, \quad \forall (\theta, x) \in K.
		\label{eq:strict_positivity_on_K}
	\end{equation}
\end{enumerate}

\subparagraph{Phase 3: Application of the Extreme Value Theorem.}
By the Weierstrass Extreme Value Theorem \cite{rudin1976}, any continuous real-valued function defined on a non-empty compact set achieves its global minimum on that set. Thus, there exists an optimal parameter state and direction pair $(\theta^*, x^*) \in K$ such that:
\begin{equation}
	\kappa_0 \equiv \min_{(\theta, x) \in K} f(\theta, x) = f(\theta^*, x^*) = g_{\theta^*}^{(1)}(x^*, x^*).
	\label{eq:weierstrass_min_attainment}
\end{equation}
Because $f(\theta, x) > 0$ strictly for all $(\theta, x) \in K$ by Equation \eqref{eq:strict_positivity_on_K}, the global minimum value $\kappa_0$ must itself be strictly positive:
\begin{equation}
	\kappa_0 = \min_{\theta \in \overline{\mathcal{U}}(\theta_0)} \min_{x \in \mathbb{S}^{d-1}} g_\theta^{(1)}(x, x) > 0.
	\label{eq:kappa0_strictly_positive}
\end{equation}

\subparagraph{Phase 4: Variational Eigenvalue Characterization and Uniformity.}
By the Rayleigh-Ritz Variational Principle for symmetric bilinear forms \cite{horn2012}, for any fixed state $\theta \in \mathcal{U}(\theta_0)$, the minimum eigenvalue of the restricted horizontal metric tensor $g_\theta^{(1)}\big|_{H_\theta \times H_\theta}$ is characterized by:
\begin{equation}
	\lambda_{\min}\left( g_\theta^{(1)}\big|_{H_\theta \times H_\theta} \right) = \min_{x \in \mathbb{S}^{d-1}} g_\theta^{(1)}(x, x).
	\label{eq:rayleigh_ritz_eigenvalue}
\end{equation}
Taking the infimum over all $\theta \in \mathcal{U}(\theta_0) \subset \overline{\mathcal{U}}(\theta_0)$ and applying Equation \eqref{eq:kappa0_strictly_positive} yields:
\begin{equation}
	\lambda_{\min}\left( g_\theta^{(1)}\big|_{H_\theta \times H_\theta} \right) \ge \min_{\theta \in \overline{\mathcal{U}}(\theta_0)} \min_{x \in \mathbb{S}^{d-1}} g_\theta^{(1)}(x, x) = \kappa_0 > 0, \quad \forall \theta \in \mathcal{U}(\theta_0).
	\label{eq:final_uniform_lower_bound}
\end{equation}
Equation \eqref{eq:final_uniform_lower_bound} establishes a uniform, state-independent lower bound $\kappa_0 > 0$ on the spectrum of $g_\theta^{(1)}\big|_{H_\theta \times H_\theta}$ across the localized chart $\mathcal{U}(\theta_0)$, completing the proof of Theorem \ref{thm:horizontal_strict_positivity}. $\blacksquare$


\end{proof}


\section{The Geometric-Operational Unification Theorem Without Ambient Non-Degeneracy}
\label{sec:module5_unification}

In Section~\ref{sec:ehresmann_connections}, we constructed the Fisher-compatible Ehresmann connection framework over the statistical fiber bundle $(\Theta, \mathcal{B}, \pi)$, decomposing the tangent bundle into mutually metric-orthogonal sub-bundles $T_\theta\Theta = H_\theta \oplus V_\theta$. By isolating the structural internal directions $V_\theta = \ker(d\pi_\theta)$ from the statistically verifiable directions $H_\theta = V_\theta^{\perp_g}$, Theorem~\ref{thm:horizontal_strict_positivity} established that the restricted Fisher metric $g_\theta\big|_{H_\theta \times H_\theta}$ is strictly positive-definite and non-singular ($\lambda_{\min}(g_\theta|_{H_\theta}) \ge \kappa_0 > 0$), even when the ambient metric $g_\theta$ exhibits massive degeneracies due to over-parameterization.

In this section, we present the formal proof of the \textbf{Geometric-Operational Unification Theorem} restricted to the horizontal distribution $H_\theta$ (and the quotient base manifold $\mathcal{B}$). We demonstrate that the Geometric Priority Paradigm (Axiom~\ref{ax:geom_priority}) and the Statistical Operationalism Paradigm (Axiom~\ref{ax:stat_oper}) are mathematically equivalent without requiring ambient parameter space identifiability or non-degeneracy. This establishes the unified bridge between differential geometry and decision theory on the non-singular horizontal carriage.

\subsection{Generalized Horizontal Log-Likelihood Decomposition}
\label{subsec:horizontal_log_likelihood}

To analyze localized asymptotic experiments over over-parameterized statistical manifolds, we restrict parameter displacement signals $h \in \mathbb{K} \subset \mathbb{R}^d$ to the horizontal distribution $H_{\theta_0} \cong T_{\pi(\theta_0)}\mathcal{B}$.

\begin{theorem}[Generalized Horizontal Log-Likelihood Representation]
	\label{thm:generalized_horizontal_log_likelihood}
	Let $\theta_0 \in \mathrm{Int}(\Theta)$ be an ambient parameter state in a $D$-dimensional parameter manifold $\Theta \subset \mathbb{R}^D$, and let $H_{\theta_0} \subset T_{\theta_0}\Theta$ be its $d$-dimensional horizontal subspace relative to the Fisher-compatible Ehresmann connection $\omega_{\theta_0}$. For any local displacement signal $h \in \mathbb{K} \subset H_{\theta_0} \cong \mathbb{R}^d$ on a compact set $\mathbb{K}$, the localized Radon--Nikodym log-likelihood ratio process
	\begin{equation}
		\Lambda_N(h) \equiv \log \frac{d P_{\theta_0 + P_{\theta_0}^H h / \sqrt{N}}^{(N)}}{d P_{\theta_0}^{(N)}}\left(X^{(N)}\right) = \sum_{n=1}^N \left[ \log p\left(X_n; \theta_0 + \frac{P_{\theta_0}^H h}{\sqrt{N}}\right) - \log p(X_n; \theta_0) \right]
		\label{eq:horizontal_log_likelihood_def}
	\end{equation}
	admits the exact algebraic decomposition built from the horizontally filtered pulled-back geometric fields:
	\begin{equation}
		\Lambda_N(h) = h^T \Delta_{N,H}(\theta_0) - \frac{1}{2} h^T \left[\psi_{N,H}^* G^{(N)}\right](\mathbf{0}) h + \mathcal{R}_{N,H}(h)
		\label{eq:horizontal_log_likelihood_decomp}
	\end{equation}
	where:
	\begin{enumerate}
		\item $\Delta_{N,H}(\theta_0) \equiv P_{\theta_0}^H \Delta_N(\theta_0) = \frac{1}{\sqrt{N}} \sum_{n=1}^N P_{\theta_0}^H \nabla_\theta \log p(X_n; \theta_0) \xrightarrow{d} \mathcal{N}\left(\mathbf{0}, g_0\big|_{\mathcal{B}}\right)$ is the horizontally projected score vector,
		\item $\left[\psi_{N,H}^* G^{(N)}\right](\mathbf{0}) \equiv P_{\theta_0}^H g^{(1)}(\theta_0) P_{\theta_0}^H = g_0\big|_{\mathcal{B}}$ is the horizontally projected pulled-back Fisher metric, and
		\item $\mathcal{R}_{N,H}(h)$ is a non-linear stochastic remainder process vanishing uniformly in probability over $\mathbb{K}$:
		\begin{equation}
			\sup_{h \in \mathbb{K} \subset H_{\theta_0}} \left| \mathcal{R}_{N,H}(h) \right| \xrightarrow{P_{\theta_0}} 0 \quad \text{as } N \to \infty.
			\label{eq:horizontal_remainder_uniform_zero}
		\end{equation}
	\end{enumerate}
\end{theorem}

\begin{proof}
	We construct a rigorous, five-step mathematical deduction without omitting intermediate algebraic or measure-theoretic steps.
	
	\subsubsection*{Step 1: Local Trajectory and Exact Taylor Expansion of Single-Observation Log-Likelihood}
	
	Let $h \in \mathbb{K} \subset H_{\theta_0} \cong \mathbb{R}^d$ be an arbitrary horizontal vector. By definition of the Fisher-compatible Ehresmann connection decomposition $T_{\theta_0}\Theta = H_{\theta_0} \oplus V_{\theta_0}$, the horizontal projection operator $P_{\theta_0}^H = \mathrm{id}_{T_{\theta_0}\Theta} - \omega_{\theta_0}$ satisfies $P_{\theta_0}^V h = \omega_{\theta_0}(h) = \mathbf{0}$, meaning that $P_{\theta_0}^H h = h \in \mathbb{R}^D$.
	
	Define the localized parameter trajectory $\theta_N(h) \in \Theta$ by:
	\begin{equation}
		\theta_N(h) \equiv \theta_0 + \frac{1}{\sqrt{N}} P_{\theta_0}^H h.
	\end{equation}
	For each observation $X_n$ ($n = 1, \dots, N$), denote the single-observation log-likelihood function by $l_n(\theta) \equiv \log p(X_n; \theta)$. Under High-Order Local Regularity assumption \ref{asm:local_regularity}, $l_n(\theta)$ is three times continuously differentiable in an open neighborhood $\mathcal{U}(\theta_0) \subset \mathrm{Int}(\Theta)$. Performing a third-order multivariate Taylor expansion of $l_n(\theta_N(h))$ around $h = \mathbf{0}$ yields:
	\begin{align}
		l_n\left(\theta_0 + \frac{P_{\theta_0}^H h}{\sqrt{N}}\right) - l_n(\theta_0) &= \frac{1}{\sqrt{N}} \sum_{i=1}^D \left( P_{\theta_0}^H h \right)^i \frac{\partial l_n(\theta_0)}{\partial \theta^i} + \frac{1}{2N} \sum_{i,j=1}^D \left( P_{\theta_0}^H h \right)^i \left( P_{\theta_0}^H h \right)^j \frac{\partial^2 l_n(\theta_0)}{\partial \theta^i \partial \theta^j} \nonumber \\
		&\quad + R_{n,H}^{(3)}(h),
		\label{eq:taylor_expansion_single_obs}
	\end{align}
	where the third-order integral remainder $R_{n,H}^{(3)}(h)$ is given in exact integral form by:
	\begin{equation}
		R_{n,H}^{(3)}(h) = \frac{1}{6 N^{3/2}} \sum_{i,j,k=1}^D \left( P_{\theta_0}^H h \right)^i \left( P_{\theta_0}^H h \right)^j \left( P_{\theta_0}^H h \right)^k \int_0^1 3(1-t)^2 \frac{\partial^3 l_n\left(\theta_0 + t P_{\theta_0}^H h / \sqrt{N}\right)}{\partial \theta^i \partial \theta^j \partial \theta^k} dt.
		\label{eq:integral_remainder_def}
	\end{equation}
	
	\subsubsection*{Step 2: Summation over $N$ Observations and Analysis of the First-Order Linear Term}
	
	Summing Equation~\eqref{eq:taylor_expansion_single_obs} over all $n = 1, \dots, N$, the joint localized log-likelihood ratio decomposes into three sums:
	\begin{equation}
		\Lambda_N(h) = \sum_{n=1}^N \left[ l_n\left(\theta_0 + \frac{P_{\theta_0}^H h}{\sqrt{N}}\right) - l_n(\theta_0) \right] = I_{1,N}(h) + I_{2,N}(h) + \sum_{n=1}^N R_{n,H}^{(3)}(h),
		\label{eq:lambda_n_three_term_sum}
	\end{equation}
	where $I_{1,N}(h)$ and $I_{2,N}(h)$ represent the first-order linear and second-order quadratic terms, respectively.
	
	We first evaluate $I_{1,N}(h)$. Expressing the sum in vector notation:
	\begin{align}
		I_{1,N}(h) &\equiv \sum_{n=1}^N \frac{1}{\sqrt{N}} \left( P_{\theta_0}^H h \right)^T \nabla_\theta l_n(\theta_0) \nonumber \\
		&= h^T P_{\theta_0}^H \left( \frac{1}{\sqrt{N}} \sum_{n=1}^N \nabla_\theta l_n(\theta_0) \right) \nonumber \\
		&= h^T \left( \frac{1}{\sqrt{N}} \sum_{n=1}^N P_{\theta_0}^H \nabla_\theta l_n(\theta_0) \right) \equiv h^T \Delta_{N,H}(\theta_0),
		\label{eq:linear_term_derivation}
	\end{align}
	where $\Delta_{N,H}(\theta_0) \equiv P_{\theta_0}^H \Delta_N(\theta_0) = \frac{1}{\sqrt{N}} \sum_{n=1}^N P_{\theta_0}^H \nabla_\theta l_n(\theta_0)$ is the horizontally projected sample score vector.
	
	Under $P_{\theta_0}$, the single-observation score vectors $S_n \equiv \nabla_\theta l_n(\theta_0)$ are independent and identically distributed ($i.i.d.$) with:
	\begin{enumerate}
		\item \textbf{Unbiasedness:} $\mathbb{E}_{\theta_0}[S_n] = \mathbf{0}$,
		\item \textbf{Covariance Structure:} $\mathrm{Var}_{\theta_0}(S_n) = \mathbb{E}_{\theta_0}[S_n S_n^T] = g^{(1)}(\theta_0)$.
	\end{enumerate}
	Consequently, the horizontally projected score vector $P_{\theta_0}^H S_n$ has mean $\mathbb{E}_{\theta_0}[P_{\theta_0}^H S_n] = \mathbf{0}$ and covariance matrix:
	\begin{equation}
		\mathrm{Var}_{\theta_0}\left( P_{\theta_0}^H S_n \right) = \mathbb{E}_{\theta_0}\left[ \left( P_{\theta_0}^H S_n \right) \left( P_{\theta_0}^H S_n \right)^T \right] = P_{\theta_0}^H g^{(1)}(\theta_0) P_{\theta_0}^H = g_0\big|_{H_{\theta_0} \times H_{\theta_0}} \equiv g_0\big|_{\mathcal{B}}.
	\end{equation}
	By the Multivariate Central Limit Theorem (CLT) applied to the $i.i.d.$ sum $\Delta_{N,H}(\theta_0) = \frac{1}{\sqrt{N}} \sum_{n=1}^N P_{\theta_0}^H S_n$:
	\begin{equation}
		\Delta_{N,H}(\theta_0) \xrightarrow{d} \mathcal{N}\left(\mathbf{0}, \, g_0\big|_{\mathcal{B}}\right).
	\end{equation}
	By the Fisher-compatible Ehresmann connection construction, $g_0\big|_{\mathcal{B}}$ is strictly positive-definite and non-singular on $H_{\theta_0}$ with $\lambda_{\min}\left(g_0\big|_{\mathcal{B}}\right) \ge \kappa_0 > 0$.
	
	\subsubsection*{Step 3: Analysis of the Second-Order Term and Quadratic Metric Stabilization}
	
	Next, we evaluate the second-order term $I_{2,N}(h)$:
	\begin{equation}
		I_{2,N}(h) \equiv \frac{1}{2N} \sum_{n=1}^N h^T P_{\theta_0}^H \left( \nabla_\theta^2 l_n(\theta_0) \right) P_{\theta_0}^H h.
	\end{equation}
	Under standard $L_2$-differentiability regularity, the expected Hessian of the single-observation log-likelihood equals the negative Fisher Information Metric:
	\begin{equation}
		\mathbb{E}_{\theta_0}\left[ \nabla_\theta^2 l_n(\theta_0) \right] = -g^{(1)}(\theta_0).
	\end{equation}
	Adding and subtracting $-g^{(1)}(\theta_0)$ inside the sum:
	\begin{align}
		I_{2,N}(h) &= \frac{1}{2} h^T P_{\theta_0}^H \left( \frac{1}{N} \sum_{n=1}^N \nabla_\theta^2 l_n(\theta_0) \right) P_{\theta_0}^H h \nonumber \\
		&= -\frac{1}{2} h^T \left( P_{\theta_0}^H g^{(1)}(\theta_0) P_{\theta_0}^H \right) h + \frac{1}{2} h^T P_{\theta_0}^H \left( \frac{1}{N} \sum_{n=1}^N \left[ \nabla_\theta^2 l_n(\theta_0) + g^{(1)}(\theta_0) \right] \right) P_{\theta_0}^H h.
		\label{eq:second_order_split}
	\end{align}
	Notice that the lead matrix in the first term matches the horizontally projected pulled-back Fisher metric:
	\begin{equation}
		P_{\theta_0}^H g^{(1)}(\theta_0) P_{\theta_0}^H = \left[ \psi_{N,H}^* G^{(N)} \right](\mathbf{0}) = g_0\big|_{\mathcal{B}}.
	\end{equation}
	
	Define the random matrix $Q_N \equiv \frac{1}{N} \sum_{n=1}^N \left[ \nabla_\theta^2 l_n(\theta_0) + g^{(1)}(\theta_0) \right] \in \mathbb{R}^{D \times D}$. Since $X_n$ are $i.i.d.$ and $\mathbb{E}_{\theta_0}[\nabla_\theta^2 l_n(\theta_0)] = -g^{(1)}(\theta_0)$, the Khintchine--Kolmogorov Weak Law of Large Numbers (WLLN) implies:
	\begin{equation}
		Q_N \xrightarrow{P_{\theta_0}} \mathbf{0}_{D \times D} \quad \text{as } N \to \infty.
	\end{equation}
	For any $h \in \mathbb{K}$, using the Cauchy--Schwarz inequality and operator norm bounds:
	\begin{equation}
		\left| \frac{1}{2} h^T P_{\theta_0}^H Q_N P_{\theta_0}^H h \right| \le \frac{1}{2} \|h\|_2^2 \left\| P_{\theta_0}^H \right\|_{\mathrm{op}}^2 \|Q_N\|_{\mathrm{op}} \le \frac{1}{2} R_{\mathbb{K}}^2 C_P^2 \|Q_N\|_{\mathrm{op}} = o_{P_{\theta_0}}(1),
	\end{equation}
	where $R_{\mathbb{K}} \equiv \sup_{h \in \mathbb{K}} \|h\|_2 < \infty$ and $C_P \equiv \|P_{\theta_0}^H\|_{\mathrm{op}} < \infty$. Therefore, the second-order term simplifies to:
	\begin{equation}
		I_{2,N}(h) = -\frac{1}{2} h^T \left[ \psi_{N,H}^* G^{(N)} \right](\mathbf{0}) h + o_{P_{\theta_0}}(1).
		\label{eq:i2_simplified}
	\end{equation}
	
	\subsubsection*{Step 4: Analysis and Uniform Boundedness of the Third-Order Remainder Process}
	
	We define the total non-linear stochastic remainder process $\mathcal{R}_{N,H}(h)$ by collecting all residual terms:
	\begin{equation}
		\mathcal{R}_{N,H}(h) \equiv \frac{1}{2} h^T P_{\theta_0}^H Q_N P_{\theta_0}^H h + \sum_{n=1}^N R_{n,H}^{(3)}(h).
		\label{eq:total_remainder_def}
	\end{equation}
	
	To bound $\sum_{n=1}^N R_{n,H}^{(3)}(h)$, recall from High-Order Local Regularity (Assumption \ref{asm:local_regularity}) that third-order log-density derivatives are uniformly dominated by a square-integrable envelope function $M_3(X)$ across $\mathcal{U}(\theta_0)$:
	\begin{equation}
		\sup_{\theta \in \mathcal{U}(\theta_0)} \left| \frac{\partial^3 \log p(X; \theta)}{\partial \theta^i \partial \theta^j \partial \theta^k} \right| \le M_3(X), \quad \text{with } \mathbb{E}_{\theta_0}[M_3(X)] < \infty.
	\end{equation}
	Applying this envelope bound to the integral remainder in Equation~\eqref{eq:integral_remainder_def}:
	\begin{align}
		\left| R_{n,H}^{(3)}(h) \right| &\le \frac{1}{6 N^{3/2}} \sum_{i,j,k=1}^D \left| P_{\theta_0}^H h \right|^i \left| P_{\theta_0}^H h \right|^j \left| P_{\theta_0}^H h \right|^k \int_0^1 3(1-t)^2 M_3(X_n) \, dt \nonumber \\
		&= \frac{1}{6 N^{3/2}} \left( \sum_{i=1}^D \left| (P_{\theta_0}^H h)^i \right| \right)^3 M_3(X_n) \nonumber \\
		&\le \frac{1}{6 N^{3/2}} \left( \sqrt{D} \left\| P_{\theta_0}^H h \right\|_2 \right)^3 M_3(X_n) \le \frac{D^{3/2} C_P^3 \|h\|_2^3}{6 N^{3/2}} M_3(X_n).
	\end{align}
	
	Summing over $n = 1, \dots, N$ and taking the uniform supremum over the compact cage $h \in \mathbb{K}$:
	\begin{align}
		\sup_{h \in \mathbb{K}} \left| \sum_{n=1}^N R_{n,H}^{(3)}(h) \right| &\le \sum_{n=1}^N \sup_{h \in \mathbb{K}} \left| R_{n,H}^{(3)}(h) \right| \nonumber \\
		&\le \frac{D^{3/2} C_P^3 R_{\mathbb{K}}^3}{6 \sqrt{N}} \left( \frac{1}{N} \sum_{n=1}^N M_3(X_n) \right).
		\label{eq:third_order_sum_bound}
	\end{align}
	By the Weak Law of Large Numbers, $\frac{1}{N} \sum_{n=1}^N M_3(X_n) \xrightarrow{P_{\theta_0}} \mathbb{E}_{\theta_0}[M_3(X)] < \infty$, which implies $\frac{1}{N} \sum_{n=1}^N M_3(X_n) = \mathcal{O}_p(1)$. Consequently:
	\begin{equation}
		\sup_{h \in \mathbb{K}} \left| \sum_{n=1}^N R_{n,H}^{(3)}(h) \right| \le \frac{D^{3/2} C_P^3 R_{\mathbb{K}}^3}{6 \sqrt{N}} \mathcal{O}_p(1) = \mathcal{O}_p\left(\frac{1}{\sqrt{N}}\right) \xrightarrow{P_{\theta_0}} 0 \quad \text{as } N \to \infty.
		\label{eq:remainder_rate_proven}
	\end{equation}
	
	Combining the bounds for both components of $\mathcal{R}_{N,H}(h)$ in Equation~\eqref{eq:total_remainder_def}:
	\begin{equation}
		\sup_{h \in \mathbb{K}} \left| \mathcal{R}_{N,H}(h) \right| \le \frac{1}{2} R_{\mathbb{K}}^2 C_P^2 \|Q_N\|_{\mathrm{op}} + \sup_{h \in \mathbb{K}} \left| \sum_{n=1}^N R_{n,H}^{(3)}(h) \right| = o_{P_{\theta_0}}(1) + \mathcal{O}_p\left(N^{-1/2}\right) \xrightarrow{P_{\theta_0}} 0,
	\end{equation}
	which establishes Equation~\eqref{eq:horizontal_remainder_uniform_zero}.
	
	\subsubsection*{Step 5: Synthesis and Final Assembly}
	
	Substituting the linear term from Equation~\eqref{eq:linear_term_derivation}, the quadratic metric term from Equation~\eqref{eq:i2_simplified}, and the remainder process from Equation~\eqref{eq:total_remainder_def} back into Equation~\eqref{eq:lambda_n_three_term_sum}:
	\begin{equation}
		\Lambda_N(h) = h^T \Delta_{N,H}(\theta_0) - \frac{1}{2} h^T \left[\psi_{N,H}^* G^{(N)}\right](\mathbf{0}) h + \mathcal{R}_{N,H}(h),
	\end{equation}
	where:
	\begin{itemize}
		\item $\Delta_{N,H}(\theta_0) \xrightarrow{d} \mathcal{N}\left(\mathbf{0}, g_0\big|_{\mathcal{B}}\right)$,
		\item $\left[\psi_{N,H}^* G^{(N)}\right](\mathbf{0}) = g_0\big|_{\mathcal{B}}$,
		\item $\sup_{h \in \mathbb{K}} |\mathcal{R}_{N,H}(h)| = \mathcal{O}_p\left(N^{-1/2}\right) \xrightarrow{P_{\theta_0}} 0$.
	\end{itemize}
	This completes the rigorous, step-by-step mathematical proof of Theorem 10.
\end{proof}

\subsection{Uniform Horizontal Metric Stabilization}
\label{subsec:horizontal_metric_stabilization}

\begin{lemma}[Algebraic Representation of Horizontally Projected Pulled-Back Fisher Metric Tensor]
	\label{lem:horizontal_pulled_back_metric_algebraic}
	Let $\theta_0 \in \mathrm{Int}(\Theta)$ be an ambient parameter state in a $D$-dimensional parameter manifold $\Theta \subset \mathbb{R}^D$, and let $H_{\theta_0} \subset T_{\theta_0}\Theta$ be its $d$-dimensional horizontal subspace relative to the Fisher-compatible Ehresmann connection $\omega_{\theta_0}$. Let $\psi_{N,H} : \mathbb{K} \to \Theta$ be the localized horizontal rescaling map defined on a compact coordinate cage $\mathbb{K} \subset H_{\theta_0} \cong \mathbb{R}^d$ by
	\begin{equation}
		\psi_{N,H}(h) \equiv \theta_0 + \frac{1}{\sqrt{N}} P_{\theta_0}^H h, \quad \text{for } h \in \mathbb{K},
	\end{equation}
	where $P_{\theta_0}^H$ denotes the horizontal projection operator at $\theta_0$.
	
	Let $\widetilde{G}_H^{(N)}(h)$ be the horizontally projected, pulled-back $N$-sample Fisher Information Metric tensor field on $\mathbb{K}$, then we have:
	\begin{align}
		\widetilde{G}_H^{(N)}(h) &\equiv \left[ \psi_{N,H}^* G^{(N)} \right](h) \nonumber \\
		&\equiv \left[ \psi_{N,H}^* \left( P_\theta^H G^{(N)}(\theta) P_\theta^H \right) \right](h) \nonumber \\
		&= P_{\psi_{N,H}(h)}^H \left( \frac{1}{N} G^{(N)}\left(\theta_0 + \frac{P_{\theta_0}^H h}{\sqrt{N}}\right) \right) P_{\psi_{N,H}(h)}^H \nonumber \\
		&= P_{\psi_{N,H}(h)}^H g^{(1)}\left(\theta_0 + \frac{P_{\theta_0}^H h}{\sqrt{N}}\right) P_{\psi_{N,H}(h)}^H,
	\end{align}
	where $G^{(N)}(\theta)$ is the $N$-sample Fisher Information Metric tensor and $g^{(1)}(\theta)$ is the single-observation Fisher Information Metric tensor.
\end{lemma}

\begin{proof}
	We execute the strict mathematical deduction in four sequential measure-theoretic and differential-geometric steps.
	
	\paragraph{Step 1: Horizontal Projection of the Ambient $N$-Sample Metric Field.}
	On the statistical fiber bundle $(\Theta, \mathcal{B}, \pi)$, the ambient $N$-sample Fisher Information Metric tensor field $G^{(N)}(\theta)$ at any point $\theta \in \Theta$ is restricted to the statistically identifiable horizontal distribution $H_\theta = \mathrm{im}(P_\theta^H)$ via the idempotent horizontal projection operator $P_\theta^H \equiv \mathrm{id}_{T_\theta\Theta} - \omega_\theta$. The horizontally filtered $2$-covariant tensor field $G_H^{(N)}(\theta)$ is defined pointwise by:
	\begin{equation}
		G_H^{(N)}(\theta) \equiv P_\theta^H G^{(N)}(\theta) P_\theta^H.
		\label{eq:filtered_metric_def}
	\end{equation}
	This filtration projects out all unidentifiable gauge directions lying in the vertical sub-bundle $V_\theta \equiv \ker(d\pi_\theta)$.
	
	\paragraph{Step 2: Differential Push-Forward and Multivariable Jacobian Scaffold.}
	Under the horizontal microscope mapping $\psi_{N,H}(h) = \theta_0 + \frac{1}{\sqrt{N}} P_{\theta_0}^H h$, the linear displacement coordinate $h \in \mathbb{K} \subset H_{\theta_0}$ is mapped into the localized parameter neighborhood $\mathcal{U}_N(\theta_0) \subset \Theta$. The differential push-forward operator $d\psi_{N,H}: T_h \mathbb{K} \to T_{\psi_{N,H}(h)} \Theta$ acts as an isotropic linear scaling transformation. 
	
	In local Cartesian derivation coordinates, its spatial Jacobian matrix scaffold $J(\psi_{N,H}) \in \mathbb{R}^{D \times d}$ is given by:
	\begin{equation}
		J_i^a(h) \equiv \frac{\partial \psi_{N,H}^a(h)}{\partial h^i} = \frac{1}{\sqrt{N}} (P_{\theta_0}^H)_i^a, \quad \forall a \in \{1, \dots, D\}, \, i \in \{1, \dots, d\}.
		\label{eq:horizontal_jacobian_scaffold}
	\end{equation}
	
	\paragraph{Step 3: Cotangent Pullback Mechanics and Exact Scale Factor Cancellation.}
	Applying the dual cotangent pullback operator $\psi_{N,H}^*$ on $2$-covariant tensor fields (Definition \ref{def:cotangent_pullback} / Theorem \ref{thm:universal_scaling_law}) to the horizontally filtered metric tensor $G_H^{(N)}(\theta) = P_\theta^H G^{(N)}(\theta) P_\theta^H$ requires contracting the tensor evaluated at the image point $\theta = \psi_{N,H}(h) = \theta_0 + \frac{P_{\theta_0}^H h}{\sqrt{N}}$ against two copies of the Jacobian scaffold matrix $J(\psi_{N,H})$:
	\begin{align}
		\widetilde{G}_H^{(N)}(h) &\equiv \left[ \psi_{N,H}^* \left( P_\theta^H G^{(N)}(\theta) P_\theta^H \right) \right](h) \nonumber \\
		&= J(\psi_{N,H})^T \left[ P_{\psi_{N,H}(h)}^H G^{(N)}\left(\psi_{N,H}(h)\right) P_{\psi_{N,H}(h)}^H \right] J(\psi_{N,H}) \nonumber \\
		&= \left( \frac{1}{\sqrt{N}} P_{\theta_0}^H \right)^T \left[ P_{\psi_{N,H}(h)}^H G^{(N)}\left(\theta_0 + \frac{P_{\theta_0}^H h}{\sqrt{N}}\right) P_{\psi_{N,H}(h)}^H \right] \left( \frac{1}{\sqrt{N}} P_{\theta_0}^H \right).
	\end{align}
	Factoring out the isotropic scalar scaling factors $\frac{1}{\sqrt{N}} \times \frac{1}{\sqrt{N}} = \frac{1}{N}$, and noting that $P_{\psi_{N,H}(h)}^H$ acts as the canonical horizontal projection operator evaluated at the target manifold state $\psi_{N,H}(h)$, the pulled-back tensor field reduces strictly to:
	\begin{equation}
		\widetilde{G}_H^{(N)}(h) = P_{\psi_{N,H}(h)}^H \left( \frac{1}{N} G^{(N)}\left(\theta_0 + \frac{P_{\theta_0}^H h}{\sqrt{N}}\right) \right) P_{\psi_{N,H}(h)}^H,
		\label{eq:third_equality_proven}
	\end{equation}
	establishing the third equality.
	
	\paragraph{Step 4: Substitution of $i.i.d.$ Metric Additivity.}
	By Lemma \ref{lem:additivity_fisher_metric}  (Additivity of the Fisher Information Metric under $i.i.d.$ Product Measures), the joint $N$-sample Fisher Information Metric tensor $G^{(N)}(\theta)$ scales linearly with sample size $N$ relative to the single-observation Fisher Information Metric tensor $g^{(1)}(\theta)$:
	\begin{equation}
		G^{(N)}(\theta) = N \cdot g^{(1)}(\theta), \quad \forall \theta \in \Theta.
		\label{eq:fim_additivity_subst}
	\end{equation}
	Substituting Equation~\eqref{eq:fim_additivity_subst} evaluated at parameter state $\theta = \theta_0 + \frac{P_{\theta_0}^H h}{\sqrt{N}}$ directly into Equation~\eqref{eq:third_equality_proven}:
	\begin{align}
		\widetilde{G}_H^{(N)}(h) &= P_{\psi_{N,H}(h)}^H \left( \frac{1}{N} \left[ N \cdot g^{(1)}\left(\theta_0 + \frac{P_{\theta_0}^H h}{\sqrt{N}}\right) \right] \right) P_{\psi_{N,H}(h)}^H \nonumber \\
		&= P_{\psi_{N,H}(h)}^H \left( \left[ \frac{N}{N} \right] \cdot g^{(1)}\left(\theta_0 + \frac{P_{\theta_0}^H h}{\sqrt{N}}\right) \right) P_{\psi_{N,H}(h)}^H \nonumber \\
		&= P_{\psi_{N,H}(h)}^H g^{(1)}\left(\theta_0 + \frac{P_{\theta_0}^H h}{\sqrt{N}}\right) P_{\psi_{N,H}(h)}^H.
	\end{align}
	The sample-size capacity expansion factor $N$ from $i.i.d.$ sampling cancels identically against the spatial cotangent contraction factor $\frac{1}{N}$ from the Jacobian pullback scaffold, establishing the fourth equality and completing the proof.
\end{proof}

\begin{lemma}[Uniform Horizontal Metric Stabilization]
	\label{lem:uniform_horizontal_metric_stabilization}
	Let $\theta_0 \in \mathrm{Int}(\Theta)$ be an ambient parameter state in a $D$-dimensional parameter manifold $\Theta \subset \mathbb{R}^D$, and let $H_{\theta_0} \subset T_{\theta_0}\Theta$ be its $d$-dimensional horizontal subspace relative to the Fisher-compatible Ehresmann connection $\omega_{\theta_0}$.
	
	Define the local horizontal rescaling map $\psi_{N,H}: \mathbb{K} \to \Theta$ on a compact coordinate cage $\mathbb{K} \subset H_{\theta_0} \cong \mathbb{R}^d$ by:
	\begin{equation}
		\psi_{N,H}(h) \equiv \theta_0 + \frac{1}{\sqrt{N}} P_{\theta_0}^H h, \quad \text{for } h \in \mathbb{K},
	\end{equation}
	where $P_{\theta_0}^H$ is the horizontal projection operator at $\theta_0$.
	
	Let $\widetilde{G}_H^{(N)}(h)$ be the horizontally projected, pulled-back $N$-sample Fisher Information Metric tensor field on $\mathbb{K}$, and $G^{(N)}(\theta) = N \cdot g^{(1)}(\theta)$ is the $N$-sample Fisher Information Metric tensor and $g^{(1)}(\theta)$ be the single-observation Fisher Information Metric.
	
	Under High-Order Local Regularity (Assumption~\ref{asm:local_regularity}), the tensor field $\widetilde{G}_H^{(N)}(h)$ converges uniformly over the compact set $\mathbb{K} \subset H_{\theta_0}$ in matrix operator norm $\|\cdot\|_{\mathrm{op}}$ to the constant, non-singular base Fisher metric tensor $g_0\big|_{\mathcal{B}} \equiv P_{\theta_0}^H g^{(1)}(\theta_0) P_{\theta_0}^H = g^{(1)}(\theta_0)\big|_{H_{\theta_0}}$:
	\begin{equation}
		\lim_{N \to \infty} \sup_{h \in \mathbb{K} \subset H_{\theta_0}} \left\| \widetilde{G}_H^{(N)}(h) - g_0\big|_{\mathcal{B}} \right\|_{\mathrm{op}} = 0.
		\label{eq:uniform_horizontal_metric_stabilization}
	\end{equation}
\end{lemma}

\begin{proof}
	Based on Lemma \ref{lem:horizontal_pulled_back_metric_algebraic}, we present a step-by-step mathematical deduction of the uniform convergence result.
	
	\paragraph{Step 1: Metric Additivity and Pullback Cancellation Mechanics.}
	By the additivity of the Fisher Information Metric under $i.i.d.$ product measures, the $N$-sample joint Fisher metric field satisfies $G^{(N)}(\theta) = N \cdot g^{(1)}(\theta)$ (Lemma~\ref{lem:additivity_fisher_metric}). Under the horizontal local microscope map $\psi_{N,H}(h) = \theta_0 + \frac{1}{\sqrt{N}} P_{\theta_0}^H h$, the differential Jacobian scaffold is given by:
	\begin{equation}
		J(\psi_{N,H}) = \frac{1}{\sqrt{N}} P_{\theta_0}^H.
	\end{equation}
	Applying the 2-covariant pullback operator $\psi_{N,H}^*$ to the horizontally filtered metric $P_\theta^H G^{(N)}(\theta) P_\theta^H$ yields:
	\begin{align}
		\widetilde{G}_H^{(N)}(h) &= J(\psi_{N,H})^T \left[ P_{\psi_{N,H}(h)}^H G^{(N)}\left(\psi_{N,H}(h)\right) P_{\psi_{N,H}(h)}^H \right] J(\psi_{N,H}) \nonumber \\
		&= \left( \frac{1}{\sqrt{N}} P_{\theta_0}^H \right)^T \left[ P_{\psi_{N,H}(h)}^H \left( N \cdot g^{(1)}\left(\theta_0 + \frac{P_{\theta_0}^H h}{\sqrt{N}}\right) \right) P_{\psi_{N,H}(h)}^H \right] \left( \frac{1}{\sqrt{N}} P_{\theta_0}^H \right) \nonumber \\
		&= P_{\theta_0}^H P_{\psi_{N,H}(h)}^H g^{(1)}\left(\theta_0 + \frac{P_{\theta_0}^H h}{\sqrt{N}}\right) P_{\psi_{N,H}(h)}^H P_{\theta_0}^H.
	\end{align}
	Notice that the sample capacity expansion factor $N$ cancels identically against the spatial scaling factor $\left(\frac{1}{\sqrt{N}}\right)^2 = \frac{1}{N}$, demonstrating exact metric capacity-cotangent scaling cancellation ($N^{1 - 2/2} = N^0 = 1$) (Theorem~\ref{thm:universal_scaling_law}).
	
	\paragraph{Step 2: Algebraic Three-Term Expansion.}
	For notation conciseness, define the localized parameter state $\theta_N(h) \equiv \psi_{N,H}(h) = \theta_0 + \frac{P_{\theta_0}^H h}{\sqrt{N}}$. We decompose the difference between $\widetilde{G}_H^{(N)}(h)$ and the base metric tensor $g_0\big|_{\mathcal{B}} \equiv P_{\theta_0}^H g^{(1)}(\theta_0) P_{\theta_0}^H$ into three operator terms:
	\begin{equation}
		\widetilde{G}_H^{(N)}(h) - g_0\big|_{\mathcal{B}} = \mathbf{T}_1(N, h) + \mathbf{T}_2(N, h) + \mathbf{T}_3(N, h),
	\end{equation}
	where:
	\begin{align}
		\mathbf{T}_1(N, h) &\equiv P_{\theta_N(h)}^H \left[ g^{(1)}(\theta_N(h)) - g^{(1)}(\theta_0) \right] P_{\theta_N(h)}^H, \nonumber \\
		\mathbf{T}_2(N, h) &\equiv \left( P_{\theta_N(h)}^H - P_{\theta_0}^H \right) g^{(1)}(\theta_0) P_{\theta_N(h)}^H, \nonumber \\
		\mathbf{T}_3(N, h) &\equiv P_{\theta_0}^H g^{(1)}(\theta_0) \left( P_{\theta_N(h)}^H - P_{\theta_0}^H \right).
	\end{align}
	Applying the triangle inequality for the matrix operator norm $\|\cdot\|_{\mathrm{op}}$:
	\begin{equation}
		\left\| \widetilde{G}_H^{(N)}(h) - g_0\big|_{\mathcal{B}} \right\|_{\mathrm{op}} \le \left\| \mathbf{T}_1(N, h) \right\|_{\mathrm{op}} + \left\| \mathbf{T}_2(N, h) \right\|_{\mathrm{op}} + \left\| \mathbf{T}_3(N, h) \right\|_{\mathrm{op}}.
		\label{eq:triangle_ineq_decomp}
	\end{equation}
	
	\paragraph{Step 3: Lipschitz Bounds on Component Tensors.}
	Under High-Order Local Regularity (Assumption~\ref{asm:local_regularity}), the single-sample Fisher metric field $g^{(1)}(\theta)$ is $C^2$-smooth on the closed neighborhood $\overline{\mathcal{U}}(\theta_0)$, and the Ehresmann connection projection $P_\theta^H = \mathrm{id} - \omega_\theta$ depends smoothly on $\theta$ (Definition~\ref{def:connection_1form_projection}).
\begin{enumerate}
\item \textbf{Metric Field Lipschitz Bound:} Since $g^{(1)}(\theta)$ is continuously differentiable on the compact set $\overline{\mathcal{U}}(\theta_0)$, it is Lipschitz continuous with constant $L_g \equiv \max_{i,j,k} \sup_{\theta \in \mathcal{U}(\theta_0)} \left| \frac{\partial g_{ij}^{(1)}(\theta)}{\partial \theta_k} \right| < \infty$:
	\begin{equation}
		\left\| g^{(1)}(\theta_N(h)) - g^{(1)}(\theta_0) \right\|_{\mathrm{op}} \le L_g \left\| \theta_N(h) - \theta_0 \right\|_2 = \frac{L_g}{\sqrt{N}} \left\| P_{\theta_0}^H h \right\|_2 \le \frac{L_g C_P R_{\mathbb{K}}}{\sqrt{N}},
	\end{equation}
	where $C_P \equiv \|P_{\theta_0}^H\|_{\mathrm{op}} < \infty$ and $R_{\mathbb{K}} \equiv \sup_{h \in \mathbb{K}} \|h\|_2 < \infty$ due to the compactness of $\mathbb{K}$ (Definition~\ref{def:microscope_map}).
	
\item  \textbf{Projection Operator Lipschitz Bound:} Smoothness of the connection 1-form implies $P_\theta^H$ is locally Lipschitz continuous with constant $L_P < \infty$:
	\begin{equation}
		\left\| P_{\theta_N(h)}^H - P_{\theta_0}^H \right\|_{\mathrm{op}} \le L_P \left\| \theta_N(h) - \theta_0 \right\|_2 \le \frac{L_P C_P R_{\mathbb{K}}}{\sqrt{N}}.
	\end{equation}
	
\item \textbf{Uniform Operator Boundedness:} On the compact neighborhood $\overline{\mathcal{U}}(\theta_0)$, the projection operator and base metric are uniformly bounded in operator norm:
	\begin{equation}
		M_P \equiv \sup_{\theta \in \mathcal{U}(\theta_0)} \left\| P_\theta^H \right\|_{\mathrm{op}} < \infty, \quad M_g \equiv \left\| g^{(1)}(\theta_0) \right\|_{\mathrm{op}} < \infty.
	\end{equation}
\end{enumerate}	
	\paragraph{Step 4: Operator Norm Bounding of Individual Terms.}
	Using sub-multiplicativity of the operator norm:
	\begin{align}
		\left\| \mathbf{T}_1(N, h) \right\|_{\mathrm{op}} &\le \left\| P_{\theta_N(h)}^H \right\|_{\mathrm{op}}^2 \cdot \left\| g^{(1)}(\theta_N(h)) - g^{(1)}(\theta_0) \right\|_{\mathrm{op}} \le M_P^2 \frac{L_g C_P R_{\mathbb{K}}}{\sqrt{N}}, \\
		\left\| \mathbf{T}_2(N, h) \right\|_{\mathrm{op}} &\le \left\| P_{\theta_N(h)}^H - P_{\theta_0}^H \right\|_{\mathrm{op}} \cdot \left\| g^{(1)}(\theta_0) \right\|_{\mathrm{op}} \cdot \left\| P_{\theta_N(h)}^H \right\|_{\mathrm{op}} \le \left(\frac{L_P C_P R_{\mathbb{K}}}{\sqrt{N}}\right) M_g M_P, \\
		\left\| \mathbf{T}_3(N, h) \right\|_{\mathrm{op}} &\le \left\| P_{\theta_0}^H \right\|_{\mathrm{op}} \cdot \left\| g^{(1)}(\theta_0) \right\|_{\mathrm{op}} \cdot \left\| P_{\theta_N(h)}^H - P_{\theta_0}^H \right\|_{\mathrm{op}} \le C_P M_g \left(\frac{L_P C_P R_{\mathbb{K}}}{\sqrt{N}}\right).
	\end{align}
	
	\paragraph{Step 5: Uniform Supremum Convergence and Non-Singularity.}
	Summing the three bounds into Equation~\eqref{eq:triangle_ineq_decomp}:
	\begin{equation}
		\left\| \widetilde{G}_H^{(N)}(h) - g_0\big|_{\mathcal{B}} \right\|_{\mathrm{op}} \le \frac{K_{\mathbb{K}}}{\sqrt{N}},
	\end{equation}
	where $K_{\mathbb{K}} \equiv C_P R_{\mathbb{K}} \left( M_P^2 L_g + L_P M_P M_g + C_P L_P M_g \right) < \infty$ is a constant independent of $N$ and $h$.
	
	Taking the uniform supremum over $h \in \mathbb{K} \subset H_{\theta_0}$:
	\begin{equation}
		\sup_{h \in \mathbb{K} \subset H_{\theta_0}} \left\| \widetilde{G}_H^{(N)}(h) - g_0\big|_{\mathcal{B}} \right\|_{\mathrm{op}} \le \frac{K_{\mathbb{K}}}{\sqrt{N}} = \mathcal{O}\left(\frac{1}{\sqrt{N}}\right).
	\end{equation}
	Taking $N \to \infty$ yields:
	\begin{equation}
		\lim_{N \to \infty} \sup_{h \in \mathbb{K} \subset H_{\theta_0}} \left\| \widetilde{G}_H^{(N)}(h) - g_0\big|_{\mathcal{B}} \right\|_{\mathrm{op}} = 0.
	\end{equation}
	Finally, by the Fisher-compatible Ehresmann connection framework, the restricted metric $g_0\big|_{\mathcal{B}} \equiv g^{(1)}(\theta_0)\big|_{H_{\theta_0}}$ is strictly positive-definite on $H_{\theta_0}$ ($\lambda_{\min}(g_0\big|_{\mathcal{B}}) \ge \kappa_0 > 0$), establishing uniform metric stabilization to a non-singular base Fisher metric tensor (Theorem~\ref{thm:horizontal_strict_positivity}).
\end{proof}

\subsection{Proof of the Unification Theorem on $H_\theta$}
\label{subsec:unification_theorem_horizontal}

We now establish the main result of this section: {\it the complete equivalence between the Geometric Priority Paradigm and the Statistical Operationalism Paradigm on the non-singular horizontal carriage $H_\theta$ (and the quotient base manifold $\mathcal{B}$).}

\begin{theorem}[Geometric-Operational Unification Theorem under Ehresmann Framework]
	\label{thm:geometric_operational_unification_horizontal}
	Under High-Order Local Regularity (Assumption~\ref{asm:local_regularity}), without assuming ambient parameter space non-degeneracy, Axiom~\ref{ax:geom_priority} (Geometric Priority) and Axiom~\ref{ax:stat_oper} (Statistical Operationalism) are logically and mathematically equivalent when restricted to the horizontal distribution $H_\theta$ (and quotient base manifold $\mathcal{B}$):
	\begin{equation}
		IG_N \Big|_{H_\theta} \xrightarrow{\text{geom}} \mathcal{M}_\infty \quad \Longleftrightarrow \quad IG_N \Big|_{H_\theta} \xrightarrow{\text{oper}} \mathcal{M}_\infty \equiv \mathrm{CS}.
		\label{eq:unification_horizontal_equivalence}
	\end{equation}
\end{theorem}

\begin{proof}
	We execute a rigorous, bi-directional mathematical proof covering necessity ($\Rightarrow$) and sufficiency ($\Leftarrow$) sequentially.
	
	\subsubsection*{Part 1: Forward Implication (Necessity: $IG_N \big|_{H_\theta} \xrightarrow{\text{geom}} \mathcal{M}_\infty \implies IG_N \big|_{H_\theta} \xrightarrow{\text{oper}} \mathcal{M}_\infty$)}
	
	Assume Axiom~\ref{ax:geom_priority} (Geometric Priority) holds when restricted to the horizontal distribution $H_{\theta_0} \cong \mathbb{R}^d$. Under the localized horizontal microscope map $\psi_{N,H}(h) = \theta_0 + \frac{1}{\sqrt{N}} P_{\theta_0}^H h$ for $h \in \mathbb{K} \subset H_{\theta_0} \cong \mathbb{R}^d$, the pulled-back horizontally filtered tensor fields satisfy uniform convergence over every compact coordinate cage $\mathbb{K} \subset H_{\theta_0}$:
	\begin{align}
		\lim_{N \to \infty} \sup_{h \in \mathbb{K}} \left\| \left[\psi_{N,H}^* G^{(N)}\right](h) - g_0\big|_{\mathcal{B}} \right\|_{op} &= 0 \quad \text{(Horizontal Metric Stabilization)}, \label{eq:proof1_metric_stabilization} \\
		\lim_{N \to \infty} \sup_{h \in \mathbb{K}} \left| \left[\psi_{N,H}^* \nabla^{(\alpha,N)}\right]_{ijk}(h) \right| &= 0 \quad \text{(Horizontal Connection Dissolution)}, \label{eq:proof1_connection_dissolution} \\
		\lim_{N \to \infty} \sup_{h \in \mathbb{K}} \left\| \left[\psi_{N,H}^* \mathcal{R}^{(\alpha,N)}\right]_{ijmn}(h) \right\|_{op} &= 0 \quad \text{(Horizontal Curvature Annihilation)}. \label{eq:proof1_curvature_annihilation}
	\end{align}
	
	\paragraph{Step 1.1: Localized Radon-Nikodym Likelihood Ratio Decomposition.}
	By Theorem~\ref{thm:generalized_horizontal_log_likelihood} (Generalized Horizontal Log-Likelihood Representation) and Lemma~\ref{lem:pillar_2_taylor_geometry} (Taylor-Geometry Identity), the localized Radon-Nikodym log-likelihood ratio process 
	$$\Lambda_N(h) \equiv \log \frac{d P_{\theta_0 + P_{\theta_0}^H h / \sqrt{N}}^{(N)}}{d P_{\theta_0}^{(N)}}(X^N)$$ 
	admits the exact algebraic decomposition built from the horizontally filtered pulled-back geometric fields:
	\begin{equation}
		\Lambda_N(h) = h^T \Delta_{N,H}(\theta_0) - \frac{1}{2} h^T \left[\psi_{N,H}^* G^{(N)}\right](0) h + \mathcal{R}_{N,H}(h),
		\label{eq:proof1_lambda_decomp}
	\end{equation}
	where $\Delta_{N,H}(\theta_0) \equiv \frac{1}{\sqrt{N}} \sum_{n=1}^N P_{\theta_0}^H \nabla_\theta \log p(X_n; \theta_0)$ is the horizontally projected sample score vector, and $\mathcal{R}_{N,H}(h)$ is the non-linear stochastic remainder process.
	
	\paragraph{Step 1.2: Uniform Bounding of the Stochastic Remainder.}
	By Lemma~\ref{lem:pillar_2_taylor_geometry}, the remainder process $\mathcal{R}_{N,H}(h)$ is governed by the uniform geometric envelope bound across $\mathbb{K}$:
	\begin{equation}
		\sup_{h \in \mathbb{K}} |\mathcal{R}_{N,H}(h)| \le \frac{\|h\|_2^3}{6} \sup_{h \in \mathbb{K}} \left| \left[\psi_{N,H}^* \nabla^{(\alpha,N)}\right]_{ijk}(h) \right| + \mathcal{O}_p\left( \sup_{h \in \mathbb{K}} \left\| \left[\psi_{N,H}^* \mathcal{R}^{(\alpha,N)}\right]_{ijmn}(h) \right\|_{op} \right).
		\label{eq:proof1_remainder_bound}
	\end{equation}
	Substituting the connection dissolution limit \eqref{eq:proof1_connection_dissolution} and accelerated curvature decay \eqref{eq:proof1_curvature_annihilation} into \eqref{eq:proof1_remainder_bound} yields:
	\begin{equation}
		\sup_{h \in \mathbb{K}} |\mathcal{R}_{N,H}(h)| = \mathcal{O}\left(N^{-1/2}\right) + \mathcal{O}_p\left(N^{-1}\right) = \mathcal{O}_p\left(N^{-1/2}\right) \xrightarrow{P_{\theta_0}} 0 \quad \text{as } N \to \infty.
		\label{eq:proof1_remainder_vanish}
	\end{equation}
	
	\paragraph{Step 1.3: Asymptotic Score Normality and LAN Limit Formulation.}
	By Theorem~\ref{thm:generalized_horizontal_log_likelihood}, under High-Order Local Regularity (Assumption~\ref{asm:local_regularity}), the horizontally projected score vector satisfies the Multivariate Central Limit Theorem:
	\begin{equation}
		\Delta_{N,H}(\theta_0) \xrightarrow{d} \Delta_{\infty, H} \sim \mathcal{N}\left(0, g_0\big|_{\mathcal{B}}\right).
		\label{eq:proof1_score_normality}
	\end{equation}
	Substituting the metric stabilization limit $\left[\psi_{N,H}^* G^{(N)}\right](0) = g_0\big|_{\mathcal{B}}$ from \eqref{eq:proof1_metric_stabilization}, the score convergence \eqref{eq:proof1_score_normality}, and the vanishing remainder \eqref{eq:proof1_remainder_vanish} into \eqref{eq:proof1_lambda_decomp} establishes the uniform linear-quadratic Local Asymptotic Normality (LAN) expansion on $H_\theta$:
	\begin{equation}
		\Lambda_N(h) = h^T \Delta_{N,H}(\theta_0) - \frac{1}{2} h^T \left(g_0\big|_{\mathcal{B}}\right) h + o_{P_{\theta_0}}(1).
		\label{eq:proof1_lan_expansion}
	\end{equation}
	
	\paragraph{Step 1.4: Application of Le Cam Deficiency Limit Lemma.}
	By Theorem~\ref{thm:horizontal_strict_positivity} (Strict Positivity and Non-Singularity of Restricted Metric), the base Fisher metric tensor $g_0\big|_{\mathcal{B}} \equiv g^{(1)}(\theta_0)\big|_{H_{\theta_0}}$ is strictly positive-definite and non-singular on $H_{\theta_0}$, with $\lambda_{\min}(g_0\big|_{\mathcal{B}}) \ge \kappa_0 > 0$. Applying Lemma~\ref{lem:lecam_deficiency_equivalence} (Le Cam Experiment Deficiency Limit Lemma) directly to the LAN representation \eqref{eq:proof1_lan_expansion} confirms contiguity and proves that Le Cam's experiment deficiency distance between the localized empirical experiment $\mathcal{E}_N(\mathbb{K})$ and the canonical Gaussian shift experiment $\mathcal{E}_\infty(\mathbb{K})$ vanishes asymptotically:
	\begin{equation}
		\lim_{N \to \infty} \Delta\left( \mathcal{E}_N(\mathbb{K}), \, \mathcal{E}_\infty(\mathbb{K}) \right) = 0.
		\label{eq:proof1_lecam_deficiency_zero}
	\end{equation}
	This fulfills Axiom~\ref{ax:stat_oper} (Statistical Operationalism) on $H_\theta$, completing the proof of Necessity:
	\begin{equation}
		IG_N \Big|_{H_\theta} \xrightarrow{\text{geom}} \mathcal{M}_\infty \implies IG_N \Big|_{H_\theta} \xrightarrow{\text{oper}} \mathcal{M}_\infty.
	\end{equation}
	
	\subsubsection*{Part 2: Converse Implication (Sufficiency: $IG_N \big|_{H_\theta} \xrightarrow{\text{oper}} \mathcal{M}_\infty \implies IG_N \big|_{H_\theta} \xrightarrow{\text{geom}} \mathcal{M}_\infty$)}
	
	Assume Axiom~\ref{ax:stat_oper} (Statistical Operationalism) holds on $H_\theta$. That is, for every compact coordinate cage $\mathbb{K} \subset H_{\theta_0} \cong \mathbb{R}^d$, the Le Cam experiment deficiency distance vanishes asymptotically:
	\begin{equation}
		\lim_{N \to \infty} \Delta\left( \mathcal{E}_N(\mathbb{K}), \, \mathcal{E}_\infty(\mathbb{K}) \right) = 0.
		\label{eq:proof2_assumption}
	\end{equation}
	
	\paragraph{Step 2.1: Convergence of Log-Likelihood Process via Le Cam Representation Theorem.}
	By Le Cam's Representation Theorem \cite{lecam1986, lecam2000}, deficiency distance convergence \eqref{eq:proof2_assumption} implies that under $P_{\theta_0}^{(N)}$, the finite-dimensional distributions of the localized log-likelihood ratio process $\Lambda_N(h)$ converge weakly to the canonical quadratic Gaussian limit:
	\begin{equation}
		\Lambda_N(h) \xrightarrow{d} h^T Z - \frac{1}{2} h^T A h \quad \text{under } P_{\theta_0}^{(N)}, \quad \text{where } Z \sim \mathcal{N}(0, A),
		\label{eq:proof2_weak_limit}
	\end{equation}
	for a symmetric, positive-definite precision matrix $A \in \mathbb{R}^{d \times d}$.
	
	\paragraph{Step 2.2: Extraction of Horizontally Projected Metric Tensor ($r = 2$).}
	Differentiating the expected log-likelihood ratio process $\mathbb{E}_{\theta_0}[\Lambda_N(h)]$ twice with respect to the horizontal displacement coordinates $h^i$ and $h^j$ at $h = 0$ recovers the negative pulled-back metric tensor:
	\begin{equation}
		\left. \frac{\partial^2 \mathbb{E}_{\theta_0}[\Lambda_N(h)]}{\partial h^i \partial h^j} \right|_{h=0} = -\left[ \psi_{N,H}^* G^{(N)} \right]_{ij}(0).
		\label{eq:proof2_metric_extract}
	\end{equation}
	From the weak Gaussian limit in \eqref{eq:proof2_weak_limit}, the expected second variation equals $-A_{ij}$. Evaluating at the background state $\theta_0$ uniquely identifies $A \equiv g_0\big|_{\mathcal{B}} = g^{(1)}(\theta_0)\big|_{H_{\theta_0}}$. Applying Lemma~\ref{lem:uniform_horizontal_metric_stabilization} (Uniform Horizontal Metric Stabilization) across the compact cage $\mathbb{K} \subset H_{\theta_0}$:
	\begin{equation}
		\lim_{N \to \infty} \sup_{h \in \mathbb{K}} \left\| \left[ \psi_{N,H}^* G^{(N)} \right](h) - g_0\big|_{\mathcal{B}} \right\|_{op} = 0,
		\label{eq:proof2_metric_proven}
	\end{equation}
	which fulfills condition \eqref{eq:proof1_metric_stabilization} of Axiom~\ref{ax:geom_priority}.
	
	\paragraph{Step 2.3: Extraction of Horizontally Projected Score 1-Form ($r = 1$).}
	Taking the first variation of $\Lambda_N(h)$ at $h = 0$ yields the horizontally projected score 1-form $\psi_{N,H}^*(dL^{(N)})(0)$. By \eqref{eq:proof2_weak_limit}, its limit covariance matrix equals $A = g_0\big|_{\mathcal{B}}$, matching the normalized Gaussian score fluctuation field $\Delta_{\infty, H} \sim \mathcal{N}(0, g_0\big|_{\mathcal{B}})$.
	
	\paragraph{Step 2.4: Extraction of Horizontally Projected Connection Symbols ($r = 3$).}
	By Lemma~\ref{lem:connection_dissolution}, differentiating $\mathbb{E}_{\theta_0}[\Lambda_N(h)]$ three times with respect to horizontal coordinates $h^i, h^j, h^k$ at $h = 0$ relates the third log-density cumulant to the pulled-back connection symbols:
	\begin{equation}
		\left. \frac{\partial^3 \mathbb{E}_{\theta_0}[\Lambda_N(h)]}{\partial h^i \partial h^j \partial h^k} \right|_{h=0} = -\sqrt{N} \left[ \psi_{N,H}^* \nabla^{(\alpha,N)} \right]_{ijk}(0) = -\Gamma_{ijk}^{(\alpha,1)}(\theta_0).
		\label{eq:proof2_third_cumulant_identity}
	\end{equation}
	Because the limiting Gaussian log-likelihood ratio $\Lambda_\infty(h) = h^T Z - \frac{1}{2} h^T A h$ is purely quadratic in $h$, its third derivative vanishes identically:
	\begin{equation}
		\frac{\partial^3 \Lambda_\infty(h)}{\partial h^i \partial h^j \partial h^k} \equiv 0.
		\label{eq:proof2_limit_third_deriv_zero}
	\end{equation}
	Matching higher-order variations under High-Order Local Regularity (Assumption~\ref{asm:local_regularity}) and Theorem~\ref{thm:connection_dissolution}:
	\begin{equation}
		\lim_{N \to \infty} \sup_{h \in \mathbb{K}} \left| \left[ \psi_{N,H}^* \nabla^{(\alpha,N)} \right]_{ijk}(h) \right| = \lim_{N \to \infty} \mathcal{O}\left(\frac{1}{\sqrt{N}}\right) = 0 \equiv \Gamma_{ijk}^{(0)},
		\label{eq:proof2_connection_proven}
	\end{equation}
	which fulfills condition \eqref{eq:proof1_connection_dissolution} of Axiom~\ref{ax:geom_priority}.
	
	\paragraph{Step 2.5: Extraction of Horizontally Projected Riemann Curvature ($r = 4$).}
	Evaluating spatial derivatives of the connection symbols under the localized horizontal microscope map $\psi_{N,H}(h) = \theta_0 + \frac{1}{\sqrt{N}} P_{\theta_0}^H h$ introduces a second chain-rule factor of $1/\sqrt{N}$. By Lemma~\ref{lem:accelerated_curvature_annihilation} and Theorem~\ref{thm:curvature_annihilation}:
	\begin{equation}
		\lim_{N \to \infty} \sup_{h \in \mathbb{K}} \left\| \left[ \psi_{N,H}^* \mathcal{R}^{(\alpha,N)} \right]_{ijmn}(h) \right\|_{op} = \lim_{N \to \infty} \mathcal{O}_p\left(\frac{1}{N}\right) = 0 \equiv \mathcal{R}_{ijmn}^{(0)},
		\label{eq:proof2_curvature_proven}
	\end{equation}
	which fulfills condition \eqref{eq:proof1_curvature_annihilation} of Axiom~\ref{ax:geom_priority}.
	
	\paragraph{Conclusion.}
	Combining Equations \eqref{eq:proof2_metric_proven}, \eqref{eq:proof2_connection_proven}, and \eqref{eq:proof2_curvature_proven} establishes that all three constitutional pillars of Axiom~\ref{ax:geom_priority} ($IG_N \big|_{H_\theta} \xrightarrow{\text{geom}} \mathcal{M}_\infty$) are satisfied on $H_\theta$, completing the proof of Sufficiency:
	\begin{equation}
		IG_N \Big|_{H_\theta} \xrightarrow{\text{oper}} \mathcal{M}_\infty \implies IG_N \Big|_{H_\theta} \xrightarrow{\text{geom}} \mathcal{M}_\infty.
	\end{equation}
	
	Combining Part 1 and Part 2 completes the bi-directional, step-by-step mathematical proof of Theorem~\ref{thm:geometric_operational_unification_horizontal}.
\end{proof}


\subsection{Epistemological Synthesis and Elimination of Non-Identifiability Pathologies}
\label{subsec:epistemological_synthesis_horizontal}

The unification established in Section~6 demonstrates that ambient parameter space non-identifiability ($\det g^{(1)}(\theta) = 0$) does not present an insurmountable barrier to Local Asymptotic Normality (LAN) or information-geometric stability. In high-dimensional, over-parameterized statistical systems—such as deep neural architectures, singular exponential families, and singular latent variable models—the parameter space $\Theta \subseteq \mathbb{R}^D$ ($D \gg d$) possesses continuous gauge symmetries. Under these symmetries, infinitely many parameter vectors $\theta \in \Theta$ yield identical observational probability distributions $P_\theta \in \mathcal{B}$.

By introducing the \emph{Fisher-Compatible Ehresmann Connection} ($\omega_\theta$), the ambient tangent bundle decomposes into two mutually orthogonal sub-bundles:
\begin{equation}
	T_\theta \Theta = H_\theta \oplus V_\theta, \quad \text{where } V_\theta \equiv \ker(d\pi_\theta) \text{ and } H_\theta \equiv V_\theta^{\perp_g}.
\end{equation}
This horizontal-vertical splitting acts as a geometric filter:
\begin{enumerate}[label=(\roman*)]
	\item \textbf{Vertical Sub-bundle ($V_\theta$):} Absorbs all unidentifiable gauge directions, internal parameter redundancies, and zero-eigenvalue modes of the Fisher information matrix.
	\item \textbf{Horizontal Sub-bundle ($H_\theta$):} Isolates the statistically verifiable, information-carrying directions that map bijectively onto the $d$-dimensional quotient base manifold $\mathcal{B} \equiv \Theta / \sim$ of identifiable probability distributions.
\end{enumerate}

By restricting localized experiments and tensor pullbacks strictly to the horizontal distribution $H_{\theta_0} \cong T_{\pi(\theta_0)}\mathcal{B}$, we eliminate all ambient metric singularities and establish a well-posed, gauge-invariant asymptotic decision framework.

\subsubsection{Formulation of Gauge-Invariant LAN}
\label{subsubsec:corrected_corollary}

\begin{corollary}[Gauge-Invariant Local Asymptotic Normality]
	\label{cor:gauge_invariant_lan}
Let $\Theta \subseteq \mathbb{R}^D$ ($D \gg d$) be an over-parameterized ambient parameter space with non-identifiable gauge fibers $\mathcal{F}_p = \pi^{-1}(p)$ and vertical sub-bundle $V_\theta = \ker(d\pi_\theta)$. Let $H_{\theta_0} \subset T_{\theta_0}\Theta$ denote the $d$-dimensional horizontal subspace relative to the Fisher-compatible Ehresmann connection $\omega_{\theta_0}$, and let $\mathcal{L}_{\theta_0} \subset \Theta$ denote the local horizontal leaf manifold passing through $\theta_0$, which is locally leaf-isomorphic to the quotient base space $\mathcal{B} \equiv \Theta / \sim$.

For any compact coordinate cage $\mathbb{K} \subset H_{\theta_0} \cong T_{\pi(\theta_0)}\mathcal{B} \cong \mathbb{R}^d$, let:
\begin{equation}
	\psi_{N,H}^* \mathcal{E}_N(\mathbb{K}) \equiv \left( \mathcal{X}^N, \, \mathcal{A}^{\otimes N}, \, \left\{ P_{\theta_0 + \frac{1}{\sqrt{N}}P_{\theta_0}^H h}^{(N)} : h \in \mathbb{K} \right\} \right)
\end{equation}
denote the localized empirical statistical experiment pulled back along the horizontal projection $P_{\theta_0}^H = \mathrm{id} - \omega_{\theta_0}$, and let:
\begin{equation}
	\mathcal{E}_\infty\left(\mathcal{B}; \mathbb{K}\right) \equiv \left( \mathbb{R}^d, \, \mathcal{B}^d, \, \left\{ \mathcal{N}\left(h, \, \left(g_0\big|_{\mathcal{B}}\right)^{-1}\right) : h \in \mathbb{K} \right\} \right)
\end{equation}
denote the canonical Gaussian shift experiment defined on the identifiable base manifold $\mathcal{B}$, equipped with the non-singular restricted precision matrix $g_0\big|_{\mathcal{B}} \equiv g^{(1)}(\theta_0)\big|_{H_{\theta_0} \times H_{\theta_0}}$.

Then, Local Asymptotic Normality holds in a fully gauge-invariant manner on the horizontal leaf space $\mathcal{L}_{\theta_0} \cong \mathcal{B}$, as quantified by the vanishing of Le Cam's experiment deficiency distance $\Delta$:
\begin{equation}
	\lim_{N \to \infty} \Delta\left( \psi_{N,H}^* \mathcal{E}_N(\mathbb{K}), \; \mathcal{E}_\infty\left(\mathcal{B}; \mathbb{K}\right) \right) = 0,
	\tag{284}
	\label{eq:corrected_equation_284}
\end{equation}
where the vertical sub-bundle $V_\theta$ absorbs all non-identifiable directions, preventing ambient metric singularity ($\det g^{(1)}(\theta) = 0$) and guaranteeing strictly positive-definite, non-singular asymptotic covariance bounds $\left(g_0\big|_{\mathcal{B}}\right)^{-1}$ for all gauge-invariant statistical estimators.
\end{corollary}

\begin{proof}
	The proof proceeds in four sequential measure-theoretic and differential-geometric stages, applying the Fisher-compatible Ehresmann connection framework \cite{lee2013} to restrict localized statistical operations to the non-singular horizontal distribution $H_{\theta_0}$ (Definition~\ref{def:ehresmann_connection}).
	
	\subsubsection*{Stage 1: Local Horizontal Trajectory and Projection}
	By definition of the Fisher-compatible Ehresmann connection decomposition $T_{\theta_0}\Theta = H_{\theta_0} \oplus V_{\theta_0}$, the horizontal projection operator $P_{\theta_0}^H = \mathrm{id}_{T_{\theta_0}\Theta} - \omega_{\theta_0}$ satisfies $P_{\theta_0}^H h = h$ for any horizontal vector $h \in \mathbb{K} \subset H_{\theta_0} \cong \mathbb{R}^d$ (Definition~\ref{def:connection_1form_projection}). 
	For each sample size $N \ge 1$, define the localized horizontal parameter trajectory $\theta_N(h) \in \Theta$ by:
	\begin{equation}
		\theta_N(h) \equiv \theta_0 + \frac{1}{\sqrt{N}} P_{\theta_0}^H h, \quad h \in \mathbb{K} \subset H_{\theta_0}.
	\end{equation}
	This parametrization restricts local parameter perturbations strictly to the horizontal leaf $\mathcal{L}_{\theta_0} \cong \mathcal{B}$, isolating and suppressing all vertical gauge variations in $V_{\theta_0} = \ker(d\pi_{\theta_0})$ (Definition~\ref{def:vertical_subbundle}).
	
	\subsubsection*{Stage 2: Algebraic Likelihood Ratio Expansion on $H_{\theta_0}$}
	Consider the localized Radon-Nikodym log-likelihood ratio process $$\Lambda_N(h) \equiv \log \frac{d P_{\theta_N(h)}^{(N)}}{d P_{\theta_0}^{(N)}}(X^N) = \sum_{n=1}^N \left[ \log p\left(X_n; \theta_0 + \frac{P_{\theta_0}^H h}{\sqrt{N}}\right) - \log p(X_n; \theta_0) \right]$$ \cite{lecam1986}.
	By Theorem~\ref{thm:generalized_horizontal_log_likelihood} (Generalized Horizontal Log-Likelihood Representation) and Lemma~\ref{lem:pillar_2_taylor_geometry} (Taylor-Geometry Identity), $\Lambda_N(h)$ admits the exact algebraic decomposition:
	\begin{equation}
		\Lambda_N(h) = h^T \Delta_{N,H}(\theta_0) - \frac{1}{2} h^T \left[ \psi_{N,H}^* G^{(N)} \right](0) h + \mathcal{R}_{N,H}(h),
	\end{equation}
	where:
	\begin{enumerate}
		\item $\Delta_{N,H}(\theta_0) \equiv P_{\theta_0}^H \Delta_N(\theta_0) = \frac{1}{\sqrt{N}} \sum_{n=1}^N P_{\theta_0}^H \nabla_\theta \log p(X_n; \theta_0)$ is the horizontally projected sample score vector (Theorem~\ref{thm:generalized_horizontal_log_likelihood}).
		\item $\left[ \psi_{N,H}^* G^{(N)} \right](0) \equiv P_{\theta_0}^H g^{(1)}(\theta_0) P_{\theta_0}^H = g_0|_{\mathcal{B}}$ is the horizontally projected pulled-back Fisher metric (Lemma~\ref{lem:horizontal_pulled_back_metric_algebraic}).
		\item $\mathcal{R}_{N,H}(h)$ is the non-linear stochastic remainder process (Theorem~\ref{thm:generalized_horizontal_log_likelihood}).
	\end{enumerate}
	
	\subsubsection*{Stage 3: Asymptotic Score Normality and Remainder Uniform Control}
	We evaluate the limiting behavior of the components in the expansion:
	\begin{enumerate}
		\item \textbf{Horizontally Projected Score Normality:} The single-observation horizontally projected score vectors $P_{\theta_0}^H S_n$ are independent and identically distributed with zero mean $\mathbb{E}_{\theta_0}[P_{\theta_0}^H S_n] = \mathbf{0}$ and covariance matrix $\mathrm{Var}_{\theta_0}(P_{\theta_0}^H S_n) = P_{\theta_0}^H g^{(1)}(\theta_0) P_{\theta_0}^H = g_0|_{\mathcal{B}}$ (Theorem~\ref{thm:score_transformation}). By the Multivariate Central Limit Theorem, the normalized sum converges weakly \cite{vandervaart1998}:
		\begin{equation}
			\Delta_{N,H}(\theta_0) \xrightarrow{d} \Delta_{\infty,H} \sim \mathcal{N}\left(\mathbf{0}, g_0|_{\mathcal{B}}\right).
		\end{equation}
		
		\item \textbf{Uniform Metric Stabilization:} By Lemma~\ref{lem:uniform_horizontal_metric_stabilization}, the horizontally projected pulled-back metric tensor stabilizes uniformly in operator norm over the compact cage $\mathbb{K}$:
		\begin{equation}
			\lim_{N \to \infty} \sup_{h \in \mathbb{K}} \left\| \left[ \psi_{N,H}^* G^{(N)} \right](h) - g_0|_{\mathcal{B}} \right\|_{\mathrm{op}} = 0.
		\end{equation}
		
		\item \textbf{Uniform Vanishing of Stochastic Remainder:} By Lemma~\ref{lem:pillar_2_taylor_geometry}, the remainder process is bounded by the connection dissolution and curvature annihilation limits on $H_{\theta_0}$:
		\begin{equation}
			\sup_{h \in \mathbb{K}} |\mathcal{R}_{N,H}(h)| \le \frac{\|h\|_2^3}{6} \sup_{h \in \mathbb{K}} \left| \left[ \psi_{N,H}^* \nabla^{(\alpha,N)} \right]_{ijk}(h) \right| + \mathcal{O}_p\left( \sup_{h \in \mathbb{K}} \left\| \left[ \psi_{N,H}^* \mathcal{R}^{(\alpha,N)} \right]_{ijmn}(h) \right\|_{\mathrm{op}} \right).
		\end{equation}
		By Lemma~\ref{lem:connection_dissolution} and Lemma~\ref{lem:accelerated_curvature_annihilation}, the pulled-back connection symbols dissolve at rate $\mathcal{O}(N^{-1/2})$ and the pulled-back Riemann curvature undergoes accelerated quadratic decay at rate $\mathcal{O}_p(N^{-1})$. Thus:
		\begin{equation}
			\sup_{h \in \mathbb{K}} |\mathcal{R}_{N,H}(h)| = \mathcal{O}_p\left(N^{-1/2}\right) \xrightarrow{P_{\theta_0}^{(N)}} 0 \quad \text{as } N \to \infty.
		\end{equation}
	\end{enumerate}
	
	Combining these limits yields the uniform linear-quadratic Local Asymptotic Normality (LAN) expansion over $\mathbb{K} \subset H_{\theta_0}$ \cite{lecam1986}:
	\begin{equation}
		\Lambda_N(h) = h^T \Delta_{N,H}(\theta_0) - \frac{1}{2} h^T \left( g_0|_{\mathcal{B}} \right) h + o_{P_{\theta_0}^{(N)}}(1), \quad \Delta_{N,H}(\theta_0) \xrightarrow{d} \mathcal{N}\left(\mathbf{0}, g_0|_{\mathcal{B}}\right).
	\end{equation}
	
	\subsubsection*{Stage 4: Strict Horizontal Positivity and Le Cam Deficiency Convergence}
	By Theorem~\ref{thm:horizontal_strict_positivity} (Strict Positivity and Non-Singularity of Restricted Metric), the restricted base metric tensor $g_0|_{\mathcal{B}} \equiv g^{(1)}(\theta_0)|_{H_{\theta_0} \times H_{\theta_0}}$ is strictly positive-definite and non-singular on $H_{\theta_0}$, satisfying $\lambda_{\min}(g_0|_{\mathcal{B}}) \ge \kappa_0 > 0$. This eliminates the ambient metric singularity ($\det g^{(1)}(\theta_0) = 0$).
	
	Applying Lemma~\ref{lem:lecam_deficiency_equivalence} (Le Cam Experiment Deficiency Limit Lemma) to the uniform LAN expansion on $H_{\theta_0}$ confirms mutual contiguity $P_{\theta_N(h)}^{(N)} \triangleleft \triangleright P_{\theta_0}^{(N)}$ and establishes that Le Cam's experiment deficiency distance between the horizontally pulled-back empirical experiment $\psi_{N,H}^* \mathcal{E}_N(\mathbb{K})$ and the canonical Gaussian experiment $\mathcal{E}_\infty(\mathcal{B}; \mathbb{K})$ vanishes asymptotically \cite{lecam1986, lecam2000}:
	\begin{equation}
		\lim_{N \to \infty} \Delta\left( \psi_{N,H}^* \mathcal{E}_N(\mathbb{K}), \, \mathcal{E}_\infty(\mathcal{B}; \mathbb{K}) \right) = 0.
	\end{equation}
	This completes the proof of Corollary~\ref{cor:gauge_invariant_lan}.
\end{proof}

\subsubsection{Resolution of Non-Identifiability Pathologies}
\label{subsubsec:pathology_elimination_table}

 Table~\ref{tab:pathology_resolution} details how the Fisher-compatible Ehresmann connection framework systematically resolves the three classic pathologies of over-parameterized statistics.

\begin{table}[h!]
\centering
\small
\begin{tabular}{|p{3.2cm}|p{5.8cm}|p{5.8cm}|}
\hline
\textbf{Pathology Category} & \textbf{Ambient Parameter Space Pathology ($\Theta \subset \mathbb{R}^D$)} & \textbf{Horizontal Ehresmann Resolution ($H_\theta \cong T\mathcal{B}$)} \\ \hline
\textbf{1. Metric Singularity \& Inverse Breakdown} & $\det g^{(1)}(\theta) = 0$ with $\mathrm{rank}(g^{(1)}) = d < D$. Matrix inverse $(g^{(1)})^{-1}$ fails to exist, blowing up Christoffel symbols and Riemann curvature tensor contractions. & $g_0\big|_{\mathcal{B}}$ is strictly positive-definite ($\lambda_{\min}\left(g_0\big|_{\mathcal{B}}\right) \ge \kappa_0 > 0$). Metric matrix inverse is well-defined and non-singular on $H_\theta$. \\ \hline
\textbf{2. Variance Divergence \& CR-Bound Collapse} & Cramér-Rao lower bound along vertical gauge directions diverges to infinity: $N \cdot \mathrm{Var}_{\theta_0}(v^T \hat{\theta}_N) \ge v^T (g^{(1)})^+ v = +\infty$ for $v \in V_{\theta_0}$. & For any physical/identifiable decision rule $\hat{g}(\cdot)$, operational estimation variance is bounded strictly by the horizontal inverse: $\mathrm{Var}(P^H \hat{\theta}_N) \le \frac{1}{N} \left(g_0\big|_{\mathcal{B}}\right)^{-1} + o(N^{-1})$. \\ \hline
\textbf{3. Non-Hausdorff Distance \& Topological Collapse} & Geodesic distance between distinct parameters $\theta_1 \neq \theta_2$ on the same gauge fiber vanishes ($d_g(\theta_1, \theta_2) = 0$), destroying Hausdorff topology. & The horizontal leaf $\mathcal{L}_{\theta_0} \cong \mathcal{B}$ inherits a non-degenerate, Hausdorff Riemannian metric $g\big|_{\mathcal{B}}$, restoring well-posed geometric topologies. \\ \hline
\end{tabular}
\caption{Systematic Elimination of Non-Identifiability Pathologies under Horizontal Splitting.}
\label{tab:pathology_resolution}
\end{table}
\newpage
\subsubsection{Logical Bridge to the Global Second Edge Phase Transition}
\label{subsubsec:logical_bridge_section6}

\begin{remark}[Logical Bridge to Global Geometric Collapse]
\label{rem:logical_bridge_global_collapse}
Theorem~11 and Corollary~\ref{cor:gauge_invariant_lan} provide the indispensable analytical foundation required to execute the complete global proof of the \textbf{Second Edge Theorem} ($IG_N \to \mathcal{M}_\infty \equiv \mathrm{CS}$). 

By filtering out the non-identifiable gauge modes along $V_\theta$ and establishing metric non-degeneracy strictly on $H_\theta$, the Ehresmann connection framework guarantees that the five global stages of geometric phase transition remain mathematically well-posed:
\begin{enumerate}
\item \textbf{Gromov-Hausdorff Leaf Collapse:} The sequence of horizontal leaves $\mathcal{L}_{N, \theta_0}$ converges in the Gromov-Hausdorff metric to the flat Euclidean space $\mathbb{R}^d$.
\item \textbf{Gauge Fibre Contraction:} The vertical gauge orbits $\mathcal{F}_p$ contract smoothly into equivalence classes, eliminating internal parameter drift.
\item \textbf{Accelerated Curvature Annihilation:} The pulled-back horizontal Riemann curvature tensor vanishes at rate $\mathcal{O}_p(N^{-1})$:
\begin{equation}
\sup_{h \in \mathbb{K}} \left\| \left[ \psi_{N,H}^* \mathcal{R}^{(\alpha,N)} \right](h) \right\|_{\mathrm{op}} = \mathcal{O}_p\left(\frac{1}{N}\right) \longrightarrow 0.
\end{equation}
\item \textbf{LAN Uniformization:} Localized log-likelihood ratio processes converge uniformly over compact displacement cages $\mathbb{K}$ to the linear-quadratic Gaussian experiment $\mathcal{E}_\infty(\mathcal{B}; \mathbb{K})$.
\item \textbf{Global Edge Collapse:} The curved finite-sample joint information manifold $IG_N$ collapses globally onto the flat canonical tangent canvas $\mathcal{M}_\infty \equiv \mathrm{CS}$ as $N \to \infty$.
\end{enumerate}
\end{remark}

\section{The Second Edge Theorem ($IG_N \to \mathcal{M}_\infty \equiv \text{CS}$): {\small Global Asymptotic Phase Transition, Topological Collapse, and Operational Unification}}
\label{sec:module7_second_edge_theorem}

\subsection{Architectural Challenges and the Necessity of Advanced Geometric Machinery}
\label{subsec:mod7_challenges_machinery}

To establish the Second Edge Theorem—the asymptotic phase transition and global geometric collapse of the sequence of $N$-sample joint information-geometric manifolds $IG_N = (\Theta, G^{(N)}, \nabla^{(\alpha,N)})$ onto the canonical flat tangent canvas of Conventional Statistics $\mathcal{M}_\infty \equiv (T_{\theta_0}IG_1, g_0, \nabla^{(0)}) \equiv \text{CS}$ as $N \to \infty$—a formidable array of mathematical instruments is strictly required. Why do we need such an aggressive differential-topological arsenal, encompassing Ehresmann connections, Gromov-Hausdorff limits, and holonomy group dissolution? Because parametric statistics in finite-sample regimes ($N < \infty$) is not a static Euclidean canvas. Instead, as established in Section~\ref{sec:module1} and Section~\ref{sec:module2}, finite-sample statistical parametric families are intrinsically curved, dynamically shifting, and often structurally ambiguous Riemannian-affine manifolds \cite{amari1985, rao1945}.

In finite samples, the information manifold sequence $IG_N = (\Theta, G^{(N)}, \nabla^{(\alpha,N)})$, defined in Definition~\ref{def:ambient_spaces}, exhibits non-zero Amari-Riemann curvature $\mathcal{R}^{(\alpha,N)}$ (Theorem~\ref{thm:curvature_annihilation}), non-Euclidean affine transport friction $\Gamma^{(\alpha,N)}$ (Theorem~\ref{thm:connection_dissolution}), and severe metric degeneracies $\det G^{(N)} = 0$ whenever over-parameterization induces non-identifiable gauge symmetries (Section~\ref{sec:ehresmann_connections}, Lemma~\ref{lem:ambient_breakdown}). Proving the Second Edge Theorem requires showing how these non-Euclidean pathologies systematically collapse as $N \to \infty$. Below is a rigorous mathematical breakdown of why each advanced geometric instrument is indispensable.

\subsubsection*{1. Resolving the Moving Target Crisis: The Localized Microscope Map}
\label{subsubsec:mod7_challenge_moving_target}

Under $i.i.d.$ sampling from a dominated family $\mathcal{P}$, direct asymptotic evaluation of $IG_N = (\Theta, G^{(N)}, \nabla^{(\alpha,N)})$ encounters a severe dual-faceted analytical barrier known as the \textbf{Moving Target Crisis} (Section~\ref{subsec:mod2_moving_target}):
\begin{enumerate}[label=(\roman*)]
	\item \textbf{Domain Collapse (Microscopic Signal Shrinkage):} By the Law of Large Numbers and the Central Limit Theorem, active sample probability mass concentrates around the background parameter $\theta_0 \in \text{Int}(\Theta)$ into a rapidly contracting coordinate neighborhood $\mathcal{U}_N(\theta_0) \subset \Theta$ whose geometric diameter shrinks at rate $\mathcal{O}_p(N^{-1/2})$. As $N \to \infty$, the active parameter domain contracts topologically to a single singular point $\{\theta_0\}$.
	\item \textbf{Metric Divergence (Macro-Capacity Explosion):} By Lemma~\ref{lem:additivity_fisher_metric}, the $N$-sample joint Fisher Information Metric tensor field satisfies $G^{(N)}(\theta) = N \cdot g^{(1)}(\theta)$. Consequently, every component of the Fisher metric tensor diverges to infinity across the domain at linear rate $\mathcal{O}(N)$:
	\begin{equation}
		\lim_{N \to \infty} G_{ab}^{(N)}(\theta) = \infty, \quad \forall \theta \in \mathcal{U}_N(\theta_0).
		\label{eq:sec7_metric_divergence}
	\end{equation}
\end{enumerate}

Attempting to take a standard pointwise limit $\lim_{N \to \infty} G^{(N)}(\theta)$ directly on the fixed parameter domain $\Theta$ fails catastrophically because the subject blurs out and contracts to a point while the metric capacity explodes.

To decouple the stationary observer tracking canvas $\mathcal{M}_\infty \equiv (\mathbb{R}^d, g_0)$ from the dynamic space $IG_N$, we construct the localized dilation microscope map $\psi_N: \mathbb{K} \to \Theta$ (Definition~\ref{def:microscope_map}):
\begin{equation}
	\psi_N(h) \equiv \theta_0 + \frac{h}{\sqrt{N}} \in \Theta, \quad \forall h \in \mathbb{K} \subset \mathbb{R}^d.
	\label{eq:sec7_microscope_map_def}
\end{equation}
The multivariable spatial Jacobian matrix scaffold $J(\psi_N) \in \mathbb{R}^{d \times d}$ acts as an isotropic linear scaling transformation (Lemma~\ref{lem:jacobian_scaffold}):
\begin{equation}
	J_i^a(h) \equiv \frac{\partial \psi_N^a(h)}{\partial h^i} = \frac{1}{\sqrt{N}} \delta_i^a.
	\label{eq:sec7_jacobian_scaffold}
\end{equation}

By Theorem~\ref{thm:universal_scaling_law}, any $r$-covariant $i.i.d.$ additive tensor field obeys the Universal Tensor Valence Scaling Law $N^{1 - r/2}$. Setting $r=2$ for the Fisher metric tensor $G^{(N)}$, the capacity expansion factor $N^1$ cancels identically against the spatial cotangent contraction factor $(\frac{1}{\sqrt{N}})^2 = N^{-1}$, yielding exact scale factor stabilization ($N^{1 - 2/2} = N^0 = 1$):
\begin{equation}
	\widetilde{G}^{(N)}(h) \equiv \left[ \psi_N^* G^{(N)} \right](h) = J(\psi_N)^T G^{(N)}(\psi_N(h)) J(\psi_N) = g^{(1)}\left(\theta_0 + \frac{h}{\sqrt{N}}\right).
	\label{eq:sec7_metric_stabilization_cancellation}
\end{equation}
By Lemma~\ref{lem:metric_freezing_proof} and Theorem~\ref{thm:metric_stabilization}, $\widetilde{G}^{(N)}(h)$ converges uniformly over compact coordinate cages $\mathbb{K}$ to the constant frozen metric $g_0 \equiv g^{(1)}(\theta_0)$ at rate $\mathcal{O}(N^{-1/2})$. The intrinsic Riemannian geodesic distance under $\psi_N^* G^{(N)}$ satisfies:
\begin{equation}
	d_{\psi_N^* G^{(N)}}(\mathbf{0}, h) \equiv d_{G^{(N)}}\left(\theta_0, \theta_0 + \frac{h}{\sqrt{N}}\right) = \sqrt{h^T g_0 h} + o(1) = \mathcal{O}(1),
	\label{eq:sec7_snr_constancy}
\end{equation}
maintaining a constant, non-degenerate Signal-to-Noise Ratio ($\text{SNR}_N(h) = \mathcal{O}(1)$) across all $N$ (Proposition~\ref{prop:resolution_mechanics}) \cite{lecam1986, vandervaart1998}.

\subsubsection*{2. Taming Non-Identifiability: The Fisher-Compatible Ehresmann Connection}
\label{subsubsec:mod7_challenge_ehresmann}

In modern over-parameterized statistical models, deep architectures, and latent variable families, the ambient parameter space $\Theta \subset \mathbb{R}^D$ has dimension $D \gg d = \dim(\mathcal{B})$, where $\mathcal{B} \equiv \Theta/\sim$ denotes the quotient manifold of observationally identifiable probability distributions (Definition \ref{def:statistical_fiber_bundle} in Section 5). Distinct parameter vectors $\theta_1 \neq \theta_2$ can map to the exact same probability distribution $P_{\theta_1} = P_{\theta_2} \in \mathcal{B}$.

Consider an intuitive analogy: the authentication of antique imperial porcelain. Two distinct vases may exhibit identical macroscopic surface features (the observable probability distribution $P_\theta \in \mathcal{B}$). However, the underlying physical kiln temperatures, cooling durations, and microscopic atmospheric fluctuations represent an unobservable, high-dimensional firing history (the vertical gauge fiber $\mathcal{F}_p = \pi^{-1}(p)$). If we attempt to measure statistical distance directly in the ambient space $\mathbb{R}^D$ without filtering out these internal firing variations, the score function along fiber directions vanishes identically (Lemma  \ref{lem:vertical_score_vanishing}):
\begin{equation}
	v\left(\log p(X; \theta_0)\right) \equiv \sum_{i=1}^D v^i \frac{\partial \log p(X; \theta_0)}{\partial \theta^i} = 0 \quad P_{\theta_0}\text{-a.e.}, \quad \forall v \in V_{\theta_0} \equiv \ker(d\pi_{\theta_0}).
	\label{eq:sec7_score_vanishing}
\end{equation}
As proved in Lemma~\ref{lem:ambient_breakdown}, this causes the ambient Fisher Information Metric $g^{(1)}(\theta_0)$ to become rank-deficient ($\det g^{(1)}(\theta_0) = 0$). Consequently, the operational estimation variance along unidentifiable gauge directions diverges to infinity:
\begin{equation}
	\lim_{N \to \infty} N \cdot \mathrm{Var}_{\theta_0}\left( v^T \hat{\theta}_N \right) = +\infty,
	\label{eq:sec7_variance_explosion}
\end{equation}
while the ambient geometric Fisher length $\|v\|_{g^{(1)}} = 0$, destroying the operational-geometric equivalence ratio of Section~\ref{sec:module1}.

The \textbf{Fisher-Compatible Ehresmann Connection} $(\Theta, \mathcal{B}, \pi, \mathcal{F})$ provides the exact differential-geometric solution (Definition \ref{def:ehresmann_connection} in Section 5) \cite{frankel2011, lee2013}. It constructs a smooth, canonically invariant splitting of the ambient tangent bundle:
\begin{equation}
	T_\theta \Theta = H_\theta \oplus V_\theta, \quad \text{where } V_\theta \equiv \ker(d\pi_\theta) \text{ and } H_\theta \equiv V_\theta^{\perp_g} = \{x \in T_\theta \Theta : g_\theta^{(1)}(x, v) = 0, \, \forall v \in V_\theta\}.
	\label{eq:sec7_tangent_splitting}
\end{equation}
By applying the horizontal projection operator $P_\theta^H \equiv \mathrm{id}_{T_\theta \Theta} - \omega_\theta$ (Definition  \ref{def:connection_1form_projection} in Section 5), the restricted differential map $d\pi_\theta\big|_{H_\theta}: H_\theta \to T_{\pi(\theta)} \mathcal{B}$ becomes a linear vector space isomorphism (Lemma \ref{lem:horizontal_isomorphism} in Section 5). Under 
 	\ref{thm:horizontal_strict_positivity}  (Section 5), the restricted Fisher metric tensor $g_\theta^{(1)}\big|_{H_\theta \times H_\theta}$ is strictly positive-definite and non-singular across the horizontal sub-bundle, satisfying:
\begin{equation}
	\lambda_{\min}\left( g_\theta^{(1)}\big|_{H_\theta \times H_\theta} \right) \ge \kappa_0 > 0, \quad \forall \theta \in \mathcal{U}(\theta_0).
	\label{eq:sec7_strict_horizontal_positivity}
\end{equation}
Filtering out the unidentifiable vertical gauge fluctuations restores metric non-degeneracy, allowing estimation risk bounds and geometric lengths to remain well-posed \cite{stoica2001}.

\subsubsection*{3. Eradicating Geometric Friction: Connection Dissolution and Accelerated Curvature Annihilation}
\label{subsubsec:mod7_challenge_friction}

Consider another physical intuition: maintaining a rigorous long-distance cycling routine. If you cycle a closed loop over a highly curved, mountainous terrain, the tilt and spatial orientation of your bicycle will shift by the time you return to the starting point. In differential geometry, this path-dependent orientation shift along a closed loop is quantified by the \textbf{parallel transport holonomy group} $\mathrm{Hol}(\nabla)$ \cite{frankel2011, lee2013}.

Finite-sample statistical parametric families suffer from this exact same "terrain friction" via non-zero Amari Christoffel connection symbols $\Gamma_{ijk}^{(\alpha,N)}$ and non-zero Amari-Riemann curvature $\mathcal{R}_{ijmn}^{(\alpha,N)}$ \cite{amari1985, efron1975}. As established in Section~\ref{subsubsec:cs_stat_definition} (Equations~\eqref{eq:test_diff_1}--\eqref{eq:test_diff_2}), this geometric friction causes the three classical hypothesis test statistics—Rao's Score Test ($W_S$), Wilks' Likelihood Ratio Test ($W_{LR}$), and Wald's Test ($W_W$)—to diverge in finite samples:
\begin{align}
	W_{LR}(h) - W_S(h) &= \frac{1}{3\sqrt{N}} \sum_{i,j,k=1}^d \Gamma_{ijk}^{(1)}(\theta_0) h^i h^j h^k + o_p\left(N^{-1/2}\right), \label{eq:sec7_test_diff_1} \\
	W_W(h) - W_{LR}(h) &= \frac{1}{3\sqrt{N}} \sum_{i,j,k=1}^d \Gamma_{ijk}^{(-1)}(\theta_0) h^i h^j h^k + o_p\left(N^{-1/2}\right). \label{eq:sec7_test_diff_2}
\end{align}

Proving the Second Edge Theorem requires demonstrating that as $N \to \infty$, this non-Euclidean terrain friction is completely eradicated:
\begin{enumerate}[label=(\roman*)]
	\item \textbf{Amari Connection Dissolution ($r=3$, Scaling $N^{-1/2}$):} By Theorem~\ref{thm:connection_dissolution} and Lemma \ref{lem:connection_dissolution} (Section \ref{sec:equivalence_proof}), the pulled-back connection symbols scale as $N^{1 - 3/2} = N^{-1/2}$:
	\begin{equation}
		\sup_{h \in \mathbb{K}} \left| \widetilde{\Gamma}_{ijk}^{(\alpha, N)}(h) \right| = \frac{1}{\sqrt{N}} \sup_{\theta \in \mathcal{U}(\theta_0)} \left| \Gamma_{ijk}^{(\alpha, 1)}(\theta) \right| = \mathcal{O}\left(N^{-1/2}\right) \longrightarrow 0 \equiv \Gamma_{ijk}^{(0)}.
		\label{eq:sec7_connection_dissolution_proof}
	\end{equation}
	
	\item \textbf{Accelerated Quadratic Curvature Annihilation ($r=4$, Scaling $N^{-1}$):} By Theorem~\ref{thm:curvature_annihilation} and Lemma \ref{lem:accelerated_curvature_annihilation} (Section \ref{sec:equivalence_proof}), the pulled-back Riemann curvature tensor scales as $N^{1 - 4/2} = N^{-1}$:
	\begin{equation}
		\sup_{h \in \mathbb{K}} \left\| \widetilde{\mathcal{R}}_{ijmn}^{(\alpha, N)}(h) \right\|_\infty = \frac{1}{N} \sup_{\theta \in \mathcal{U}(\theta_0)} \left\| \mathcal{R}_{ijmn}^{(\alpha, 1)}(\theta) \right\|_\infty = \mathcal{O}_p\left(N^{-1}\right) \longrightarrow 0 \equiv \mathcal{R}_{ijmn}^{(0)}.
		\label{eq:sec7_curvature_annihilation_proof}
	\end{equation}
\end{enumerate}

Because Riemann curvature undergoes accelerated quadratic decay ($\mathcal{O}_p(N^{-1})$)—decaying twice as fast as connection dissolution ($\mathcal{O}(N^{-1/2})$)—the holonomy group collapses to the trivial identity group ($\mathrm{Hol}(\nabla^{(\alpha,N)}) \to \{\mathrm{Id}\}$). Parallel transport becomes path-independent, flattening the curved statistical manifold into a perfectly smooth, zero-friction Euclidean canvas and guaranteeing the unconditional asymptotic invariance of the test trinity ($W_S = W_{LR} = W_W + o_p(1)$) \cite{kass1997}.

\subsubsection*{4. Bridging Two Worlds: Geometric-Operational Unification}
\label{subsubsec:mod7_challenge_unification}

Finally, there exist two historically divergent paradigms for evaluating the convergence $IG_N \to \mathcal{M}_\infty$:
\begin{itemize}
	\item \textbf{The Differential Geometer's View (Way 1, Geometric Priority):} Focuses on intrinsic spatial flattening, defining convergence via the Cheeger-Gromov $C^\infty$-manifold topology ($\psi_N^* IG_N \xrightarrow{C^\infty} \mathcal{M}_\infty$, Axiom~\ref{ax:geom_priority}, Proposition~\ref{prop:pillar_1_completeness}) \cite{cheeger1970, gromov2007}.
	\item \textbf{The Operational Statistician's View (Way 2, Statistical Operationalism):} Focuses strictly on operational decision performance, defining convergence via the vanishing of Le Cam's experiment deficiency distance ($\lim_{N \to \infty} \Delta(\mathcal{E}_N(\mathbb{K}), \mathcal{E}_\infty(\mathbb{K})) = 0$, Axiom~\ref{ax:stat_oper}, Definition~\ref{def:localized_experiments}) \cite{lecam1960, lecam1986}.
\end{itemize}

Complex differential topology and measure theory are required to construct the \textbf{Double Completeness Architecture} (Section~\ref{subsec:mod3_double_completeness}, Remark \ref{rem:double_completeness_architecture}), proving that measuring structural manifold geometry and evaluating operational decision risk are mathematically identical actions.

By Lemma~\ref{lem:pillar_2_taylor_geometry} and Theorem \ref{thm:generalized_horizontal_log_likelihood} (Section \ref{sec:module5_unification}), the localized Radon-Nikodym log-likelihood ratio process $\Lambda_N(h)$ admits an exact algebraic functional representation built exclusively from the four pulled-back geometric tensor fields ($r \in \{1, 2, 3, 4\}$):
\begin{equation}
	\Lambda_N(h) = \sum_{a=1}^d h^a \left[ \psi_N^*(dL^{(N)}) \right]_a(\mathbf{0}) - \frac{1}{2} \sum_{i,j=1}^d h^i h^j \left[ \psi_N^* G^{(N)} \right]_{ij}(\mathbf{0}) + \mathcal{R}_N(h),
	\label{eq:sec7_taylor_geometry_identity}
\end{equation}
where the non-linear stochastic remainder process $\mathcal{R}_N(h)$ is strictly bounded by connection dissolution $\widetilde{\Gamma}^{(\alpha,N)}$ and curvature annihilation $\widetilde{\mathcal{R}}^{(\alpha,N)}$:
\begin{equation}
	\sup_{h \in \mathbb{K}} |\mathcal{R}_N(h)| \le \frac{\|h\|_2^3}{6} \sup_{h \in \mathbb{K}} \left| \widetilde{\Gamma}_{ijk}^{(\alpha,N)}(h) \right| + \mathcal{O}_p\left( \sup_{h \in \mathbb{K}} \left\| \widetilde{\mathcal{R}}_{ijmn}^{(\alpha,N)}(h) \right\|_\infty \right) = \mathcal{O}_p\left(N^{-1/2}\right) \xrightarrow{P_{\theta_0}} 0.
	\label{eq:sec7_remainder_bound}
\end{equation}

As proved in Theorem~\ref{thm:geom_oper_equivalence} (Section \ref{sec:equivalence_proof}) and Theorem \ref{thm:geometric_operational_unification_horizontal} (Section \ref{sec:module5_unification}), this algebraic identity locks differential geometry and decision theory together, establishing the full Geometric-Operational Equivalence Theorem:
\begin{equation}
	IG_N \xrightarrow{\text{geom}} \mathcal{M}_\infty \quad \Longleftrightarrow \quad IG_N \xrightarrow{\text{oper}} \mathcal{M}_\infty \equiv \text{CS}.
	\label{eq:sec7_unification_equivalence}
\end{equation}
Thus, structural manifold flattening and decision risk condensation are non-separable dual manifestations of the exact same asymptotic phase transition.

\subsection{Architectural Blueprint and Strategic Roadmap of the Global Second Edge Phase Transition}
\label{subsec:roadmap_second_edge}

In Section~\ref{sec:module1}, we established the Epistemological Equivalence Theorem ($\text{CS}_{\text{stat}} \iff \text{CS}_{\text{geom}} \iff \mathcal{M}_\infty$, Theorem~\ref{thm:theorem2_full}), proving that Conventional Statistics ($\text{CS}$) is mathematically identical to the canonical flat tangent canvas $\mathcal{M}_\infty \equiv (T_{\theta_0} IG_1, g_0, \nabla^{(0)})$. In Section~\ref{sec:module2}, we constructed the localized microscope map $\psi_N(h) = \theta_0 + \frac{h}{\sqrt{N}}$ (Definition~\ref{def:microscope_map}), derived the multivariable Jacobian scaffold $J(\psi_N)$ (Lemma~\ref{lem:jacobian_scaffold}), and proved the Universal Tensor Valence Scaling Law $N^{1-r/2}$ (Theorem~\ref{thm:universal_scaling_law}). In Section~\ref{sec:module3} and Section~\ref{sec:equivalence_proof}, we axiomatized the dual phase transition paradigms (Axiom~\ref{ax:geom_priority} and Axiom~\ref{ax:stat_oper}) and proved the Double Completeness Architecture alongside the Geometric-Operational Equivalence Theorem (Theorem~\ref{thm:geom_oper_equivalence}). Finally, in Section~\ref{sec:ehresmann_connections} and Section~\ref{sec:module5_unification}, we constructed the Fisher-Compatible Ehresmann Connection framework ($T_\theta \Theta = H_\theta \oplus V_\theta$, Definition~\ref{def:ehresmann_connection}) to filter out non-identifiable gauge modes ($V_\theta = \ker(d\pi_\theta)$) and proved the Geometric-Operational Unification Theorem on the non-singular horizontal distribution $H_\theta$ (Theorem~\ref{thm:geometric_operational_unification_horizontal}, Corollary~\ref{cor:gauge_invariant_lan}).

Building upon these foundational results, this section executes the global proof of the Second Edge Theorem: {\it  the asymptotic phase transition, topological collapse, and operational unification of the sequence of $N$-sample joint information-geometric manifolds $IG_N = (\Theta, G^{(N)}, \nabla^{(\alpha,N)})$ onto the flat canonical tangent canvas of Conventional Statistics $\mathcal{M}_\infty \equiv \text{CS}$ as $N \to \infty$.}

To maintain mathematical rigor, the global transition $IG_N \to \mathcal{M}_\infty$ is established through a five-stage strategic roadmap:
\begin{description}[leftmargin=*]
	\item[\textbf{Stage I: Gromov-Hausdorff Leaf Collapse (Section~\ref{subsec:gromov_hausdorff_collapse}):}] We prove that the sequence of Riemannian horizontal leaf manifolds $(\mathcal{L}_{N,\theta_0}, \widetilde{G}_H^{(N)})$ converges in the Gromov-Hausdorff metric topology to the flat Euclidean space $(\mathcal{M}_\infty, d_{g_0})$ at rate $\mathcal{O}(N^{-1/2})$.
	\item[\textbf{Stage II: Accelerated Holonomy Annihilation (Section~\ref{subsec:holonomy_annihilation}):}] We prove that the restricted holonomy group $\mathrm{Hol}(\widetilde{\nabla}^{(\alpha,N)})$ collapses to the trivial identity group $\{\mathbf{I}_{d \times d}\}$ at the accelerated quadratic rate $\mathcal{O}_p(N^{-1})$, eradicating all non-Euclidean parallel transport friction.
	\item[\textbf{Stage III: Global LAN Uniformization (Section~\ref{subsec:global_lan_uniformization}):}] We prove that the localized log-likelihood ratio process $\Lambda_N(h)$ converges uniformly in probability over compact coordinate cages to the linear-quadratic Gaussian experiment under Le Cam deficiency distance $\Delta(\mathcal{E}_N, \mathcal{E}_\infty) \to 0$.
	\item[\textbf{Stage IV: The Master Second Edge Theorem (Section~\ref{subsec:master_second_edge_theorem}):}] We synthesize Stages I--III to state and prove the primary Master Second Edge Theorem, formalizing the complete asymptotic collapse $IG_N \to \mathcal{M}_\infty \equiv \text{CS}$.
	\item[\textbf{Stage V: Epistemological Synthesis and Gauge Contraction (Section~\ref{subsec:epistemological_synthesis_master}):}] We demonstrate how vertical gauge fibers $\mathcal{F}_p = \pi^{-1}(p)$ contract topologically into singular points, eliminating over-parameterization pathologies and unifying differential geometry with asymptotic decision theory.
\end{description}

\subsection{Gromov-Hausdorff Leaf Collapse and Metric-Topological Shrinkage}
\label{subsec:gromov_hausdorff_collapse}

We begin by formalizing the metric-topological convergence of the horizontal leaf manifolds $(\mathcal{L}_{N,\theta_0}, \widetilde{G}_H^{(N)})$ to the canonical flat tangent space $\mathcal{M}_\infty \equiv (T_{\theta_0} IG_1, g_0|_{\mathcal{B}}, \nabla^{(0)})$.

\begin{definition}[Gromov-Hausdorff Distance for Information Leaf Spaces]
	\label{def:gromov_hausdorff_distance}
	Let $(\mathcal{X}, d_{\mathcal{X}})$ and $(\mathcal{Y}, d_{\mathcal{Y}})$ be two compact metric spaces. The Gromov-Hausdorff distance $d_{GH}\left((\mathcal{X}, d_{\mathcal{X}}), (\mathcal{Y}, d_{\mathcal{Y}})\right)$ is defined as the infimum of the Hausdorff distance $d_H^{\mathcal{Z}}(i(\mathcal{X}), j(\mathcal{Y}))$ over all isometric embeddings $i: \mathcal{X} \to \mathcal{Z}$ and $j: \mathcal{Y} \to \mathcal{Z}$ into a common metric space $(\mathcal{Z}, d_{\mathcal{Z}})$:
	\begin{equation}
		d_{GH}\left((\mathcal{X}, d_{\mathcal{X}}), (\mathcal{Y}, d_{\mathcal{Y}})\right) \equiv \inf_{\mathcal{Z}, i, j} d_H^{\mathcal{Z}}\left(i(\mathcal{X}), j(\mathcal{Y})\right),
	\end{equation}
	where $d_H^{\mathcal{Z}}(A, B) \equiv \max\left\{ \sup_{a \in A} \inf_{b \in B} d_{\mathcal{Z}}(a,b), \, \sup_{b \in B} \inf_{a \in A} d_{\mathcal{Z}}(a,b) \right\}$.
\end{definition}

\begin{lemma}[Flat Euclidean-Fisher Distance Characterization]
	\label{lem:horizontal_pulled_back_metric_algebraic}
	Let $\mathbb{K} \subset H_{\theta_0} \cong \mathbb{R}^d$ be a compact coordinate domain, and let $g_0|_{\mathcal{B}}$ denote the restriction of the Fisher information metric $g_0$ to the horizontal subspace $\mathcal{B} \subset T_{\theta_0}\mathcal{M}$ \cite{amari2000, kullback1951}, evaluated as a constant positive-definite matrix field across $\mathbb{K}$. Then, the intrinsic Riemannian geodesic distance $d_{g_0}(h_1, h_2)$ between any points $h_1, h_2 \in \mathbb{K}$ on $\mathcal{M}_\infty$ is given by the flat Euclidean-Fisher distance:
	\begin{equation}
		\label{eq:flat_fisher_dist}
		d_{g_0}(h_1, h_2) \equiv \sqrt{(h_1 - h_2)^T (g_0|_{\mathcal{B}}) (h_1 - h_2)}.
	\end{equation}
\end{lemma}

\begin{proof}
	The proof proceeds by establishing zero Riemannian curvature \cite{docarmo1992}, deriving the constant-velocity geodesic curves, and explicitly calculating the path length integral.
	
	\paragraph{1. Metric Constancy and Flat Geometry.}
	By definition of the limit manifold metric field, {\it $g_0|_{\mathcal{B}}$ is a constant positive-definite matrix independent of the coordinate position $h \in \mathbb{K}$.} Consequently, all partial derivatives of the metric tensor components vanish identically:
	\begin{equation}
		\label{eq:metric_deriv_zero}
		\frac{\partial (g_0|_{\mathcal{B}})_{ij}}{\partial h^k} = 0, \quad \forall i, j, k \in \{1, \dots, d\}.
	\end{equation}
	It follows directly from equation  \eqref{eq:metric_deriv_zero} that the Christoffel symbols of the second kind \cite{lee2018} vanish everywhere on $\mathbb{K}$:
	\begin{equation}
		\label{eq:christoffel_zero}
		\Gamma^i_{jk} = \frac{1}{2} (g_0|_{\mathcal{B}})^{il} \left( \frac{\partial (g_0|_{\mathcal{B}})_{lj}}{\partial h^k} + \frac{\partial (g_0|_{\mathcal{B}})_{lk}}{\partial h^j} - \frac{\partial (g_0|_{\mathcal{B}})_{jk}}{\partial h^l} \right) = 0.
	\end{equation}
	Because $\Gamma^i_{jk} = 0$, the Riemann curvature tensor vanishes identically ($R^i{}_{\mu\alpha\beta} = 0$), {\it defining a flat Riemannian geometry on $(\mathbb{K}, g_0|_{\mathcal{B}})$.}
	
	\paragraph{2. Straight-Line Geodesics.}
	The geodesic equation for a curve $\gamma(t): [0, 1] \to \mathbb{K}$ parameterised by $t \in [0, 1]$ is given by:
	\begin{equation}
		\label{eq:geodesic_eq}
		\ddot{\gamma}^i(t) + \Gamma^i_{jk}\dot{\gamma}^j(t)\dot{\gamma}^k(t) = 0.
	\end{equation}
	Substituting \eqref{eq:christoffel_zero} into \eqref{eq:geodesic_eq} yields $\ddot{\gamma}(t) = 0$. Integrating twice with boundary conditions $\gamma(0) = h_1$ and $\gamma(1) = h_2$ yields the unique length-minimizing geodesic:
	\begin{equation}
		\label{eq:geodesic_path}
		\gamma(t) = h_1 + t(h_2 - h_1), \quad t \in [0, 1],
	\end{equation}
	{\it with constant tangent velocity vector $\dot{\gamma}(t) = h_2 - h_1$.}
	
	\paragraph{3. Geodesic Distance Integration.}
	The intrinsic Riemannian distance $d_{g_0}(h_1, h_2)$ in Lemma~\ref{lem:horizontal_pulled_back_metric_algebraic} is defined as the arc length of the minimizing geodesic $\gamma(t)$ given in \eqref{eq:geodesic_path} under the inner product induced by $g_0|_{\mathcal{B}}$ \cite{amari2000}:
	\begin{align}
		d_{g_0}(h_1, h_2) &= \int_0^1 \sqrt{\langle \dot{\gamma}(t), \dot{\gamma}(t) \rangle_{g_0|_{\mathcal{B}}}} \, dt \nonumber \\
		&= \int_0^1 \sqrt{\dot{\gamma}(t)^T (g_0|_{\mathcal{B}}) \dot{\gamma}(t)} \, dt \nonumber \\
		&= \int_0^1 \sqrt{(h_2 - h_1)^T (g_0|_{\mathcal{B}}) (h_2 - h_1)} \, dt. \label{eq:distance_integral}
	\end{align}
	Since the integrand in \eqref{eq:distance_integral} is independent of $t$, integration over $t \in [0, 1]$ simplifies directly to:
	\begin{equation}
		d_{g_0}(h_1, h_2) = \sqrt{(h_1 - h_2)^T (g_0|_{\mathcal{B}}) (h_1 - h_2)},
	\end{equation}
	which completes the proof for expression \eqref{eq:flat_fisher_dist}.
\end{proof}

\begin{lemma}[Uniform Geodesic Metric Distortion Bound]
	\label{lem:geodesic_distortion_bound}
	Let $\mathbb{K} \subset H_{\theta_0} \cong \mathbb{R}^d$ be a compact coordinate cage. For each $N \ge 1$, let $d_{\widetilde{G}_H^{(N)}}(h_1, h_2)$ denote the intrinsic Riemannian geodesic distance between $h_1, h_2 \in \mathbb{K}$ under the pulled-back horizontal Fisher metric field $\widetilde{G}_H^{(N)}(h) \equiv [\psi_{N,H}^* G^{(N)}](h)$ (Lemma~\ref{lem:horizontal_pulled_back_metric_algebraic}). Under High-Order Local Regularity (Assumption~\ref{asm:local_regularity}), the metric distortion satisfies:
	\begin{equation}
		\sup_{h_1, h_2 \in \mathbb{K}} \left| d_{\widetilde{G}_H^{(N)}}(h_1, h_2) - d_{g_0}(h_1, h_2) \right| \le \frac{K_{\mathbb{K}}}{\sqrt{N}} = \mathcal{O}\left(N^{-1/2}\right),
		\label{eq:geodesic_distortion_formula}
	\end{equation}
	where $K_{\mathbb{K}} < \infty$ is a constant independent of $N, h_1,$ and $h_2$.
\end{lemma}

\begin{proof}
	We execute the proof in four sequential measure-theoretic and differential-geometric steps:
	
	\paragraph{Step 1: Path Integral Representation of Geodesic Distance.}
	Let $h_1, h_2 \in \mathbb{K} \subset \mathbb{R}^d$ be arbitrary points. The intrinsic Riemannian geodesic distance $d_{\widetilde{G}_H^{(N)}}(h_1, h_2)$ is defined as the infimum of the length functional $L_{\widetilde{G}_H^{(N)}}(\gamma)$ over all piecewise smooth curves $\gamma: [0, 1] \to \mathbb{K}$ connecting $\gamma(0) = h_1$ to $\gamma(1) = h_2$:
	\begin{equation}
		d_{\widetilde{G}_H^{(N)}}(h_1, h_2) \equiv \inf_{\gamma \in C^1([0,1], \mathbb{K})} \int_0^1 \sqrt{\dot{\gamma}(t)^T \widetilde{G}_H^{(N)}(\gamma(t)) \dot{\gamma}(t)} \, dt.
	\end{equation}
	Using Lemma \ref{lem:horizontal_pulled_back_metric_algebraic}, the flat distance $d_{g_0}(h_1, h_2)$ is achieved along the straight line segment $\gamma_0(t) = h_1 + t(h_2 - h_1)$:
	\begin{equation}
		d_{g_0}(h_1, h_2) = \int_0^1 \sqrt{(h_2 - h_1)^T (g_0|_{\mathcal{B}}) (h_2 - h_1)} \, dt = \sqrt{(h_2 - h_1)^T (g_0|_{\mathcal{B}}) (h_2 - h_1)}.
	\end{equation}
	
	\paragraph{Step 2: Operator Norm Metric Difference Bound.}
	By Lemma~\ref{lem:uniform_horizontal_metric_stabilization} (Uniform Horizontal Metric Stabilization), the pulled-back metric tensor field satisfies uniform operator norm stabilization over $\mathbb{K}$:
	\begin{equation}
		\sup_{h \in \mathbb{K}} \left\| \widetilde{G}_H^{(N)}(h) - g_0|_{\mathcal{B}} \right\|_{op} \le \frac{K_{\mathbb{K}}^{(1)}}{\sqrt{N}},
	\end{equation}
	for a finite constant $K_{\mathbb{K}}^{(1)} < \infty$. For any tangent vector $v \in T_h \mathbb{K} \cong \mathbb{R}^d$, by Rayleigh-Ritz bounds:
	\begin{equation}
		\left| v^T \widetilde{G}_H^{(N)}(h) v - v^T (g_0|_{\mathcal{B}}) v \right| \le \left\| \widetilde{G}_H^{(N)}(h) - g_0|_{\mathcal{B}} \right\|_{op} \|v\|_2^2 \le \frac{K_{\mathbb{K}}^{(1)}}{\sqrt{N}} \|v\|_2^2.
	\end{equation}

\paragraph{Step 3: Linear Matrix Inequality and Integrand Bounds.}
Fix $h \in \mathbb{K}$ and a non-zero vector $v \in T_h \mathbb{K} \cong \mathbb{R}^d$. To rigorously bound the point-wise integrand difference between the perturbed metric length $\sqrt{v^T \widetilde{G}_H^{(N)}(h) v}$ and the unperturbed base metric length $\sqrt{v^T (g_0|_{\mathcal{B}}) v}$, we define the positive quadratic forms:
\begin{equation}
	x \coloneqq v^T \widetilde{G}_H^{(N)}(h) v, \quad y \coloneqq v^T (g_0|_{\mathcal{B}}) v.
\end{equation}
We execute step-by-step mathematical deduction through four sequential sub-steps:

\begin{enumerate}
	\item \textbf{Lower Bound on the Unperturbed Quadratic Form $y$:} \\
	By Theorem~\ref{thm:horizontal_strict_positivity}, $g_0|_{\mathcal{B}}$ is strictly positive-definite on $H_{\theta_0}$, with minimal eigenvalue bounded below by $\lambda_{\min}(g_0|_{\mathcal{B}}) \ge \kappa_0 > 0$. By the Rayleigh-Ritz variational principle:
	\begin{equation}
		y = v^T (g_0|_{\mathcal{B}}) v \ge \lambda_{\min}(g_0|_{\mathcal{B}}) \|v\|_2^2 \ge \kappa_0 \|v\|_2^2 > 0 \implies \sqrt{y} = \sqrt{v^T (g_0|_{\mathcal{B}}) v} \ge \sqrt{\kappa_0} \|v\|_2.
	\end{equation}
	
	\item \textbf{Operator Norm Perturbation Bound on $|x - y|$:} \\
	Applying the uniform matrix operator norm bound established in Step~2, $\|\widetilde{G}_H^{(N)}(h) - g_0|_{\mathcal{B}}\|_{\mathrm{op}} \le \frac{K_{\mathbb{K}}^{(1)}}{\sqrt{N}}$, the absolute difference between the quadratic forms satisfies:
	\begin{equation}
		|x - y| = \left| v^T \left( \widetilde{G}_H^{(N)}(h) - g_0|_{\mathcal{B}} \right) v \right| \le \left\| \widetilde{G}_H^{(N)}(h) - g_0|_{\mathcal{B}} \right\|_{\mathrm{op}} \|v\|_2^2 \le \frac{K_{\mathbb{K}}^{(1)}}{\sqrt{N}} \|v\|_2^2.
	\end{equation}
	
	\item \textbf{Uniform Denominator Lower Bound for $\sqrt{x} + \sqrt{y}$:} \\
	Under uniform horizontal metric stabilization, the perturbed metric satisfies $\lambda_{\min}(\widetilde{G}_H^{(N)}(h)) \ge \kappa_0 - \mathcal{O}(N^{-1}{}^/{}^2)$. For sufficiently large $N$, $\sqrt{x} \ge \sqrt{\kappa_0} \|v\|_2 + o(1)$. Summing the square roots of the quadratic forms yields the lower bound:
	\begin{equation}
		\sqrt{x} + \sqrt{y} = \sqrt{v^T \widetilde{G}_H^{(N)}(h) v} + \sqrt{v^T (g_0|_{\mathcal{B}}) v} \ge 2 \sqrt{\kappa_0} \|v\|_2.
	\end{equation}
	
	\item \textbf{Algebraic Rationalization and Final Integrand Evaluation:} \\
	Applying the algebraic difference-of-squares identity $|\sqrt{x} - \sqrt{y}| = \frac{|x - y|}{\sqrt{x} + \sqrt{y}}$ for positive $x, y$:
	\begin{align}
		\left| \sqrt{v^T \widetilde{G}_H^{(N)}(h) v} - \sqrt{v^T (g_0|_{\mathcal{B}}) v} \right| 
		&= \frac{\left| v^T \widetilde{G}_H^{(N)}(h) v - v^T (g_0|_{\mathcal{B}}) v \right|}{\sqrt{v^T \widetilde{G}_H^{(N)}(h) v} + \sqrt{v^T (g_0|_{\mathcal{B}}) v}} \nonumber \\
		&\le \frac{\frac{K_{\mathbb{K}}^{(1)}}{\sqrt{N}} \|v\|_2^2}{2 \sqrt{\kappa_0} \|v\|_2} \nonumber \\
		&= \left( \frac{K_{\mathbb{K}}^{(1)}}{2\sqrt{\kappa_0}} \right) \frac{\|v\|_2}{\sqrt{N}}.
	\end{align}
\end{enumerate}
	
\paragraph{Step 4: Integration along Straight and Optimal Geodesics.}
To derive a rigorous two-sided bound on the geodesic distortion $|d_{\widetilde{G}_H^{(N)}}(h_1, h_2) - d_{g_0}(h_1, h_2)|$, we analyze the length functional along two distinct curves connecting $h_1, h_2 \in \mathbb{K}$: the Euclidean straight line segment $\gamma_0$ (for the upper bound) and the metric-minimizing geodesic $\gamma_N^*$ under $\widetilde{G}_H^{(N)}$ (for the lower bound).

\noindent
\subparagraph{(1) Upper Bound via Trial Path $\gamma_0$:}
Consider the standard linear path $\gamma_0(t) = h_1 + t(h_2 - h_1)$ for $t \in [0, 1]$, whose tangent velocity is constant: $\dot{\gamma}_0(t) = h_2 - h_1$. By definition of the intrinsic Riemannian distance $d_{\widetilde{G}_H^{(N)}}(h_1, h_2)$ as the infimum of arc lengths over all piecewise $C^1$ curves connecting $h_1$ and $h_2$, evaluating the length along $\gamma_0$ provides an immediate upper bound:
\begin{equation}
	d_{\widetilde{G}_H^{(N)}}(h_1, h_2) \le L_{\widetilde{G}_H^{(N)}}(\gamma_0) = \int_0^1 \sqrt{\dot{\gamma}_0(t)^T \widetilde{G}_H^{(N)}(\gamma_0(t)) \dot{\gamma}_0(t)} \, dt.
\end{equation}
Substituting the point-wise integrand upper bound derived in Step 3, namely
\begin{equation*}
	\sqrt{v^T \widetilde{G}_H^{(N)}(h) v} \le \sqrt{v^T (g_0|_{\mathcal{B}}) v} + \frac{K_{\mathbb{K}}^{(1)}}{2\sqrt{\kappa_0}\sqrt{N}} \|v\|_2
\end{equation*}
evaluated at $v = \dot{\gamma}_0(t) = h_2 - h_1$ and $h = \gamma_0(t)$, we obtain:
\begin{align}
	d_{\widetilde{G}_H^{(N)}}(h_1, h_2) &\le \int_0^1 \left( \sqrt{\dot{\gamma}_0(t)^T (g_0|_{\mathcal{B}}) \dot{\gamma}_0(t)} + \frac{K_{\mathbb{K}}^{(1)}}{2\sqrt{\kappa_0}\sqrt{N}} \|\dot{\gamma}_0(t)\|_2 \right) dt \nonumber \\
	&= \int_0^1 \sqrt{(h_2 - h_1)^T (g_0|_{\mathcal{B}}) (h_2 - h_1)} \, dt + \frac{K_{\mathbb{K}}^{(1)} \|h_2 - h_1\|_2}{2\sqrt{\kappa_0}\sqrt{N}} \int_0^1 dt \nonumber \\
	&= d_{g_0}(h_1, h_2) + \frac{K_{\mathbb{K}}^{(1)} \|h_2 - h_1\|_2}{2\sqrt{\kappa_0}\sqrt{N}}. \label{eq:geodesic_upper_bound}
\end{align}
Subtracting $d_{g_0}(h_1, h_2)$ from both sides yields the one-sided upper distortion bound:
\begin{equation}
	d_{\widetilde{G}_H^{(N)}}(h_1, h_2) - d_{g_0}(h_1, h_2) \le \frac{K_{\mathbb{K}}^{(1)} \|h_2 - h_1\|_2}{2\sqrt{\kappa_0}\sqrt{N}}.
\end{equation}

\noindent
\subparagraph{(2) Lower Bound via Optimal Geodesic $\gamma_N^*$:}
Conversely, let $\gamma_N^*: [0, 1] \to \mathbb{K}$ be the length-minimizing geodesic curve under metric $\widetilde{G}_H^{(N)}$ satisfying $\gamma_N^*(0) = h_1$ and $\gamma_N^*(1) = h_2$, such that $L_{\widetilde{G}_H^{(N)}}(\gamma_N^*) = d_{\widetilde{G}_H^{(N)}}(h_1, h_2)$. Since $d_{g_0}(h_1, h_2)$ represents the absolute minimum length between $h_1$ and $h_2$ under the flat metric $g_0|_{\mathcal{B}}$, evaluating the flat length of $\gamma_N^*$ yields:
\begin{equation}
	d_{g_0}(h_1, h_2) \le L_{g_0}(\gamma_N^*) = \int_0^1 \sqrt{\dot{\gamma}_N^*(t)^T (g_0|_{\mathcal{B}}) \dot{\gamma}_N^*(t)} \, dt.
\end{equation}
Applying the reverse point-wise integrand bound from Step 3 along $\gamma_N^*(t)$:
\begin{equation*}
	\sqrt{v^T (g_0|_{\mathcal{B}}) v} \le \sqrt{v^T \widetilde{G}_H^{(N)}(h) v} + \frac{K_{\mathbb{K}}^{(1)}}{2\sqrt{\kappa_0}\sqrt{N}} \|v\|_2,
\end{equation*}
we integrate along $\gamma_N^*$:
\begin{align}
	d_{g_0}(h_1, h_2) &\le \int_0^1 \left( \sqrt{\dot{\gamma}_N^*(t)^T \widetilde{G}_H^{(N)}(\gamma_N^*(t)) \dot{\gamma}_N^*(t)} + \frac{K_{\mathbb{K}}^{(1)}}{2\sqrt{\kappa_0}\sqrt{N}} \|\dot{\gamma}_N^*(t)\|_2 \right) dt \nonumber \\
	&= L_{\widetilde{G}_H^{(N)}}(\gamma_N^*) + \frac{K_{\mathbb{K}}^{(1)}}{2\sqrt{\kappa_0}\sqrt{N}} \int_0^1 \|\dot{\gamma}_N^*(t)\|_2 \, dt \nonumber \\
	&= d_{\widetilde{G}_H^{(N)}}(h_1, h_2) + \frac{K_{\mathbb{K}}^{(1)}}{2\sqrt{\kappa_0}\sqrt{N}} L_{\text{Euc}}(\gamma_N^*).
\end{align}
Because $\lambda_{\min}(\widetilde{G}_H^{(N)}) \ge \kappa_0 / 2 > 0$ uniformly for large $N$, the Euclidean length $L_{\text{Euc}}(\gamma_N^*)$ of the minimizing curve is bounded by the straight-line distance $\|h_2 - h_1\|_2$ up to higher-order $\mathcal{O}(N^{-1/2})$ corrections, giving:
\begin{equation}
	d_{g_0}(h_1, h_2) \le d_{\widetilde{G}_H^{(N)}}(h_1, h_2) + \frac{K_{\mathbb{K}}^{(1)} \|h_2 - h_1\|_2}{2\sqrt{\kappa_0}\sqrt{N}}. \label{eq:geodesic_lower_bound}
\end{equation}
Rearranging \eqref{eq:geodesic_lower_bound} establishes the complementary lower distortion bound:
\begin{equation}
	d_{g_0}(h_1, h_2) - d_{\widetilde{G}_H^{(N)}}(h_1, h_2) \le \frac{K_{\mathbb{K}}^{(1)} \|h_2 - h_1\|_2}{2\sqrt{\kappa_0}\sqrt{N}}.
\end{equation}

\subparagraph{3. Uniform Supremum over Compact Cage $\mathbb{K}$:}
Combining the upper and lower bounds from Equations \eqref{eq:geodesic_upper_bound} and \eqref{eq:geodesic_lower_bound}, we obtain the two-sided point-wise distortion inequality for any $h_1, h_2 \in \mathbb{K}$:
\begin{equation}
	\left| d_{\widetilde{G}_H^{(N)}}(h_1, h_2) - d_{g_0}(h_1, h_2) \right| \le \frac{K_{\mathbb{K}}^{(1)} \|h_2 - h_1\|_2}{2\sqrt{\kappa_0}\sqrt{N}}.
\end{equation}
Since $\mathbb{K}$ is compact, its Euclidean diameter is strictly finite: $R_{\mathbb{K}} \equiv \sup_{h \in \mathbb{K}} \|h\|_2 < \infty$, which implies $\sup_{h_1, h_2 \in \mathbb{K}} \|h_2 - h_1\|_2 \le 2 R_{\mathbb{K}}$. Taking the uniform supremum over all parameter pairs $h_1, h_2 \in \mathbb{K}$ yields:
\begin{equation}
	\sup_{h_1, h_2 \in \mathbb{K}} \left| d_{\widetilde{G}_H^{(N)}}(h_1, h_2) - d_{g_0}(h_1, h_2) \right| \le \frac{K_{\mathbb{K}}^{(1)} (2 R_{\mathbb{K}})}{2\sqrt{\kappa_0}\sqrt{N}} = \frac{K_{\mathbb{K}}^{(1)} R_{\mathbb{K}}}{\sqrt{\kappa_0}\sqrt{N}} \equiv \frac{K_{\mathbb{K}}}{\sqrt{N}},
\end{equation}
where $K_{\mathbb{K}} \equiv \frac{K_{\mathbb{K}}^{(1)} R_{\mathbb{K}}}{\sqrt{\kappa_0}} < \infty$ is a finite constant independent of $N, h_1, h_2$. This completes the proof of Lemma~\ref{lem:geodesic_distortion_bound}.
\end{proof}

\begin{theorem}[Gromov-Hausdorff Leaf Collapse Theorem]
	\label{thm:gh_leaf_collapse}
	Under High-Order Local Regularity (Assumption \ref{asm:local_regularity}), the sequence of horizontal Riemannian leaf manifolds $(\mathcal{L}_{N,\theta_0}, \tilde{G}_H^{(N)})$ restricted to compact coordinate cages $\mathbb{K} \subset \mathbb{R}^d$ converges in the Gromov-Hausdorff metric topology to the canonical flat tangent canvas $\mathcal{M}_\infty \equiv (T_{\theta_0}IG_1, g_0|_{\mathcal{B}}, \nabla^{(0)})$:
	\begin{equation}
		\label{eq:gh_leaf_collapse_limit}
		\lim_{N \to \infty} d_{GH}\left( (\mathbb{K}, d_{\tilde{G}_H^{(N)}}), \, (\mathbb{K}, d_{g_0}) \right) = 0.
	\end{equation}
	More specifically, the rate of topological collapse satisfies $d_{GH}\left( (\mathbb{K}, d_{\tilde{G}_H^{(N)}}), (\mathbb{K}, d_{g_0}) \right) = \mathcal{O}\left(N^{-1/2}\right)$.
\end{theorem}

\begin{proof}
	We construct a detailed, step-by-step mathematical proof structured into four sequential stages: metric space formulation, mapping distortion estimation, explicit isometric reference space embedding, and asymptotic convergence analysis.
	
	\subsubsection*{Stage 1: Formulation of Compact Metric Spaces and Canonical Mapping}
	Let $\mathbb{K} \subset H_{\theta_0} \cong \mathbb{R}^d$ be a fixed, $N$-invariant compact coordinate cage enclosing the origin $h = \mathbf{0}$. We define two compact metric spaces on the common domain $\mathbb{K}$:
	\begin{enumerate}
		\item $X_N \equiv \left(\mathbb{K}, \, d_{\tilde{G}_H^{(N)}}\right)$, representing the localized horizontal leaf manifold equipped with the pulled-back Riemannian geodesic distance $d_{\tilde{G}_H^{(N)}}$;
		\item $Y \equiv \left(\mathbb{K}, \, d_{g_0}\right)$, representing the canonical asymptotic target space equipped with the flat Euclidean-Fisher distance $d_{g_0}$ established in Lemma~\ref{lem:horizontal_pulled_back_metric_algebraic} (equation (\ref{eq:flat_fisher_dist})).
	\end{enumerate}
	Consider the canonical identity map $f = \mathrm{id}_{\mathbb{K}}: X_N \to Y$ defined pointwise by $f(h) = h$ for all $h \in \mathbb{K}$. The map $f$ is a global bijection between $X_N$ and $Y$.
	
	\subsubsection*{Stage 2: Distortion Bound of the Identity Correspondence}
	By classical metric geometry \cite{burago2001, gromov2007}, for any surjective mapping $f: X_N \to Y$, the \textbf{metric distortion} $\mathrm{dis}(f)$ of $f$ is defined as:
	\begin{equation}
		\label{eq:distortion_def_proof}
		\mathrm{dis}(f) \equiv \sup_{h_1, h_2 \in X_N} \left| d_Y\left(f(h_1), f(h_2)\right) - d_{X_N}(h_1, h_2) \right|.
	\end{equation}
	Substituting $f = \mathrm{id}_{\mathbb{K}}$, $d_{X_N} = d_{\tilde{G}_H^{(N)}}$, and $d_Y = d_{g_0}$ into \eqref{eq:distortion_def_proof}:
	\begin{equation}
		\label{eq:identity_distortion_eval}
		\mathrm{dis}(\mathrm{id}_{\mathbb{K}}) = \sup_{h_1, h_2 \in \mathbb{K}} \left| d_{g_0}(h_1, h_2) - d_{\tilde{G}_H^{(N)}}(h_1, h_2) \right|.
	\end{equation}
	Applying Lemma~\ref{lem:geodesic_distortion_bound} equation (equation \eqref{eq:geodesic_distortion_formula}) directly to \eqref{eq:identity_distortion_eval}, the uniform geodesic metric distortion is strictly bounded under Assumption~\ref{asm:local_regularity}:
	\begin{equation}
		\label{eq:distortion_bound_applied}
		\mathrm{dis}(\mathrm{id}_{\mathbb{K}}) \le \frac{K_{\mathbb{K}}}{\sqrt{N}} = \mathcal{O}\left(N^{-1/2}\right),
	\end{equation}
	where $K_{\mathbb{K}} \equiv \frac{K_{\mathbb{K}}^{(1)} R_{\mathbb{K}}}{\sqrt{\kappa_0}} < \infty$ is an $N$-independent finite constant.
	
	\subsubsection*{Stage 3: Explicit Embedding into Reference Metric Space}
	To compute the Gromov-Hausdorff distance according to Definition~\ref{def:gromov_hausdorff_distance} {\it (Gromov-Hausdorff Distance for Information Leaf Spaces)}, we construct an explicit joint metric space $(\mathcal{Z}, d_{\mathcal{Z}})$ and isometric embeddings $i: X_N \hookrightarrow \mathcal{Z}$ and $j: Y \hookrightarrow \mathcal{Z}$.
	
	Let $\mathcal{Z} \equiv \mathbb{K}$, equipped with the target flat metric $d_{\mathcal{Z}} \equiv d_{g_0}$. We define the candidate embeddings:
	\begin{equation}
		\label{eq:embeddings_def}
		j = \mathrm{id}_{\mathbb{K}}: Y \hookrightarrow \mathcal{Z}, \quad \text{and} \quad i = \mathrm{id}_{\mathbb{K}}: X_N \hookrightarrow \mathcal{Z}.
	\end{equation}
	The embedding $j: (Y, d_{g_0}) \hookrightarrow (\mathcal{Z}, d_{g_0})$ is an exact isometry ($d_{\mathcal{Z}}(j(h_1), j(h_2)) = d_{g_0}(h_1, h_2)$).
	
	Now we evaluate the Hausdorff distance $d_H^{\mathcal{Z}}(i(X_N), j(Y))$ between the embedded images $i(\mathbb{K})$ and $j(\mathbb{K})$ in $\mathcal{Z}$. By Definition~\ref{def:gromov_hausdorff_distance}, and leveraging the distortion inequality for surjective mappings \cite[Theorem~7.3.25]{burago2001}:
	\begin{equation}
		\label{eq:gh_distortion_theorem}
		d_{GH}\left( X_N, \, Y \right) \le \frac{1}{2} \mathrm{dis}(\mathrm{id}_{\mathbb{K}}),
	\end{equation}
where $\mathrm{dis}(\mathrm{id}_{\mathbb{K}})$ is the metric distortion of the canonical identity correspondence $\mathrm{id}_{\mathbb{K}}$ mapping between the metric spaces $(\mathbb{K}, d_{\tilde{G}_H^{(N)}})$ and $(\mathbb{K}, d_{g_0})$, defined as:  
\begin{equation}
\mathrm{dis}(\mathrm{id}_{\mathbb{K}}) = \sup_{h_1, h_2 \in \mathbb{K}} \left\vert{} d_{\tilde{G}_H^{(N)}}(h_1, h_2) - d_{g_0}(h_1, h_2) \right\vert{}
\end{equation}

	For completeness, we verify \eqref{eq:gh_distortion_theorem} directly: for any point $x = i(h_1) \in i(X_N)$, choosing $y = j(h_1) \in j(Y)$ yields $d_{\mathcal{Z}}(i(h_1), j(h_1)) = d_{g_0}(h_1, h_1) = 0$. However, the metric distance between pairs in $X_N$ differs from that in $\mathcal{Z}$ by at most $\mathrm{dis}(\mathrm{id}_{\mathbb{K}})$. Thus, the minimal Hausdorff distance attainable over all isometric embeddings into arbitrary metric spaces is upper-bounded by half the maximum pairwise distortion:
	\begin{equation}
		\label{eq:gh_half_distortion_bound}
		d_{GH}\left( (\mathbb{K}, d_{\tilde{G}_H^{(N)}}), \, (\mathbb{K}, d_{g_0}) \right) \le \frac{1}{2} \sup_{h_1, h_2 \in \mathbb{K}} \left| d_{\tilde{G}_H^{(N)}}(h_1, h_2) - d_{g_0}(h_1, h_2) \right|.
	\end{equation}
	
	\subsubsection*{Stage 4: Asymptotic Limit and Rate Evaluation}
	Substituting the uniform bound \eqref{eq:distortion_bound_applied} from Lemma~\ref{lem:geodesic_distortion_bound} into \eqref{eq:gh_half_distortion_bound}:
	\begin{equation}
		\label{eq:gh_explicit_rate_bound}
		d_{GH}\left( (\mathbb{K}, d_{\tilde{G}_H^{(N)}}), \, (\mathbb{K}, d_{g_0}) \right) \le \frac{1}{2} \left( \frac{K_{\mathbb{K}}}{\sqrt{N}} \right) = \frac{K_{\mathbb{K}}}{2\sqrt{N}} = \mathcal{O}\left(N^{-1/2}\right).
	\end{equation}
	Taking the thermodynamic / large-sample limit as $N \to \infty$ on both sides of \eqref{eq:gh_explicit_rate_bound}:
	\begin{equation}
		\label{eq:final_limit_eval}
		\lim_{N \to \infty} d_{GH}\left( (\mathbb{K}, d_{\tilde{G}_H^{(N)}}), \, (\mathbb{K}, d_{g_0}) \right) \le \lim_{N \to \infty} \frac{K_{\mathbb{K}}}{2\sqrt{N}} = 0.
	\end{equation}
	Since $d_{GH}(\cdot, \cdot) \ge 0$ by definition of a metric on compact metric space isomorphism classes, we have:
	\begin{equation}
		\lim_{N \to \infty} d_{GH}\left( (\mathbb{K}, d_{\tilde{G}_H^{(N)}}), \, (\mathbb{K}, d_{g_0}) \right) = 0.
	\end{equation}
	This completes the rigorous, step-by-step mathematical proof of Theorem~\ref{thm:gh_leaf_collapse}.
\end{proof}

\begin{remark}[Topological and Statistical Consequences of Theorem~\ref{thm:gh_leaf_collapse}]
	The Gromov-Hausdorff leaf collapse proved in Theorem~\ref{thm:gh_leaf_collapse} carries profound structural implications across differential geometry and large-sample decision theory:
	\begin{enumerate}
		\item \textbf{Dimension Preservation (Absence of Lower-Dimensional Collapse):} Unlike general collapse phenomena in Riemannian geometry where manifolds can collapse to lower-dimensional spaces \cite{cheeger1986}, the horizontal leaf $\mathcal{L}_{N,\theta_0}$ collapses onto a flat canvas $(\mathbb{K}, d_{g_0})$ of \textbf{identical dimension} $d = \mathrm{dim}(\mathcal{B})$. This is guaranteed because $g_0|_{\mathcal{B}}$ is strictly positive-definite ($\lambda_{\min} \ge \kappa_0 > 0$) under Assumption~\ref{asm:local_regularity}.
		
		\item \textbf{Exact CLT Matching Rate:} The convergence rate $\mathcal{O}(N^{-1/2})$ in Gromov-Hausdorff metric distance matches the micro-scale fluctuation rate $\beta_N = 1/\sqrt{N}$ governed by the Central Limit Theorem.
		
		\item \textbf{Geometric Foundation for Le Cam Deficiency Convergence:} Combining Theorem~\ref{thm:gh_leaf_collapse} with Cheeger-Gromov $C^\infty$-convergence \cite{cheeger1986, gromov2007} provides the exact differential-topological substrate for Le Cam's experiment deficiency convergence $\Delta(\mathcal{E}_N(\mathbb{K}), \mathcal{E}_\infty(\mathbb{K})) \to 0$ \cite{lecam1986}.
	\end{enumerate}
\end{remark}

\subsection{Accelerated Holonomy Group Annihilation and Global Parallel Transport Flattening}
\label{subsec:holonomy_annihilation}

Having established Gromov-Hausdorff metric leaf collapse, we now analyze the affine connection structure and parallel transport holonomy governing the phase transition $IG_N \to \mathcal{M}_\infty$. 

To study phrase transition, the Gromov-Hausdorff metric leaf collapse only tracks the metric degeneration of distance scales, whereas the physical and geometric phase transition is defined by the evolution of differential and topological structures:
\begin{itemize}
\item {\bf Affine Connection Structure (Local Differential Dynamics):} While metric collapse compresses the metric space, the affine connection controls how directional derivatives, covariant differentiation, and Christoffel symbols transform under the limit. It governs whether the differential operator field remains stable or degenerates, dictating how the differential geometry on $IG_N$ transitions into the smooth manifold structure of $\mathcal{M}_\infty$.

\item {\bf Parallel Transport Holonomy (Global Topological \& Gauge Dynamics):} Holonomy measures the non-trivial path-dependent phase shift and curvature accumulation acquired when transporting vectors around closed loops, especially along the collapsing leaves. As leaves collapse, holonomy tracks how global topological invariants and internal gauge degrees of freedom concentrate, decouple, or reduce into boundary/gauge conditions on $\mathcal{M}_\infty$.	
	
\end{itemize}
Together, the affine connection determines {\it the local differential continuum}, and parallel transport holonomy determines {\it the global gauge/topological structure}, completely defining the dynamic mechanism behind the phase transition.

\begin{definition}[Restricted Holonomy Group $\mathrm{Hol}_{\theta_0}(\nabla)$] \footnote{The $GL(d, \mathbb{R})$ and $\mathfrak{gl}(d, \mathbb{R})$ structures in this subsection represent the holonomy group of linear endomorphisms acting on vector tangent spaces via the Ambrose--Singer Theorem; they do not assume $\Theta$ itself is a Lie group, thus preserving the full validity of the Fisher-compatible Ehresmann connection framework.}
	\label{def:restricted_holonomy_group}
	Let $(\mathcal{M}, \nabla)$ be an affine manifold. For a base point $p \in \mathcal{M}$, the restricted holonomy group $\mathrm{Hol}_p^0(\nabla) \subseteq \mathrm{GL}(T_p \mathcal{M})$ is the Lie subgroup of invertible linear endomorphisms generated by parallel transport operators $\mathcal{P}_\gamma: T_p \mathcal{M} \to T_p \mathcal{M}$ along smooth piecewise closed loops $\gamma: [0, 1] \to \mathcal{M}$ ($\gamma(0) = \gamma(1) = p$) that are piecewise contractible to $p$.
\end{definition}

\begin{remark}
	The appearance of the Lie group structure $GL(T_{\theta_0}\mathcal{M}) \cong GL(d, \mathbb{R})$ stems strictly from the linear endomorphisms acting on the vector tangent space $T_{\theta_0}\mathcal{M}$ at base point $\theta_0$, rather than requiring the underlying parameter manifold $\Theta$ to be a Lie group or principal bundle. Thus, this formulation maintains complete compatibility with the Fisher-compatible Ehresmann connection framework.
\end{remark}

\begin{lemma}[Ambrose-Singer Holonomy Lie Algebra Bound]
	\label{lem:ambrose_singer_bound}
	Let $\widetilde{\nabla}^{(\alpha,N)}$ denote the pulled-back Amari $\alpha$-connection on $\mathbb{K} \subset \mathcal{M}_\infty$ (Theorem~\ref{thm:connection_dissolution}), and let $\mathfrak{hol}_0(\widetilde{\nabla}^{(\alpha,N)}) \subset \mathfrak{gl}(d, \mathbb{R})$ denote its Lie algebra. Under High-Order Local Regularity (Assumption~\ref{asm:local_regularity}), every generator $A \in \mathfrak{hol}_0(\widetilde{\nabla}^{(\alpha,N)})$ satisfies the accelerated quadratic decay bound in matrix operator norm:
	\begin{equation}
		\sup_{A \in \mathfrak{hol}_0(\widetilde{\nabla}^{(\alpha,N)}), \|A\|_{op}=1} \|A\|_{op} \le C_{\mathbb{K}}^{(\mathcal{R})} \cdot \sup_{h \in \mathbb{K}} \left\| [\psi_{N,H}^* \mathcal{R}^{(\alpha,N)}](h) \right\|_{op} = \mathcal{O}_p\left(N^{-1}\right),
		\label{eq:ambrose_singer_decay_bound}
	\end{equation}
	where $C_{\mathbb{K}}^{(\mathcal{R})} < \infty$ is a constant independent of $N$.
\end{lemma}

\begin{remark}
	By virtue of the Ambrose--Singer Holonomy Theorem \cite{ambrose1953}, the operator norm bound on the Lie algebra $Hol_{\theta_0}(\widetilde{\nabla}^{(\alpha,N)}) \subseteq GL(d, \mathbb{R})$ directly measures curvature decay. Explicitly indexing by base point $\theta_0$, we have $Hol_{\theta_0}^0(\widetilde{\nabla}^{(\alpha,N)}) \subseteq GL(d, \mathbb{R})$ (where subscript $\theta_0$ designates the base point and superscript $0$ denotes restricted holonomy via contractible loops). As $N \to \infty$, the $\mathcal{O}_p(N^{-1})$ bound forces all Lie algebra generators in $Hol_{\theta_0}(\widetilde{\nabla}^{(\alpha,N)})$ to vanish, collapsing $Hol_{\theta_0}^0(\widetilde{\nabla}^{(\alpha,N)})$ to the trivial identity subgroup $\{I_{d \times d}\}$ and establishing global path-independence for parallel transport on the horizontal leaf.
\end{remark}

\begin{proof}
	We execute a step-by-step mathematical deduction using the Ambrose-Singer Holonomy Theorem.
	
	\paragraph{Step 1: Application of the Ambrose-Singer Holonomy Theorem.}
	By the Ambrose-Singer Theorem \cite{ambrose1953, frankel2011}, the Lie algebra $\mathfrak{hol}_0(\widetilde{\nabla}^{(\alpha,N)})$ of the holonomy group at $h=\mathbf{0}$ is linearly spanned by transformation operators of the form:
	\begin{equation}
		A(x, y; \tau) \equiv \mathcal{P}_\tau^{-1} \circ \widetilde{\mathcal{R}}^{(\alpha,N)}(\tau(1))(\mathcal{P}_\tau x, \mathcal{P}_\tau y) \circ \mathcal{P}_\tau,
		\label{eq:ambrose_singer_generator}
	\end{equation}
	where $\tau: [0, 1] \to \mathbb{K}$ is a smooth path starting at $\tau(0) = \mathbf{0}$, $\mathcal{P}_\tau: T_{\mathbf{0}} \mathbb{K} \to T_{\tau(1)} \mathbb{K}$ is the parallel transport operator along $\tau$ under connection $\widetilde{\nabla}^{(\alpha,N)}$, $x, y \in T_{\mathbf{0}} \mathbb{K}$ are tangent vectors,\footnote{Since $\tau(0) = \mathbf{0}$, $T_{\mathbf{0}}\mathbb{K}$ and $T_{\tau(0)}\mathbb{K}$ denote the exact same tangent space; using $T_{\mathbf{0}}\mathbb{K}$ explicitly emphasizes that parallel transport originates at the base point $h = \mathbf{0}$.} and $\widetilde{\mathcal{R}}^{(\alpha,N)} \equiv \psi_{N,H}^* \mathcal{R}^{(\alpha,N)}$ is the pulled-back Riemann curvature tensor.

\paragraph{Step 2: Operator Norm Bounding of Parallel Transport Operators.}
\label{step:parallel_transport_norm_bounding}
By Theorem~\ref{thm:connection_dissolution} (Amari Connection Dissolution), the Christoffel symbols of the pulled-back connection $\widetilde{\nabla}^{(\alpha,N)}$ satisfy the uniform coordinate bound:
\begin{equation}
	\sup_{h \in \mathbb{K}} \left| \widetilde{\Gamma}_{ijk}^{(\alpha,N)}(h) \right| = \frac{1}{\sqrt{N}} \sup_{\theta \in \mathcal{U}(\theta_0)} \left| \Gamma_{ijk}^{(\alpha,1)}(\theta) \right| \le \frac{M_\Gamma}{\sqrt{N}} = \mathcal{O}\left(N^{-1/2}\right) \longrightarrow 0 \quad \text{as } N \to \infty,
	\label{eq:christoffel_uniform_bound}
\end{equation}
where $M_\Gamma \equiv \sup_{\theta \in \overline{\mathcal{U}}(\theta_0)} \max_{i,j,k} \left| \Gamma_{ijk}^{(\alpha,1)}(\theta) \right| < \infty$ due to $C^0$-continuity over the compact neighborhood closure $\overline{\mathcal{U}}(\theta_0)$.

Let $\tau: [0, 1] \to \mathbb{K} \subset \mathbb{R}^d$ be a piecewise smooth curve with velocity vector field $\dot{\tau}(t) \in T_{\tau(t)}\mathbb{K}$ and total arc length:
\begin{equation}
	L_\tau \equiv \int_0^1 \|\dot{\tau}(t)\|_2 \, dt < \infty.
\end{equation}
The linear system of differential equations governing the parallel transport of a tangent vector field $v(t) \in T_{\tau(t)}\mathbb{K}$ along $\tau(t)$ with initial state $v(0) \in T_{\tau(0)}\mathbb{K}$ under $\widetilde{\nabla}^{(\alpha,N)}$ is given by:
\begin{equation}
	\frac{d v^k(t)}{dt} = -\sum_{i,j=1}^d \widetilde{\Gamma}_{ij}^{(\alpha,N)k}(\tau(t)) \dot{\tau}^i(t) v^j(t), \quad k = 1, \dots, d.
	\label{eq:parallel_transport_ode_enhanced}
\end{equation}

To rigorously bound the growth of the Euclidean norm $\|v(t)\|_2 = \sqrt{\sum_{k=1}^d (v^k(t))^2}$, we differentiate $\|v(t)\|_2^2$ with respect to $t$:
\begin{align}
	\frac{d}{dt} \|v(t)\|_2^2 &= 2 \sum_{k=1}^d v^k(t) \frac{d v^k(t)}{dt} \nonumber \\
	&= -2 \sum_{i,j,k=1}^d \widetilde{\Gamma}_{ij}^{(\alpha,N)k}(\tau(t)) \dot{\tau}^i(t) v^j(t) v^k(t).
	\label{eq:norm_sq_derivative}
\end{align}
Applying the Cauchy--Schwarz inequality together with the uniform Christoffel symbol bound from \eqref{eq:christoffel_uniform_bound}:
\begin{align}
	\left| \frac{d}{dt} \|v(t)\|_2^2 \right| &\le 2 \sum_{k=1}^d |v^k(t)| \sum_{j=1}^d \left( \sum_{i=1}^d \left| \widetilde{\Gamma}_{ij}^{(\alpha,N)k}(\tau(t)) \right| |\dot{\tau}^i(t)| \right) |v^j(t)| \nonumber \\
	&\le 2 \left( \frac{M_\Gamma}{\sqrt{N}} \right) \|\dot{\tau}(t)\|_2 \|v(t)\|_2^2.
	\label{eq:norm_sq_inequality}
\end{align}
Since $\frac{d}{dt} \|v(t)\|_2^2 = 2 \|v(t)\|_2 \frac{d}{dt} \|v(t)\|_2$, dividing both sides by $2 \|v(t)\|_2$ for $\|v(t)\|_2 > 0$ yields the differential inequality:
\begin{equation}
	\left| \frac{d}{dt} \|v(t)\|_2 \right| \le \frac{M_\Gamma}{\sqrt{N}} \|\dot{\tau}(t)\|_2 \|v(t)\|_2.
	\label{eq:differential_inequality_v}
\end{equation}

Integrating equation \eqref{eq:differential_inequality_v} via Grönwall's Lemma \cite{gronwall1919} over $t \in [0, 1]$ gives:
\begin{equation}
	\|v(1)\|_2 \le \|v(0)\|_2 \exp\left( \int_0^1 \frac{M_\Gamma}{\sqrt{N}} \|\dot{\tau}(t)\|_2 \, dt \right) = \|v(0)\|_2 \exp\left( \frac{M_\Gamma L_\tau}{\sqrt{N}} \right).
	\label{eq:gronwall_forward_bound}
\end{equation}

Let $\mathcal{P}_\tau: T_{\tau(0)}\mathbb{K} \to T_{\tau(1)}\mathbb{K}$ denote the parallel transport operator along $\tau$, defined by $\mathcal{P}_\tau v(0) \equiv v(1)$. Its operator norm induced by the inner product (metric) norms on the respective tangent spaces is defined as:
\begin{equation}
	\|\mathcal{P}_\tau\|_{op} \equiv \sup_{\substack{v \in T_{\tau(0)}\mathbb{K} \\ \|v\|_{T_{\tau(0)}\mathbb{K}} = 1}} \|\mathcal{P}_\tau v\|_{T_{\tau(1)}\mathbb{K}}.\footnote{For a general non-metric-compatible affine connection such as the Amari $\alpha$-connection $\widetilde{\nabla}^{(\alpha,N)}$, parallel transport is a general linear transformation rather than an isometry, so $\|\mathcal{P}_\tau\|_{op}$ quantifies the maximum norm-stretching factor applied to tangent vectors along $\tau$.}
\end{equation}
Taking the supremum over all $\|v(0)\|_2 = 1$ in \eqref{eq:gronwall_forward_bound} and using the Taylor expansion $\exp(x) = 1 + x + \mathcal{O}(x^2)$ as $x \to 0$:
\begin{equation}
	\|\mathcal{P}_\tau\|_{op} \le \exp\left( \frac{M_\Gamma L_\tau}{\sqrt{N}} \right) = 1 + \frac{M_\Gamma L_\tau}{\sqrt{N}} + \mathcal{O}\left(N^{-1}\right) = 1 + \mathcal{O}\left(N^{-1/2}\right).
	\label{eq:pt_op_norm_upper}
\end{equation}

To derive the corresponding bound for the inverse parallel transport operator $\mathcal{P}_\tau^{-1}$, observe that $\mathcal{P}_\tau^{-1} = \mathcal{P}_{\overline{\tau}}$, where $\overline{\tau}(t) \equiv \tau(1-t)$ ($t \in [0, 1]$) is the time-reversed curve. The total arc length satisfies $L_{\overline{\tau}} = L_\tau$. Applying Grönwall's Lemma to parallel transport along $\overline{\tau}$:
\begin{equation}
	\|\mathcal{P}_\tau^{-1}\|_{op} = \|\mathcal{P}_{\overline{\tau}}\|_{op} \le \exp\left( \frac{M_\Gamma L_\tau}{\sqrt{N}} \right) = 1 + \mathcal{O}\left(N^{-1/2}\right).
	\label{eq:pt_inv_op_norm_upper}
\end{equation}

Furthermore, sub-multiplicativity of the operator norm enforces $1 = \|\mathrm{Id}\|_{op} = \|\mathcal{P}_\tau \mathcal{P}_\tau^{-1}\|_{op} \le \|\mathcal{P}_\tau\|_{op} \|\mathcal{P}_\tau^{-1}\|_{op}$, providing matching lower bounds:
\begin{equation}
	\|\mathcal{P}_\tau\|_{op} \ge \frac{1}{\|\mathcal{P}_\tau^{-1}\|_{op}} \ge \exp\left( -\frac{M_\Gamma L_\tau}{\sqrt{N}} \right) = 1 - \mathcal{O}\left(N^{-1/2}\right).
\end{equation}

Thus, both the forward and inverse parallel transport operator norms satisfy exact two-sided asymptotic bounds:
\begin{equation}
	\|\mathcal{P}_\tau\|_{op} = 1 + \mathcal{O}\left(N^{-1/2}\right) \quad \text{and} \quad \|\mathcal{P}_\tau^{-1}\|_{op} = 1 + \mathcal{O}\left(N^{-1/2}\right).
\end{equation}

	\paragraph{Step 3: Contraction against Accelerated Curvature Annihilation.}
	Taking the operator norm of generator $A(x,y;\tau)$ in Equation \eqref{eq:ambrose_singer_generator}:
	\begin{align}
		\|A(x,y;\tau)\|_{op} &\le \|\mathcal{P}_\tau^{-1}\|_{op} \cdot \|\widetilde{\mathcal{R}}^{(\alpha,N)}(\tau(1))\|_{op} \cdot \|\mathcal{P}_\tau\|_{op}^3 \|x\|_2 \|y\|_2 \nonumber \\
		&\le \left(1 + \mathcal{O}\left(N^{-1/2}\right)\right)^4 \|\widetilde{\mathcal{R}}^{(\alpha,N)}(\tau(1))\|_{op} \|x\|_2 \|y\|_2.
	\end{align}
	By Theorem~\ref{thm:curvature_annihilation} (Object 4: Accelerated Quadratic Curvature Annihilation),
	 Lemma \ref{lem:accelerated_curvature_annihilation}  (Accelerated Curvature Annihilation), and Theorem \ref{thm:universal_scaling_law}  (Universal Tensor Valence Scaling Law $N^{1-r/2}$), the pulled-back Riemann curvature scales as $N^{1-4/2} = N^{-1}$:
	\begin{equation}
		\sup_{h \in \mathbb{K}} \|\widetilde{\mathcal{R}}^{(\alpha,N)}(h)\|_{op} = \frac{1}{N} \sup_{\theta \in \mathcal{U}(\theta_0)} \|\mathcal{R}^{(\alpha,1)}(\theta)\|_{op} = \mathcal{O}_p\left(N^{-1}\right).
		\label{eq:curvature_annihilation_step_ref}
	\end{equation}
	Substituting Equation \eqref{eq:curvature_annihilation_step_ref} into Equation \eqref{eq:ambrose_singer_generator} proves Equation \eqref{eq:ambrose_singer_decay_bound}, completing the proof of Lemma~\ref{lem:ambrose_singer_bound}.
\end{proof}


\begin{theorem}[Accelerated Holonomy Annihilation and Global Parallel Transport Trivialization]
	\label{thm:holonomy_annihilation}
	Under High-Order Local Regularity (Assumption \ref{asm:local_regularity}), as $N\rightarrow\infty$, the restricted holonomy group $Hol_{0}^{0}(\tilde{\nabla}^{(\alpha,N)})$ of the pulled-back information manifold sequence collapses to the trivial identity group $\{I_{d\times d}\}$ at rate $\mathcal{O}_{p}(N^{-1})$:
	\begin{equation}
		\lim_{N\rightarrow\infty}\text{diam}_{op}\left(Hol_{0}^{0}(\tilde{\nabla}^{(\alpha,N)})\right)=0 \iff Hol_{0}^{0}(\tilde{\nabla}^{(\alpha,N)})\xrightarrow{N\rightarrow\infty}\{I_{d\times d}\}.
		\label{eq:holonomy_collapse}
	\end{equation}
	Consequently, parallel transport on the limiting canvas $\mathcal{M}_{\infty}$ is path-independent, establishing global affine flat geometry ($\tilde{\Gamma}^{(0)}\equiv0,\mathcal{R}^{(0)}\equiv0$).
\end{theorem}

\begin{proof}
	We execute the proof through a rigorous four-step mathematical deduction, translating local curvature decay into global parallel transport trivialization via the non-Abelian Stokes' Theorem.
\begin{itemize}
\item \textbf{Step 1: Surface-Ordered Exponential Representation.}
	Let $\gamma:[0,1]\rightarrow\mathbb{K}$ be an arbitrary piecewise smooth contractible closed loop based at the origin $0$. Because the compact coordinate cage $\mathbb{K}\subset\mathbb{R}^d$ is simply connected, $\gamma$ continuously bounds a smooth 2-dimensional oriented surface $\Sigma\subset\mathbb{K}$ such that its boundary is exactly $\gamma$ ($\partial\Sigma=\gamma$).
	
	By the non-Abelian Stokes' Theorem for affine connections \cite{frankel2011}, the holonomy transformation $\mathcal{P}_{\gamma}\in Hol_{0}^{0}(\tilde{\nabla}^{(\alpha,N)})$ along $\gamma$ is given exactly by the surface-ordered exponential integral over $\Sigma$:
	\begin{equation}
		\mathcal{P}_{\gamma}=\mathcal{P}\exp\left(\iint_{\Sigma}\mathcal{P}_{\tau_{\xi}}^{-1}\tilde{\mathcal{R}}^{(\alpha,N)}(\xi)\mathcal{P}_{\tau_{\xi}}d\Sigma(\xi)\right),
		\label{eq:stokes_non_abelian}
	\end{equation}
	where
\begin{itemize}
	\item \textbf{Surface and Boundary ($\Sigma$ and $\gamma$):}
	$\Sigma \subset \mathbb{K}$ is a smooth 2-dimensional oriented surface continuously bounded by the contractible loop $\gamma$ ($\partial\Sigma = \gamma$), based at a reference origin point $0$.
	
	\item \textbf{Reference Radial Paths ($\tau_{\xi}$):}
	For every point $\xi \in \Sigma$, $\tau_{\xi}$ represents a smooth family of reference paths connecting the base point $0$ to $\xi$.
	
	\item \textbf{Local Curvature Conjugation ($\mathcal{P}_{\tau_{\xi}}^{-1}\tilde{\mathcal{R}}^{(\alpha,N)}(\xi)\mathcal{P}_{\tau_{\xi}}$):}
	The term $\tilde{\mathcal{R}}^{(\alpha,N)}(\xi)$ is the local Riemann curvature tensor (or Lie-algebra-valued curvature 2-form) evaluated at $\xi \in \Sigma$. Because curvature tensors at different points $\xi$ act on different tangent spaces/fibers, the operators $\mathcal{P}_{\tau_{\xi}}$ (parallel transport from $0$ to $\xi$) and $\mathcal{P}_{\tau_{\xi}}^{-1}$ (parallel transport back from $\xi$ to $0$) conjugate the local curvature to translate its action back to the single, common tangent Lie algebra space at the base point $0$.
	
	\item \textbf{Surface-Ordering Operator ($\mathcal{P}\exp$):}
	In non-Abelian geometry, curvature operators at different points do not commute ($\left[A(\xi_1), A(\xi_2)\right] \neq 0$). The operator $\mathcal{P}\exp$ generalizes the 1D path-ordered exponential (Dyson series) to 2D surfaces. It dictates that infinitesimal surface elements $d\Sigma(\xi)$ are continuously ordered along a parameterized family of expanding loops/curves sweeping across $\Sigma$. Formally, it is defined via the series expansion:
	\begin{equation*}
		\mathcal{P}_{\gamma} = I + \iint_{\Sigma} A(\xi) d\Sigma(\xi) + \frac{1}{2!} \mathcal{P} \left( \iint_{\Sigma} A(\xi) d\Sigma(\xi) \right)^2 + \dots
	\end{equation*}
	where $A(\xi) = \mathcal{P}_{\tau_{\xi}}^{-1}\tilde{\mathcal{R}}^{(\alpha,N)}(\xi)\mathcal{P}_{\tau_{\xi}}$.
\end{itemize}
	
\item 	\textbf{Step 2: Operator Norm Bound of the Lie Algebra Generator.}
	Let $A(\xi)\equiv\mathcal{P}_{\tau_{\xi}}^{-1}\tilde{\mathcal{R}}^{(\alpha,N)}(\xi)\mathcal{P}_{\tau_{\xi}}$ denote the integrand in Equation \eqref{eq:stokes_non_abelian}, which acts as a Lie algebra generator. Applying the sub-multiplicative property of the matrix operator norm $||\cdot||_{op}$:
	\begin{equation}
		||A(\xi)||_{op} \le ||\mathcal{P}_{\tau_{\xi}}^{-1}||_{op} \cdot ||\tilde{\mathcal{R}}^{(\alpha,N)}(\xi)||_{op} \cdot ||\mathcal{P}_{\tau_{\xi}}||_{op}.
		\label{eq:generator_norm_bound}
	\end{equation}
	By the parallel transport operator bounds established in Lemma \ref{lem:ambrose_singer_bound}, $||\mathcal{P}_{\tau_{\xi}}||_{op} = 1 + \mathcal{O}(N^{-1/2})$ and $||\mathcal{P}_{\tau_{\xi}}^{-1}||_{op} = 1 + \mathcal{O}(N^{-1/2})$. Concurrently, Theorem \ref{thm:curvature_annihilation} (Accelerated Quadratic Curvature Annihilation) guarantees that the pulled-back Riemann curvature is uniformly bounded across $\mathbb{K}$ by:
	\begin{equation}
		\sup_{\xi\in\mathbb{K}}||\tilde{\mathcal{R}}^{(\alpha,N)}(\xi)||_{op} = \mathcal{O}_{p}(N^{-1}).
	\end{equation}
	Substituting these rates into Equation \eqref{eq:generator_norm_bound}, we bound the generator uniformly over $\Sigma\subset\mathbb{K}$:
	\begin{equation}
		\sup_{\xi\in\Sigma}||A(\xi)||_{op} \le \left(1 + \mathcal{O}(N^{-1/2})\right)^2 \mathcal{O}_{p}(N^{-1}) = \mathcal{O}_{p}(N^{-1}).
		\label{eq:uniform_integrand_bound}
	\end{equation}
	
\item 	\textbf{Step 3: Dyson Series Expansion and Identity Deviation.}
	The ordered exponential operator $\mathcal{P}\exp$ in Equation \eqref{eq:stokes_non_abelian} is formally defined by its Dyson series expansion:
	\begin{equation}
		\mathcal{P}_{\gamma} = I_{d\times d} + \iint_{\Sigma}A(\xi)d\Sigma(\xi) + \frac{1}{2!}\mathcal{P}\left(\iint_{\Sigma}A(\xi)d\Sigma(\xi)\right)^2 + \dots
	\end{equation}
	Subtracting the identity matrix $I_{d\times d}$ and applying the triangle inequality to the infinite sum yields the standard exponential majorization bound for the operator norm:
	\begin{equation}
		||\mathcal{P}_{\gamma} - I_{d\times d}||_{op} \le \exp\left(\iint_{\Sigma}||A(\xi)||_{op}d\Sigma(\xi)\right) - 1.
	\end{equation}
	Using the analytic scalar inequality $e^x - 1 \le x e^x$ for $x \ge 0$, we extract the explicit linear-exponential bound:
	\begin{equation}
		||\mathcal{P}_{\gamma} - I_{d\times d}||_{op} \le \left(\iint_{\Sigma}||A(\xi)||_{op}d\Sigma(\xi)\right) \exp\left(\iint_{\Sigma}||A(\xi)||_{op}d\Sigma(\xi)\right).
		\label{eq:exponential_majorization}
	\end{equation}
	
\item \textbf{Step 4: Asymptotic Annihilation of the Holonomy Group.}
	By integrating the uniform bound from Equation \eqref{eq:uniform_integrand_bound} over the surface $\Sigma$, and recognizing that the area of $\Sigma$ is strictly upper-bounded by the maximal area of the compact cage $\mathbb{K}$:
	\begin{equation}
		\iint_{\Sigma}||A(\xi)||_{op}d\Sigma(\xi) \le \text{Area}(\Sigma) \cdot \mathcal{O}_{p}(N^{-1}) \le \text{Area}(\mathbb{K}) \cdot \mathcal{O}_{p}(N^{-1}) = \mathcal{O}_{p}(N^{-1}).
	\end{equation}
	Substituting this integral bound into Equation \eqref{eq:exponential_majorization}:
	\begin{equation}
		||\mathcal{P}_{\gamma} - I_{d\times d}||_{op} \le \mathcal{O}_{p}(N^{-1}) \cdot \exp\left(\mathcal{O}_{p}(N^{-1})\right).
	\end{equation}
	Since $\exp(\mathcal{O}_{p}(N^{-1})) \xrightarrow{N\rightarrow\infty} 1$, the right-hand side simplifies asymptotically to exactly $\mathcal{O}_{p}(N^{-1})$.
	Because this bound depends solely on the geometric area of the domain $\mathbb{K}$ and is completely independent of the specific choice of the contractible loop $\gamma$, we can take the supremum over all possible loops $\gamma \subset \mathbb{K}$:
	\begin{equation}
		\sup_{\gamma\subset\mathbb{K}}||\mathcal{P}_{\gamma} - I_{d\times d}||_{op} = \mathcal{O}_{p}(N^{-1}).
		\label{eq:supremum_bound_holonomy}
	\end{equation}
	The operator norm diameter of the restricted holonomy group is precisely defined by this supremum. Taking the thermodynamic limit $N\rightarrow\infty$ on both sides of Equation \eqref{eq:supremum_bound_holonomy} yields:
	\begin{equation}
		\lim_{N\rightarrow\infty}\text{diam}_{op}\left(Hol_{0}^{0}(\tilde{\nabla}^{(\alpha,N)})\right) = \lim_{N\rightarrow\infty}\sup_{\gamma\subset\mathbb{K}}||\mathcal{P}_{\gamma} - I_{d\times d}||_{op} = 0.
	\end{equation}
	This proves that $Hol_{0}^{0}(\tilde{\nabla}^{(\alpha,N)}) \rightarrow \{I_{d\times d}\}$ identically. 
	Consequently, parallel transport along any path within the limiting canvas $\mathcal{M}_{\infty}$ induces exclusively the identity transformation, ensuring absolute path-independence. This globally trivializes the affine geometry, locking the limits strictly to $\tilde{\Gamma}^{(0)}\equiv0$ and $\mathcal{R}^{(0)}\equiv0$, completing the proof of Theorem \ref{thm:holonomy_annihilation}.
\end{itemize}	
\end{proof}


\subsection{Global LAN Uniformization Across Compact Coordinate Cages}
\label{subsec:global_lan_uniformization}

In Section~\ref{sec:module3} and Section~\ref{sec:equivalence_proof}, the equivalence between Geometric Priority (\ref{ax:geom_priority}) and Statistical Operationalism (\ref{ax:stat_oper}) was proven for localized point neighborhoods. Furthermore, in Section~\ref{sec:module2}, Section~\ref{subsec:mod2_degeneration}, Section~\ref{sec:module1}, and Section~\ref{sec:equivalence_proof}, we established that the horizontal Riemannian leaf sequence $(\mathcal{L}_{N,\theta_0}, \widetilde{G}_H^{(N)})$ undergoes Gromov--Hausdorff metric collapse onto the flat Euclidean canvas $\mathcal{M}_\infty \equiv (T_{\theta_0}IG_1, g_0\big|_{\mathcal{B}}, \nabla^{(0)})$ at rate $\mathcal{O}(N^{-1/2})$ (Theorem \ref{thm:gh_leaf_collapse}), while the restricted holonomy group $\mathrm{Hol}_{\theta_0}^0(\widetilde{\nabla}^{(\alpha,N)})$ collapses to the trivial identity subgroup $\{\mathrm{Id}_{d \times d}\}$ at the accelerated rate $\mathcal{O}_p(N^{-1})$ (Theorem \ref{thm:holonomy_annihilation}).

In this subsection (Stage III of the global blueprint outlined in Section~\ref{sec:module1}), we translate these differential-topological leaf collapse results into global decision-theoretic uniformization across arbitrary compact coordinate cages $\mathbb{K} \subset H_{\theta_0} \cong \mathbb{R}^d$. We demonstrate that the localized Radon--Nikodym log-likelihood ratio process $\Lambda_{N,H}(h)$ converges uniformly in probability over $\mathbb{K}$ to a linear-quadratic Gaussian experiment, forcing Le Cam's experiment deficiency distance $\Delta\left(\mathcal{E}_{N,H}(\mathbb{K}), \mathcal{E}_\infty(\mathcal{B}; \mathbb{K})\right)$ to vanish at rate $\mathcal{O}(N^{-1/2})$.

\begin{definition}[Horizontal Localized Experiment Sequence and Canonical Limit Experiment]
	\label{def:horizontal_localized_experiment}
	Let $(\Theta, \mathcal{B}, \pi, \mathcal{F})$ be the statistical fiber bundle equipped with the Fisher-compatible Ehresmann connection $T_\theta \Theta = H_\theta \oplus V_\theta$ (Definition \ref{def:ehresmann_connection}). For a fixed background parameter state $\theta_0 \in \mathrm{Int}(\Theta)$ and an $N$-invariant compact coordinate cage $\mathbb{K} \subset H_{\theta_0} \cong \mathbb{R}^d$ enclosing the origin $h = \mathbf{0}$:
	\begin{enumerate}[label=\textnormal{(\roman*)}]
		\item The \textbf{horizontally restricted $N$-sample empirical experiment sequence} $\mathcal{E}_{N,H}(\mathbb{K})$ is defined by:
		\begin{equation}
			\mathcal{E}_{N,H}(\mathbb{K}) \equiv \left( \mathcal{X}^N, \, \mathcal{A}^{\otimes N}, \, \left\{ P_{\theta_0 + \frac{1}{\sqrt{N}} P_{\theta_0}^H h}^{(N)} : h \in \mathbb{K} \right\} \right),
			\label{eq:horizontal_empirical_exp_def}
		\end{equation}
		where $P_{\theta_0}^H \equiv \mathrm{id}_{T_{\theta_0}\Theta} - \omega_{\theta_0}$ is the horizontal projection operator (Definition  \ref{def:connection_1form_projection}).
		
		\item The \textbf{canonical limiting Gaussian shift experiment} $\mathcal{E}_\infty(\mathcal{B}; \mathbb{K})$ defined on the quotient base manifold $\mathcal{B} \equiv \Theta / \sim$ is defined by:
		\begin{equation}
			\mathcal{E}_\infty(\mathcal{B}; \mathbb{K}) \equiv \left( \mathbb{R}^d, \, \mathcal{B}^d, \, \left\{ \mathcal{N}\left( h, \, \left(g_0\big|_{\mathcal{B}}\right)^{-1} \right) : h \in \mathbb{K} \right\} \right),
			\label{eq:canonical_gaussian_limit_exp_def}
		\end{equation}
		where $g_0\big|_{\mathcal{B}} \equiv P_{\theta_0}^H g^{(1)}(\theta_0) P_{\theta_0}^H$ is the strictly positive-definite, non-singular restricted Fisher metric tensor on $H_{\theta_0}$ (Theorem \ref{thm:horizontal_strict_positivity}).
	\end{enumerate}
\end{definition}

\begin{lemma}[Uniform Stochastic Expansion of the Horizontal Log-Likelihood Process]
	\label{lem:uniform_horizontal_log_likelihood_expansion}
	Under High-Order Local Regularity (Assumption \ref{asm:local_regularity}), the localized horizontal log-likelihood ratio process
	\begin{equation}
		\Lambda_{N,H}(h) \equiv \log \frac{d P_{\theta_0 + \frac{1}{\sqrt{N}} P_{\theta_0}^H h}^{(N)}}{d P_{\theta_0}^{(N)}}(X^N) = \sum_{n=1}^N \left[ \log p\left(X_n; \theta_0 + \frac{P_{\theta_0}^H h}{\sqrt{N}}\right) - \log p(X_n; \theta_0) \right]
		\label{eq:lambda_nh_definition_sec7}
	\end{equation}
	admits the uniform stochastic linear-quadratic expansion over the compact coordinate cage $\mathbb{K} \subset H_{\theta_0}$:
	\begin{equation}
		\Lambda_{N,H}(h) = h^\top \Delta_{N,H}(\theta_0) - \frac{1}{2} h^\top \left( g_0\big|_{\mathcal{B}} \right) h + R_{N,H}(h),
		\label{eq:uniform_lan_expansion_formula_sec7}
	\end{equation}
	where:
	\begin{enumerate}[label=\textnormal{(\roman*)}]
		\item The horizontally projected sample score vector $\Delta_{N,H}(\theta_0) \equiv \frac{1}{\sqrt{N}} \sum_{n=1}^N P_{\theta_0}^H \nabla_\theta \log p(X_n; \theta_0)$ converges weakly under $P_{\theta_0}^{(N)}$ to a $d$-dimensional multivariate Gaussian vector:
		\begin{equation}
			\Delta_{N,H}(\theta_0) \xrightarrow{d} \Delta_{\infty,H} \sim \mathcal{N}\left( \mathbf{0}, \, g_0\big|_{\mathcal{B}} \right);
			\label{eq:score_weak_limit_h_sec7}
		\end{equation}
		\item The non-linear stochastic remainder process $R_{N,H}(h)$ vanishes uniformly in probability over $\mathbb{K}$:
		\begin{equation}
			\sup_{h \in \mathbb{K}} \left| R_{N,H}(h) \right| = \mathcal{O}_p\left( N^{-1/2} \right) \xrightarrow{P_{\theta_0}^{(N)}} 0 \quad \text{as } N \to \infty.
			\label{eq:remainder_uniform_p_convergence_sec7}
		\end{equation}
	\end{enumerate}
\end{lemma}

\begin{proof}
	We execute the deduction in four sequential steps.
	
	\paragraph{\textbf{Step 1: Multivariable Taylor Expansion of Single-Observation Log-Likelihoods.}}
	Fix $h \in \mathbb{K} \subset H_{\theta_0}$. Under High-Order Local Regularity (Assumption \ref{asm:local_regularity}), the single-observation log-likelihood $l_n(\theta) \equiv \log p(X_n; \theta)$ is three times continuously differentiable on an open neighborhood $\mathcal{U}(\theta_0) \subset \mathrm{Int}(\Theta)$. Performing a third-order multivariate Taylor series expansion of $l_n\left(\theta_0 + \frac{P_{\theta_0}^H h}{\sqrt{N}}\right)$ around $h = \mathbf{0}$:
	\begin{align}
		\small 
		l_n\left(\theta_0 + \frac{P_{\theta_0}^H h}{\sqrt{N}}\right) - l_n(\theta_0) &= \frac{1}{\sqrt{N}} \left(P_{\theta_0}^H h\right)^\top \nabla_\theta l_n(\theta_0) + \frac{1}{2N} \left(P_{\theta_0}^H h\right)^\top \nabla_\theta^2 l_n(\theta_0) \left(P_{\theta_0}^H h\right) \nonumber \\
		&\quad + \frac{1}{6 N^{3/2}} \sum_{i,j,k=1}^D \left(P_{\theta_0}^H h\right)^i \left(P_{\theta_0}^H h\right)^j \left(P_{\theta_0}^H h\right)^k \int_0^1 3(1-t)^2 \frac{\partial^3 l_n\left(\theta_0 + t \frac{P_{\theta_0}^H h}{\sqrt{N}}\right)}{\partial \theta^i \partial \theta^j \partial \theta^k} dt.
		\label{eq:taylor_expansion_single_obs_sec7}
	\end{align}
	
	\paragraph{\textbf{Step 2: Linear Term and Multivariate Central Limit Theorem.}}
	Summing the first-order linear term in Equation~\eqref{eq:taylor_expansion_single_obs_sec7} over $n = 1, \dots, N$:
	\begin{equation}
		\sum_{n=1}^N \frac{1}{\sqrt{N}} \left(P_{\theta_0}^H h\right)^\top \nabla_\theta l_n(\theta_0) = h^\top P_{\theta_0}^H \left( \frac{1}{\sqrt{N}} \sum_{n=1}^N \nabla_\theta l_n(\theta_0) \right) = h^\top \Delta_{N,H}(\theta_0).
		\label{eq:linear_term_sum_sec7}
	\end{equation}
	The horizontally projected score vectors $S_{n,H} \equiv P_{\theta_0}^H \nabla_\theta l_n(\theta_0)$ are independent and identically distributed ($i.i.d.$) under $P_{\theta_0}^{(N)}$, satisfying:
	\begin{align}
		\mathbb{E}_{\theta_0}\left[ S_{n,H} \right] &= P_{\theta_0}^H \mathbb{E}_{\theta_0}\left[ \nabla_\theta l_n(\theta_0) \right] = \mathbf{0}, \label{eq:score_mean_zero_sec7} \\
		\mathrm{Var}_{\theta_0}\left( S_{n,H} \right) &= \mathbb{E}_{\theta_0}\left[ S_{n,H} S_{n,H}^\top \right] = P_{\theta_0}^H g^{(1)}(\theta_0) P_{\theta_0}^H \equiv g_0\big|_{\mathcal{B}}. \label{eq:score_cov_g0_sec7}
	\end{align}
	By the Multivariate Central Limit Theorem applied to the normalized sum $$\Delta_{N,H}(\theta_0) = \frac{1}{\sqrt{N}} \sum_{n=1}^N S_{n,H},$$ we obtain weak Gaussian convergence:
	\begin{equation}
		\Delta_{N,H}(\theta_0) \xrightarrow{d} \Delta_{\infty,H} \sim \mathcal{N}\left( \mathbf{0}, \, g_0\big|_{\mathcal{B}} \right).
		\label{eq:clt_score_proven_sec7}
	\end{equation}
	
	\paragraph{\textbf{Step 3: Second-Order Term and Metric Stabilization.}}
	Summing the second-order term in Equation~\eqref{eq:taylor_expansion_single_obs_sec7} over $n = 1, \dots, N$ and introducing the expected single-observation Hessian $\mathbb{E}_{\theta_0}\left[ \nabla_\theta^2 l_n(\theta_0) \right] = -g^{(1)}(\theta_0)$:
	\begin{align}
		\frac{1}{2N} \sum_{n=1}^N \left(P_{\theta_0}^H h\right)^\top \nabla_\theta^2 l_n(\theta_0) \left(P_{\theta_0}^H h\right) &= -\frac{1}{2} h^\top \left( P_{\theta_0}^H g^{(1)}(\theta_0) P_{\theta_0}^H \right) h \nonumber \\
		&\quad + \frac{1}{2} h^\top P_{\theta_0}^H \left( \frac{1}{N} \sum_{n=1}^N \nabla_\theta^2 l_n(\theta_0) + g^{(1)}(\theta_0) \right) P_{\theta_0}^H h.
		\label{eq:second_order_sum_split_sec7}
	\end{align}
	The lead term matches $-\frac{1}{2} h^\top \left( g_0\big|_{\mathcal{B}} \right) h$. For the matrix deviation term $$Q_N \equiv \frac{1}{N} \sum_{n=1}^N \nabla_\theta^2 l_n(\theta_0) + g^{(1)}(\theta_0),$$ Khintchine's Weak Law of Large Numbers implies $Q_N \xrightarrow{P_{\theta_0}^{(N)}} \mathbf{0}_{D \times D}$. Because $h \in \mathbb{K}$ with $\|h\|_2 \le R_\mathbb{K} < \infty$ and $\|P_{\theta_0}^H\|_{\mathrm{op}} < \infty$:
	\begin{equation}
		\sup_{h \in \mathbb{K}} \left| \frac{1}{2} h^\top P_{\theta_0}^H Q_N P_{\theta_0}^H h \right| \le \frac{1}{2} R_\mathbb{K}^2 \left\| P_{\theta_0}^H \right\|_{\mathrm{op}}^2 \left\| Q_N \right\|_{\mathrm{op}} = \mathcal{O}_p\left( N^{-1/2} \right).
		\label{eq:second_order_dev_bound_sec7}
	\end{equation}
	
	\paragraph{\textbf{Step 4: Geometric Control of the Third-Order Remainder.}}
	By Lemma~\ref{lem:pillar_2_taylor_geometry} and Theorem~\ref{thm:generalized_horizontal_log_likelihood}, the third-order log-density sum is governed by the pulled-back connection symbols $\left[\psi_{N,H}^* \nabla^{(\alpha,N)}\right]_{ijk}(h)$ and pulled-back Riemann curvature $\left[\psi_{N,H}^* \mathcal{R}^{(\alpha,N)}\right]_{ijmn}(h)$:
	\begin{equation}
		\sup_{h \in \mathbb{K}} \left| \sum_{n=1}^N r_n^{(3)}(h) \right| \le \frac{\|h\|_2^3}{6} \sup_{h \in \mathbb{K}} \left| \left[ \psi_{N,H}^* \nabla^{(\alpha,N)} \right]_{ijk}(h) \right| + \mathcal{O}_p\left( \sup_{h \in \mathbb{K}} \left\| \left[ \psi_{N,H}^* \mathcal{R}^{(\alpha,N)} \right]_{ijmn}(h) \right\|_\infty \right).
		\label{eq:third_order_geometric_envelope_sec7}
	\end{equation}
	Applying Theorem~\ref{thm:connection_dissolution} (connection dissolution at rate $\mathcal{O}(N^{-1/2})$) and Theorem~\ref{thm:curvature_annihilation} (accelerated curvature annihilation at rate $\mathcal{O}_p(N^{-1})$) directly to Equation~\eqref{eq:third_order_geometric_envelope_sec7}:
	\begin{equation}
		\sup_{h \in \mathbb{K}} \left| \sum_{n=1}^N r_n^{(3)}(h) \right| \le \frac{R_\mathbb{K}^3}{6} \mathcal{O}\left( N^{-1/2} \right) + \mathcal{O}_p\left( N^{-1} \right) = \mathcal{O}_p\left( N^{-1/2} \right).
		\label{eq:third_order_rate_proven_sec7}
	\end{equation}
	Combining Equations~\eqref{eq:second_order_dev_bound_sec7} and \eqref{eq:third_order_rate_proven_sec7} completes the proof of Equation~\eqref{eq:remainder_uniform_p_convergence_sec7} and establishes Lemma~\ref{lem:uniform_horizontal_log_likelihood_expansion}.
\end{proof}

\begin{lemma}[Mutual Contiguity on Horizontal Leaves]
	\label{lem:horizontal_leaf_contiguity}
	For any sequence of local displacement vectors $h_N \in \mathbb{K} \subset H_{\theta_0}$, the sequence of $N$-sample probability measures $P_{\theta_N(h_N)}^{(N)}$ along the horizontal trajectory $\theta_N(h_N) \equiv \theta_0 + \frac{1}{\sqrt{N}} P_{\theta_0}^H h_N$ and the background measure $P_{\theta_0}^{(N)}$ are mutually contiguous:
	\begin{equation}
		P_{\theta_N(h_N)}^{(N)} \triangleleft \triangleright P_{\theta_0}^{(N)}.
		\label{eq:mutual_contiguity_leaf_formula_sec7}
	\end{equation}
\end{lemma}

\begin{proof}
	We apply Le Cam's First Lemma \cite{lecam1986, vandervaart1998}. Under $P_{\theta_0}^{(N)}$, the horizontal log-likelihood ratio process $\Lambda_{N,H}(h_N)$ defined in Lemma~\ref{lem:uniform_horizontal_log_likelihood_expansion} satisfies the weak limit:
	\begin{equation}
		\Lambda_{N,H}(h_N) \xrightarrow{d} \Lambda_{\infty,H} \sim \mathcal{N}\left( -\frac{1}{2} \sigma_H^2, \, \sigma_H^2 \right), \quad \text{where } \sigma_H^2 \equiv h^\top \left( g_0\big|_{\mathcal{B}} \right) h.
		\label{eq:weak_limit_likelihood_ratio_contiguity}
	\end{equation}
	Evaluating the expectation of the limiting Radon--Nikodym derivative $L_\infty \equiv \exp(\Lambda_{\infty,H})$:
	\begin{equation}
		\mathbb{E}\left[ e^{\Lambda_{\infty,H}} \right] = \exp\left( -\frac{1}{2} \sigma_H^2 + \frac{1}{2} \sigma_H^2 \right) = e^0 = 1.
		\label{eq:expectation_radon_nikodym_one}
	\end{equation}
	By Le Cam's First Lemma (Condition 2), an expectation of 1 for the limiting likelihood ratio is necessary and sufficient for mutual contiguity $P_{\theta_N(h_N)}^{(N)} \triangleleft \triangleright P_{\theta_0}^{(N)}$, completing the proof.
\end{proof}

\begin{theorem}[Global Experiment Deficiency Convergence under Le Cam Distance]
	\label{thm:global_experiment_deficiency_convergence}
	Under High-Order Local Regularity (Assumption \ref{asm:local_regularity}) and the Fisher-compatible Ehresmann Connection framework (Definition \ref{def:ehresmann_connection}), for every compact coordinate cage $\mathbb{K} \subset H_{\theta_0} \cong \mathbb{R}^d$, the Le Cam experiment deficiency distance between the horizontally restricted empirical experiment $\mathcal{E}_{N,H}(\mathbb{K})$ and the canonical Gaussian shift limit experiment $\mathcal{E}_\infty(\mathcal{B}; \mathbb{K})$ vanishes asymptotically:
	\begin{equation}
		\lim_{N \to \infty} \Delta\left( \mathcal{E}_{N,H}(\mathbb{K}), \, \mathcal{E}_\infty(\mathcal{B}; \mathbb{K}) \right) = 0.
		\label{eq:global_lecam_deficiency_limit_formula_sec7}
	\end{equation}
	Furthermore, the rate of operational experiment convergence satisfies:
	\begin{equation}
		\Delta\left( \mathcal{E}_{N,H}(\mathbb{K}), \, \mathcal{E}_\infty(\mathcal{B}; \mathbb{K}) \right) = \mathcal{O}\left( N^{-1/2} \right).
		\label{eq:lecam_deficiency_rate_formula_sec7}
	\end{equation}
\end{theorem}

\begin{proof}
	We execute a step-by-step measure-theoretic and decision-theoretic proof in four stages.
	
	\paragraph{\textbf{Step 1: Construction of Randomization Markov Kernels.}}
	To bound the Le Cam deficiency distance $\Delta(\mathcal{E}_{N,H}, \mathcal{E}_\infty) \equiv \max\left\{ \delta(\mathcal{E}_{N,H}, \mathcal{E}_\infty), \, \delta(\mathcal{E}_\infty, \mathcal{E}_{N,H}) \right\}$, we construct explicit Markov transition kernels (randomization operators) between the sample observation space $\mathcal{X}^N$ and the Gaussian space $\mathbb{R}^d$:
	\begin{enumerate}
		\item \textbf{Forward Kernel $K_N : \mathcal{X}^N \to \mathcal{P}(\mathbb{R}^d)$:} Maps sample data $X^N \in \mathcal{X}^N$ to a randomized location statistic $Y \in \mathbb{R}^d$ via $Y \equiv \left( g_0\big|_{\mathcal{B}} \right)^{-1} \Delta_{N,H}(\theta_0)$, where $$\Delta_{N,H}(\theta_0) = \frac{1}{\sqrt{N}} \sum_{n=1}^N P_{\theta_0}^H \nabla_\theta \log p(X_n; \theta_0).$$
		\item \textbf{Reverse Kernel $K_\infty : \mathbb{R}^d \to \mathcal{P}(\mathcal{X}^N)$:} Draws synthetic observation vectors $X^N \in \mathcal{X}^N$ from the localized product probability measure $P_{\theta_0 + \frac{1}{\sqrt{N}} P_{\theta_0}^H Y}^{(N)}$ given a Gaussian observation $Y \sim \mathcal{N}\left( h, \left(g_0\big|_{\mathcal{B}}\right)^{-1} \right)$.
	\end{enumerate}
	
	\paragraph{\textbf{Step 2: Bounding Forward Deficiency $\delta\left( \mathcal{E}_{N,H}(\mathbb{K}), \mathcal{E}_\infty(\mathcal{B}; \mathbb{K}) \right)$.}}
	By definition of experiment deficiency \cite{lecam1986, torgersen1991}:
	\begin{equation}
		\delta\left( \mathcal{E}_{N,H}(\mathbb{K}), \mathcal{E}_\infty(\mathcal{B}; \mathbb{K}) \right) \le \sup_{h \in \mathbb{K}} \left\| K_N P_{\theta_0 + \frac{P_{\theta_0}^H h}{\sqrt{N}}}^{(N)} - \mathcal{N}\left( h, \, \left(g_0\big|_{\mathcal{B}}\right)^{-1} \right) \right\|_{\mathrm{TV}}.
		\label{eq:forward_deficiency_tv_bound_sec7}
	\end{equation}
	Applying Scheff\'e's Lemma and Pinsker's Inequality \cite{vandervaart1998}, the total variation distance is bounded by the Kullback--Leibler (KL) divergence between probability measures:
	\begin{equation}
		\left\| P - Q \right\|_{\mathrm{TV}} \le \sqrt{\frac{1}{2} D_{\mathrm{KL}}(P \parallel Q)}.
		\label{eq:pinskers_inequality_sec7}
	\end{equation}
	Evaluating the KL divergence between the pushed-forward score statistic distribution $K_N P_{\theta_N(h)}^{(N)}$ and the target Gaussian distribution $\mathcal{N}\left( h, \left(g_0\big|_{\mathcal{B}}\right)^{-1} \right)$:
	\begin{align}
		D_{\mathrm{KL}}\left( K_N P_{\theta_N(h)}^{(N)} \parallel \mathcal{N}\left( h, \, \left(g_0\big|_{\mathcal{B}}\right)^{-1} \right) \right) &= \mathbb{E}_{P_{\theta_N(h)}^{(N)}} \left[ \Lambda_{N,H}(h) - \left( h^\top \Delta_{\infty,H} - \frac{1}{2} h^\top \left(g_0\big|_{\mathcal{B}}\right) h \right) \right].
		\label{eq:kl_divergence_expansion_sec7}
	\end{align}
	Substituting the uniform stochastic expansion from Lemma~\ref{lem:uniform_horizontal_log_likelihood_expansion}:
	\begin{equation}
		\sup_{h \in \mathbb{K}} D_{\mathrm{KL}}\left( K_N P_{\theta_N(h)}^{(N)} \parallel \mathcal{N}\left( h, \, \left(g_0\big|_{\mathcal{B}}\right)^{-1} \right) \right) = \sup_{h \in \mathbb{K}} \mathbb{E}_{P_{\theta_N(h)}^{(N)}} \left[ \left| R_{N,H}(h) \right| \right] = \mathcal{O}\left( N^{-1} \right).
		\label{eq:kl_divergence_rate_proven_sec7}
	\end{equation}
	Substituting Equation~\eqref{eq:kl_divergence_rate_proven_sec7} into Pinsker's Inequality Equation~\eqref{eq:pinskers_inequality_sec7} yields:
	\begin{equation}
		\delta\left( \mathcal{E}_{N,H}(\mathbb{K}), \mathcal{E}_\infty(\mathcal{B}; \mathbb{K}) \right) \le \sqrt{\frac{1}{2} \mathcal{O}\left( N^{-1} \right)} = \mathcal{O}\left( N^{-1/2} \right).
		\label{eq:forward_deficiency_rate_proven_sec7}
	\end{equation}
	
	\paragraph{\textbf{Step 3: Bounding Reverse Deficiency $\delta\left( \mathcal{E}_\infty(\mathcal{B}; \mathbb{K}), \mathcal{E}_{N,H}(\mathbb{K}) \right)$.}}
	Symmetrically, applying the reverse kernel $K_\infty$ under mutual contiguity $P_{\theta_N(h)}^{(N)} \triangleleft \triangleright P_{\theta_0}^{(N)}$ (Lemma~\ref{lem:horizontal_leaf_contiguity}):
	\begin{equation}
		\delta\left( \mathcal{E}_\infty(\mathcal{B}; \mathbb{K}), \mathcal{E}_{N,H}(\mathbb{K}) \right) \le \sup_{h \in \mathbb{K}} \left\| K_\infty \mathcal{N}\left( h, \, \left(g_0\big|_{\mathcal{B}}\right)^{-1} \right) - P_{\theta_0 + \frac{P_{\theta_0}^H h}{\sqrt{N}}}^{(N)} \right\|_{\mathrm{TV}} = \mathcal{O}\left( N^{-1/2} \right).
		\label{eq:reverse_deficiency_rate_proven_sec7}
	\end{equation}
	
	\paragraph{\textbf{Step 4: Assembly of Symmetric Le Cam Distance.}}
	Combining Equations~\eqref{eq:forward_deficiency_rate_proven_sec7} and \eqref{eq:reverse_deficiency_rate_proven_sec7}:
	\begin{align}
		\Delta\left( \mathcal{E}_{N,H}(\mathbb{K}), \, \mathcal{E}_\infty(\mathcal{B}; \mathbb{K}) \right) &= \max\left\{ \delta\left(\mathcal{E}_{N,H}, \mathcal{E}_\infty\right), \, \delta\left(\mathcal{E}_\infty, \mathcal{E}_{N,H}\right) \right\} = \mathcal{O}\left( N^{-1/2} \right).
		\label{eq:lecam_distance_max_assembly_sec7}
	\end{align}
	Taking $N \to \infty$ proves Equation~\eqref{eq:global_lecam_deficiency_limit_formula_sec7} and completes the proof of Theorem~\ref{thm:global_experiment_deficiency_convergence}.
\end{proof}


\subsection{The Master Second Edge Theorem}
\label{subsec:master_second_edge_theorem}

Building upon the metric-topological leaf shrinkage established in Section~\ref{subsec:mod2_degeneration} (Theorem~\ref{thm:metric_stabilization}), the Gromov-Hausdorff convergence in Section 7.3 (Theorem~\ref{thm:gh_leaf_collapse}), the accelerated holonomy group annihilation in Section 7.4 (Theorem~\ref{thm:holonomy_annihilation}), and the global Le Cam experiment deficiency uniformization in Section 7.5 (Theorem~\ref{thm:global_experiment_deficiency_convergence}), we now state and prove the central theoretical result of this paper: the \textbf{Master Second Edge Theorem}.

This theorem synthesizes differential geometry, geometric analysis, and asymptotic decision theory into a single, unified phase transition law, proving that the sequence of curved $N$-sample joint information manifolds $IG_N|_{H_\theta}$ collapses globally onto the canonical flat tangent canvas of Conventional Statistics $\mathcal{M}_\infty \equiv (T_{\theta_0}IG_1, g_0|_\mathcal{B}, \nabla^{(0)}) \equiv \mathrm{CS}$ as $N \to \infty$.

\begin{theorem}[The Master Second Edge Theorem: Global Phase Transition and Geometric Collapse]
	\label{thm:master_second_edge}
	Let $\mathcal{M}_{\mathrm{univ}} \equiv \left( \mathcal{X}, \mu, \Theta, \mathcal{P}, IG_1, IG_N, (\mathcal{X}^N, \mathcal{A}^{\otimes N}, P_{\theta_0}^{\otimes N}) \right)$ be the ambient statistical universe (Definition~\ref{def:ambient_spaces}) satisfying High-Order Local Regularity (Assumption~\ref{asm:local_regularity}). Let $(\Theta, \mathcal{B}, \pi, \mathcal{F})$ be the statistical fiber bundle (Definition~\ref{def:statistical_fiber_bundle}) equipped with the Fisher-compatible Ehresmann connection $T_\theta \Theta = H_\theta \oplus V_\theta$ (Definition~\ref{def:ehresmann_connection}).
	
	As the sample size $N \to \infty$, the sequence of $N$-sample joint information-geometric manifolds restricted to the non-singular horizontal distribution $IG_N|_{H_\theta} \equiv \left(\Theta|_{H_\theta}, G^{(N)}|_{H_\theta}, \nabla^{(\alpha,N)}|_{H_\theta}\right)$ undergoes a global asymptotic phase transition and metric-topological-operational collapse onto the canonical flat tangent canvas of Conventional Statistics $\mathcal{M}_\infty \equiv \left( T_{\theta_0} IG_1, g_0|_\mathcal{B}, \nabla^{(0)} \right) \equiv \mathrm{CS}$.
	
	Specifically, the global phase transition $IG_N|_{H_\theta} \longrightarrow \mathcal{M}_\infty \equiv \mathrm{CS}$ is governed by four synchronized limit theorems across every compact coordinate cage $\mathbb{K} \subset H_{\theta_0} \cong \mathbb{R}^d$:
	\begin{enumerate}[label=\textnormal{(\roman*)}]
		\item \textbf{Gromov-Hausdorff Metric Leaf Collapse:}
		The horizontal Riemannian leaf space sequence $(\mathcal{L}_{N, \theta_0}, \widetilde{G}_H^{(N)})$ converges in the Gromov-Hausdorff metric topology to the flat Euclidean-Fisher space $(\mathbb{K}, d_{g_0|_\mathcal{B}})$:
		\begin{equation}
			\lim_{N \to \infty} d_{\mathrm{GH}}\left( \left(\mathbb{K}, d_{\widetilde{G}_H^{(N)}}\right), \, \left(\mathbb{K}, d_{g_0|_\mathcal{B}}\right) \right) = 0,
			\label{eq:master_gh_collapse}
		\end{equation}
		with an exact metric shrinkage rate of $\mathcal{O}\left(N^{-1/2}\right)$.
		
		\item \textbf{Accelerated Holonomy Annihilation and Parallel Transport Trivialization:}
		The restricted holonomy Lie group $\mathrm{Hol}_{\theta_0}^0\left(\widetilde{\nabla}^{(\alpha,N)}\right) \subseteq GL(d, \mathbb{R})$ collapses to the trivial identity subgroup $\{I_{d \times d}\}$:
		\begin{equation}
			\lim_{N \to \infty} \mathrm{diam}_{\mathrm{op}}\left( \mathrm{Hol}_{\theta_0}^0\left(\widetilde{\nabla}^{(\alpha,N)}\right) \right) = 0 \iff \mathrm{Hol}_{\theta_0}^0\left(\widetilde{\nabla}^{(\alpha,N)}\right) \xrightarrow{N \to \infty} \{I_{d \times d}\},
			\label{eq:master_holonomy_collapse}
		\end{equation}
		at the accelerated quadratic rate $\mathcal{O}_p\left(N^{-1}\right)$, forcing pulled-back Amari connection symbols to dissolve ($\widetilde{\Gamma}^{(\alpha,N)} \to \Gamma^{(0)} \equiv 0$) and Riemann curvature to vanish ($\widetilde{\mathcal{R}}^{(\alpha,N)} \to \mathcal{R}^{(0)} \equiv 0$).
		
		\item \textbf{Le Cam Operational Experiment Deficiency Convergence:}
		The sequence of localized empirical statistical experiments $\mathcal{E}_{N,H}(\mathbb{K})$ converges locally in the sense of Le Cam to the canonical Gaussian shift limit experiment $\mathcal{E}_\infty(\mathcal{B}; \mathbb{K})$:
		\begin{equation}
			\lim_{N \to \infty} \Delta\left( \mathcal{E}_{N,H}(\mathbb{K}), \, \mathcal{E}_\infty(\mathcal{B}; \mathbb{K}) \right) = 0,
			\label{eq:master_lecam_collapse}
		\end{equation}
		under Le Cam deficiency distance at rate $\mathcal{O}\left(N^{-1/2}\right)$.
		
		\item \textbf{Cheeger-Gromov $C^\infty$-Manifold Collapse and Epistemological Identity:}
		The pulled-back manifold sequence $\psi_{N,H}^*(IG_N|_{H_\theta})$ converges smoothly in the Cheeger-Gromov $C^\infty$-manifold topology to $\mathcal{M}_\infty$:
		\begin{equation}
			\psi_{N,H}^*(IG_N|_{H_\theta}) \xrightarrow[\mathrm{Cheeger-Gromov}]{C^\infty} \mathcal{M}_\infty \equiv \mathrm{CS}.
			\label{eq:master_cheeger_gromov}
		\end{equation}
	\end{enumerate}
\end{theorem}
\subsubsection{Motivation and Explanation of Theorem \ref{thm:master_second_edge} (The Master Second Edge Theorem)}
\label{subsubsec:theorem15_motivation_explanation}

\paragraph{1. Architectural Motivation}
Classical asymptotic statistical theory historically treats the transition from curved, finite-sample statistical models to flat limit experiments as an ad-hoc analytical convenience or a passive collection of scalar point limits. However, finite-sample statistical parametric families are intrinsically curved, dynamically expanding Riemannian-affine manifolds $IG_N = (\Theta, G^{(N)}, \nabla^{(\alpha,N)})$ possessing non-zero Amari-Riemann curvature $\mathcal{R}^{(\alpha,N)}$, non-Euclidean parallel transport friction $\Gamma^{(\alpha,N)}$, and potential rank-degeneracies ($\det G^{(N)} = 0$) in over-parameterized regimes.

The primary objective of Theorem \ref{thm:master_second_edge} (\emph{The Master Second Edge Theorem: Global Phase Transition and Geometric Collapse}) is to construct a rigorous, unified, and coordinate-free foundation for the global asymptotic metamorphosis $IG_N|_{H_\theta} \longrightarrow \mathcal{M}_\infty \equiv \mathrm{CS}$ as $N \to \infty$. It establishes that Conventional Statistics ($\mathrm{CS}$) is not merely an intuitive linear approximation, but rather the unique, mathematically invariant, zero-curvature attractor state $\mathcal{M}_\infty \equiv \left(T_{\theta_0}IG_1, g_0|_{\mathcal{B}}, \nabla^{(0)}\right)$ toward which all regular information manifolds dynamically collapse in the thermodynamic limit.

\paragraph{2. Explanation of the Four Synchronized Limit Conditions}
The core assertion of the Second Edge Theorem is that the global phase transition 
\[
IG_N|_{H_\theta} \longrightarrow \mathcal{M}_\infty \equiv \mathrm{CS}
\]
is governed by four synchronized limit theorems across every compact coordinate cage $\mathbb{K} \subset H_{\theta_0} \cong \mathbb{R}^d$. Crucially, these four conditions are \textbf{necessary and sufficient conditions} for the complete statement and realization of the Second Edge Theorem. They operate in tight structural locks across the differential, topological, metric, and decision-theoretic dimensions of the manifold:

\begin{enumerate}[label=\textnormal{(\roman*)}]
	\item \textbf{Gromov-Hausdorff Metric Leaf Collapse (Metric-Topological Necessity \& Sufficiency):}
	\begin{equation}
		\lim_{N\to\infty} d_{\mathrm{GH}}\left(\left(\mathbb{K}, d_{\tilde{G}_H^{(N)}}\right), \left(\mathbb{K}, d_{g_0|_{\mathcal{B}}}\right)\right) = 0 \quad \text{at rate } \mathcal{O}\left(N^{-1/2}\right).
	\end{equation}
	\emph{Explanation:} This condition resolves the Moving Target Crisis by proving that the sequence of horizontal Riemannian leaf spaces $(\mathcal{L}_{N,\theta_0}, \tilde{G}_H^{(N)})$ collapses in the Gromov-Hausdorff metric topology onto the flat Euclidean-Fisher space $(\mathbb{K}, d_{g_0|_{\mathcal{B}}})$. It guarantees exact metric capacity-cotangent scale cancellation ($N^{1 - 2/2} = N^0 = 1$), fixing the Signal-to-Noise Ratio ($\mathrm{SNR}_N(h) = \mathcal{O}(1)$) without loss of topological dimension.
	
	\item \textbf{Accelerated Holonomy Annihilation and Parallel Transport Trivialization (Differential-Affine Necessity \& Sufficiency):}
	\begin{equation}
		\lim_{N\to\infty} \mathrm{diam}_{\mathrm{op}}\left(\mathrm{Hol}_{\theta_0}^0\left(\tilde{\nabla}^{(\alpha,N)}\right)\right) = 0 \iff \mathrm{Hol}_{\theta_0}^0\left(\tilde{\nabla}^{(\alpha,N)}\right) \xrightarrow{N\to\infty} \{I_{d \times d}\} \quad \text{at rate } \mathcal{O}_p\left(N^{-1}\right).
	\end{equation}
	\emph{Explanation:} By the Ambrose-Singer Holonomy Theorem, the restricted holonomy Lie group collapses to the trivial identity subgroup $\{I_{d \times d}\}$ at the accelerated quadratic decay rate $\mathcal{O}_p(N^{-1})$. This condition is necessary and sufficient to force connection friction to dissolve ($\tilde{\Gamma}^{(\alpha,N)} \to \Gamma^{(0)} \equiv 0$) and Riemann curvature to vanish ($\tilde{\mathcal{R}}^{(\alpha,N)} \to \mathcal{R}^{(0)} \equiv 0$). It eradicates non-Euclidean path-dependent terrain friction, ensuring path-independent parallel transport and enforcing the unconditional asymptotic identity of the classical test trinity ($W_S = W_{LR} = W_W + o_p(1)$).
	
	\item \textbf{Le Cam Operational Experiment Deficiency Convergence (Decision-Theoretic Necessity \& Sufficiency):}
	\begin{equation}
		\lim_{N\to\infty} \Delta\left(\mathcal{E}_{N,H}(\mathbb{K}), \mathcal{E}_\infty(\mathcal{B};\mathbb{K})\right) = 0 \quad \text{at rate } \mathcal{O}\left(N^{-1/2}\right).
	\end{equation}
	\emph{Explanation:} This condition ensures that the sequence of localized empirical statistical experiments $\mathcal{E}_{N,H}(\mathbb{K})$ converges uniformly in operational decision risk to the canonical Gaussian shift experiment $\mathcal{E}_\infty(\mathcal{B};\mathbb{K})$ under Le Cam deficiency distance $\Delta$. It proves that decision rules on $IG_N|_{H_\theta}$ achieve zero second-order information loss ($\Delta I_N = 0$).
	
	\item \textbf{Cheeger-Gromov $C^\infty$-Manifold Collapse and Epistemological Identity (Global Structural Necessity \& Sufficiency):}
	\begin{equation}
		\psi_{N,H}^*\left(IG_N|_{H_\theta}\right) \xrightarrow[\text{Cheeger-Gromov}]{C^\infty} \mathcal{M}_\infty \equiv \mathrm{CS}.
	\end{equation}
	\emph{Explanation:} Grounded in Élie Cartan's Structural Determinacy Theorem, uniform $C^\infty$-bounds on metric components, connection symbols, and curvature tensors guarantee smooth manifold convergence in the Cheeger-Gromov topology, uniquely determining $\mathcal{M}_\infty$ without higher-order topological anomalies.
\end{enumerate}

\paragraph{3. Necessity and Sufficiency of the Fourfold Unification}
These four conditions are strictly \textbf{necessary and sufficient} because, by the Double Completeness Architecture and the Geometric-Operational Equivalence Theorem:
\begin{itemize}
	\item \textbf{Necessity:} If any single condition fails (e.g., if holonomy does not collapse or if Le Cam deficiency remains non-zero), residual curvature or decision risk discrepancies persist, preventing global collapse onto the flat, zero-curvature canvas $\mathcal{M}_\infty \equiv \mathrm{CS}$.
	\item \textbf{Sufficiency:} Fulfilling conditions (i)--(iv) completely specifies the limiting differential-geometric and decision-theoretic state, leaving no remaining degrees of freedom or unobserved topological distortions outside these four geometric fields ($r \in \{1, 2, 3, 4\}$).
\end{itemize}

Thus, structural manifold flattening ($\mathrm{Way\ 1}$) and operational decision risk condensation ($\mathrm{Way\ 2}$) are demonstrated to be non-separable dual projections of a single, unified asymptotic phase transition governed by these four synchronized limit theorems.

\begin{proof} {\bf Proof of Theorem \ref{thm:master_second_edge}:}
	
	The proof proceeds in four unified mathematical stages, synthesizing the structural, topological, differential, and decision-theoretic results established throughout the manuscript.
	
	\paragraph{Stage 1: Proof of Metric-Topological Leaf Collapse (Statement i).}
	Let $\mathbb{K} \subset H_{\theta_0} \cong \mathbb{R}^d$ be a compact coordinate cage. Under the horizontal localized microscope map $\psi_{N,H}(h) \equiv \theta_0 + \frac{1}{\sqrt{N}} P_{\theta_0}^H h$ (Definition~\ref{def:microscope_map}), the multivariable spatial Jacobian scaffold $J(\psi_{N,H}) = \frac{1}{\sqrt{N}} P_{\theta_0}^H$ (Lemma~\ref{lem:jacobian_scaffold}) contracts $r=2$ covariant tensor slots against $N$-sample metric additivity $G^{(N)}(\theta) = N g^{(1)}(\theta)$ (Lemma~\ref{lem:additivity_fisher_metric}). 
	
	By Lemma~\ref{lem:horizontal_pulled_back_metric_algebraic}, the pulled-back horizontal metric field obeys exact scale factor cancellation:
	\begin{equation}
		\widetilde{G}_H^{(N)}(h) \equiv \left[\psi_{N,H}^* G^{(N)}\right](h) = P_{\psi_{N,H}(h)}^H g^{(1)}\left(\theta_0 + \frac{P_{\theta_0}^H h}{\sqrt{N}}\right) P_{\psi_{N,H}(h)}^H.
		\label{eq:proof_stage1_metric_cancellation}
	\end{equation}
	By Lemma~
	\ref{lem:uniform_horizontal_metric_stabilization}, 
	 $\widetilde{G}_H^{(N)}(h)$ converges uniformly in matrix operator norm over $\mathbb{K}$ to the constant, non-singular base metric $g_0|_\mathcal{B} \equiv P_{\theta_0}^H g^{(1)}(\theta_0) P_{\theta_0}^H$:
	\begin{equation}
		\sup_{h \in \mathbb{K}} \left\| \widetilde{G}_H^{(N)}(h) - g_0|_\mathcal{B} \right\|_{\mathrm{op}} \le \frac{K_\mathbb{K}^{(1)}}{\sqrt{N}} = \mathcal{O}\left(N^{-1/2}\right).
		\label{eq:proof_stage1_operator_bound}
	\end{equation}
	By Lemma~\ref{lem:geodesic_distortion_bound}, the intrinsic Riemannian geodesic distance $d_{\widetilde{G}_H^{(N)}}(h_1, h_2)$ between any $h_1, h_2 \in \mathbb{K}$ satisfies the uniform point-wise distortion bound relative to the flat Euclidean-Fisher distance $d_{g_0|_\mathcal{B}}(h_1, h_2) = \sqrt{(h_1 - h_2)^T (g_0|_\mathcal{B})(h_1 - h_2)}$ (Lemma~\ref{lem:uniform_horizontal_metric_stabilization}):
	\begin{equation}
		\sup_{h_1, h_2 \in \mathbb{K}} \left| d_{\widetilde{G}_H^{(N)}}(h_1, h_2) - d_{g_0|_\mathcal{B}}(h_1, h_2) \right| \le \frac{K_\mathbb{K}}{\sqrt{N}} = \mathcal{O}\left(N^{-1/2}\right).
		\label{eq:proof_stage1_geodesic_bound}
	\end{equation}
	Applying Theorem~\ref{thm:gh_leaf_collapse}, embedding $(\mathbb{K}, d_{\widetilde{G}_H^{(N)}})$ and $(\mathbb{K}, d_{g_0|_\mathcal{B}})$ into a common metric space yields the Gromov-Hausdorff distance bound:
	\begin{equation}
		d_{\mathrm{GH}}\left( \left(\mathbb{K}, d_{\widetilde{G}_H^{(N)}}\right), \, \left(\mathbb{K}, d_{g_0|_\mathcal{B}}\right) \right) \le \frac{1}{2} \sup_{h_1, h_2 \in \mathbb{K}} \left| d_{\widetilde{G}_H^{(N)}}(h_1, h_2) - d_{g_0|_\mathcal{B}}(h_1, h_2) \right| \le \frac{K_\mathbb{K}}{2\sqrt{N}} = \mathcal{O}\left(N^{-1/2}\right).
		\label{eq:proof_stage1_gh_final}
	\end{equation}
	Taking $N \to \infty$ establishes Statement (i).
	
	\paragraph{Stage 2: Proof of Holonomy Annihilation and Parallel Transport Trivialization (Statement ii).}
	By Theorem~\ref{thm:universal_scaling_law}, setting valence $r=3$ for Amari connection symbols and $r=4$ for the Riemann curvature tensor yields:
	\begin{align}
		\sup_{h \in \mathbb{K}} \left| \widetilde{\Gamma}_{ijk}^{(\alpha,N)}(h) \right| &= \frac{1}{\sqrt{N}} \sup_{\theta \in \mathcal{U}(\theta_0)} \left| \Gamma_{ijk}^{(\alpha,1)}(\theta) \right| = \mathcal{O}\left(N^{-1/2}\right) \xrightarrow{N \to \infty} 0, \label{eq:proof_stage2_connection} \\
		\sup_{h \in \mathbb{K}} \left\| \widetilde{\mathcal{R}}_{ijmn}^{(\alpha,N)}(h) \right\|_{\mathrm{op}} &= \frac{1}{N} \sup_{\theta \in \mathcal{U}(\theta_0)} \left\| \mathcal{R}_{ijmn}^{(\alpha,1)}(\theta) \right\|_{\mathrm{op}} = \mathcal{O}_p\left(N^{-1}\right) \xrightarrow{N \to \infty} 0. \label{eq:proof_stage2_curvature}
	\end{align}
	By Lemma~\ref{lem:ambrose_singer_bound}, applying the Ambrose-Singer Holonomy Theorem \cite{ambrose1953} bounds every generator $A \in \mathfrak{hol}_0(\widetilde{\nabla}^{(\alpha,N)})$ of the holonomy Lie algebra by $\|A\|_{\mathrm{op}} = \mathcal{O}_p(N^{-1})$. 
	
	For any contractible loop $\gamma = \partial \Sigma \subset \mathbb{K}$, applying the non-Abelian Stokes surface-ordered exponential representation \cite{frankel2011}:
	\begin{equation}
		\mathcal{P}_\gamma = \mathcal{P}\exp\left( \iint_\Sigma \mathcal{P}_{\tau_\xi}^{-1} \widetilde{\mathcal{R}}^{(\alpha,N)}(\xi) \mathcal{P}_{\tau_\xi} d\Sigma(\xi) \right) = I_{d \times d} + \iint_\Sigma A(\xi) d\Sigma(\xi) + \dots
		\label{eq:proof_stage2_stokes}
	\end{equation}
	By Grönwall bounds on parallel transport operators $\|\mathcal{P}_{\tau_\xi}\|_{\mathrm{op}} = 1 + \mathcal{O}(N^{-1/2})$ (Lemma~\ref{lem:ambrose_singer_bound}), the Dyson series expansion yields:
	\begin{equation}
		\mathrm{diam}_{\mathrm{op}}\left( \mathrm{Hol}_{\theta_0}^0\left(\widetilde{\nabla}^{(\alpha,N)}\right) \right) = \sup_{\gamma \subset \mathbb{K}} \left\| \mathcal{P}_\gamma - I_{d \times d} \right\|_{\mathrm{op}} \le \mathrm{Area}(\mathbb{K}) \cdot \mathcal{O}_p\left(N^{-1}\right) = \mathcal{O}_p\left(N^{-1}\right).
		\label{eq:proof_stage2_holonomy_final}
	\end{equation}
	Taking $N \to \infty$ forces $\mathrm{Hol}_{\theta_0}^0\left(\widetilde{\nabla}^{(\alpha,N)}\right) \to \{I_{d \times d}\}$, proving Statement (ii) and trivializing parallel transport.
	
	\paragraph{Stage 3: Proof of Le Cam Experiment Deficiency Convergence (Statement iii).}
	Consider the localized horizontal empirical experiment $\mathcal{E}_{N,H}(\mathbb{K}) \equiv \left( \mathcal{X}^N, \mathcal{A}^{\otimes N}, \left\{ P_{\theta_0 + \frac{1}{\sqrt{N}} P_{\theta_0}^H h}^{(N)} : h \in \mathbb{K} \right\} \right)$ and the limit Gaussian experiment $\mathcal{E}_\infty(\mathcal{B}; \mathbb{K}) \equiv \left( \mathbb{R}^d, \mathcal{B}^d, \left\{ \mathcal{N}\left(h, (g_0|_\mathcal{B})^{-1}\right) : h \in \mathbb{K} \right\} \right)$ (Definition~\ref{def:horizontal_localized_experiment}).
	
	By Lemma~\ref{lem:uniform_horizontal_log_likelihood_expansion}, the localized log-likelihood ratio process admits the exact functional decomposition:
	\begin{equation}
		\Lambda_{N,H}(h) = h^T \Delta_{N,H}(\theta_0) - \frac{1}{2} h^T (g_0|_\mathcal{B}) h + R_{N,H}(h),
		\label{eq:proof_stage3_lan_exp}
	\end{equation}
	where $\Delta_{N,H}(\theta_0) \xrightarrow{d} \mathcal{N}(0, g_0|_\mathcal{B})$ and $\sup_{h \in \mathbb{K}} |R_{N,H}(h)| = \mathcal{O}_p(N^{-1/2})$. By Lemma~\ref{lem:horizontal_leaf_contiguity}, probability measures $P_{\theta_N(h)}^{(N)}$ and $P_{\theta_0}^{(N)}$ are mutually contiguous.
	
	Applying Theorem~\ref{thm:global_experiment_deficiency_convergence}, we construct forward randomization Markov kernels $K_N(X^N) \equiv (g_0|_\mathcal{B})^{-1} \Delta_{N,H}(\theta_0)$. Bounding the total variation distance via Pinsker's inequality and KL divergence:
	\begin{equation}
		\delta\left( \mathcal{E}_{N,H}(\mathbb{K}), \, \mathcal{E}_\infty(\mathcal{B}; \mathbb{K}) \right) \le \sup_{h \in \mathbb{K}} \sqrt{\frac{1}{2} D_{\mathrm{KL}}\left( K_N P_{\theta_N(h)}^{(N)} \,||\, \mathcal{N}\left(h, (g_0|_\mathcal{B})^{-1}\right) \right)} = \sqrt{\mathcal{O}\left(N^{-1}\right)} = \mathcal{O}\left(N^{-1/2}\right).
		\label{eq:proof_stage3_deficiency_bound}
	\end{equation}
	Symmetric evaluation for reverse deficiency $\delta(\mathcal{E}_\infty, \mathcal{E}_{N,H}) = \mathcal{O}(N^{-1/2})$ yields:
	\begin{equation}
		\Delta\left( \mathcal{E}_{N,H}(\mathbb{K}), \, \mathcal{E}_\infty(\mathcal{B}; \mathbb{K}) \right) = \mathcal{O}\left(N^{-1/2}\right) \xrightarrow{N \to \infty} 0,
		\label{eq:proof_stage3_lecam_final}
	\end{equation}
	establishing Statement (iii).
	
	\paragraph{Stage 4: Proof of Cheeger-Gromov $C^\infty$-Convergence and Epistemological Identity (Statement iv).}
	By Proposition~\ref{prop:pillar_1_completeness} (Pillar I: Differential-Geometric Completeness), a smooth Riemannian-affine manifold $(\mathcal{M}, g, \nabla)$ is locally uniquely determined up to isometric affine diffeomorphisms by its metric tensor field $g$, connection symbols $\Gamma_{ij}^k$, and Riemann curvature tensor field $\mathcal{R}_{jmn}^i$ \cite{cheeger1970, gromov2007}.
	
	Combining Statements (i) and (ii):
	\begin{enumerate}
		\item Metric tensor field $\widetilde{G}_H^{(N)}(h) \xrightarrow{C^\infty} g_0|_\mathcal{B}$ uniformly across $\mathbb{K}$ (Theorem~\ref{thm:metric_stabilization});
		\item Connection symbols $\widetilde{\Gamma}_{ijk}^{(\alpha,N)}(h) \xrightarrow{C^\infty} \Gamma_{ijk}^{(0)} \equiv 0$ uniformly across $\mathbb{K}$ (Theorem~\ref{thm:connection_dissolution});
		\item Riemann curvature tensor $\widetilde{\mathcal{R}}_{ijmn}^{(\alpha,N)}(h) \xrightarrow{C^\infty} \mathcal{R}_{ijmn}^{(0)} \equiv 0$ uniformly across $\mathbb{K}$ (Theorem~\ref{thm:curvature_annihilation}).
	\end{enumerate}
	
	According to Cheeger-Gromov compactness and convergence theory for smooth Riemannian manifolds \cite{cheeger1986, gromov2007}, uniform $C^\infty$-bounds on metric components, connection symbols, and curvature components guarantee that the sequence of pulled-back manifolds $\psi_{N,H}^*(IG_N|_{H_\theta}) = (\mathbb{K}, \widetilde{G}_H^{(N)}, \widetilde{\nabla}^{(\alpha,N)})$ converges smoothly in the Cheeger-Gromov $C^\infty$-topology to the canonical flat tangent canvas $\mathcal{M}_\infty \equiv \left(T_{\theta_0}IG_1, g_0|_\mathcal{B}, \nabla^{(0)}\right)$.
	
	By Theorem~\ref{thm:theorem2_full} (Epistemological Equivalence Theorem), $\mathcal{M}_\infty$ is mathematically identical to Conventional Statistics ($\mathrm{CS}$). This completes the proof of Statement (iv) and establishes the Master Second Edge Theorem.
\end{proof}

\subsection{Epistemological Synthesis, Gauge Orbit Contraction, and Statistical Implications}
\label{subsec:epistemological_synthesis_master}

The completion of the Master Second Edge Theorem (Theorem~\ref{thm:master_second_edge}) marks a fundamental paradigm shift in mathematical statistics, differential geometry, and theoretical machine learning. Beyond providing a technical convergence proof, this section executes a deep epistemological synthesis, formalizing the topological contraction of vertical gauge orbits, the systematic elimination of non-identifiability pathologies in over-parameterized models, and the conceptual duality between the "{\it Curved Geometry of Finite Sample}" (finite-sample Riemannian geometry) and the "{\it Second Edge"} (asymptotic flat tangent collapse). In \cite{cheng_first_edge}, we have built the First Edge Theorem that as $N\to\infty$, IG lives at the edge or boundary of SMG geometry. Here the Second Edge Theorem proves that, as the sample size $N\to\infty$, CS lives at edge or boundary of IG. 

\subsubsection{Topological Gauge Orbit Contraction and Fiber Decoupling}
\label{subsubsec:gauge_orbit_contraction}

In Section~\ref{subsec:fiber_bundle_architecture}, we formalized the parameter domain as a smooth statistical fiber bundle $(\Theta, \mathcal{B}, \pi, \mathcal{F})$ (Definition~\ref{def:statistical_fiber_bundle}), where the $D$-dimensional ambient space $\Theta \subset \mathbb{R}^D$ projects onto a $d$-dimensional quotient base manifold $\mathcal{B} \equiv \Theta / \sim$ of observationally identifiable probability distributions ($d < D$). 

For each point $p \in \mathcal{B}$, the fiber $\mathcal{F}_p \equiv \pi^{-1}(p) \subset \Theta$ represents a $(D-d)$-dimensional vertical gauge orbit consisting of observationally equivalent parameters ($P_{\theta_1} = P_{\theta_2}$ $\mu$-a.e. for all $\theta_1, \theta_2 \in \mathcal{F}_p$).

We now prove that as $N \to \infty$, the vertical gauge orbits contract topologically under horizontal projection, decoupling unidentifiable parameter drift from active statistical inference.

\begin{lemma}[Topological Fiber Contraction and Gauge Decoupling]
	\label{lem:fiber_contraction}
	Let $(\Theta, \mathcal{B}, \pi, \mathcal{F})$ be the statistical fiber bundle equipped with the Fisher-compatible Ehresmann connection $T_\theta \Theta = H_\theta \oplus V_\theta$ (Definition~\ref{def:ehresmann_connection}). For any base distribution state $p_0 \in \mathcal{B}$ and any vertical fiber $\mathcal{F}_{p_0} = \pi^{-1}(p_0) \subset \Theta$:
	\begin{enumerate}[label=\textnormal{(\roman*)}]
		\item \textbf{Vanishing Vertical Metric Capacity:}
		The restriction of the joint $N$-sample Fisher Information Metric tensor $G^{(N)}$ to the vertical sub-bundle $V_\theta \equiv \mathrm{ker}(d\pi_\theta)$ vanishes identically for all $N \ge 1$:
		\begin{equation}
			G^{(N)}|_{V_\theta \times V_\theta} \equiv 0, \quad \forall \theta \in \mathcal{F}_{p_0}.
			\label{eq:vertical_metric_zero}
		\end{equation}
		
		\item \textbf{Topological Fiber Collapse onto Quotient Base Points:}
		Under the horizontal projection operator $P_\theta^H \equiv \mathrm{id}_{T_\theta \Theta} - \omega_\theta$, the vertical gauge orbit $\mathcal{F}_{p_0}$ contracts topologically to the single equivalence point $[p_0] \in \mathcal{B}$ in the metric topology induced by the restricted Fisher metric $g_0|_\mathcal{B}$:
		\begin{equation}
			P_\theta^H \left( \mathcal{F}_{p_0} \right) = \{\mathbf{0} \in H_\theta\}, \quad \forall \theta \in \mathcal{F}_{p_0}.
			\label{eq:fiber_point_collapse}
		\end{equation}
	\end{enumerate}
\end{lemma}

\begin{proof}
	\textbf{Proof of (i):} Let $v, w \in V_\theta = \mathrm{ker}(d\pi_\theta)$ be vertical tangent vectors at $\theta \in \mathcal{F}_{p_0}$. By Lemma~\ref{lem:vertical_score_vanishing}, directional derivatives of the single-observation log-likelihood function along vertical vectors vanish $\mu$-almost everywhere:
	\begin{equation}
		v(\log p(X; \theta)) = 0, \quad w(\log p(X; \theta)) = 0 \quad P_\theta\text{-a.e.}
		\label{eq:proof_fiber_score_zero}
	\end{equation}
	By $N$-sample Fisher metric additivity $G^{(N)}(\theta) = N g^{(1)}(\theta)$ (Lemma~\ref{lem:additivity_fisher_metric}):
	\begin{equation}
		G^{(N)}(v, w) = N \cdot \mathbb{E}_\theta \left[ \big(v(\log p(X; \theta))\big) \cdot \big(w(\log p(X; \theta))\big) \right] = N \cdot \mathbb{E}_\theta [0 \cdot 0] = 0.
		\label{eq:proof_vertical_metric_eval}
	\end{equation}
	This proves Equation~\eqref{eq:vertical_metric_zero}.
	
	\textbf{Proof of (ii):} By Definition~\ref{def:connection_1form_projection}, the Ehresmann connection 1-form $\omega_\theta : T_\theta \Theta \to V_\theta$ acts as the canonical projection onto $V_\theta$ along $H_\theta$. For any vertical vector $v \in V_\theta$, $\omega_\theta(v) = v$. Applying the horizontal projection operator $P_\theta^H \equiv \mathrm{id}_{T_\theta \Theta} - \omega_\theta$:
	\begin{equation}
		P_\theta^H(v) = v - \omega_\theta(v) = v - v = \mathbf{0} \in H_\theta.
		\label{eq:proof_horizontal_proj_zero}
	\end{equation}
	Integrating Equation~\eqref{eq:proof_horizontal_proj_zero} along smooth integral curves in the fiber $\mathcal{F}_{p_0}$ maps the entire $(D-d)$-dimensional submanifold $\mathcal{F}_{p_0}$ onto the origin $\mathbf{0} \in H_\theta \cong T_{p_0} \mathcal{B}$. Thus, vertical gauge variations generate zero Riemannian geodesic distance ($d_{G^{(N)}}(\theta_1, \theta_2) = 0$ for all $\theta_1, \theta_2 \in \mathcal{F}_{p_0}$), collapsing the gauge orbit $\mathcal{F}_{p_0}$ to a single point $[p_0] \in \mathcal{B}$.
\end{proof}

\subsubsection{Resolution of Non-Identifiability Pathologies in Deep Learning and High-Dimensional Models}
\label{subsubsec:deep_learning_pathologies}

In modern statistical learning theory, over-parameterized models (such as deep neural networks, singular mixture models, and over-parameterized latent factor models) operate in regimes where $D \gg d$. Historically, this over-parameterization induces three severe mathematical pathologies that cause standard asymptotic theory to break down (Lemma~\ref{lem:ambient_breakdown}).

Table~\ref{tab:pathology_resolution_master} summarizes how the Master Second Edge Theorem and the Fisher-compatible Ehresmann connection framework systematically resolve these pathologies.

\begin{table}[h!]
	\centering
	\small
	\begin{tabular}{|p{3.2cm}|p{5.8cm}|p{6.0cm}|}
		\hline
		\textbf{Pathology Category} & \textbf{Ambient Parameter Space Pathology ($\Theta \subset \mathbb{R}^D$)} & \textbf{Horizontal Ehresmann Resolution ($H_\theta \cong T\mathcal{B}$)} \\ \hline
		\textbf{1. Metric Singularity \& Matrix Inverse Failure} & $\det g^{(1)}(\theta) = 0$ with $\mathrm{rank}(g^{(1)}) = d < D$. The matrix inverse $(g^{(1)})^{-1}$ fails to exist, blowing up Christoffel symbols and Riemann curvature tensor contractions. & $g_0|_\mathcal{B} \equiv P_\theta^H g^{(1)}(\theta) P_\theta^H$ is strictly positive-definite ($\lambda_{\min}(g_0|_\mathcal{B}) \ge \kappa_0 > 0$, Theorem~\ref{thm:horizontal_strict_positivity}). The restricted inverse matrix $(g_0|_\mathcal{B})^{-1}$ is well-defined and non-singular on $H_\theta$. \\ \hline
		\textbf{2. Variance Explosion \& Cramér-Rao Collapse} & Cramér-Rao lower bound along vertical gauge directions diverges to infinity: $N \cdot \mathrm{Var}_{\theta_0}(v^T \widehat{\theta}_N) \ge v^T (g^{(1)})^+ v = +\infty$ for non-zero $v \in V_{\theta_0}$. & For any physical, gauge-invariant estimator $\widehat{g}(X^N)$, operational estimation variance projected onto $H_\theta$ is bounded strictly by the horizontal inverse: $\mathrm{Var}(P^H \widehat{\theta}_N) \le \frac{1}{N} (g_0|_\mathcal{B})^{-1} + o(N^{-1})$ (Corollary~\ref{cor:gauge_invariant_lan}). \\ \hline
		\textbf{3. Topological Degeneracy \& Flat Energy Valleys} & Geodesic distance between distinct parameter states $\theta_1 \neq \theta_2$ on the same gauge fiber vanishes ($d_g(\theta_1, \theta_2) = 0$), destroying Hausdorff metric topology. & The horizontal leaf space $\mathcal{L}_{\theta_0} \cong \mathcal{B}$ inherits a non-degenerate, Hausdorff Riemannian metric $g|_\mathcal{B}$, restoring well-posed Gromov-Hausdorff topology ($d_{\mathrm{GH}} \to 0$, Theorem~\ref{thm:gh_leaf_collapse}). \\ \hline
	\end{tabular}
	\caption{Systematic Elimination of Non-Identifiability Pathologies under Horizontal Ehresmann Splitting.}
	\label{tab:pathology_resolution_master}
\end{table}

\begin{corollary}[Optimal Asymptotic Efficiency of Gauge-Invariant Estimators]
	\label{cor:gauge_invariant_efficiency}
	Let $\widehat{\theta}_N(X^N) \in \Theta \subset \mathbb{R}^D$ be a parameter estimator in an over-parameterized model ($D \gg d$). Under High-Order Local Regularity (Assumption~\ref{asm:local_regularity}), if the projected estimator $\widehat{h}_N \equiv \sqrt{N} P_{\theta_0}^H (\widehat{\theta}_N - \theta_0) \in H_{\theta_0} \cong \mathbb{R}^d$ is locally asymptotically unbiased for the identifiable base distribution $p \in \mathcal{B}$, then its asymptotic covariance matrix is strictly lower-bounded by the inverse horizontal Fisher metric:
	\begin{equation}
		\lim_{N \to \infty} N \cdot \mathrm{Var}_{\theta_0}\left( P_{\theta_0}^H \widehat{\theta}_N \right) \ge \left( g_0|_\mathcal{B} \right)^{-1},
		\label{eq:horizontal_cramer_rao}
	\end{equation}
	where equality is achieved if and only if $P_{\theta_0}^H \widehat{\theta}_N$ is asymptotically linear in the horizontally projected score vector:
	\begin{equation}
		P_{\theta_0}^H \left( \widehat{\theta}_N - \theta_0 \right) = \frac{1}{N} \left( g_0|_\mathcal{B} \right)^{-1} \sum_{n=1}^N P_{\theta_0}^H \nabla_\theta \log p(X_n; \theta_0) + o_p\left(N^{-1/2}\right).
		\label{eq:asymptotic_score_linearity_horizontal}
	\end{equation}
\end{corollary}

\begin{proof}
	By Corollary~\ref{cor:gauge_invariant_lan}, localized statistical experiments pulled back along $P_{\theta_0}^H$ converge in Le Cam deficiency distance to the limit Gaussian shift experiment $\mathcal{E}_\infty(\mathcal{B}; \mathbb{K}) = \left(\mathbb{R}^d, \mathcal{B}^d, \left\{\mathcal{N}\left(h, (g_0|_\mathcal{B})^{-1}\right)\right\}\right)$. Applying Hajek's Local Asymptotic Minimax Theorem \cite{hajek1972, vandervaart1998} to the non-singular limit experiment $\mathcal{E}_\infty(\mathcal{B}; \mathbb{K})$ enforces $(g_0|_\mathcal{B})^{-1}$ as the strict lower bound for quadratic estimation risk on $H_{\theta_0}$. Asymptotic linearity (Equation~\eqref{eq:asymptotic_score_linearity_horizontal}) achieves this lower bound, completing the proof.
\end{proof}

\subsubsection{The "Second Edge": Cramér-Rao Boundary vs. Asymptotic Geometric Collapse}
\label{subsubsec:first_vs_second_edge}

We conclude this epistemological synthesis by formalizing the fundamental distinction between the {\it Curved Geometry of Finite Sample} and the {\it Second Edge} in statistical physics and information geometry.

\begin{definition}[The Dual Edge Framework]
	\label{def:dual_edge_framework}
	In the structured evolution of information manifolds $IG_N = (\Theta, G^{(N)}, \nabla^{(\alpha,N)})$, statistical inference is bounded by two distinct mathematical thresholds:
	\begin{enumerate}[label=\textnormal{(\roman*)}]
		\item \textbf{The Curved Geometry of Finite Sample (Finite-Sample Cramér-Rao Information Boundary):}
		Defined in finite-sample regimes ($N < \infty$), the Curved Geometry of Finite Sample represents the point-wise lower bound on estimation variance enforced by the Fisher Information Metric $g^{(1)}(\theta_0)$ via the Cramér-Rao inequality:
		\begin{equation}
			\mathrm{Var}_{\theta_0}(\widehat{\theta}) \ge \frac{1}{N} \left[ g^{(1)}(\theta_0) \right]^{-1}.
			\label{eq:first_edge_cramer_rao}
		\end{equation}
		At the Curved Geometry of Finite Sample, the manifold $IG_N$ remains intrinsically curved ($\mathcal{R}^{(\alpha,N)} \neq 0$), parallel transport is path-dependent ($\Gamma^{(\alpha,N)} \neq 0$), and first-order efficient estimators suffer second-order information loss $\Delta I_N = \gamma_e^2 + \frac{1}{2}\gamma_m^2 + \mathcal{O}(N^{-1})$ due to non-zero Efron statistical curvatures ($\gamma_e, \gamma_m > 0$) \cite{efron1975, amari1982}.
		
		\item \textbf{The Second Edge (Global Asymptotic Geometric Collapse):}
		Defined in the thermodynamic limit as $N \to \infty$, the Second Edge represents the structural phase transition where the dynamic sequence of curved, expanding manifolds $IG_N|_{H_\theta}$ collapses globally onto the canonical flat tangent canvas $\mathcal{M}_\infty \equiv \mathrm{CS}$.
		At the Second Edge:
		\begin{itemize}
			\item Riemannian curvature is completely annihilated ($\widetilde{\mathcal{R}}^{(\alpha,N)} \to 0$ at rate $\mathcal{O}_p(N^{-1})$);
			\item Connection friction dissolves ($\widetilde{\Gamma}^{(\alpha,N)} \to 0$ at rate $\mathcal{O}(N^{-1/2})$);
			\item Holonomy collapses to the identity ($\mathrm{Hol}^0 \to \{I_{d \times d}\}$);
			\item Metric capacity freezes ($\widetilde{G}_H^{(N)} \to g_0|_\mathcal{B}$);
			\item Second-order information loss vanishes identically ($\Delta I_N = 0$).
		\end{itemize}
	\end{enumerate}
\end{definition}

Table~\ref{tab:first_vs_second_edge_comparison} provides a complete architectural comparison highlighting the physical and mathematical transformations occurring between the Curved Geometry of Finite Sample and the Second Edge.

\begin{table}[h!]
	\centering
	\small
	\begin{tabular}{|p{3.5cm}|p{5.5cm}|p{5.8cm}|}
		\hline
		\textbf{Theoretical Feature} & \textbf{The Curved Geometry of Finite Sample ($N < \infty$, Finite Samples)} & \textbf{The Second Edge ($N \to \infty$, Thermodynamic Limit)} \\ \hline
		\textbf{Ontological Nature} & Point-wise variance lower bound on a fixed, curved manifold $IG_N$. & Global phase transition and structural manifold collapse $IG_N \to \mathcal{M}_\infty \equiv \mathrm{CS}$. \\ \hline
		\textbf{Manifold Geometry} & Intrinsically curved Riemannian-affine space ($\mathcal{R}^{(\alpha)} \neq 0$). & Flat Euclidean-Fisher tangent space $\mathcal{M}_\infty$ ($\mathcal{R}^{(0)} \equiv 0$). \\ \hline
		\textbf{Parallel Transport} & Path-dependent transport with non-trivial holonomy ($\mathrm{Hol}^0 \neq \{I\}$). & Path-independent transport with trivial holonomy ($\mathrm{Hol}^0 = \{I_{d \times d}\}$). \\ \hline
		\textbf{Test Trinity Behavior} & Score ($W_S$), LRT ($W_{LR}$), and Wald ($W_W$) tests diverge due to $\Gamma_{ijk}^{(\alpha)} \neq 0$. & Test trinity becomes unconditionally identical: $W_S = W_{LR} = W_W + o_p(1)$. \\ \hline
		\textbf{Information Loss} & Second-order loss exists: $\Delta I_N = \gamma_e^2 + \frac{1}{2}\gamma_m^2 + \mathcal{O}(N^{-1}) > 0$. & Zero second-order information loss: $\Delta I_N = 0$ due to spatial metric freezing ($\nabla g = 0$). \\ \hline
		\textbf{Decision Paradigm} & Complex finite-sample decision rules on curved manifolds. & Universal Gaussian shift experiment $\mathcal{E}_\infty = \left(\mathbb{R}^d, \mathcal{B}^d, \{\mathcal{N}(h, g_0^{-1})\}\right)$. \\ \hline
	\end{tabular}
	\caption{Architectural Comparison between the Curved Geometry of Finite Sample and the Second Edge.}
	\label{tab:first_vs_second_edge_comparison}
\end{table}

\begin{philosophical}[Conventional Statistics as an Invariant Attractor State]
	The Master Second Edge Theorem demonstrates that Conventional Statistics ($\mathrm{CS}$) is not an ad-hoc, heuristic linear approximation invented for computational convenience. Instead, $\mathrm{CS} \equiv \mathcal{M}_\infty$ is the unique, mathematically invariant, zero-curvature attractor state toward which all regular parametric statistical manifolds $IG_N$ dynamically collapse under $N \to \infty$. 
	
	By proving that structural geometric collapse ($IG_N \xrightarrow{\mathrm{geom}} \mathcal{M}_\infty$) and operational decision condensation ($IG_N \xrightarrow{\mathrm{oper}} \mathcal{M}_\infty$) are two mathematically equivalent projections of a single underlying phase transition (Theorem~\ref{thm:master_second_edge}), the Second Edge Theorem establishes a unified foundation bridging differential geometry, geometric analysis, and asymptotic decision theory.
\end{philosophical}

\newpage
\section{Modern Statistical Implications of the 2nd Edge Theorem}
\label{sec:section8_implications}

\subsection{Historical Paradigm Rift: Aesthetic Geometry vs. Operational Utility}
\label{subsec:ig_vs_cs_dilemma}

For over half a century, mathematical statistics has experienced a fundamental epistemological rift between Conventional Statistics ($\text{CS}$) \cite{fisher1922, lecam1986, vandervaart1998} and Information Geometry ($\text{IG}$) \cite{amari1985, amari2000, rao1945}. While Information Geometry is widely celebrated as a triumph of mathematical beauty—synthesizing Riemannian geometry, dual affine connections, and tensor calculus over statistical manifolds—the vast majority of practicing statisticians treat $\text{IG}$ as an aesthetic luxury rather than an essential operational tool \cite{kass1997}.

This widespread skepticism among mainstream statisticians stems from three deep-seated structural reasons:

\begin{enumerate}[leftmargin=*, label={\bf (\arabic*)}]
	\item \textbf{High Formalism with Asymptotic Redundancy:} 
	In classical decision theory \cite{lecam1960, lecam1986}, large-sample inference is governed by Local Asymptotic Normality (LAN). Under the scale transformation $h = \sqrt{N}(\theta - \theta_0)$, the localized log-likelihood ratio process $\Lambda_N(h)$ converges weakly to a linear-quadratic Gaussian shift experiment:
	\begin{equation}
		\Lambda_N(h) = h^\top \Delta_N(\theta_0) - \frac{1}{2} h^\top g^{(1)}(\theta_0) h + o_p(1), \quad \Delta_N(\theta_0) \xrightarrow{d} \mathcal{N}\left(0, g^{(1)}(\theta_0)\right).
	\end{equation}
	Because classical asymptotic results (such as first-order efficiency of the Maximum Likelihood Estimator, asymptotic optimality of the likelihood ratio test, and the Cramér-Rao bound) depend solely on the point evaluation of the frozen Fisher Information Metric $g_0 \equiv g^{(1)}(\theta_0)$, statisticians view finite-sample manifold curvature $\mathcal{R}^{(\alpha)}$ and connection Christoffel symbols $\Gamma_{ijk}^{(\alpha)}$ as higher-order nuisances of order $\mathcal{O}(N^{-1})$ and $\mathcal{O}(N^{-1/2})$ \cite{efron1975, amari1982}. Consequently, the sophisticated differential-geometric apparatus of $\text{IG}$ appears to offer no practical gain over standard flat Euclidean asymptotic calculations.
	
	\item \textbf{Absence of Actionable Algorithms for Complex Data:} 
	Traditional Information Geometry focuses primarily on well-behaved, finite-dimensional, strictly identifiable parametric models (such as regular exponential families and mixture models) \cite{amari2000, chentsov1982}. However, when applied to modern complex problems—such as non-parametric estimation, high-dimensional inference, or real-time online learning—classical $\text{IG}$ often yields non-linear geodesic differential equations that are computationally intractable, offering few actionable algorithms for real-world statistical computing.
	
	\item \textbf{Breakdown in Over-Parameterized and Singular Regimes:} 
	Standard $\text{IG}$ relies critically on the strict positive-definiteness of the single-sample Fisher Information Metric $g^{(1)}(\theta)$. In modern machine learning, deep neural networks, and singular latent variable models, the parameter dimension $D$ vastly exceeds the sample size $N$ ($D \gg N$), inducing massive non-identifiability gauge orbits where $\det g^{(1)}(\theta) = 0$ \cite{watanabe2009, stoica2001}. When $g^{(1)}$ is rank-deficient, classical Riemannian metrics and inverse contractions break down entirely, reinforcing the perception that $\text{IG}$ is an idealized theoretical artifact restricted to regular non-singular models.
\end{enumerate}

\subsection{Why the 2nd Edge Theorem Transcends Conventional Statistics}
\label{subsec:transcending_cs}

The Second Edge Theorem radically alters this classical paradigm. It demonstrates that Conventional Statistics ($\text{CS}$) is not an independent or self-contained baseline of statistical reality, but rather the flat, zero-curvature limit canvas $\mathcal{M}_\infty \equiv (T_{\theta_0}IG_1, g_0, \nabla^{(0)})$ onto which the dynamic, curved information manifold sequence $IG_N = (\Theta, G^{(N)}, \nabla^{(\alpha,N)})$ globally collapses as $N \to \infty$.

Crucially, \textbf{conventional statistical models cannot prove the Second Edge Theorem}. Conventional statistics operates \emph{inside} the static, flat limit experiment $\mathcal{E}_\infty$, taking LAN and asymptotic Euclidean flatness for granted as local approximations \cite{lecam1986, vandervaart1998}. Conventional statistics lacks the analytical machinery required to establish the 2nd Edge Theorem for three fundamental reasons:

\begin{enumerate}[leftmargin=*, label={\bf (\roman*)}]
	\item \textbf{Lack of Manifold Sequence Mechanics:} 
	$\text{CS}$ treats $N \to \infty$ as a sequence of scalar random variables or probability measures concentrating around a point $\theta_0$. It possesses no framework to model $IG_N$ as a dynamic sequence of Riemannian-affine manifolds undergoing metric capacity scaling $G^{(N)} = N g^{(1)}$, connection dissolution $\widetilde{\Gamma}^{(\alpha,N)} \sim \mathcal{O}(N^{-1/2})$, and quadratic curvature decay $\widetilde{\mathcal{R}}^{(\alpha,N)} \sim \mathcal{O}_p(N^{-1})$.
	
	\item \textbf{Absence of Metric-Topological and Holonomy Convergence Machinery:} 
	Proving the global collapse $IG_N \to \mathcal{M}_\infty$ requires proving:
\begin{itemize}
\item \text{Gromov-Hausdorff Metric Leaf Collapse:} \begin{equation}
\lim_{N\to\infty} d_{GH}\left((\mathbb{K}, d_{\widetilde{G}_H^{(N)}}), (\mathbb{K}, d_{g_0|_{\mathcal{B}}})\right) = 0, \label{eq:gh_collapse_sec8} 
\end{equation} 
\item \text{Accelerated Holonomy Annihilation:}
\begin{equation}
\lim_{N\to\infty} \text{diam}_{op}\left(\text{Hol}_{\theta_0}^0(\widetilde{\nabla}^{(\alpha,N)})\right) = 0 \iff \text{Hol}_{\theta_0}^0 \to \{I_{d \times d}\}. \label{eq:holonomy_collapse_sec8}
\end{equation}
\end{itemize}
	Conventional statistics possesses no differential-topological tools (such as Gromov-Hausdorff metric spaces \cite{gromov2007, burago2001}, Cheeger-Gromov $C^\infty$-manifold convergence \cite{cheeger1986}, or the Ambrose-Singer Holonomy Theorem \cite{ambrose1953}) to evaluate metric space distortion or the trivialization of parallel transport holonomy groups.
	
	\item \textbf{Inability to Resolve Singular Fiber Degeneracies:} 
	When $\det g^{(1)}(\theta_0) = 0$, conventional statistics suffers from catastrophic breakdown: parameter variances along kernel directions $v \in \ker(g^{(1)})$ diverge to infinity ($N \text{Var}(v^\top \hat{\theta}_N) \to +\infty$), rendering the classical Cramér-Rao bound useless \cite{stoica2001}. $\text{CS}$ cannot resolve this without the Fisher-Compatible Ehresmann Connection Instrument, which constructs a smooth Fiber Bundle $(\Theta, \mathcal{B}, \pi, \mathcal{F})$ and projects localized operations onto the non-singular horizontal distribution $H_\theta \equiv V_\theta^{\perp_g}$ via $P_\theta^H = \text{id} - \omega_\theta$.
\end{enumerate}

\subsection{The Grand Hierarchy: $CS \subset \partial IG \subset \partial SMG \subset SMG$}
\label{subsec:dual_edge_hierarchy}

By unifying the \textbf{First Edge Theorem} ($IG \subset \partial SMG$) and the \textbf{Second Edge Theorem} ($CS \subset \partial IG$), we establish the grand structural architecture of statistical science:

\begin{enumerate}[leftmargin=*, label={\bf 1.}]
	\item \textbf{The First Edge Theorem ($IG \subset \partial SMG$):} 
	Proves that finite-sample parametric Information Geometry ($IG$), which restricts attention to closed parametric families $\mathcal{P} = \{P_\theta : \theta \in \Theta\}$, constitutes a finite-dimensional boundary slice ($\partial SMG$) of the ambient, infinite-dimensional, non-parametric, open-system Statistical Manifold Geometry ($SMG$). $SMG$ encompasses non-parametric probability measures, open environment-system interactions, and active dynamic connection transport.
	
	\item \textbf{The Second Edge Theorem ($CS \subset \partial IG$):} 
	Proves that as sample size $N \to \infty$, the sequence of curved joint information manifolds $IG_N$ undergoes global geometric phase transition and collapses onto the flat, canonical tangent space $\mathcal{M}_\infty \equiv CS$. Conventional Statistics ($CS$) is the zero-curvature asymptotic boundary slice ($\partial IG$) of $IG$.
\end{enumerate}

Combining both results reveals the \textbf{Nested Dual-Edge Hierarchy}:
\begin{equation}
	\mathcal{M}_\infty \equiv \text{CS} \quad \subset \quad \partial IG_N \quad \subset \quad IG_N \quad \subset \quad \partial SMG \quad \subset \quad SMG.
	\label{eq:grand_hierarchy}
\end{equation}

\begin{philosophical}[Decades of Statistics as Boundary Phenomena]
	\label{phil:boundary_phenomena}
	The Nested Dual-Edge Hierarchy \eqref{eq:grand_hierarchy} delivers a profound epistemological realization: \textbf{both Conventional Statistics ($\text{CS}$) and Information Geometry ($\text{IG}$) live exclusively at the boundary edges of the vast Statistically  Meaningful Geometry ($SMG$)}. 
	
	Conventional Statistics ($\text{CS}$) is two steps removed from ambient statistical reality. $\text{CS}$ operates strictly at the asymptotic boundary where sample-size thermodynamic cooling $\beta_N = 1/\sqrt{N}$ has frozen the metric tensor and completely quenched manifold curvature ($\mathcal{R}^{(0)} \equiv 0$). Information Geometry ($\text{IG}$) relaxes the flatness restriction by incorporating finite-sample curvature ($\mathcal{R}^{(\alpha)} \neq 0$), but remains trapped on parametric boundary slices ($\partial SMG$). The true, unrestricted space of data, complex systems, and open environments is governed by $SMG$.
\end{philosophical}

\subsection{Toward a New Statistical Paradigm: Methodology for Over-Parameterization and Beyond}
\label{subsec:new_paradigm_overparam}

The Dual-Edge Hierarchy \eqref{eq:grand_hierarchy} is not merely a theoretical synthesis; it dictates a revolutionary new paradigm and actionable methodology for modern statistical challenges that completely upturns classical $\text{CS}$ and traditional $\text{IG}$.

\subsubsection{Addressing Over-Parameterization and High-Dimensionality ($D \gg N$)}

In modern artificial intelligence, deep learning, large language models, and high-dimensional genomic inference, models operate in regimes where the parameter count $D$ vastly exceeds the sample size $N$ ($D \gg N$). In these over-parameterized regimes:
\begin{itemize}
	\item Classical $\text{CS}$ fails because asymptotic limits ($N \to \infty$ with fixed $D$) are invalid, and sample covariance matrices are strictly singular.
	\item Traditional $\text{IG}$ fails because the ambient Fisher metric has a non-trivial kernel $\ker(g^{(1)}) \neq \{0\}$, causing geometric curvature expressions to blow up.
\end{itemize}

Under our new $SMG$/Ehresmann framework, we resolve over-parameterization by exploiting the geometric fiber bundle structure $(\Theta, \mathcal{B}, \pi, \mathcal{F})$, projecting statistical operations onto the non-singular horizontal carriage $H_\theta \equiv V_\theta^{\perp_g}$ and providing three actionable methodologies:

\begin{methodology}[Gauge-Invariant Horizontal Optimization (Ehresmann Stochastic Gradient Descent)]
	\label{meth:esgd}
	In over-parameterized deep learning architectures, parameter updates along vertical gauge fibers $V_\theta = \ker(d\pi_\theta)$ generate zero reduction in empirical loss while causing representation drift and numerical instability. By constructing the horizontal projection operator $P_\theta^H = \text{id} - \omega_\theta$, we formulate \textbf{Ehresmann Stochastic Gradient Descent (E-SGD)}:
	\begin{equation}
		\theta_{t+1} = \theta_t - \eta_t \, P_{\theta_t}^H \left[ \left(g_{\theta_t}^{(1)}|_{H_{\theta_t}}\right)^{-1} \nabla_\theta \mathcal{L}_N(X^N; \theta_t) \right],
		\label{eq:esgd_update}
	\end{equation}
	where $g_{\theta_t}^{(1)}|_{H_{\theta_t}}$ is the strictly positive-definite restricted horizontal Fisher metric. E-SGD restricts optimization strictly to the horizontal leaf $\mathcal{L}_{\theta_0} \cong \mathcal{B}$, eliminating unidentifiable gauge drift and accelerating convergence.
\end{methodology}

\begin{methodology}[Curvature-Aware Geometric Regularization]
	\label{meth:geom_regularization}
	To prevent over-fitting in finite-sample over-parameterized models, we introduce explicit geometric penalization based on the pulled-back connection symbols $\widetilde{\Gamma}_{ijk}^{(\alpha,N)}$ and Riemann curvature tensor $\widetilde{\mathcal{R}}_{ijmn}^{(\alpha,N)}$. The geometrically regularized objective function is:
	\begin{equation}
		\mathcal{J}(\theta) = \mathcal{L}_N(X^N; \theta) + \lambda_1 \left\| \widetilde{\Gamma}^{(\alpha,N)}(\theta) \right\|_F^2 + \lambda_2 \left\| \widetilde{\mathcal{R}}^{(\alpha,N)}(\theta) \right\|_F^2,
		\label{eq:geom_regularization_loss}
	\end{equation}
	where $\|\cdot\|_F$ denotes the tensor Frobenius norm. By actively penalizing connection friction and Riemannian curvature, \eqref{eq:geom_regularization_loss} forces the finite-sample manifold $IG_N$ to accelerate its phase transition toward the flat, optimal generalization canvas $\mathcal{M}_\infty \equiv CS$, providing a principled geometric defense against over-fitting.
\end{methodology}

\begin{methodology}[Open-System Dynamic Inference in SMG]
	\label{meth:smg_open_systems}
	Real-world data generation frequently violates the i.i.d. closed-system assumption of $\text{CS}$ due to non-stationary environments, distribution shifts, and active system-environment feedback loops. $SMG$ models data dynamics as an open thermodynamic system governed by the Triadic Variable Spectrum $\mathcal{T}_N = (\mathcal{E}_N, \mathcal{S}_N, \mathcal{F}_N)$. By tracking the environmental measure concentration $\beta_N$ against system metric capacity expansion $\alpha_N$, $SMG$ provides robust, out-of-distribution dynamic inference mechanisms that adaptively update connection parallel transport across shifting environmental manifolds.
\end{methodology}

\section{Future Research Directions: Generative AI Applications and Beyond}
\label{sec:future_directions}

The Master Second Edge Theorem (Theorem \ref{thm:master_second_edge}) establishes that the sequence of curved $N$-sample joint information manifolds $IG_N\vert{}_{H_\theta} = (\Theta\vert{}_{H_\theta}, G^{(N)}\vert{}_{H_\theta}, \nabla^{(\alpha,N)}\vert{}_{H_\theta})$ undergoes an asymptotic phase transition and global geometric collapse onto the flat canonical tangent canvas $\mathcal{M}_\infty \equiv (T_{\theta_0}IG_1, g_0\vert{}_{\mathcal{B}}, \nabla^{(0)}) \equiv \text{CS}$ as $N \to \infty$ \cite{amari2000, lecam1986}.

By formalizing the systematic decay of Amari connection friction $\widetilde{\Gamma}^{(\alpha,N)} = \mathcal{O}(N^{-1/2})$, accelerated quadratic Riemann curvature annihilation $\widetilde{\mathcal{R}}^{(\alpha,N)} = \mathcal{O}_p(N^{-1})$, and restricted holonomy group collapse $\mathrm{Hol}_{\theta_0}^0(\widetilde{\nabla}^{(\alpha,N)}) \to \{I_{d \times d}\}$ \cite{ambrose1953, cheeger1986}, this mathematical framework bridges finite-sample differential geometry and asymptotic decision theory.

In this section, we extend the Second Edge Theorem to modern frontiers, focusing on Generative Artificial Intelligence (Diffusion Models, Flow Matching, and Transformer Architectures), Singular Learning Theory, and Non-Equilibrium Open-System Information Geometry.

\subsection{Information-Geometric Phase Transitions in Score-Based Diffusion and Flow Matching}
\label{subsec:future_diffusion_flow}

Score-based generative diffusion models \cite{song2021} and continuous-time flow matching frameworks \cite{lipman2023} generate high-dimensional data distributions by reversing a continuous-time stochastic degradation process. We establish the information-geometric foundation governing the reverse-time probability trajectories over the statistical fiber bundle $(\Theta, \mathcal{B}, \pi, \mathcal{F})$ \cite{lee2013}.

\begin{definition}[Continuous Diffusion Stochastic Differential Equation on Information Manifolds]
	\label{def:diffusion_sde}
	Let $X_t \in \mathbb{R}^D$ ($t \in [0, T]$) be a continuous-time stochastic process governed by the forward Ito stochastic differential equation (SDE):
	\begin{equation}
		dX_t = f(X_t, t) , dt + g(t) , dW_t, \quad X_0 \sim p_0(x),
		\label{eq:forward_sde}
	\end{equation}
	where $f(\cdot, t): \mathbb{R}^D \to \mathbb{R}^D$ is a smooth drift vector field, $g(t) \in \mathbb{R}$ is a scalar diffusion coefficient, and $W_t$ is a $D$-dimensional standard Brownian motion. Let $p_t(x)$ denote the marginal probability density of $X_t$ under Lebesgue measure $d\mu(x)$.
\end{definition}

By Anderson's reverse-time theorem \cite{anderson1982}, the reverse process trajectory satisfies the reverse time-dependent SDE:
\begin{equation}
	dX_t = \left[ f(X_t, t) - g(t)^2 \nabla_x \log p_t(X_t) \right] dt + g(t) , d\bar{W}_t,
	\label{eq:reverse_sde}
\end{equation}
where $\bar{W}_t$ is a backward standard Brownian motion and $\nabla_x \log p_t(x)$ is the score 1-form field \cite{amari1985}.

\begin{theorem}[Duality Between Diffusion Time $t$ and Sample Scaling $N$]
	\label{thm:diffusion_dilation_duality}
	Let $IG_N(t) \equiv (\Theta, G^{(N)}(t), \nabla^{(\alpha,N)}(t))$ denote the dynamic information manifold associated with the time-marginal family $\mathcal{P}_t = \{ p_t(x;\theta) : \theta \in \Theta \}$. Under the variance-preserving (VP) diffusion process where $f(x,t) = -\frac{1}{2}\beta(t)x$ and $g(t) = \sqrt{\beta(t)}$, as diffusion time $t \to T \to \infty$, the score field $\nabla_x \log p_t(x;\theta)$ degrades to a standard Gaussian distribution $\mathcal{N}(0, I_{D \times D})$.
	
	The forward diffusion time evolution $t \to T$ induces an exact geometric phase transition mathematically dual to the sample size limit $N \to \infty$:
	\begin{enumerate}[label=\textnormal{(\roman*)}]
		\item \textbf{Metric Freezing Along Diffusion Time:} The time-dependent single-sample Fisher metric tensor $g_{ij}^{(1)}(\theta; t) \equiv \mathbb{E}_{p_t} \left[ \frac{\partial \log p_t(X;\theta)}{\partial \theta^i} \frac{\partial \log p_t(X;\theta)}{\partial \theta^j} \right]$ satisfies exponential metric decay:
		\begin{equation}
			\left| g^{(1)}(\theta; t) - e^{-\int_0^t \beta(s)ds} g^{(1)}(\theta; 0) \right|*\infty = \mathcal{O}\left(e^{-\int_0^t \beta(s)ds}\right) \longrightarrow 0 \quad \text{as } t \to \infty.
		\end{equation}
		\item \textbf{Curvature Annihilation Under Reverse-Time Flow:} Under reverse-time score-matching estimation with empirical sample size $N$, the pulled-back Amari-Riemann curvature tensor $\widetilde{\mathcal{R}}*{ijmn}^{(\alpha,N)}(t)$ on compact horizontal coordinate cages $\mathbb{K} \subset H_{\theta_0}$ satisfies the joint space-time collapse law:
		\begin{equation}
			\sup_{h \in \mathbb{K}} \left| \widetilde{\mathcal{R}}*{ijmn}^{(\alpha,N)}(h; t) \right|*{\mathrm{op}} = \mathcal{O}_p\left( \frac{e^{-\int_0^t \beta(s)ds}}{N} \right) \xrightarrow{N \to \infty, , t \to \infty} 0.
		\end{equation}
	\end{enumerate}
\end{theorem}

\begin{proof}
	We execute a four-step mathematical deduction.
	
	\textbf{Step 1: Marginal Density Evolution under Gaussian Kernels.} Under the VP-SDE, the transition kernel $p_{t\vert{}0}(x_t \vert{} x_0)$ is exact Gaussian:
	\begin{equation}
		p_{t|0}(x_t | x_0) = \mathcal{N}\left( x_t; , \alpha_t x_0, , \sigma_t^2 I_{D \times D} \right),
	\end{equation}
	where $\alpha_t \equiv \exp\left( -\frac{1}{2}\int_0^t \beta(s)ds \right)$ and $\sigma_t^2 \equiv 1 - \alpha_t^2$. The marginal density is the convolution $p_t(x;\theta) = \int p_0(x_0;\theta) \mathcal{N}\left(x; \alpha_t x_0, \sigma_t^2 I\right) dx_0$.
	
	\textbf{Step 2: Score Function Expansion.} Differentiating the marginal log-density with respect to parameter $\theta^i$:
	\begin{equation}
		\frac{\partial \log p_t(x;\theta)}{\partial \theta^i} = \frac{1}{p_t(x;\theta)} \int \frac{\partial p_0(x_0;\theta)}{\partial \theta^i} \mathcal{N}\left(x; \alpha_t x_0, \sigma_t^2 I\right) dx_0.
	\end{equation}
	As $t \to \infty$, $\alpha_t \to 0$ and $\sigma_t^2 \to 1$. A Taylor expansion of the kernel around $x_0 = 0$ yields:
	\begin{equation}
		\mathcal{N}\left(x; \alpha_t x_0, \sigma_t^2 I\right) = \mathcal{N}(x; 0, I) \left( 1 + \alpha_t x^\top x_0 + \mathcal{O}(\alpha_t^2) \right).
	\end{equation}
	Substituting this expansion into the score derivative and integrating over $x_0$:
	\begin{equation}
		\int \frac{\partial p_0(x_0;\theta)}{\partial \theta^i} dx_0 = \frac{\partial}{\partial \theta^i}(1) = 0,
	\end{equation}
	the leading non-vanishing term scales directly with $\alpha_t$:
	\begin{equation}
		\frac{\partial \log p_t(x;\theta)}{\partial \theta^i} = \alpha_t \cdot \mathbb{E}_{X_0 \sim p_0(\cdot;\theta)} \left[ \frac{\partial \log p_0(X_0;\theta)}{\partial \theta^i} x^\top X_0 \right] + \mathcal{O}(\alpha_t^2).
	\end{equation}
	
	\textbf{Step 3: Fisher Metric Exponential Decay.} Evaluating the Fisher metric quadratic form $g_{ij}^{(1)}(\theta; t) = \mathbb{E}_{X_t \sim p_t} \left[ \frac{\partial \log p_t(X_t;\theta)}{\partial \theta^i} \frac{\partial \log p_t(X_t;\theta)}{\partial \theta^j} \right]$:
	\begin{equation}
		g_{ij}^{(1)}(\theta; t) = \alpha_t^2 \cdot g_{ij}^{(1)}(\theta; 0) + \mathcal{O}(\alpha_t^3) = e^{-\int_0^t \beta(s)ds} g_{ij}^{(1)}(\theta; 0) + \mathcal{O}\left(e^{-\frac{3}{2}\int_0^t \beta(s)ds}\right),
	\end{equation}
	proving statement (i).
	
	\textbf{Step 4: Joint Curvature Collapse.} Applying the Universal Tensor Valence Scaling Law (Theorem \ref{thm:universal_scaling_law}) for $r=4$ covariant Riemann curvature tensor fields on $IG_N(t)$, under sample size scaling $N$ and time parameter $t$:
	\begin{equation}
		\widetilde{\mathcal{R}}*{ijmn}^{(\alpha,N)}(h; t) = \frac{1}{N} \mathcal{R}*{ijmn}^{(\alpha,1)}\left(\theta_0 + \frac{h}{\sqrt{N}}; t\right).
	\end{equation}
	Since $\mathcal{R}_{ijmn}^{(\alpha,1)}(\theta; t) = \mathcal{O}(\alpha_t^2) = \mathcal{O}\left(e^{-\int_0^t \beta(s)ds}\right)$, we obtain:
	\begin{equation}
		\sup_{h \in \mathbb{K}} \left| \widetilde{\mathcal{R}}*{ijmn}^{(\alpha,N)}(h; t) \right|*{\mathrm{op}} \le \frac{M_{\mathcal{R}}}{N} e^{-\int_0^t \beta(s)ds} = \mathcal{O}_p\left( \frac{e^{-\int_0^t \beta(s)ds}}{N} \right) \xrightarrow{N \to \infty, , t \to \infty} 0.
	\end{equation}
	This completes the proof.
\end{proof}

\subsection{Attention Manifolds, Transformer Dynamics, and Asymptotic Holonomy Collapse}
\label{subsec:future_transformer}

Transformer architectures \cite{vaswani2017} compute context-dependent token representations via soft-attention probability distributions defined over sequence length $L$. We formalize the sequence length $L$ as an information-geometric dimension and prove the collapse of parallel transport holonomy across deep attention layers.

\begin{definition}[Self-Attention Information Manifold]
	Let $X \in \mathbb{R}^{L \times d_{model}}$ denote a sequence matrix of $L$ tokens. For a multi-head self-attention layer with Query, Key, and Value projection matrices $W_Q, W_K, W_V \in \mathbb{R}^{d_{model} \times d_k}$, the attention weight matrix $A(X) \in \mathbb{R}^{L \times L}$ is defined row-wise via the softmax operator on the simplex $\Delta^{L-1}$:
	\begin{equation}
		A_{ij}(X) = \mathrm{softmax}\left( \frac{(X W_Q)*i (X W_K)*j^\top}{\sqrt{d_k}} \right) = \frac{\exp\left( \frac{\langle q_i, k_j \rangle}{\sqrt{d_k}} \right)}{\sum*{m=1}^L \exp\left( \frac{\langle q_i, k_m \rangle}{\sqrt{d_k}} \right)}.
		\label{eq:softmax_attention_def}
	\end{equation}
	The row vectors $A_i(X) = (A*{i1}, \dots, A_{iL}) \in \Delta^{L-1}$ define a family of conditional categorical probability measures over the token index space $\mathcal{L} \equiv \{1, \dots, L\}$.
\end{definition}

\begin{theorem}[Attention-Induced Holonomy Collapse Across Deep Layers]
	\label{thm:transformer_holonomy_collapse}
	Let $\mathcal{M}_{Attn}^{(d_k)} \equiv (\Theta_{Attn}, g^{(Attn)}, \nabla^{(\alpha)})$ denote the information manifold of self-attention parameter matrices $\Theta_{Attn} = (W_Q, W_K)$. As layer depth $D_{layer} \to \infty$ or sequence length $L \to \infty$, the softmax attention metric tensor $g^{(Attn)}$ concentrates onto a low-rank sub-manifold, forcing the restricted holonomy group $\mathrm{Hol}^0(\nabla^{(\alpha)})$ along attention layers to collapse to the trivial identity group.
	
	Specifically, under the Fisher-compatible Ehresmann connection splitting $T \Theta_{Attn} = H \oplus V$, the horizontal holonomy operator $\mathcal{P}_{\gamma}^{(\ell)}$ across layer $\ell \in \{1, \dots, D_{layer}\}$ satisfies:
	\begin{equation}
		\lim_{D_{layer} \to \infty} \max_{1 \le \ell \le D_{layer}} \mathrm{diam}*{\mathrm{op}}\left( \mathrm{Hol}*{\theta_0}^0\left( \nabla^{(\alpha, \ell)} \right) \right) = 0 \iff \mathrm{Hol}*{\theta_0}^0\left( \nabla^{(\alpha, \ell)} \right) \xrightarrow{D*{layer} \to \infty} { I_{d_k \times d_k} }.
	\end{equation}
\end{theorem}

\begin{proof}
	We execute a three-step mathematical deduction.
	
	\textbf{Step 1: Softmax Metric Concentration.} The Fisher Information Metric tensor components $g_{ab}^{(Attn)}$ for token $i$ with parameter vector $\theta = \mathrm{vec}(W_Q W_K^\top)$ are derived from the categorical log-likelihood derivatives of the softmax distribution $A_i \in \Delta^{L-1}$:
	\begin{equation}
		g_{ab}^{(Attn)}(\theta) = \sum_{j=1}^L A_{ij}(\theta) \left[ \frac{\partial \log A_{ij}(\theta)}{\partial \theta^a} \right] \left[ \frac{\partial \log A_{ij}(\theta)}{\partial \theta^b} \right] = \sum_{j=1}^L A_{ij}(\theta) \frac{\partial \eta_j}{\partial \theta^a} \frac{\partial \eta_j}{\partial \theta^b} - \left( \sum_{j=1}^L A_{ij}(\theta) \frac{\partial \eta_j}{\partial \theta^a} \right) \left( \sum_{m=1}^L A_{im}(\theta) \frac{\partial \eta_m}{\partial \theta^b} \right),
	\end{equation}
	where $\eta_j \equiv \frac{\langle q_i, k_j \rangle}{\sqrt{d_k}}$. In matrix notation, this is the covariance matrix of score derivatives under categorical measure $A_i$:
	\begin{equation}
		g^{(Attn)}(\theta) = J_\eta^\top \left[ \mathrm{Diag}(A_i) - A_i A_i^\top \right] J_\eta,
	\end{equation}
	where $J_\eta \in \mathbb{R}^{L \times d_k}$ is the Jacobian matrix of linear key projections.
	
	\textbf{Step 2: Layer-wise Softmax Contraction.} In deep Transformer architectures ($D_{layer} \gg 1$), self-attention maps exhibit value-rank contraction and entropy reduction (oversmoothing / token phase transition \cite{watanabe2009}). As layer depth increases, $A_{ij} \to \delta_{j, j^*(i)}$ (one-hot extreme point of the simplex $\Delta^{L-1}$).
	
	Evaluating the covariance matrix $\mathrm{Diag}(A_i) - A_i A_i^\top$ at a deterministic vertex $e_k = (0,\dots,1,\dots,0)^\top$:
	\begin{equation}
		\lim_{A_i \to e_k} \left[ \mathrm{Diag}(e_k) - e_k e_k^\top \right] = 0_{L \times L}.
	\end{equation}
	Thus, the single-layer metric tensor component vanishes:
	\begin{equation}
		\left| g^{(Attn; , \ell)}(\theta) \right|_{\mathrm{op}} = \mathcal{O}\left( e^{-\gamma \ell} \right) \longrightarrow 0 \quad \text{as } \ell \to \infty.
	\end{equation}
	
	\textbf{Step 3: Holonomy Collapse via Ambrose-Singer.} By Lemma  \ref{lem:ambrose_singer_bound} and the Ambrose-Singer Holonomy Theorem \cite{ambrose1953}, the Lie algebra generators $A \in \mathfrak{hol}_0(\nabla^{(\alpha, \ell)})$ are bounded by the Riemann curvature tensor:
	\begin{equation}
		| A |*{\mathrm{op}} \le C*\mathbb{K}^{(\mathcal{R})} \left| \mathcal{R}^{(\alpha, \ell)}(\theta) \right|*{\mathrm{op}}.
	\end{equation}
	Since the Riemann curvature tensor $\mathcal{R}^{(\alpha, \ell)}$ is built from quadratic products and partial derivatives of $g^{(Attn; \, \ell)}$, metric concentration forces exponential curvature decay:
	\begin{equation}
		\left| \mathcal{R}^{(\alpha, \ell)}(\theta) \right|*{\mathrm{op}} = \mathcal{O}\left( e^{-2\gamma \ell} \right).
	\end{equation}
	Applying the surface-ordered exponential expansion along layer loops $\gamma$:
	\begin{equation}
		\mathrm{diam}*{\mathrm{op}}\left( \mathrm{Hol}*{\theta_0}^0\left( \nabla^{(\alpha, \ell)} \right) \right) \le \mathrm{Area}(\Sigma) \cdot \mathcal{O}\left( e^{-2\gamma \ell} \right) \xrightarrow{\ell \to \infty} 0.
	\end{equation}
	This establishes the complete collapse of parallel transport holonomy across deep Transformer attention layers.
\end{proof}


\newpage

\end{document}